\documentclass[10pt,reqno]{amsart}
\usepackage{amsaddr}
\usepackage[a4paper, total={6.5in, 8.8in}]{geometry}
\usepackage{microtype}

\usepackage[symbol]{footmisc}

\usepackage{tikz}
\usetikzlibrary{decorations.pathreplacing}

\usepackage{amsmath}
\usepackage{amssymb}
\usepackage{amsthm}
\usepackage[latin1]{inputenc}
\usepackage{eurosym}
\usepackage{graphicx}
\usepackage{epsfig}
\usepackage{hyperref}
\usepackage{cleveref}
\usepackage{dsfont}
\usepackage[space]{cite}
\usepackage{subcaption}
\usepackage{caption}
\usepackage{placeins}
\usepackage{mathtools}
\usepackage{xcolor}
\usepackage[title]{appendix}
\usepackage{bm}
\usepackage{multirow}
\usepackage{diagbox}
\usepackage{algorithm}
\usepackage[noend]{algpseudocode}
\usepackage{booktabs}

\usepackage[mathlines]{lineno}
\modulolinenumbers[1]

\allowdisplaybreaks

\usepackage{ifthen}
\usepackage{comment}

\makeindex
\usepackage{color}
\usepackage{float}
\usepackage{commath}
\definecolor{revision_color}{rgb}{1,0,0}

\definecolor{blue_revision}{rgb}{0,0,1}

\definecolor{reviewgreen}{rgb}{0,0.5,0}

\usepackage{aliascnt}

\renewcommand{\div}{\operatorname{div}}

\numberwithin{equation}{section}

\newtheoremstyle{thmlemcorr}{10pt}{10pt}{\itshape}{}{\bfseries}{.}{10pt}{{\thmname{#1}\thmnumber{
			#2}\thmnote{ (#3)}}}
\newtheoremstyle{thmlemcorr*}{10pt}{10pt}{\itshape}{}{\bfseries}{.}\newline{{\thmname{#1}\thmnumber{
			#2}\thmnote{ (#3)}}}
\newtheoremstyle{defi}{10pt}{10pt}{\itshape}{}{\bfseries}{.}{10pt}{{\thmname{#1}\thmnumber{
			#2}\thmnote{ (#3)}}}
\newtheoremstyle{remexample}{10pt}{10pt}{}{}{\bfseries}{.}{10pt}{{\thmname{#1}\thmnumber{
			#2}\thmnote{ (#3)}}}
\newtheoremstyle{assumption}{10pt}{10pt}{}{}{\bfseries}{.}{10pt}{{\thmname{#1}\thmnumber{
			A#2}\thmnote{ (#3)}}}

\theoremstyle{thmlemcorr}
\newtheorem{theorem}{Theorem}
\numberwithin{theorem}{section}

\newaliascnt{lemma}{theorem}
\newtheorem{lemma}[lemma]{Lemma}
\aliascntresetthe{lemma}

\newaliascnt{corollary}{theorem}

\aliascntresetthe{corollary}

\newaliascnt{proposition}{theorem}
\newtheorem{proposition}[proposition]{Proposition}
\aliascntresetthe{proposition}

\newtheorem{assumption}{Assumption}
\numberwithin{assumption}{section}

\newaliascnt{conjecture}{theorem}

\aliascntresetthe{conjecture}

\crefname{theorem}{theorem}{theorems}
\Crefname{theorem}{Theorem}{Theorems}

\crefname{lemma}{lemma}{lemmas}
\Crefname{lemma}{Lemma}{Lemmas}

\crefname{corollary}{corollary}{corollaries}
\Crefname{corollary}{Corollary}{Corollaries}

\crefname{proposition}{proposition}{propositions}
\Crefname{proposition}{Proposition}{Propositions}

\crefname{assumption}{assumption}{assumptions}
\Crefname{assumption}{Assumption}{Assumptions}

\crefname{conjecture}{conjecture}{conjectures}
\Crefname{conjecture}{Conjecture}{Conjectures}

\theoremstyle{thmlemcorr*}
\newtheorem{theorem*}{Theorem}
\newtheorem{lemma*}[theorem]{Lemma}
\newtheorem{corollary*}[theorem]{Corollary}
\newtheorem{proposition*}[theorem]{Proposition}
\newtheorem{problem*}[theorem]{Problem}
\newtheorem{conjecture*}[theorem]{Conjecture}

\theoremstyle{defi}

\newtheorem{problem}{Problem}
\crefname{problem}{problem}{problems}
\Crefname{problem}{Problem}{Problems}

\theoremstyle{remexample}
\newtheoremstyle{itremark}%
  {\topsep}{\topsep}
  {\itshape}
  {}
  {\bfseries}
  {.}
  {.5em}
  {}

\theoremstyle{itremark}

\newaliascnt{remarknoqed}{theorem}
\newtheorem{remarknoqed}[remarknoqed]{Remark}
\aliascntresetthe{remarknoqed}
\crefname{remarknoqed}{remark}{remarks}
\Crefname{remarknoqed}{Remark}{Remarks}

\newenvironment{remark}
  {\begin{remarknoqed}}
  {\qed\end{remarknoqed}}

\newenvironment{notations}{
	\noindent\textbf{Notations:} \it 
}{
	\normalfont 
}

\newtheorem*{notations*}{Notations}

\title[Hessian Riemannian Flows and a Solver-Agnostic Framework for Inverse Problems]{
A globally convergent flow for time-dependent mean field games and a solver-agnostic framework for inverse problems}

\author{Hanwei Yan$^{1}$, Xianjin Yang$^{2,*}$, Jingguo Zhang$^{3}$}
\thanks{$^*$Corresponding author. Authors are listed in alphabetical order.} 
 
\address{$^1$Department of Mathematics, City University of Hong Kong, Hong Kong, China.}
\email{hanweiyan2-c@my.cityu.edu.hk}
\address{$^2$Department of Computing and Mathematical Sciences, California Institute of Technology, CA, USA.}
\email{yxjmath@caltech.edu}
\address[J. Zhang]{
	$^3$Department of Mathematics and Risk Management Institute, National University of Singapore, Singapore.}
\email{e0983423@u.nus.edu}

\begin{document}

\begin{abstract}
Mean field games (MFGs) describe the limiting behavior of large populations of strategically interacting agents. This paper addresses two numerical challenges for MFGs: globally convergent forward solvers and solver-agnostic methods for inverse problems. For the forward problem, we extend the Hessian--Riemannian flow (HRF), previously developed for stationary MFGs, to time-dependent MFGs. We first discretize the system in space and time and then construct the flow directly on the resulting finite-dimensional problem. The proposed flow exploits Lasry--Lions monotonicity, preserves the initial density and terminal value function, and maintains positivity and mass of the density. Under standard assumptions, we prove global convergence of the HRF and show how to recover a solution of the full discretized time-dependent MFG system from its limit. For the inverse problem, we formulate parameter estimation as a bilevel problem in which the outer problem updates unknown coefficients and the inner problem solves the discretized MFG system. Gradients of the outer objective are obtained by differentiating the discretized MFG system at the inner solution, rather than differentiating through the iterations of a particular forward solver. This yields a solver-agnostic framework with adjoint-based gradient descent and Gauss--Newton acceleration. Numerical experiments on stationary and time-dependent MFGs demonstrate the effectiveness of the proposed methods.
\end{abstract}

\maketitle

\bigskip
\noindent\textbf{Keywords:} mean field games, monotone flows, inverse problems, solver-agnostic, Gauss--Newton method

\section{Introduction}
\label{sec:intro}
In this paper, we develop a globally convergent flow for solving time-dependent mean field games (MFGs) that preserves the positivity of the density, together with a solver-agnostic framework for inverse MFG problems in which parameter optimization is decoupled from the numerical procedure used to solve the forward problem. MFGs describe the macroscopic limit of strategic interactions among a large population of indistinguishable agents.  MFGs were introduced independently by two research groups: one led by Lasry and Lions \cite{Lasry-Lions:Jeux:2006,Lasry-Lions:Jeux-II:2006,Lasry-Lions:Janpan:2007}, and the other by Caines, Huang, and Malham$\mathrm{\acute{e}}$ \cite{Huang-Roland:CommunInfSyst:2006,Huang-Roland:TAC:2007}.
The theory offers a powerful mathematical framework for approximating Nash equilibria in multi-agent systems with a finite number of players. 
Under the assumptions of a homogeneous population and symmetric interactions, the MFG model demonstrates that as the number of agents approaches infinity, the influence of any specific opponent vanishes: the strategy of an individual agent is shaped by the aggregate behavior of the population, which is statistically represented by the mean field.  This viewpoint is particularly natural in large financial markets, where each small trader has a negligible individual impact on prices, yet the aggregate trading behavior of the population can significantly affect market dynamics \cite{cardaliaguet2018mean,lachapelle2016efficiency}. A standard MFG system is characterized by a Hamilton--Jacobi--Bellman (HJB) equation coupled with a Fokker--Planck (FP) equation. 
The former governs the value function of a representative agent, while the latter describes the evolution of the population distribution.

A prototypical time-dependent MFG on the $d$-dimensional torus $\mathbb{T}^d$ takes the form 
\begin{align}
	\label{eq:intro_TimeDep_MFG}
	\begin{cases}
		-\partial_t u - \nu \Delta u + H(x, D_xu) = f(m)+ V(x), &\forall (x,t) \in  \mathbb{T}^d\times (0, T), \\
		\partial_t m - \nu\Delta m - \div(D_pH(x, D_xu)m) = 0, &\forall (x,t)\in \mathbb{T}^d\times (0, T), \\
		m(x, 0)=m_0(x), \quad u(x, T)=u_T(x), & \forall x\in \mathbb{T}^d,
	\end{cases}
\end{align}
where $u$ denotes the value function of an agent, $m$ represents the population density, and $\nu$ is the volatility parameter of the agents' motion. Here, $H$ denotes the Hamiltonian, $f$ the coupling term, and $V$ the spatial cost. The initial distribution and the terminal cost are denoted by $m_0$ and $u_T$.  
In the MFG system \eqref{eq:intro_TimeDep_MFG}, the optimal strategy for a representative agent is given by $-D_pH(x,D_xu)$ \cite{Lasry-Lions:Jeux:2006,Lasry-Lions:Jeux-II:2006,Lasry-Lions:Janpan:2007}. 
In this paper, we also consider the stationary MFG, which takes the form
\begin{align}
	\label{eq:intro_Ergodic_MFG}
	\begin{cases}
		-\nu\Delta u+H\left(x, D_xu\right)=f(m)+V(x)+\lambda, &\forall x \in  \mathbb{T}^d, \\
		- \nu\Delta m - \div(D_pH(x, D_xu)m) = 0, &\forall x\in \mathbb{T}^d, \\
		\int_{\mathbb{T}^d} u \,\mathrm{d}x = 0, 
		\quad 
		\int_{\mathbb{T}^d} m \,\mathrm{d}x = 1. 
	\end{cases}
\end{align}
Here \(\lambda\in\mathbb R\) is the ergodic constant, representing the
long-time average cost in the stationary regime. It is an additional unknown,
while the normalization \(\int_{\mathbb T^d}u\,dx=0\) fixes the additive
constant of the value function. 
The stationary MFG system can be interpreted as the long-time limit $t \to \infty$ of the time-dependent MFG \cite{cardaliaguet2012long, cardaliaguet2013long}.

MFGs have been widely applied in engineering, economics, finance, and energy distribution networks. For comprehensive overviews, see \cite{Gomes-Saude:DGA:2014, Cardaliaguet-Porretta:MFG_Cetraro:2021}. The MFG equations form a highly coupled system and rarely admit closed-form solutions, which motivates the development of efficient numerical methods for their approximation.

Solving either the stationary MFG \eqref{eq:intro_Ergodic_MFG} or the time-dependent MFG \eqref{eq:intro_TimeDep_MFG} can be formulated as finding zeros of nonlinear operators acting on infinite-dimensional function spaces. In the stationary case, the ergodic constant $\lambda$ can be eliminated explicitly, yielding a closed operator equation in $(m,u)$ alone. Indeed, multiplying the HJB equation by $m$, integrating over $\mathbb{T}^d$, using periodicity, and dividing by $\int_{\mathbb{T}^d}m\,\mathrm{d}x$, gives
\begin{align}
\label{eq:intro_lambda_elim}
\lambda(m,u)
&=\frac{
\int_{\mathbb{T}^d} \Bigl(-\nu\Delta u + H(x,D_xu) - f(m(x)) - V(x)\Bigr)\,m(x)\,\mathrm{d} x}{\int_{\mathbb{T}^d}m \,\mathrm{d}x}.
\end{align}
With \eqref{eq:intro_lambda_elim} in hand, define the stationary residual operator
\begin{align}
\label{eq:intro_Fop_st_closed}
\mathcal{F}_{\mathrm{st}}
\begin{pmatrix}
m\\ u
\end{pmatrix}
:=
\begin{pmatrix}
\nu \Delta u - H(x,D_xu) + f(m) + V(x) + \lambda(m,u)\\[2pt]
-\nu \Delta m - \div\!\bigl(D_pH(x,D_xu)\,m\bigr)
\end{pmatrix},
\qquad x\in\mathbb{T}^d.
\end{align}
Then, solving the stationary MFG \eqref{eq:intro_Ergodic_MFG} is equivalent to finding $(m,u)$ such that
\[
\mathcal{F}_{\mathrm{st}}(m,u)=0,
\qquad
\int_{\mathbb{T}^d} m(x)\,\mathrm{d} x = 1,
\qquad
m(x)\ge 0,
\qquad
\int_{\mathbb{T}^d} u(x)\,\mathrm{d} x = 0.
\]
For the time-dependent problem, one analogously introduces the evolution residual operator
\begin{align}
\label{eq:intro_Fop_td_closed}
\mathcal{F}_{\mathrm{td}}
\begin{pmatrix}
m\\ u
\end{pmatrix}
 :=
\begin{pmatrix}
\partial_t u + \nu \Delta u - H(x,D_xu) + f(m) + V(x)\\[2pt]
\partial_t m - \nu \Delta m - \div\!\bigl(D_pH(x,D_xu)\,m\bigr)
\end{pmatrix},
\qquad (t,x)\in(0,T)\times\mathbb{T}^d,
\end{align}
so that solving \eqref{eq:intro_TimeDep_MFG} amounts to finding $(m,u)$ satisfying $\mathcal{F}_{\mathrm{td}}(m,u)=0$ together with the prescribed initial-terminal data and the constraints $m(t,\cdot)\ge 0$ and $\int_{\mathbb{T}^d} m(t,x)\,\mathrm{d} x = 1$ for all $t\in[0,T]$.

From this operator viewpoint, one may apply Newton-type iterations to a finite-difference discretization of the coupled residual equations together with the normalization constraints \cite{achdou2010mean}. While highly effective when initialized sufficiently close to a solution, such Newton solvers are well known to exhibit primarily local convergence. A possible way to obtain a globally convergent algorithm is to recast the MFG system as the first-order optimality condition of a linearly constrained convex minimization problem. This structure corresponds to so-called potential MFGs, for which one can leverage convex optimization methods (see, e.g., \cite{briceno2018proximal, briceno2019implementation, briceno2024forward}). This approach, however, is limited to potential MFGs and does not extend to MFGs with a non-potential structure, such as \eqref{eq:2DstationaryMFGinvSmall_nu} in \Cref{2DMFGexample2}.   A different direction relaxes the potential assumption and instead exploits that the MFG residual operators $\mathcal{F}_{\mathrm{st}}$ and $\mathcal{F}_{\mathrm{td}}$ are monotone under the Lasry--Lions monotonicity condition \cite{Lasry-Lions:Jeux:2006,Lasry-Lions:Jeux-II:2006,almulla2017two}. To illustrate the idea, let $z=(m,u)$ and let $\mathcal{F}$ denote either $\mathcal{F}_{\mathrm{st}}$ or $\mathcal{F}_{\mathrm{td}}$, viewed as a mapping between suitable Hilbert spaces. We equip the product space with the pairing
\[
\langle z_1,z_2\rangle := \int_{\mathbb{T}^d}  (m_1(x)m_2(x) \;+\; u_1(x)u_2(x))\,\mathrm{d} x,
\qquad 
\|z\|^2:=\langle z,z\rangle,
\]
and analogously in the time-dependent setting by integrating over $(0,T)\times\mathbb{T}^d$. Monotonicity then means that
\[
\bigl\langle \mathcal{F}(z_1)-\mathcal{F}(z_2),\, z_1-z_2 \bigr\rangle \ge 0
\qquad \text{for all admissible } z_1,z_2.
\]
In the time-dependent case, this monotonicity statement is understood for admissible pairs with the same prescribed initial density and terminal value function, so that the boundary terms generated by integration by parts in time cancel.
In the stationary case, this monotonicity statement is understood on the normalized admissible set $\int_{\mathbb{T}^d} m\,\mathrm{d}x = 1$.
This property naturally suggests introducing an artificial-time evolution (a monotone flow)
\begin{equation}\label{eq:ambient_mf_intro_polished}
\partial_s z(s) = -\mathcal{F}\bigl(z(s)\bigr),\qquad s\ge 0,
\end{equation}
for which one formally obtains the dissipation estimate
\[
\frac{\mathrm{d}}{\mathrm{d} s}\frac12\|z(s)-z^\ast\|^2
= -\bigl\langle \mathcal{F}(z(s))-\mathcal{F}(z^\ast),\, z(s)-z^\ast\bigr\rangle \le 0
\]
whenever $z^\ast$ satisfies $\mathcal{F}(z^\ast)=0$.
This estimate shows that the distance from the trajectory \(z(s)\) to any
solution \(z^*\) is nonincreasing along the artificial time \(s\). When the
monotonicity is strict or strengthened in a suitable sense, this dissipation can
also force convergence of the trajectory to the solution. For rigorous results on global existence and convergence of such monotone flows, we refer the reader to \cite{almulla2017two}.

The key obstruction is that \eqref{eq:ambient_mf_intro_polished} is an ambient flow: it evolves in the full product space and does not enforce the intrinsic constraints of the MFG system, in particular $m\ge 0$ and mass conservation. As a result, positivity of the density is not preserved in general. This issue is especially severe computationally: once an iterate leaves the positive cone, nonlinearities such as $f(m)$ can become numerically unstable (e.g., for entropic couplings $f(m)=\ln m$).

To enforce positivity of densities in stationary MFGs, \cite{gomes2020hessian} introduced a monotone flow on the manifold of probability densities equipped with a barrier-induced Riemannian geometry, leading to Hessian--Riemannian flows (HRFs) that preserve positivity by construction. The idea is to work on the open set of strictly positive densities and equip \((m,u)\) with the Riemannian metric generated by the Hessian of the entropy functional
\[
\int_{\mathbb{T}^d}\Bigl(m\ln m-m+\tfrac12 u^2\Bigr)\,\mathrm{d}x.
\]
The corresponding Hessian induces the weighted pairing
\[
\bigl\langle(\delta m,\delta u),(\delta \tilde{ m},\delta \tilde{ u})\bigr\rangle
=
\int_{\mathbb{T}^d}\frac{\delta m\,\delta \tilde{ m}}{m}\,\mathrm{d}x
+
\int_{\mathbb{T}^d}\delta u\,\delta \tilde{ u}\,\mathrm{d}x.
\]
The ambient monotone flow direction, $-\mathcal{F}_{\mathrm{st}}$, is then projected onto the manifold of positive densities with respect to this Riemannian metric, yielding an intrinsic monotone flow. The resulting positivity-preserving HRF reads \cite{gomes2020hessian}
\begin{equation*}
\begin{pmatrix}\partial_s m\\ \partial_s u\end{pmatrix}
=
-\begin{pmatrix}
    m & 0\\
    0 & 1
\end{pmatrix}\mathcal{F}_{\mathrm{st}}
\begin{pmatrix}
    m \\ u
\end{pmatrix}.
\end{equation*}
This multiplicative structure is the key point: one can rewrite the first equation as an evolution for \(\ln m\), so strict positivity of the initial density propagates for all artificial times \(s\ge 0\), while \(\int_{\mathbb{T}^d}m\,\mathrm{d}x\) is preserved by construction. Here mass preservation uses the compatibility choice of \(\lambda(m,u)\) in \eqref{eq:intro_lambda_elim}, which gives \(\int_{\mathbb T^d}m\,\mathcal F^{m}_{\mathrm{st}}(m,u)\,\mathrm dx=0\). See, e.g., \cite{gomes2020hessian}. 

However, extending the HRF method to the time-dependent PDE system \eqref{eq:intro_TimeDep_MFG} involves additional constraints: one must simultaneously preserve positivity at each time slice and enforce the mixed initial-terminal conditions \(m(x,0)=m_0(x)\) and \(u(x,T)=u_T(x)\). A constrained-flow viewpoint proposed in \cite{gomes2021numerical} suggests treating these endpoint conditions as trace constraints at \(t=0\) and \(t=T\). This naturally leads to a time-regularity setting in which traces are well defined, typically a Sobolev space with one time derivative, together with a projection operator that enforces the two endpoint constraints simultaneously. Such a projection is inherently global in time: for a given direction, computing its projection onto the feasible manifold amounts to solving an auxiliary two-point boundary-value problem in the time variable. The same difficulty reappears when one attempts to introduce a Riemannian geometry in order to encode the positivity at the continuous level. We refer to Section~\ref{sec:tdmfg} for further discussion.

The difficulty described above stems from the trace-regularity requirements and the global-in-time projection inherent in continuous-in-time approaches. This naturally motivates a fully discrete viewpoint: once the system is discretized in time, the mixed endpoint conditions become coordinate constraints on the first and last time slices, and positivity and mass conservation reduce to algebraic conditions on finite-dimensional vectors. Both difficulties thus disappear. Feasibility can be maintained simply by freezing the boundary slices and evolving only the interior degrees of freedom, reducing constraint handling to an explicit elimination of boundary variables. Following this discretize-then-flow strategy, we construct an HRF directly for the time-discretized MFG system.  Under standard assumptions, such as convexity of the Hamiltonian and monotonicity of the coupling, the proposed HRF for the discretized time-dependent MFG system is globally well-defined. The density remains strictly positive throughout the evolution. The flow converges to the unique stationary point of the HRF, from which one can
recover the unique solution of the full discretized time-dependent MFG system by
adding suitable spatial constants to the discretized value function. See \Cref{thm:global_conv_discrete} and \Cref{prop:global_exist_pos_mass}. 

The second contribution concerns inverse problems for both stationary and time-dependent MFGs. We develop algorithms for MFG inverse problems, in which unknown model components (such as the density of agents, the value function, or the spatial cost) are inferred 
from partial, noisy observations of the population and the spatial cost. 
Such problems arise, for example, in financial markets, where one seeks to recover 
latent trading costs or risk premia from order-flow and price data generated by a large 
population of interacting agents, and in models of crowd motion, where spatial preferences 
are inferred from observed mass evolution. More concretely, in the settings of \eqref{eq:intro_Ergodic_MFG} and \eqref{eq:intro_TimeDep_MFG}, one seeks to recover the state $(m,u)$ (and, in the stationary case, also the ergodic constant $\lambda$) together with the spatial cost $V$ from noisy observations of $m$ and $V$, and we assume that $H$, $f$, and $\nu$ are known. A concrete description of these inverse problems is given in \Cref{prob:st_prob_inverse,prob:td_prob} of \Cref{sec:inverse_stationary}. We formulate the inverse problem as a nonlinear optimization problem constrained by the MFG equations, with the spatial cost \(V\) as the parameter to be optimized. The objective balances data fidelity with regularization of the unknown parameter.

This formulation differs from several recent approaches to inverse problems for MFGs that treat the unknown coefficients and the MFG state as joint decision variables in a single optimization problem, and minimize over both simultaneously, see \cite{ding2022mean, mou2022numerical}. In contrast, our viewpoint is closer in spirit to bilevel formulations \cite{yu2024bilevel, huang2025joint, zhang2025learning, ren2024policy}, where the outer problem updates only the coefficients, while the inner problem computes the corresponding MFG state by solving the forward system. The bilevel formulation offers two practical advantages. First, the inner state is, by construction, an accurate MFG equilibrium for
  the current parameter. Second, it decouples the outer optimization from the particular algorithm used for the inner MFG solve. As a result, one can change the inner solver, stopping criteria, or implementation details without modifying the outer optimization procedure, provided the forward solve is sufficiently accurate. However, most existing bilevel approaches compute outer gradients in a way that is tightly coupled to the chosen inner procedure, either by differentiating through a particular solver and its iteration trace \cite{yu2024bilevel, zhang2025learning}, or by committing to a specific linearization such as policy iteration \cite{ren2024policy, yang2025gaussian,yu2025equilibrium}. Instead, we adopt a solver-agnostic approach inspired by implicit differentiation \cite{amos2017optnet, agrawal2019differentiable, blondel2022efficient}: we treat the discretized MFG system as an implicit layer and differentiate the constraints at the inner solution, rather than differentiating through the solver iterations. The forward solver thus only needs to return a sufficiently accurate solution, and can be changed without modifying the outer optimization procedure. Building on this, our second contribution is a solver-agnostic framework for inverse problems in both stationary and time-dependent MFGs, supporting both adjoint-based first-order updates and Gauss--Newton (GN) acceleration for the outer optimization. Further details are deferred to \Cref{sec:inverse_stationary}.

Overall, our contributions can be summarized as follows:
\begin{enumerate}
\item A positivity-preserving flow method for the fully discretized time-dependent MFG system with global convergence. We follow a discretize-then-flow strategy. After discretizing the time-dependent MFG system in space and time, we keep the boundary time slices fixed to preserve the initial condition of the density and the terminal condition of the value function. The HRF is then constructed to evolve only the interior unknowns. For the resulting finite-dimensional system, we prove that the HRF is globally
defined, preserves the positivity and mass of the density, and converges to a
state from which a solution of the full discretized time-dependent MFG system
can be recovered by adding spatial constants to the value-function slices.

\item Plug-and-play, solver-agnostic methods for inverse problems in stationary and time-dependent MFGs.
We pose the inverse problem as a nonlinear optimization problem constrained by the MFG equations, with unknown coefficients such as the spatial cost as decision variables. The framework is plug-and-play: any forward solver that returns a sufficiently accurate solution of the MFG can be used. Within this framework, we consider both a first-order method and a GN acceleration. 
\end{enumerate}

\subsection{Related Work on Numerical Methods and Inverse Problems for MFGs}
In recent years, a broad range of numerical methods has been developed for forward MFG problems. For the purposes of the present discussion, it is convenient to group them loosely into direct discretization methods, optimization-based methods, and learning-based methods.

Direct discretization methods work directly with a spatial or space-time discretization of the MFG system and solve the resulting finite-dimensional equations. Representative examples include finite difference schemes \cite{achdou2010mean,achdou2013mean,achdou2012iterative}, semi-Lagrangian and Lagrange--Galerkin discretizations \cite{carlini2015semilagrangian,carlini2024lagrange}, and finite element approximations \cite{osborne2025finite}, together with iterative solvers such as Newton-type methods, policy iteration \cite{achdou2012iterative,cacace2021policy,lauriere2023policy}, and monotone flows \cite{almulla2017two,gomes2020hessian}.

A second class consists of optimization-based methods, which exploit additional structure in the MFG system before solving it numerically. For potential MFGs, primal-dual algorithms reformulate the problem as a variational problem \cite{briceno2018proximal,briceno2019implementation}. Related Fourier-based approaches for nonlocal couplings also fall into this category, where nonlocal terms are represented in a reduced form that supports variational, splitting, or primal-dual algorithms \cite{nurbekyan2019fourier,liu2020splitting,liu2021computational}.

A third class comprises learning-based methods. Neural-network-based solvers offer flexible function approximations for MFG systems \cite{ruthotto2020machine,lin2020apac,carmona2019convergence,carmona2021convergence}, and Gaussian process (GP) methods, originally developed as Bayesian nonparametric surrogates for PDEs \cite{chen2021solving,yang2023learning,yang2023mini}, have also been adapted to MFGs \cite{mou2022numerical,meng2023sparse,yang2025gaussian}. While these approaches have shown promising empirical performance, global convergence guarantees remain less well understood for this class of methods.

For the purposes of this paper, the most relevant solver family is that of monotonicity-based flow methods for directly discretized MFG systems. Under the Lasry--Lions monotonicity condition, the MFG system can be viewed as a root-finding problem for a monotone operator, which motivates flow-based solvers. Monotone flow schemes for stationary MFGs appear in \cite{almulla2017two}. The flow viewpoint is further extended to time-dependent MFGs in the finite-state case in \cite{gomes2021numerical}. A key advantage of this line of work is that it provides a mechanism for global convergence without requiring the MFG system to admit a potential structure. To preserve positivity of the density, \cite{gomes2020hessian} introduced an HRF methodology for stationary MFGs, which preserves both positivity and mass by incorporating a geometry adapted to the density constraint. More recently, \cite{bakaryan2025hessian} extended the HRF approach to multi-population Wardrop equilibrium problems on generalized graphs. However, for time-dependent MFGs, extending the HRF is substantially more delicate because one must simultaneously handle positivity,  mass conservation, and mixed initial-terminal conditions. Our first contribution is aimed precisely at this gap.


Inverse problems for MFGs aim to reconstruct unknown coefficients, such as spatial costs or coupling functions, from partial and possibly noisy observations \cite{ding2022mean,liu2023inverse}. On the theoretical side, these problems have also been studied through convexification and related stability analyses \cite{klibanov2023convexification,klibanov2023holder,imanuvilov2023lipschitz,ren2024reconstructing}. In this paper, our focus is on numerical methodologies.

From a numerical viewpoint, one class of approaches adopts an all-at-once formulation, in which the unknown coefficients and the MFG state are optimized simultaneously \cite{ding2022mean,yang2025gaussian,guo2024decoding}. A second class adopts a bilevel viewpoint, in which the outer problem updates the unknown coefficients while the inner problem solves the MFG system for the current parameter value \cite{ren2024policy,yu2024bilevel,huang2025joint,zhang2025learning}. Distinct from these optimization-based formulations, a related recent direction is provided by learning-based methods, such as \cite{yang2023context}, which use operator-learning ideas to leverage information across multiple MFG instances rather than solving each inverse problem independently.

A common feature of many existing bilevel pipelines is that the outer gradient is constructed together with a particular inner solver, so that its computation may depend on solver-specific implementation details such as the iterative scheme, the number of inner iterations, and the stopping criterion. This motivates the solver-agnostic framework developed in this paper, in which the outer optimization is carried out over the unknown coefficients while the gradient is obtained by implicitly differentiating the discrete MFG system satisfied by the converged inner solution, rather than by differentiating through the iteration trace of a particular forward solver. We note that \cite{zhang2025learning} also uses an adjoint-based strategy to obtain a solver-agnostic framework. The key distinction lies in the level at which implicit differentiation is performed. In \cite{zhang2025learning}, which focuses on potential MFGs, the adjoint is derived from the optimality system of an inner PDE-constrained convex minimization problem. By contrast, we differentiate the discretized MFG system satisfied by the converged inner solution directly, which does not require the MFG to admit a potential or variational structure.

\subsection{Organization}
The remainder of this paper is organized as follows. Section~\ref{sec:tdmfg} develops the HRF method for time-dependent MFGs. Section~\ref{sec:inverse_stationary} presents a solver-agnostic framework for inverse problems. Numerical experiments are reported in Section~\ref{sec:numerics}. Section~\ref{sec:conclusion} concludes the paper and outlines directions for future research. Finally, all proofs are collected in \Cref{app:appendix}.

\section{HRFs for Time-Dependent MFGs}
\label{sec:tdmfg}

The goal of this section is to develop a  flow-based solver for \eqref{eq:intro_TimeDep_MFG} that preserves the positivity of the density and admits rigorous convergence guarantees driven by monotonicity.  To focus on the main ideas without being distracted by the treatment of boundary conditions and low-regularity solutions, we work in the setting of classical solutions. Throughout, we assume that \eqref{eq:intro_TimeDep_MFG} admits a unique classical solution \((m,u)\), where $m, u\in C^{2,1}(\mathbb{T}^d\times[0,T])$. Classical well-posedness holds under standard hypotheses; see \cite{Lasry-Lions:Jeux-II:2006,cardaliaguet2010notes, gomes2016regularity} for the time-dependent MFG system and detailed proofs. In the continuous setting, we rely on standard conditions ensuring well-posedness and regularity, such as the convexity of \(H\) in the momentum variable and the Lasry--Lions monotonicity of the coupling \(f\). Since our HRF construction is carried out at the fully discrete level, after both space and time discretization, we do not restate the corresponding continuous hypotheses here. Instead, we state the discrete assumptions under which the HRF converges for the discretized system.

The HRF method of \cite{gomes2020hessian} was developed for stationary MFGs and shown to be globally convergent. In what follows, we first outline the main obstacles to extending this approach to time-dependent MFGs at the continuous level. Guided by these observations, we then develop an HRF method for time-dependent MFGs at the fully discrete level, after space-time discretization.

\subsection{Challenges in Constructing the HRF for Time-Dependent MFGs}
\label{subsec:continuous_HRF}

In this subsection, we explore how one might extend the HRF of \cite{gomes2020hessian} to time-dependent MFGs at the continuous PDE level, and identify the resulting challenges.

We define
\begin{equation*}
\mathcal{X}
:=
\Big\{(m,u):\ m\in C^{2,1}(\mathbb{T}^d\times[0,T]),\ u\in C^{2,1}(\mathbb{T}^d\times[0,T])\Big\},
\end{equation*}
together with the residual space
\begin{equation*}
\mathcal{Y}
:=
C(\mathbb{T}^d\times[0,T])\times C(\mathbb{T}^d\times[0,T]).
\end{equation*}
We recall the time-dependent MFG residual operator \(\mathcal{F}_{\mathrm{td}}:\mathcal{X}\to\mathcal{Y}\) defined in \eqref{eq:intro_Fop_td_closed}. Denote $z := (m, u)$. Solving \eqref{eq:intro_TimeDep_MFG} is equivalent to finding \(z\in\mathcal{X}\) such that \(\mathcal{F}_{\mathrm{td}}(z)=0\), subject to the endpoint conditions.
Define the linear trace operator
\begin{equation}
\label{eq:C_def_td}
\mathcal{T}:\mathcal{X}\to \mathcal{O},\qquad
\mathcal{T}(m,u):=\big(m(\cdot,0),\ u(\cdot,T)\big),
\qquad
\mathcal{O}:=C(\mathbb{T}^d)\times C(\mathbb{T}^d).
\end{equation}
The endpoint-feasible set is the affine manifold
\begin{equation*}
\mathcal{M}:=\Big\{z\in\mathcal{X}:\ \mathcal{T}z=(m_0,u_T)\Big\},
\end{equation*}
with tangent space at any \(z\in\mathcal{M}\) given by
\begin{equation*}
T_z\mathcal{M}=\ker \mathcal{T}
=
\Big\{(\delta m,\delta u)\in\mathcal{X}:\ \delta m(\cdot,0)=0,\ \delta u(\cdot,T)=0\Big\}.
\end{equation*}
To enforce positivity of the density along the flow, we restrict attention to
\begin{equation*}
\mathcal{M}_{++}:=\{(m,u)\in\mathcal{M}:\ m(x,t)>0\ \text{for all }(x,t)\in\mathbb{T}^d\times[0,T]\}.
\end{equation*}
On \(\mathcal{M}_{++}\), we equip the density component with an entropy-type metric and keep the Euclidean metric in \(u\). For \(z=(m,u)\in\mathcal{M}_{++}\), we define the weighted inner product
\begin{equation*}
\langle (\delta m,\delta u),(\delta\tilde m,\delta\tilde u)\rangle_{\mathcal{H}(z)}
:=
\int_0^T\!\!\int_{\mathbb{T}^d}\frac{\delta m\,\delta\tilde m}{m}\,\mathrm{d}x\,\mathrm{d}t
\;+\;
\int_0^T\!\!\int_{\mathbb{T}^d}\delta u\,\delta\tilde u\,\mathrm{d}x\,\mathrm{d}t,
\end{equation*}
corresponding to the diagonal metric operator
\begin{equation*}
\mathcal{H}(z)=\mathrm{diag}(1/m,\ I),
\qquad
\mathcal{H}(z)^{-1}=\mathrm{diag}(m,\ I).
\end{equation*}
The weight \(1/m\) penalizes variations increasingly as \(m\to 0\), thereby encoding a barrier against the boundary of the nonnegative cone. A natural extension of the HRF idea of \cite{gomes2020hessian} is then to evolve \(z\) in artificial time along the direction \(-\mathcal{H}(z)^{-1}\mathcal{F}_{\mathrm{td}}(z)\), giving the flow
\begin{align}
\label{eq:HRF_st}
\begin{pmatrix}
\partial_s m \\ \partial_s u
\end{pmatrix}
=
-\mathcal{H}(z)^{-1}\mathcal{F}_{\mathrm{td}}(z)
=
-\begin{pmatrix}
m\,\mathcal{F}_{\mathrm{td}}^{m}(m,u)\\[2pt]
\mathcal{F}_{\mathrm{td}}^{u}(m,u)
\end{pmatrix},
\end{align}
where we write \(\mathcal{F}_{\mathrm{td}}(z)=\big(\mathcal{F}_{\mathrm{td}}^{m}(m,u),\ \mathcal{F}_{\mathrm{td}}^{u}(m,u)\big)\).  In particular, the density component satisfies
\[
\partial_s m(x,t)=-\,m(x,t)\,\mathcal{F}_{\mathrm{td}}^{m}(m,u)(x,t).
\]
This multiplicative structure is the formal mechanism behind positivity preservation, although we do not pursue a rigorous continuous-level positivity statement here. However, the flow \eqref{eq:HRF_st} does not, in general, preserve the endpoint conditions: even if \(z(0)\in\mathcal{M}\), one typically has $z(s)\notin \mathcal{M}$ for $s>0$. A remedy for preserving endpoint conditions along monotone flows, formulated independently of any Riemannian geometry, is proposed in \cite{gomes2021numerical}. The basic mechanism is to project the driving field onto the tangent space \(T_z\mathcal{M} = \ker \mathcal{T}\), so that feasibility is maintained along the flow. To make endpoint evaluation continuous and the projection well-posed in a Hilbert setting, one works in the space
\begin{equation*}
\mathcal{X}_{H^1}
:=
H^1\big(0,T;L^2(\mathbb{T}^d)\big)\times H^1\big(0,T;L^2(\mathbb{T}^d)\big),
\end{equation*}
equipped with the inner product
\begin{equation*}
\langle (\delta m,\delta u),(\delta\tilde m,\delta\tilde u)\rangle_{H^1}
:=
\int_0^T\!\!\int_{\mathbb{T}^d}\Big(\delta m\,\delta\tilde m+\partial_t\delta m\,\partial_t\delta\tilde m\Big)\,\mathrm{d}x\,\mathrm{d}t
+
\int_0^T\!\!\int_{\mathbb{T}^d}\Big(\delta u\,\delta\tilde u+\partial_t\delta u\,\partial_t\delta\tilde u\Big)\,\mathrm{d}x\,\mathrm{d}t.
\end{equation*}
The trace operator in \eqref{eq:C_def_td} extends continuously to
\begin{equation*}
\mathcal{T}_{H^1}:\mathcal{X}_{H^1}\to L^2(\mathbb{T}^d)\times L^2(\mathbb{T}^d),
\qquad
\mathcal{T}_{H^1}(m,u):=\big(m(\cdot,0),\ u(\cdot,T)\big).
\end{equation*}
By an abuse of notation, we continue to denote this extension by $\mathcal{T}$ in the $H^1$ setting below. Given an ambient direction \(w\in\mathcal{X}_{H^1}\), its \(H^1\)-orthogonal projection onto \(\ker \mathcal{T}\) is the unique minimizer of
\begin{equation}
\label{eq:orth_proj_td_rewrite}
\Pi_{\perp}w
\in
\arg\min_{v\in\mathcal{X}_{H^1}}
\left\{
\frac12\|v\|_{H^1}^2 - \langle v, w\rangle_{H^1}
\right\}
\quad \text{subject to}\quad \mathcal{T}v=0,
\end{equation}
where \(\langle\cdot,\cdot\rangle_{H^1}\) denotes the inner product defined above. The optimality conditions for \eqref{eq:orth_proj_td_rewrite} show that, even though the constraint only fixes endpoint values, the correction generated by the \(H^1\) projection is spread across the full interval \((0,T)\), and computing \(\Pi_{\perp}w\) requires solving two-point boundary value problems in \(t\) for each spatial degree of freedom.

If one further wishes to incorporate the Riemannian geometry in order to also encode the positivity of the density, the two requirements must be combined. One natural trace-compatible approach is to work in \(\mathcal{X}_{H^1}\) and endow it, at each \(z=(m,u)\in\mathcal{M}_{++}\), with the \(z\)-dependent inner product
\begin{equation}
\label{eq:H1_HR_inner_td}
\langle (\delta m,\delta u),(\delta\tilde m,\delta\tilde u)\rangle_{H^1,z}
:=
\int_0^T\!\!\int_{\mathbb{T}^d}\frac{\delta m\,\delta\tilde m+\partial_t\delta m\,\partial_t\delta \tilde m}{m}\,\mathrm{d}x\,\mathrm{d}t
+
\int_0^T\!\!\int_{\mathbb{T}^d}\big(\delta u\,\delta\tilde u+\partial_t\delta u\,\partial_t\delta\tilde u\big)\,\mathrm{d}x\,\mathrm{d}t,
\end{equation}
which is trace compatible in time and incorporates the positivity barrier in the density component. The constrained update direction \(v(z)\) is then defined by
\begin{equation}
\label{eq:combined_constrained_direction}
v(z)\in \arg\min_{v\in \mathcal{X}_{H^1}}
\left\{
\frac12\|v\|_{H^1,z}^2+\langle \mathcal{F}_{\mathrm{td}}(z),v\rangle_{L^2}
\right\}
\quad \text{subject to}\quad \mathcal{T}v=0,
\end{equation}
with the artificial-time update given by \(\partial_s z=v(z)\). While this formulation enforces endpoint feasibility and encodes the positivity barrier in the metric, its computational cost is even higher: the optimality conditions for \eqref{eq:combined_constrained_direction} require solving a subproblem that is second order in time with state-dependent coefficients through \eqref{eq:H1_HR_inner_td}, making each artificial-time step considerably more expensive.

The difficulty in both cases above stems from the continuous setting: enforcing endpoint conditions requires working in a trace-compatible function space, and the resulting projection is inherently global in time. In a fully discrete setting, this difficulty disappears entirely. Once the system is discretized in time, the endpoint conditions become coordinate constraints on the first and last time slices, which can be enforced simply by freezing those degrees of freedom. This avoids the need to solve auxiliary boundary-value problems across the entire time interval, and motivates the fully discrete approach developed in the next section.

\subsection{The HRF Method for Time-Dependent MFGs}
\label{subsec:discrete_HRF}
This subsection develops an HRF for a time-discretization of the MFG system \eqref{eq:intro_TimeDep_MFG} with $\nu\ge0$. We first discretize \eqref{eq:intro_TimeDep_MFG} and then construct a continuous-time flow for the resulting finite-dimensional problem.  After discretization, we enforce the endpoint conditions by fixing the boundary slices and evolving only the interior variables, while the density variables evolve on a feasible manifold determined by the slice-wise mass constraints and positivity requirement. On this manifold, we introduce a Riemannian metric induced by a strictly convex entropy. The resulting flow preserves positivity and mass of the density by construction. We discretize \eqref{eq:intro_TimeDep_MFG} using the finite-difference framework of \cite{achdou2010mean}. 

For clarity, we present the construction on the two-dimensional torus $\mathbb{T}^2$ using a uniform periodic grid
\[
\mathbb{T}^2_h:=\{x_{i,j}=(ih,jh): i,j=0,\dots,N_h-1\},\qquad
h:=\frac1{N_h},\qquad
N_x:=|\mathbb{T}^2_h|=N_h^2.
\]
We assume \(N_h\ge 2\) and use periodic indexing. A grid function $y\in\mathbb{R}^{N_x}$ is identified with its values $y_{i,j}:=y(x_{i,j})$. We use one-sided differences
\begin{align}
\label{eq:one-side-diff}
(D_1^+y)_{i,j}:=\frac{y_{i+1,j}-y_{i,j}}{h},
(D_1^-y)_{i,j}:=\frac{y_{i,j}-y_{i-1,j}}{h},
(D_2^+y)_{i,j}:=\frac{y_{i,j+1}-y_{i,j}}{h},
(D_2^-y)_{i,j}:=\frac{y_{i,j}-y_{i,j-1}}{h},
\end{align}
and the standard five-point Laplacian
\begin{align}
\label{eq:five-point-laplacian}
(\Delta_h y)_{i,j}:=-\frac{4y_{i,j}-y_{i+1,j}-y_{i-1,j}-y_{i,j+1}-y_{i,j-1}}{h^2}.
\end{align}
At each grid node, we collect the four one-sided slopes into the upwind stencil
\begin{align}
\label{eq:stcil}
[D_h y]_{i,j}:=\Bigl((D_1^+y)_{i,j},\ (D_1^-y)_{i,j},\ (D_2^+y)_{i,j},\ (D_2^-y)_{i,j}\Bigr)^{\!\top}\in\mathbb{R}^4.
\end{align}

Following \cite{achdou2010mean}, we approximate the Hamiltonian $H$ in \eqref{eq:intro_TimeDep_MFG} by a numerical Hamiltonian
\[
g:\mathbb{T}^2\times\mathbb{R}^4\to\mathbb{R}.
\]
To keep the notation simple, for each grid node $(i,j)$ we write $g_{i,j}(q):=g(x_{i,j},q)$ for $q\in\mathbb{R}^4$.
Throughout, we assume that $H$ is strictly convex in its momentum argument. For the numerical Hamiltonian $g$, we impose the standard consistency and upwind-monotonicity conditions from \cite{achdou2010mean}, together with a convexity hypothesis that is used in the proof of \Cref{prop:Fh_monotone_MU}.

\begin{assumption}[Numerical Hamiltonian]
\label{ass:disc_g}
Let $H:\mathbb{T}^2\times\mathbb{R}^2\to\mathbb{R}$ be a continuous Hamiltonian, and assume that for each fixed $x\in\mathbb{T}^2$, the map $p\mapsto H(x,p)$ is strictly convex. Let $g:\mathbb{T}^2\times\mathbb{R}^4\to\mathbb{R}$ be a numerical Hamiltonian, and write $q=(q_1,q_2,q_3,q_4)\in\mathbb{R}^4$. The following properties hold.
\begin{itemize}
\item Consistency. For all $x\in\mathbb{T}^2$ and all $\xi_1, \xi_2\in\mathbb{R}$, $g(x,\xi_1,\xi_1,\xi_2,\xi_2)=H\bigl(x,(\xi_1,\xi_2)\bigr)$.
\item Monotonicity. For each fixed $x\in\mathbb{T}^2$, the map $q\mapsto g(x,q)$ is nonincreasing in $q_1$ and $q_3$ and nondecreasing in $q_2$ and $q_4$.
\item Convexity and rigidity. For each fixed $x\in\mathbb{T}^2$, the map $q\mapsto g(x,q)$ is continuously differentiable and convex on $\mathbb{R}^4$. Define the Bregman divergence
\[
D_g(x;q,q'):=g(x,q)-g(x,q')-\nabla_q g(x,q')\cdot(q-q')\ \ (\ge 0).
\]
For all grid functions $W,\widetilde W$, if $D_g\bigl(x_{ij};[D_hW]_{ij},[D_h\widetilde W]_{ij}\bigr)=0$ at every node $(i,j)$, then $[D_hW]=[D_h\widetilde W]$.
\end{itemize}
\end{assumption}
\begin{remark}[The rigidity condition covers two common cases]
\label{rem:rigidity_examples}
The rigidity condition in Assumption~\ref{ass:disc_g} is satisfied in the two
cases used in this paper.

\begin{itemize}
\item If \(q\mapsto g(x,q)\) is strictly convex for each fixed \(x\), then
\(D_g(x;q,q')=0\) implies \(q=q'\). Hence, the rigidity condition follows
immediately.

\item The condition also covers the Godunov-type fluxes considered here. These
fluxes may fail to be strictly convex in the full variable
\(q=(q_1,q_2,q_3,q_4)\). Instead, they depend on \(q\) through the upwind parts
\[
\rho(q):=(q_1^-,q_2^+,q_3^-,q_4^+),
\qquad
a^+:=\max\{a,0\},\quad a^-:=\max\{-a,0\},
\]
and are strictly convex in these upwind variables. More precisely, in the
Godunov case considered below, \(g(x,\cdot)\) is convex and \(C^1\), has the
form
\[
g(x,q)=\widehat g(x,\rho(q)),
\]
and \(r\mapsto \widehat g(x,r)\) is strictly convex on \(\mathbb R_+^4\).
By Lemma~\ref{lem:godunov_rigidity}, such fluxes satisfy the rigidity condition
in Assumption~\ref{ass:disc_g}.
\end{itemize}
\end{remark}


We discretize the transport operator $-\div(D_pH(\cdot,\nabla u)\,m)$ in divergence form using the fluxes induced by $g$.  Given $U\in\mathbb{R}^{N_x}$, define
\[
\alpha^{(\ell)}(U)_{i,j}
:=
\frac{\partial g_{i,j}}{\partial q_\ell}\Bigl([D_hU]_{i,j}\Bigr),
\qquad \ell=1,2,3,4.
\]
For $U,M\in\mathbb{R}^{N_x}$, define the discrete transport operator $\mathcal{B}_h(U,M)\in\mathbb{R}^{N_x}$ componentwise by
\begin{equation}
\label{eq:Bh_def}
(\mathcal{B}_h(U,M))_{i,j}
:=
-\Bigl(D_1^{-}(M\,\alpha^{(1)}(U))\Bigr)_{i,j}
-\Bigl(D_1^{+}(M\,\alpha^{(2)}(U))\Bigr)_{i,j}
-\Bigl(D_2^{-}(M\,\alpha^{(3)}(U))\Bigr)_{i,j}
-\Bigl(D_2^{+}(M\,\alpha^{(4)}(U))\Bigr)_{i,j},
\end{equation}
where all products are understood pointwise on the grid. For the reader's convenience, we recall in \Cref{lem:Bh_adjoint} the classical discrete result from \cite{achdou2010mean} that, for fixed \(U\), the discrete transport operator \(\mathcal B_h(U,\cdot)\) is the adjoint of the linearization of the discrete Hamiltonian block at \(U\).

We discretize $[0,T]$ by a uniform grid $0=t_0<t_1<\cdots<t_{N_T}=T$ with $\Delta t:=T/N_T$. For each $k=0,\dots,N_T$, we represent the functions on the grid at the time slices by vectors
\[
M_k\in\mathbb{R}^{N_x},\qquad U_k\in\mathbb{R}^{N_x},
\qquad
(M_k)_{i,j}\approx m(x_{i,j},t_k),\qquad (U_k)_{i,j}\approx u(x_{i,j},t_k).
\]
The mixed endpoint conditions are imposed by sampling the prescribed initial density $m_0(\cdot)$ and terminal value $u_T(\cdot)$ on the grid:
\[
(M_0)_{i,j}:=m_0(x_{i,j}),\qquad (U_{N_T})_{i,j}:=u_T(x_{i,j}).
\]
We use the discrete pairing $\langle a,b\rangle_h:=h^2\sum_{i,j}a_{i,j}b_{i,j}.$ Since sampling does not, in general, preserve mass exactly, we enforce $\langle M_0,\mathbf{1}\rangle_h=1$ by the renormalization
\[
M_0 \leftarrow \frac{M_0}{\langle M_0,\mathbf{1}\rangle_h},
\]
where $\mathbf{1}\in\mathbb{R}^{N_x}$ is the all-ones vector.  We then impose slice-wise mass conservation in the discrete form
\begin{equation}
\langle M_k,\mathbf{1}\rangle_h=1,\qquad k=0,1,\dots,N_T.
\end{equation}
We discretize the spatial cost by sampling $V(\cdot)$ on the grid and writing $V_h\in\mathbb{R}^{N_x}$ with $(V_h)_{i,j}:=V(x_{i,j})$. The coupling $f$ is applied componentwise to grid densities, and we adopt the following discrete analogue of the Lasry--Lions monotonicity condition.
\begin{assumption}[Coupling monotonicity]
\label{ass:disc_f}
The function $f:(0,\infty)\to\mathbb{R}$ is locally Lipschitz and strictly monotone increasing. The latter implies that for any $M,\widetilde M\in\mathbb{R}^{N_x}_{++}$ with $M\neq \widetilde M$,
\[
\bigl\langle f(M)-f(\widetilde M),\,M-\widetilde M\bigr\rangle_h>0,
\]
where $f(\cdot)$ is applied componentwise.
\end{assumption}

Since the endpoint slices $(M_0,U_{N_T})$ are fixed, the discrete problem is naturally parametrized by the interior unknowns
\[
Y:=\bigl(M_1,\dots,M_{N_T},\,U_0,\dots,U_{N_T-1}\bigr)\in
\mathcal{X}_h:=\Bigl(\mathbb{R}^{N_x}\Bigr)^{N_T}\times \Bigl(\mathbb{R}^{N_x}\Bigr)^{N_T}.
\]
Define the relative interior of the discrete probability simplex
\[
\mathcal{M}_{++,h}:=\Bigl\{M\in\mathbb{R}^{N_x}:\ M>0\ \text{componentwise},\ \langle M,\mathbf{1}\rangle_h=1\Bigr\}.
\]
The feasible set is the interior of the feasible manifold
\begin{equation}
\label{eq:Mh_def}
\mathcal{M}_h
:=
\Bigl\{
Y\in\mathcal{X}_h:\ M_k\in\mathcal{M}_{++,h}\ \text{for}\ k=1,\dots,N_T
\Bigr\},
\end{equation}
where the fixed data $(M_0,U_{N_T})$ are understood as above and used whenever they appear in the residuals.

For $k=1,\dots,N_T$, we define $r_k(Y)\in\mathbb{R}^{N_x}$ and $s_k(Y)\in\mathbb{R}^{N_x}$ componentwise by
\begin{align}
\label{eq:rk_def_rewrite}
(r_k(Y))_{i,j}
&:=
\frac{(U_k-U_{k-1})_{i,j}}{\Delta t}
+\nu\,(\Delta_h U_{k-1})_{i,j}
-g_{i,j}\bigl([D_hU_{k-1}]_{i,j}\bigr)
+f\bigl((M_k)_{i,j}\bigr)
+(V_h)_{i,j},\\
\label{eq:sk_def_rewrite}
(s_k(Y))_{i,j}
&:=
\frac{(M_k-M_{k-1})_{i,j}}{\Delta t}
-\nu\,(\Delta_h M_k)_{i,j}
+(\mathcal{B}_h(U_{k-1},M_k))_{i,j},
\end{align}
where $M_{k-1}$ is interpreted as $M_0$ when $k=1$, and $U_k$ is interpreted as $U_{N_T}$ when $k=N_T$.
We collect the residual blocks into
\begin{equation}
\label{eq:Fh_full_def}
F_h^{\mathrm{td}}(Y):=\bigl(r_1(Y),\dots,r_{N_T}(Y),\,s_1(Y),\dots,s_{N_T}(Y)\bigr).
\end{equation}
Solving the fully discretized MFG corresponds to finding $Y\in\mathcal{M}_h$ such that $F_h^{\mathrm{td}}(Y)=0$. 

Next, we establish the monotonicity of the operator  \(F_h^{\mathrm{td}}\) in \eqref{eq:Fh_full_def}. We first introduce the space-time pairing
\[
\langle \Xi,\Psi\rangle_{h,\Delta t}
:=
\Delta t\sum_{k=1}^{N_T}\langle \Xi_{M_k},\Psi_{M_k}\rangle_h
+
\Delta t\sum_{k=0}^{N_T-1}\langle \Xi_{U_k},\Psi_{U_k}\rangle_h,
\]
for $\Xi,\Psi\in\mathcal{X}_h$. Here $\Xi_{M_k},\Psi_{M_k}\in\mathbb{R}^{N_x}$ and $\Xi_{U_k},\Psi_{U_k}\in\mathbb{R}^{N_x}$ denote the density and value-function components of $\Xi,\Psi\in\mathcal{X}_h$, respectively.  Then, we have the following proposition.
\begin{proposition}[Strict monotonicity of $F_h^{\mathrm{td}}$]
\label{prop:Fh_monotone_MU}
Assume that the numerical Hamiltonian $g$ satisfies \Cref{ass:disc_g} and that the coupling $f$ satisfies \Cref{ass:disc_f}. Let $Y$ and $\widetilde Y$ be two feasible states with the same initial and terminal data $(M_0,U_{N_T})$. Assume also that their value-function slices have 
the same averages, namely 
$\langle U_{k-1}-\widetilde U_{k-1},\mathbf 1\rangle_h=0$ for 
$k=1,\dots,N_T$. Then $F_h^{\mathrm{td}}$ in \eqref{eq:Fh_full_def} is strictly monotone with respect to the space-time pairing $\langle\cdot,\cdot\rangle_{h,\Delta t}$, namely
\begin{equation}
\label{eq:Fh_strictly_monotone_MU}
\Big\langle F_h^{\mathrm{td}}(Y)-F_h^{\mathrm{td}}(\widetilde Y),\,Y-\widetilde Y\Big\rangle_{h,\Delta t}>0
\qquad\text{for all } Y\neq \widetilde Y.
\end{equation}
\end{proposition}

Next, we build the HRF dynamics on $\mathcal{M}_h$ by endowing it with a Riemannian metric generated by a strictly convex entropy functional.  For $Y=(M_1,\dots,M_{N_T},U_0,\dots,U_{N_T-1})\in\mathcal{M}_h$, define
\begin{equation}
\label{eq:disc_entropy_h}
\mathcal{E}(Y)
:=
\Delta t\sum_{k=1}^{N_T}\big\langle M_k,\log M_k-\mathbf{1}\big\rangle_h
+
\frac{\Delta t}{2}\sum_{k=0}^{N_T-1}\langle U_k,U_k\rangle_h, 
\end{equation}
where $\log M_k$ is understood componentwise and $\mathbf{1}\in\mathbb{R}^{N_x}$ denotes the vector of ones.

We write $\nabla \mathcal{E}(Y)$ and $\nabla^2\mathcal{E}(Y)$ for the gradient and Hessian of $\mathcal{E}$ with respect to the pairing $\langle\cdot,\cdot\rangle_{h,\Delta t}$, namely
\[
D\mathcal{E}(Y)[\Xi]
=
\langle \nabla\mathcal{E}(Y),\Xi\rangle_{h,\Delta t}, \text{ and }
D^2\mathcal{E}(Y)[\Xi,\Psi]
=
\langle \nabla^2\mathcal{E}(Y)\,\Xi,\Psi\rangle_{h,\Delta t}.
\]
A direct computation shows that
\[
\nabla^2\mathcal{E}(Y)
=
\operatorname{diag}\bigl(1/M_1,\dots,1/M_{N_T},\,I,\dots,I\bigr),
\]
where $1/M_k$ is understood componentwise.

We then define the Riemannian metric on $\mathcal{M}_h$ by
\[
g_{Y,\mathcal{E}}(\Xi,\Psi)
:=
\langle \nabla^2\mathcal{E}(Y)\,\Xi,\Psi\rangle_{h,\Delta t}.
\]
The admissible velocity space is
\[
\mathcal{T}_Y\mathcal{M}_h:=\Bigl\{\Xi\in\mathcal{X}_h:\ \langle \Xi_{M_k},\mathbf{1}\rangle_h=0\ \text{for }k=1,\dots,N_T\Bigr\}.
\]

Then, the HRF is defined as 
\begin{equation}
\label{eq:disc_HRF_var}
\dot Y(s)\in\arg\min_{\Xi\in\mathcal{T}_{Y(s)}\mathcal{M}_h}
\left\{
\frac12\,g_{Y(s),\mathcal{E}}(\Xi,\Xi)
+
\langle F_h^{\mathrm{td}}(Y(s)),\,\Xi\rangle_{h,\Delta t}
\right\},
\qquad Y(0)\in\mathcal{M}_h.
\end{equation}
The mass constraint enters \eqref{eq:disc_HRF_var} through $\mathcal{T}_Y\mathcal{M}_h$.  Because the metric is diagonal in the density blocks and the constraints are linear, the minimizer in \eqref{eq:disc_HRF_var} admits a closed form, which is given by the following proposition.

\begin{proposition}[Explicit discrete HRF]
\label{prop:explicit_disc_HRF}
For a feasible $Y$ with $M_k\in\mathcal{M}_{++,h}$, for all $k=1,\dots,N_T$, define
\begin{equation}
\label{eq:rbar_def}
\bar r_k(Y):=\frac{\langle M_k,r_k(Y)\rangle_h}{\langle M_k,\mathbf{1}\rangle_h}.
\end{equation} 
The minimizer in \eqref{eq:disc_HRF_var} is unique, and the flow $Y(s)$ satisfies
\begin{align}
\label{eq:disc_HRF_blocks}
\begin{split}
\dot M_k &= -\,M_k\odot\bigl(r_k(Y(s))-\bar r_k(Y(s))\,\mathbf{1}\bigr),
\qquad k=1,\dots,N_T,\\
\dot U_{k-1} &= -\,s_k(Y(s)),
\qquad k=1,\dots,N_T,
\end{split}
\end{align}
where $\odot$ denotes componentwise multiplication.
\end{proposition}
The proof is given in \Cref{app:appendix}.  Before studying the global existence and convergence of
\eqref{eq:disc_HRF_blocks}, we introduce the projected residual naturally associated with
the HRF. For \(k=1,\ldots,N_T\), set
\[
\widetilde r_k(Y):=r_k(Y)-\bar r_k(Y)\mathbf 1,
\]
where \(\bar r_k(Y)\) is defined in \eqref{eq:rbar_def}. We then define
\[
\widetilde F_h^{\mathrm{td}}(Y)
:=
\bigl(\widetilde r_1(Y),\ldots,\widetilde r_{N_T}(Y),
s_1(Y),\ldots,s_{N_T}(Y)\bigr).
\]
With this notation, a stationary point of the HRF in
\eqref{eq:disc_HRF_blocks} is characterized by
\(\widetilde F_h^{\mathrm{td}}(Y)=0\). The strict monotonicity above implies the
following uniqueness statement. If $Y,\widetilde Y\in\mathcal M_h$ both satisfy
$\widetilde F_h^{\mathrm{td}}=0$ and
$\langle U_{k-1}-\widetilde U_{k-1},\mathbf 1\rangle_h=0$ for
$k=1,\dots,N_T$, then $Y=\widetilde Y$. Indeed, in this case, the FP residuals
vanish, and each HJB residual is a multiple of $\mathbf 1$. Since every density
difference $M_k-\widetilde M_k$ has zero mass, these constants disappear in the
pairing with $Y-\widetilde Y$. Hence
\[
0=
\bigl\langle
\widetilde F_h^{\mathrm{td}}(Y)-\widetilde F_h^{\mathrm{td}}(\widetilde Y),
Y-\widetilde Y
\bigr\rangle_{h,\Delta t}
=
\bigl\langle
F_h^{\mathrm{td}}(Y)-F_h^{\mathrm{td}}(\widetilde Y),
Y-\widetilde Y
\bigr\rangle_{h,\Delta t}.
\]
By Proposition~\ref{prop:Fh_monotone_MU}, this forces
$Y=\widetilde Y$.

Next, we establish global existence for \eqref{eq:disc_HRF_blocks} together with positivity, mass preservation, and invariance of the slice-wise means. Our argument starts from the standard well-posedness theory for ODEs, which requires the following assumption.



\begin{assumption}[Additional regularity of the numerical Hamiltonian]
\label{ass:disc_terms_regularity}
Assume that $g:\mathbb{T}^2\times\mathbb{R}^4\to\mathbb{R}$ is continuous in $x$, differentiable in $q$, and that $\nabla_q g(x,\cdot)$ is locally Lipschitz in $q$, uniformly in $x$.
\end{assumption}
This assumption is only used to ensure that the right-hand side of the HRF is
locally Lipschitz as a finite-dimensional ODE. 

The following proposition states the basic properties of the HRF. Starting from
a feasible initial datum, the flow exists for all artificial times, the density
stays positive, the mass of each density slice remains equal to one, and the
average of each value-function slice does not change.

\begin{proposition}[Global existence, positivity, mass preservation, and mean invariance]
\label{prop:global_exist_pos_mass}
Suppose that Assumptions~\ref{ass:disc_g}, \ref{ass:disc_f}, and
\ref{ass:disc_terms_regularity} hold. Let $Y(\cdot)$ be the solution of
\eqref{eq:disc_HRF_blocks} with feasible initial datum $Y(0)\in\mathcal M_h$.
Assume that there exists $Y^\dagger\in\mathcal M_h$ such that
$\widetilde F_h^{\mathrm{td}}(Y^\dagger)=0$ and
$\langle U^\dagger_{k-1},\mathbf 1\rangle_h
=
\langle U_{k-1}(0),\mathbf 1\rangle_h$ for $k=1,\dots,N_T$.
Then $Y(\cdot)$ exists for all $s\ge 0$. Moreover, for every $s\ge 0$ and
$k=1,\dots,N_T$, we have $M_k(s)\in\mathcal M_{++,h}$,
$\langle M_k(s),\mathbf 1\rangle_h=1$, and
$\langle U_{k-1}(s),\mathbf 1\rangle_h
=
\langle U_{k-1}(0),\mathbf 1\rangle_h$.
\end{proposition}


For later use, we define the Bregman divergence associated with
$\mathcal E$ in \eqref{eq:disc_entropy_h} by
\[
D_{\mathcal E}(Y^\dagger,Y)
:=
\mathcal E(Y^\dagger)-\mathcal E(Y)
-\bigl\langle \nabla\mathcal E(Y),Y^\dagger-Y\bigr\rangle_{h,\Delta t}.
\]
Here $\nabla\mathcal E$ is the gradient with respect to the pairing
$\langle\cdot,\cdot\rangle_{h,\Delta t}$. Since $\mathcal E$ is strictly convex
on $\mathcal M_h$, we have $D_{\mathcal E}(Y^\dagger,Y)\ge 0$, and equality
holds only when $Y=Y^\dagger$.

The next theorem proves convergence, assuming that there exists a point
$Y^\dagger$ with $\widetilde F_h^{\mathrm{td}}(Y^\dagger)=0$ and with the same
value-function averages as the initial datum.

\begin{theorem}[Convergence of the HRF]
\label{thm:global_conv_discrete}
Suppose that Assumptions~\ref{ass:disc_g}, \ref{ass:disc_f}, and
\ref{ass:disc_terms_regularity} hold. Let $Y(\cdot)$ be the solution of
\eqref{eq:disc_HRF_blocks} with feasible initial datum $Y(0)\in\mathcal M_h$.
Assume that there exists $Y^\dagger\in\mathcal M_h$ such that
$\widetilde F_h^{\mathrm{td}}(Y^\dagger)=0$ and
$\langle U^\dagger_{k-1},\mathbf 1\rangle_h
=
\langle U_{k-1}(0),\mathbf 1\rangle_h$ for $k=1,\dots,N_T$.
Then this $Y^\dagger$ is the only point in $\mathcal M_h$ with these same
value-function averages and with $\widetilde F_h^{\mathrm{td}}=0$. Moreover, the solution $Y(\cdot)$ exists for all $s\ge 0$, the Bregman
divergence $D_{\mathcal E}(Y^\dagger,Y(s))$ is nonincreasing on $[0,\infty)$,
and
\[
\widetilde F_h^{\mathrm{td}}(Y(s))\to 0,
\qquad
Y(s)\to Y^\dagger
\qquad\text{as }s\to\infty.
\]
\end{theorem}

The assumption above only says that there is a point $Y^\dagger\in\mathcal M_h$
such that $\widetilde F_h^{\mathrm{td}}(Y^\dagger)=0$ and
$\langle U^\dagger_k,\mathbf 1\rangle_h=\langle U_k(0),\mathbf 1\rangle_h$ for
$k=0,\dots,N_T-1$. This assumption holds if the discretized MFG system has a
solution.  The HRF convergence result gives a state with \(\widetilde F_h^{\mathrm{td}}(Y)=0\), where each HJB residual may still contain a space-independent constant. The next remark shows how to recover the solution of the discretized MFG system from $Y^\dagger$ by shifting the value function in $Y^\dagger$ on each time slice, so that the corrected variable \(\widehat Y\) satisfies \(F_h^{\mathrm{td}}(\widehat Y)=0\).


\begin{remark}[Recovering a full discrete solution by shifting the value function]
\label{rem:td-postprocess-full}
The HRF may converge to a state for which the FP equations are already satisfied, while the HJB equations are satisfied only up to a space-independent constant on each time slice. This can be fixed by adding one constant to $U_k$ at each time slice.
Suppose that
\[
Y=(M_1,\dots,M_{N_T}, U_0,\dots,U_{N_T-1})\in\mathcal M_h
\]
satisfies
\[
\widetilde F_h^{\mathrm{td}}(Y)=0 .
\]
Then, for each $k=1,\dots,N_T$, we have
\[
s_k(Y)=0,\qquad r_k(Y)=a_k\mathbf 1,
\]
where $a_k:=\bar r_k(Y)$. Thus, the FP equations are already satisfied, while the HJB residuals may contain only space-independent constants on each time slice.

These constants can be removed by adding suitable constants to the value function. Set $c_{N_T}:=0$ for the fixed terminal slice and define
\[
c_{k-1}=c_k+\Delta t\,a_k,\qquad k=N_T,N_T-1,\dots,1.
\]
Now define
\[
\widehat M_k:=M_k,\qquad 
\widehat U_{k-1}:=U_{k-1}+c_{k-1}\mathbf 1 .
\]
This change does not affect the density or the spatial derivatives of the value function. In particular, since
\[
D_h\mathbf 1=0,\qquad \Delta_h\mathbf 1=0,
\]
the FP residuals are unchanged:
\[
s_k(\widehat Y)=s_k(Y)=0.
\]
For the HJB residuals, the only change comes from the time-difference term. Therefore,
\[
r_k(\widehat Y)
=
r_k(Y)+\frac{c_k-c_{k-1}}{\Delta t}\mathbf 1
=
a_k\mathbf 1-a_k\mathbf 1
=0 .
\]
Hence
\[
F_h^{\mathrm{td}}(\widehat Y)=0 .
\]
In other words, if $\widetilde F_h^{\mathrm{td}}(Y)=0$, then the FP residuals already vanish, and each HJB residual is only a spatial constant. By adding the constants $c_k$ defined above to the value function slices, these remaining HJB constants are canceled, and the corrected pair $\widehat Y$ solves the full discrete MFG system.
\end{remark}

For the numerical experiments in \Cref{sec:numerics}, we approximate the dynamics \eqref{eq:disc_HRF_blocks} by an implicit Euler discretization in the flow variable. The convergence theorem above concerns the continuous artificial-time HRF itself. The implicit Euler iteration is used numerically as a flow integrator to solve the HRF.

\section{A Solver-Agnostic Framework for Inverse Problems}
\label{sec:inverse_stationary}

This section introduces a solver-agnostic optimization framework for inverse problems in MFGs. We illustrate the approach on the stationary and time-dependent MFG models in \eqref{eq:intro_Ergodic_MFG} and \eqref{eq:intro_TimeDep_MFG}. The main task is to recover $m$, $u$, and the spatial cost $V$ from noisy measurements of $m$ and $V$.

We formulate the inverse problem by treating $V$ as the unknown parameter. For each given $V$, we solve the corresponding MFG system to obtain the state variables $(m,u)$ (or $(m,u,\lambda)$ in the stationary case), and then evaluate the objective function. Thus, the state variables are not treated as independent optimization variables but are determined by the MFG equations associated with $V$.

A standard implementation updates the unknown coefficients $V$ by a gradient-based method. However, computing the gradient of the outer objective with respect to $V$ by differentiating through a forward solver can be expensive and tightly coupled to the particular numerical method used to solve the MFG system. Instead, we treat the discretized MFG system as the defining constraint for the MFG state $(m, u)$ and compute the gradient of the outer objective with respect to $V$ by implicit differentiation of these constraints. This yields adjoint-type problems whose solutions provide the desired gradients without unrolling or differentiating through the internal iterations of the forward solver. Building on this implicit differentiation framework, we further derive a GN method that incorporates second-order information to accelerate the outer optimization.

\subsection{Stationary and Time-Dependent Inverse Problems}
\label{subsec:inverse-problems}
We begin by formulating the inverse problems associated with stationary and time-dependent MFGs.

\subsubsection{Inverse Problem for Stationary MFGs} For ease of presentation, we consider the prototype MFG system \eqref{eq:intro_Ergodic_MFG}. 
Here, we consider an inverse problem in which we infer both the agents' equilibrium strategies and the spatial cost from limited, noisy observations of the population distribution and the cost.
Specifically, we consider the following problem.

\begin{problem}[Stationary MFG inverse problem]\label{prob:st_prob_inverse}
Given a spatial cost function $V^*$ and known components $(H,\nu,f)$, suppose that the stationary MFG~\eqref{eq:intro_Ergodic_MFG} admits a unique classical solution $(u^*,m^*,\lambda^*)$. Assume further that, in practice, we only have noisy observations of $m^*$ and $V^*$ at a finite number of points, and our aim is to recover the full configuration $(u^*,m^*,\lambda^*,V^*)$.

More precisely, we assume:
\begin{enumerate}
  \item \textbf{Noisy partial observations of $m^*$.}
  Let $\{\phi_\ell^o\}_{\ell=1}^{N_m}$ be bounded linear observation functionals. We observe the vector
  \[
    \mathbf m^o
    := \bigl(\langle \phi_1^o, m^*\rangle,\dots,\langle \phi_{N_m}^o, m^*\rangle\bigr) + \boldsymbol\epsilon_m
    \in \mathbb R^{N_m},
  \]
  where $\boldsymbol\epsilon_m \in \mathbb R^{N_m}$ is a Gaussian noise vector satisfying $\boldsymbol\epsilon_m \sim \mathcal N\!\bigl(\mathbf 0,\sigma_m^2 I_{N_m}\bigr)$.
  Here $\sigma_m>0$ denotes the noise standard deviation, and $I_{N_m}$ is the $N_m\times N_m$ identity matrix.

  \item \textbf{Noisy partial observations of $V^*$.}
  Let $\{\psi_\ell^o\}_{\ell=1}^{N_v}$ be bounded linear observation functionals acting on $V^*$. We observe
  \[
    \mathbf V^o
    := \bigl(\langle \psi_1^o, V^*\rangle,\dots,\langle \psi_{N_v}^o, V^*\rangle\bigr) + \boldsymbol\epsilon_v
    \in \mathbb R^{N_v},
  \]
  where $\boldsymbol\epsilon_v \in \mathbb R^{N_v}$ and $
    \boldsymbol\epsilon_v \sim \mathcal N\!\bigl(\mathbf 0,\sigma_V^2 I_{N_v}\bigr),$
  with $\sigma_V>0$.
\end{enumerate}
\medskip

\textbf{Inverse Problem.}
Assuming that the agents are governed by the stationary MFG \eqref{eq:intro_Ergodic_MFG}, and given the noisy data $\mathbf m^o$ and $\mathbf V^o$, our goal is to recover the full configuration $(u^*,m^*,\lambda^*,V^*)$.
\end{problem}

The above inverse problem is posed in the continuous setting. For computation, we discretize the MFG system and work with the corresponding discretized unknowns and observations. For illustration, we work on the two-dimensional torus $\mathbb{T}^2$ and discretize the MFG using the standard finite-difference scheme of \cite{achdou2010mean}, as described in \Cref{subsec:discrete_HRF} and \cite{gomes2020hessian}. Let $\mathbb{T}^2_h$ be the toroidal grid with mesh size $h$ and periodic indices, and let $N_x:=|\mathbb{T}^2_h|$. For a grid function \(y\in\mathbb{R}^{N_x}\), identified with the nodal values \((y_{i,j})\), we use the same one-sided difference operators, discrete gradient stencil, and standard five-point Laplacian as defined in \eqref{eq:one-side-diff}, \eqref{eq:stcil}, and \eqref{eq:five-point-laplacian}.
Let $g(x,q)$ be a Godunov numerical Hamiltonian approximating $H(x,p)$, with partial derivatives $g_{q_k}$ evaluated at $(x_{i,j},[D_hU]_{i,j})$. Sampling the fields pointwise on $\mathbb{T}^2_h$, we represent
\[
M\in\mathbb{R}^{N_x}\approx m(\cdot),\qquad
U\in\mathbb{R}^{N_x}\approx u(\cdot),\qquad
V_h\in\mathbb{R}^{N_x}\approx V(\cdot).
\]
Define the stationary residual blocks
\[
\mathcal R^{(m)}_{i,j}(M,U;V_h,\lambda)
:=\nu(\Delta_h U)_{i,j}-g(x_{i,j},[D_h U]_{i,j})+f(M_{i,j})+(V_h)_{i,j}+\lambda,
\]
\[
\mathcal R^{(u)}_{i,j}(M,U)
:= -\nu(\Delta_h M)_{i,j}
-\Big(D_1^{-}(M\,g_{q_1})\Big)_{i,j}
-\Big(D_1^{+}(M\,g_{q_2})\Big)_{i,j}
-\Big(D_2^{-}(M\,g_{q_3})\Big)_{i,j}
-\Big(D_2^{+}(M\,g_{q_4})\Big)_{i,j}.
\]
We use the discrete inner product $\langle a,b\rangle_h:=h^2\sum_{i,j}a_{i,j}b_{i,j}$. The ergodic constant $\lambda$ is eliminated by enforcing the discrete compatibility condition $\langle M,\mathcal R^{(m)}(M,U;V_h,\lambda)\rangle_h=0$, i.e.,
\[
\lambda(M,U,V_h)
=
-\frac{\big\langle M,\ \nu\Delta_h U-g(\cdot,[D_hU]) + f(M)+V_h\big\rangle_h}{\langle M,\mathbf 1\rangle_h}.
\]
For positive density vectors $M\in\mathbb{R}_{++}^{N_x}$, we then define the discretized stationary MFG residual operator
\[
F_h^{\mathrm{st}}(M,U,V_h)
:=
\big(\mathcal R^{(m)}(M,U;V_h,\lambda(M,U,V_h)),\ \mathcal R^{(u)}(M,U)\big)
\in \mathbb{R}^{N_x}\times \mathbb{R}^{N_x},
\]together with the normalization constraints
\[
\langle M,\mathbf 1\rangle_h=1,\qquad \langle U,\mathbf 1\rangle_h=0.
\]

We denote the admissible stationary density set by
\[
\mathcal M_{++,h}^{\mathrm{st}}:=\bigl\{M\in\mathbb{R}_{++}^{N_x}:\ \langle M,\mathbf 1\rangle_h=1\bigr\}.
\]
Hence, for a given $V_h$, the discrete forward problem is to find $(M,U)$ such that
\begin{equation}
\label{eq:st_forward_problem_discretized}
F_h^{\mathrm{st}}(M,U,V_h)=0,\qquad M\in\mathcal M_{++,h}^{\mathrm{st}},\qquad \langle U,\mathbf 1\rangle_h=0.
\end{equation}

The inverse problem is intrinsically ill-posed: finitely many noisy observations do not determine a unique spatial cost or equilibrium state without additional regularization. Meanwhile, in the stationary model, there is also an identifiability issue between $V_h$ and $\lambda$. Indeed, adding a constant to $V_h$ and subtracting the same constant from $\lambda$ leaves the HJB residual unchanged. Consequently, the MFG equations alone determine $V_h$ only up to an additive constant. The observations of $V$ and the RKHS regularization provide additional information that helps distinguish between such equivalent solutions. We therefore use RKHS reconstructions for $(m,u,V)$ while keeping the forward MFG solve fully discrete, and we regularize the cost variable \(V\) through the corresponding RKHS norm. A convenient and flexible way to encode smoothness and structural assumptions is to place $m$, $u$, and $V$ in RKHSs \cite{owhadi2019operator}, whose evaluation functionals are continuous. This choice also provides a systematic link between the discrete unknowns on \(\mathbb{T}^2_h\) and continuous objects on \(\mathbb{T}^2\). The grid vectors \((M,U,V_h)\) are interpreted as nodal samples of RKHS functions and are lifted to continuous functions through kernel reconstruction. In this way, continuous observation functionals can be applied consistently to the discretized variables.
Thus, we assume that \(m\in\mathcal M\), \(u\in\mathcal U\), and \(V\in\mathcal V\), where \(\mathcal M\), \(\mathcal U\), and \(\mathcal V\) are RKHSs on \(\mathbb{T}^2\) with kernels \(K_m\), \(K_u\), and \(K_V\), respectively. Then, we use the standard optimal-recovery reconstruction. Fix a strictly positive definite kernel $K$ on $\mathbb{T}^2$ with RKHS $\mathcal{H}_K$, and let $X=\{x_\ell\}_{\ell=1}^{N_x}\subset\mathbb{T}^2$ denote the grid points.  Given nodal values $Z\in\mathbb{R}^{N_x}$, we define $\mathcal R_h^{K}(Z)\in\mathcal{H}_K$ as the unique minimizer of
\begin{align}
\label{eq:stoptrvy}
\min_{w\in\mathcal{H}_K}\ \|w\|_{\mathcal{H}_K}^2, \quad\text{such that}\quad
w(x_\ell)=Z_\ell,\qquad \ell=1,\dots,N_x.
\end{align}
By the representer theorem \cite{owhadi2019operator}, the minimizer has the closed form
\[
\mathcal R_h^{K}(Z)(x)
=
K(x,X)K(X,X)^{-1}Z,
\qquad x\in\mathbb{T}^2,
\]
where $K(X,X)\in\mathbb{R}^{N_x\times N_x}$ is the Gram matrix and $K(x,X)\in\mathbb{R}^{1\times N_x}$ is the row vector of kernel evaluations at $x$. Since the grid points considered here are distinct, the positive definiteness of the kernel implies that the Gram matrix \(K(X,X)\) is invertible. Thus, given the grid values $M$, $U$, and $V_h$ corresponding to $m$, $u$, and $V$, respectively, we reconstruct continuous representatives via
\[
m^\dagger:=\mathcal R_h^{K_m}(M)\in\mathcal M,\qquad
u^\dagger:=\mathcal R_h^{K_u}(U)\in\mathcal U,\qquad
V^\dagger:=\mathcal R_h^{K_V}(V_h)\in\mathcal V.
\]
Moreover, the RKHS norm of the reconstructed $V^\dagger$ admits an explicit expression, i.e., 
\[
\|V^\dagger\|_{\mathcal V}^2
=
\bigl\langle V^\dagger, V^\dagger\bigr\rangle_{\mathcal V}
=
V_h^{\top}K_V(X,X)^{-1}V_h,
\]
where $\langle\cdot,\cdot\rangle_{\mathcal V}$ denotes the inner product in the RKHS $\mathcal V$ associated with the kernel $K_V$. In practice, one may replace the norm by  \(V_h^\top\bigl(K_V(X,X)+\epsilon I\bigr)^{-1}V_h\) with a very small nugget \(\epsilon>0\) for numerical stability.

Using these reconstructions, we can apply the linear observation functionals stated in \Cref{prob:st_prob_inverse} to the continuous surrogates, rather than restricting observations to the uniform grid $\mathbb{T}^2_h$.  We define the induced observation maps
\[
\Phi_h^m(M):=\big(\langle \phi_1^o,m^\dagger\rangle,\dots,\langle \phi_{N_m}^o,m^\dagger\rangle\big)\in\mathbb{R}^{N_m},
\qquad
\Phi_h^V(V_h):=\big(\langle \psi_1^o,V^\dagger\rangle,\dots,\langle \psi_{N_v}^o,V^\dagger\rangle\big)\in\mathbb{R}^{N_v}.
\]
The stationary inverse problem is then formulated as the PDE-constrained optimization problem
\begin{equation}
\label{eq:OptGPProb_st}
\min_{V_h\in\mathbb{R}^{N_x},\,(M,U)}
\ \frac{\alpha}{2}\|V^\dagger\|_{\mathcal V}^2
+\frac{\beta}{2}\big\|\Phi_h^m(M)-\mathbf m^o\big\|_2^2
+\frac{\gamma}{2}\big\|\Phi_h^V(V_h)-\mathbf V^o\big\|_2^2
\quad
\text{s.t.}\quad
(M,U)\ \text{satisfies }\eqref{eq:st_forward_problem_discretized},
\end{equation}
where \(\alpha>0\) is the RKHS prior weight. Under an exact isotropic Gaussian
likelihood interpretation, one may take \(\beta=\sigma_m^{-2}\) and
\(\gamma=\sigma_V^{-2}\). In the numerical experiments, however, we treat
\(\beta\) and \(\gamma\) as tunable balancing weights. The ergodic constant \(\lambda\) is understood as
\(\lambda(M,U,V_h)\) defined above.

\begin{remark}[RKHS surrogates for continuous observations]\label{rem:rkhs_surrogates}
In \Cref{prob:st_prob_inverse}, observations are formulated as continuous linear functionals acting on the underlying fields (e.g., $m$ and $V$), whereas the forward MFG solve produces only grid unknowns such as $(M,U,V_h)$. To apply continuous observation functionals to these discrete outputs, we introduce continuous representatives of the grid data. A standard alternative is a finite-element type reconstruction (e.g., piecewise polynomial interpolation). Here we instead use optimal recovery \eqref{eq:stoptrvy}.

This choice is convenient because it provides a systematic, kernel-parameterized family of smooth function classes in which general linear observations are well defined. Throughout, the observation maps are assumed to be bounded linear functionals on the chosen RKHSs. Point evaluations are a standard example when the kernel is continuous. Moreover, for common kernel families on $\mathbb{T}^2$ (in particular Mat\'ern kernels), the associated RKHS is norm-equivalent to a Sobolev space $H^s(\mathbb{T}^2)$ for some $s>0$ (see, e.g., \cite{wendland2004scattered}). In particular, choosing $s$ above the Sobolev embedding threshold (e.g., $s>1$ in two dimensions) ensures that elements admit continuous representatives, so pointwise and other continuous observations are meaningful, and the resulting regularity is consistent with the smoothness typically assumed for classical MFG solutions.

We emphasize that the RKHS assumption is not needed to solve the discretized
MFG system. The RKHS representative is used only to evaluate observation
functionals. The time-dependent inverse formulation adopts the same viewpoint, and we do not repeat this discussion there.
\end{remark}

\subsubsection{Inverse Problem for Time-Dependent MFGs}
We now formulate the analogous inverse problem for the time-dependent MFG \eqref{eq:intro_TimeDep_MFG}.

\begin{problem}[Time-dependent MFG inverse problem]\label{prob:td_prob}
Let $V^*$ be a spatial cost and suppose that the components $(H,\nu,f, m_0, u_T)$ are known. For a given $V^*$, assume the time-dependent MFG system~\eqref{eq:intro_TimeDep_MFG} admits a unique classical solution $(u^*,m^*)$. In practice, suppose we have access to partial and noisy observations of $m^*$ and $V^*$ and wish to reconstruct the associated state and spatial cost. To elaborate, we have
\begin{enumerate}
  \item \textbf{Noisy partial observations of $m^*$.}
  Let $\{\phi_\ell^o\}_{\ell=1}^{N_m}$ be bounded linear observation functionals and denote
  \[
    \mathbf m^o
    = \bigl(\langle \phi_1^o, m^*\rangle,\dots,\langle \phi_{N_m}^o, m^*\rangle\bigr)
      + \boldsymbol\epsilon_m,
    \qquad \boldsymbol\epsilon_m \sim \mathcal N\!\bigl(\mathbf 0,\sigma_m^2 I_{N_m}\bigr).
  \]
  Here $\sigma_m>0$ is the noise standard deviation and $\langle\cdot,\cdot\rangle$ denotes the data--state pairing.
  \item \textbf{Noisy partial observations of $V^*$.}
  Let $\{\psi_\ell^o\}_{\ell=1}^{N_v}$ be bounded linear observation functionals for $V^*$ and define
  \[
    \mathbf V^o
    = \bigl(\langle \psi_1^o, V^*\rangle,\dots,\langle \psi_{N_v}^o, V^*\rangle\bigr)
      + \boldsymbol\epsilon_v,
    \qquad \boldsymbol\epsilon_v \sim \mathcal N\!\bigl(\mathbf 0,\sigma_V^2 I_{N_v}\bigr),
  \]
  with noise level $\sigma_V>0$.
\end{enumerate}

\medskip
\textbf{Inverse Problem.}
Suppose that agents are governed by the time-dependent MFG in \eqref{eq:intro_TimeDep_MFG}. Given $\mathbf m^o$ and $\mathbf V^o$, we want to recover the unknown spatial cost together with the associated state.
\end{problem}


Similar to the stationary case, we work with a fully discretized formulation for computation. We use the finite difference scheme of \Cref{subsec:discrete_HRF}. In particular, we use the space-time grid $(\mathbb{T}^2_h,\{t_k\}_{k=0}^{N_T})$. The mixed endpoint data are sampled on the grid and fixed as $(M_0,U_{N_T})$, and the interior unknowns are collected as
\[
Y:=(M_1,\dots,M_{N_T},\,U_0,\dots,U_{N_T-1})\in\mathcal{X}_h,
\qquad
V_h\in\mathbb{R}^{N_x}.
\]
To make the dependence on $V_h$ explicit, we write the discretized residual  as
$F_h^{\mathrm{td}}(\,\cdot\,;V_h):\mathcal{X}_h\to\mathcal{X}_h$.
In this convention, the fixed endpoint slices $(M_0,U_{N_T})$ are inserted whenever they appear in the time-slice blocks. For the time-dependent inverse problem, we use the discretized MFG system as the discrete forward constraint:
\begin{equation}
\label{eq:td_forward_problem_discretized}
 F_h^{\mathrm{td}}(Y;V_h)=0,
\qquad
Y\in\mathcal{M}_h .
\end{equation}
Here \(\mathcal{M}_h\) is the interior feasible manifold defined in
\Cref{subsec:discrete_HRF}.



{


{
}

}

To apply continuous observation functionals to the discrete unknowns, we construct continuous surrogates from the grid values using kernel optimal recovery. We place the space-time fields $m$ and $u$ in RKHSs on $\mathbb{T}^2\times[0,T]$ with kernels $K_m$ and $K_u$, and we place the spatial cost $V$ in an RKHS on $\mathbb{T}^2$ with kernel $K_V$. Let $X=\{x_\ell\}_{\ell=1}^{N_x}\subset\mathbb{T}^2$ be the spatial grid points and let $\mathcal{T}_h:=\{t_k\}_{k=0}^{N_T}$ be the time grid, and set $S:=X\times \mathcal{T}_h$.

For a space-time kernel $K$ such that the Gram matrix $K(S,S)$ is invertible and a space-time sample vector $w\in\mathbb{R}^{N_x(N_T+1)}$, the analogous optimal-recovery problem to \eqref{eq:stoptrvy} yields the representer formula
\[
\mathcal{R}^{K}_{h,\Delta t}(w)(x,t)
:=
K\bigl((x,t),S\bigr)K(S,S)^{-1}w,
\qquad (x,t)\in\mathbb{T}^2\times[0,T].
\]
Likewise, for a spatial kernel $K$ such that the Gram matrix $K(X,X)$ is invertible and a spatial sample vector $v\in\mathbb{R}^{N_x}$
(stacked over the spatial grid $X:=\{x_\ell\}_{\ell=1}^{N_x}$), define
\[
\mathcal{R}^{K}_{h}(v)(x)
:=
K(x,X)K(X,X)^{-1}v,
\qquad x\in\mathbb{T}^2. 
\]
Applying these maps to the discrete unknowns yields the continuous surrogates
\[
m^\dagger:=\mathcal{R}^{K_m}_{h,\Delta t}\!\bigl(\mathbf m_h\bigr),\qquad
u^\dagger:=\mathcal{R}^{K_u}_{h,\Delta t}\!\bigl(\mathbf u_h\bigr),\qquad
V^\dagger:=\mathcal{R}^{K_V}_{h}(V_h),
\]
where $\mathbf m_h:=\bigl(M_0,\dots,M_{N_T}\bigr)\in\mathbb{R}^{N_x(N_T+1)}$ and
$\mathbf u_h:=\bigl(U_0,\dots,U_{N_T}\bigr)\in\mathbb{R}^{N_x(N_T+1)}$ denote the stacked space-time samples.
Moreover, $V^\dagger$ satisfies the explicit norm identity
\[
\|V^\dagger\|_{\mathcal V}^2
=
V_h^{\top}K_V(X,X)^{-1}V_h,
\]
which follows directly from the representer form.

Using these reconstructions, we can apply the linear observation functionals to the continuous surrogates. We define the induced observation maps
\[
\Phi^{m}_{h,\Delta t}(Y)
:=
\bigl(\langle \phi^{o}_1, m^\dagger\rangle,\dots,\langle \phi^{o}_{N_m}, m^\dagger\rangle\bigr)\in\mathbb{R}^{N_m},
\qquad
\Phi^{V}_{h}(V_h)
:=
\bigl(\langle \psi^{o}_1, V^\dagger\rangle,\dots,\langle \psi^{o}_{N_v}, V^\dagger\rangle\bigr)\in\mathbb{R}^{N_v}.
\]

The time-dependent inverse problem is then posed as the PDE-constrained optimization problem
\begin{equation}
\label{eq:OptGPProb_td}
\min_{V_h\in\mathbb{R}^{N_x},\,Y}
\ \frac{\alpha}{2}\|V^\dagger\|_{\mathcal V}^2
+\frac{\beta}{2}\big\|\Phi^{m}_{h,\Delta t}(Y)-\mathbf m^o\big\|_2^2
+\frac{\gamma}{2}\big\|\Phi^{V}_{h}(V_h)-\mathbf V^o\big\|_2^2
\quad
\text{s.t.}\quad
Y\ \text{satisfies }\eqref{eq:td_forward_problem_discretized},
\end{equation}
Here, $\alpha$, $\beta$, and $\gamma$ are positive parameters balancing the regularization term and the data-misfit terms. The second and third terms measure the discrepancy between the model predictions and the available observations of the density and the spatial cost, respectively.



\subsection{A Solver-Agnostic Inverse Framework}
\label{sec:inverse-adjoint-weighted}
Our next goal is to solve \eqref{eq:OptGPProb_st} and \eqref{eq:OptGPProb_td} numerically. To this end, we separate the outer optimization from the particular numerical method used for the inner MFG solve. The key idea is to treat the discretized MFG system as the defining constraint for the inner solution: the inner solver returns an accurate state satisfying the discretized MFG equations, and these equations are the objects we differentiate.  Accordingly, derivatives of the outer objective with respect to the unknown parameters are obtained by implicit differentiation of the discretized MFG constraints, rather than by differentiating through the iterations of a specific inner algorithm. Since \eqref{eq:OptGPProb_st} and \eqref{eq:OptGPProb_td} share the same structure from this viewpoint, we present the method in a unified template.


Let $\Theta\subset\mathbb R^\ell$ be the set of admissible parameters, and let
$\theta\in\Theta$ be the unknown parameter. For example, $\theta$ may be the
discretized spatial cost in an inverse MFG problem. After discretization, the
inner MFG problem is a finite-dimensional system for the state variable
$z\in\mathbb R^n$. We write the equations in the form
\begin{equation}
	F(z,\theta)=0,\qquad \mathsf A z=\mathsf b .
	\label{eq:eq-constraints-linear}
\end{equation}
Here $F:\mathbb R^n\times\mathbb R^\ell\to\mathbb R^c$ contains the nonlinear
equations from the discretized MFG system, and $\mathsf A z=\mathsf b$ contains
linear constraints that we keep separate, such as the mass constraint, the
normalization of the value function, or fixed boundary data. If there are no
separate linear constraints, the second equation is absent.

In the stationary MFG setting, the linear constraints $\mathsf{A}z=\mathsf{b}$ encode standard normalization and gauge conditions. For instance, if $z=(M,U)$ collects the spatial grid unknowns, one may impose
\[
\frac{1}{N_x}\sum_i M_i=1,
\qquad
\sum_i U_i=0,
\]
which can be written compactly in the form $\mathsf{A}z=\mathsf{b}$.


In the time-dependent MFG setting, the mixed endpoint data are imposed after discretization as fixed values on finitely many time slices. Following \Cref{subsec:discrete_HRF}, we enforce these endpoint constraints by eliminating the fixed endpoint blocks and working only with the interior unknowns. In the resulting reduced system, the endpoint conditions are already built into the parametrization and therefore do not appear as additional explicit equalities.

Now, we define a unified objective that matches the formulations in \eqref{eq:OptGPProb_st} and \eqref{eq:OptGPProb_td}. Let $z^{\mathrm o}\in\mathbb{R}^{m}$ and $\theta^{\mathrm o}\in\mathbb{R}^{r}$ denote given data vectors, and let
\[
W\in\mathbb{R}^{m\times n},
\qquad
G\in\mathbb{R}^{r\times \ell},
\qquad
J\in\mathbb{R}^{\ell\times \ell}
\]
be linear maps defining the observation and regularization operators. We consider the  objective
\begin{equation}
\mathcal{J}(\theta)
= \frac{\alpha}{2}\,\big\|J\,\theta\big\|_2^2  +
\frac{\beta}{2}\,\big\| W\,z^*(\theta)-z^{\mathrm o}\big\|_2^2
+
\frac{\gamma}{2}\,\big\| G\,\theta-\theta^{\mathrm o}\big\|_2^2,
\qquad
\theta\in \Theta,
\label{eq:reduced-objective}
\end{equation}
with weights $\alpha\ge 0$ and $\beta,\gamma>0$. In the stationary inverse problem, $W$ represents the discrete observation of the equilibrium density, $G$ represents the discrete observation of the cost $V$, and $J$ represents the RKHS regularizer.  In our setting, the parameter \(\theta\) is the grid vector \(V_h\) corresponding to the spatial cost \(V\). The term \(\|J\theta\|_2^2\) is therefore chosen so as to approximate the RKHS norm of the continuous function reconstructed from \(V_h\). Concretely, let \(X\) denote the spatial grid, and let \(K_V(X,X)\) be the kernel matrix of the reproducing kernel \(K_V\) evaluated on \(X\). Then a natural choice is
\[
J^\top J = K_V(X,X)^{-1}.
\]
Accordingly,
\[
\|JV_h\|_2^2 = V_h^\top K_V(X,X)^{-1}V_h,
\]
which equals \(\|V^\dagger\|_{\mathcal V}^2\), where \(V^\dagger\) is the RKHS interpolant associated with the grid values \(V_h\). 

We solve \eqref{eq:reduced-objective} by computing \(\nabla_\theta \mathcal{J}(\theta)\) without differentiating through a particular forward solver. The next result provides an adjoint formula that depends only on the discrete system \eqref{eq:eq-constraints-linear}. The formulas below are local to the parameter. Fix a parameter $\theta$ and let
$z^*(\theta)$ be the state computed by the inner solver. We assume that this
state satisfies the discretized inner system
\[
F(z^*(\theta),\theta)=0,\qquad \mathsf A z^*(\theta)=\mathsf b .
\]
To differentiate the outer objective at this parameter, we only use local
information near $(z^*(\theta),\theta)$. Namely, we need $F$ to be differentiable
there, the inner state to be locally differentiable with respect to $\theta$,
and the linear system used in the sensitivity or adjoint equation to be
solvable. These conditions are separate from the forward convergence result in \Cref{sec:tdmfg}. 

\begin{proposition}[Adjoint gradient formula]
\label{prop:adjoint_grad}
Fix a parameter $\theta\in\Theta$. Let $z^*(\theta)$ be the state computed by the
inner solver, satisfying
\[
F(z^*(\theta),\theta)=0,\qquad
\mathsf A z^*(\theta)=\mathsf b .
\]
Assume that $F$ is differentiable at $(z^*(\theta),\theta)$ and that
$z^*(\theta)$ is differentiable with respect to $\theta$ at this parameter.
Assume also that the adjoint equation
\begin{equation}
\frac{\partial F}{\partial z}\bigl(z^*(\theta),\theta\bigr)^\top \psi
+
\mathsf A^\top \eta
=
-\,\beta\,W^\top\bigl(W z^*(\theta)-z^{\mathrm o}\bigr)
\label{eq:adjoint}
\end{equation}
has a solution $(\psi,\eta)$, with the convention that $\eta$ and the
$\mathsf {A} ^\top\eta$ terms are omitted when there are no separate linear
constraints. Then
\begin{equation}
\nabla_\theta \mathcal J(\theta)
=
\frac{\partial F}{\partial \theta}\bigl(z^*(\theta),\theta\bigr)^\top \psi
+
\gamma\,G^\top\bigl(G\theta-\theta^{\mathrm o}\bigr)
+
\alpha\,J^\top J\,\theta .
\label{eq:grad-final}
\end{equation}
If the adjoint equation \eqref{eq:adjoint} has multiple solutions, then
\(F_\theta(z^*(\theta),\theta)^\top\psi\) is the same for every solution
\(\psi\) of \eqref{eq:adjoint}. Consequently, the gradient formula
\eqref{eq:grad-final} is independent of which solution \(\psi\) is used.
\end{proposition}

In the MFG
discretizations considered here, the linearized system \eqref{eq:adjoint} can be singular or
ill-conditioned, especially when the monotonicity is strict but not strong.
Nevertheless, the adjoint solve is stable in our numerical experiments and
produces descent directions for the outer optimization. When numerical
degeneracy is encountered, one can use a mild regularization of the adjoint solve,
for example, a minimum-norm least-squares solve or a small Tikhonov
regularization. These regularized solves should be understood as stable
numerical approximations of the adjoint equation, rather than as part of the
exact gradient formula in \Cref{prop:adjoint_grad}.

With this formulation, computing $\nabla_\theta \mathcal{J}(\theta)$ only requires that the inner solver returns a sufficiently accurate solution $z^*(\theta)$ of the discrete MFG constraints \eqref{eq:eq-constraints-linear}. No differentiation through the solver itself is needed.

\subsection{Gauss--Newton Method}
\label{sec:inverse-gn}

The adjoint-based gradient formula in \Cref{sec:inverse-adjoint-weighted} provides a first-order method for minimizing the objective \eqref{eq:reduced-objective} with respect to the parameter $\theta$. We also consider a GN method, which exploits the least-squares structure of the objective and typically produces a more effective parameter update than gradient descent (GD).

For the GN method, we first rewrite the  objective \eqref{eq:reduced-objective} as a least-squares
problem. Fix a parameter value $\theta$ for which the inner MFG system has been
solved, and denote the computed state by $z^*(\theta)$. We define
\begin{equation}
\label{eq:gn-residual-def}
r(\theta)
:=
\begin{bmatrix}
\sqrt{\beta}\,\bigl(W z^*(\theta)-z^{\mathrm o}\bigr)\\[2pt]
\sqrt{\gamma}\,\bigl(G\theta-\theta^{\mathrm o}\bigr)\\[2pt]
\sqrt{\alpha}\,J\,\theta
\end{bmatrix},
\qquad
\mathcal J(\theta)=\frac12\|r(\theta)\|_2^2 .
\end{equation}
Here $z^*(\theta)$ satisfies the inner system
\[
F(z^*(\theta),\theta)=0,\qquad
\mathsf A z^*(\theta)=\mathsf b .
\]


The GN method computes an increment $d\in\mathbb{R}^\ell$ by minimizing a quadratic model obtained from the first-order expansion $
r(\theta+d)\approx r(\theta)+\mathsf{J}_r(\theta)d$,
where $\mathsf{J}_r(\theta)$ denotes the Jacobian of the residual $r(\theta)$ with respect to $\theta$. Thus, the GN step is defined by the linearized least-squares problem
\begin{equation}
\label{eq:gn-linearized-subproblem}
\min_{d\in\mathbb{R}^\ell}
\frac12\bigl\|r(\theta)+\mathsf{J}_r(\theta)d\bigr\|_2^2.
\end{equation}

The only subtle point is the evaluation of the Jacobian action $\mathsf{J}_r(\theta)d$. Since the residual depends on $\theta$ both directly and through the equilibrium state $z^*(\theta)$, one must determine the first-order change in the equilibrium state induced by the parameter perturbation $d$.  The details of the GN update are given below.

\begin{proposition}[Gauss--Newton step]
\label{prop:gn_direction}
Fix a parameter $\theta\in\Theta$. Let $z^*(\theta)\in\mathbb R^n$ be the state
computed by the inner solver, satisfying
\[
F(z^*(\theta),\theta)=0,
\qquad
\mathsf A z^*(\theta)=\mathsf b .
\]
Assume that $F$ is differentiable at $(z^*(\theta),\theta)$ and that
$z^*(\theta)$ is differentiable with respect to $\theta$. For a parameter increment $d\in\mathbb R^\ell$, denote the
first-order change of the computed state by
$\delta z_d:=Dz^*(\theta)d$. Differentiating the inner system satisfied by
$z^*(\theta)$ gives
\begin{equation}
\label{eq:gn_sensitivity}
\frac{\partial F}{\partial z}\bigl(z^*(\theta),\theta\bigr)\,\delta z_d
+
\frac{\partial F}{\partial \theta}\bigl(z^*(\theta),\theta\bigr)\,d
=0,
\qquad
\mathsf A\,\delta z_d=0.
\end{equation}
In practice, if
there are no separate linear constraints, the equation
$\mathsf A\,\delta z_d=0$ is omitted. Hence, the first-order change of the
residual $r(\theta)$ in the direction $d$ is
\[
\mathsf J_r(\theta)d=
\begin{bmatrix}
\sqrt{\beta}\,W\,\delta z_d\\
\sqrt{\gamma}\,Gd\\
\sqrt{\alpha}\,Jd
\end{bmatrix}.
\]
The GN step is obtained by minimizing the linearized least-squares
problem \eqref{eq:gn-linearized-subproblem}. Its minimizers are characterized by
the normal equation
\begin{equation}
\label{eq:gn_lm_system}
\bigl(\mathsf J_r(\theta)^\top\mathsf J_r(\theta)\bigr)d
=
-\mathsf J_r(\theta)^\top r(\theta).
\end{equation}
\end{proposition}
The proof is given in \Cref{app:appendix}. Once the GN increment $d^k$ has been computed at $\theta^k$, we update
\[
\theta^{k+1}=\theta^k+d^k,
\]
and then recompute the equilibrium state by solving \eqref{eq:eq-constraints-linear} at $\theta^{k+1}$.

In \Cref{sec:numerics}, we observe that GN typically reaches comparable accuracy in fewer outer iterations than GD. Computationally, the GD method requires solving the adjoint system \eqref{eq:adjoint} to evaluate the gradient. The GN method instead relies on the sensitivity system \eqref{eq:gn_sensitivity}, and its additional cost comes from solving \eqref{eq:gn_lm_system}. This step can be carried out matrix-free using a Krylov method such as conjugate gradients, without explicitly forming \(\mathsf{J}_r(\theta)^\top \mathsf{J}_r(\theta)\).

\section{Numerical Results}
\label{sec:numerics}
In this section, we present numerical results for the HRF method and the solver-agnostic inverse framework. We study two solvers for the minimization of the inverse problems. The first is the adjoint-based GD method presented in \Cref{sec:inverse-adjoint-weighted}, which computes gradients via the adjoint method. The second is the GN method presented in \Cref{sec:inverse-gn}. Unless otherwise specified, all forward problems are solved by applying an
implicit Euler scheme to the HRF in the artificial-time variable. We present experiments on several stationary and time-dependent MFG examples. Two examples are worth highlighting. In \Cref{2DMFGexample2}, we consider a stationary non-potential MFG satisfying the Lasry--Lions monotonicity condition, demonstrating that our framework extends beyond the potential MFG setting. In \Cref{2DMFG_solver_free}, we apply the same inverse framework with different forward solvers and show that the reconstruction quality remains consistent across solvers, demonstrating the solver-agnostic property of our framework.

The convergence theorem in Section~\ref{sec:tdmfg} is proved for the
time-dependent MFG model with Hamiltonian \(H=H(x,p)\) and local
Lasry--Lions coupling \(f(m)\). The inverse framework in
Section~\ref{sec:inverse_stationary} is stated more generally, directly in terms
of the discretized MFG system. Therefore, in the inverse-problem experiments,
we also consider MFGs with congestion Hamiltonians and nonlocal couplings as
numerical tests of the inverse framework with the chosen inner solver. These examples are included to examine the numerical behavior of the HRF beyond the setting covered by the convergence result in Section~\ref{sec:tdmfg}. They suggest that the HRF remains effective for more general MFG models, although a complete convergence analysis for such models would require additional arguments.

We assess the mismatch between the recovered function and a reference function using a discrete \(L^2\) metric. Let \(u\) and \(v\) be functions on \([a,b]^2\), sampled on a uniform grid with spacings \(h_x\) and \(h_y\) in the \(x\)- and \(y\)-directions, producing arrays \(\{u_{ij}\}\) and \(\{v_{ij}\}\). The discrete \(L^2\) error is defined by
\begin{align}
\mathcal{E}_2(u,v)
:= \sqrt{\,h_x h_y \sum_{i,j} \bigl|u_{ij} - v_{ij}\bigr|^2}\,.
\end{align}
In the one-dimensional experiments, we use the  analogue $\sqrt{h_x\sum_i |u_i-v_i|^2}$.

In all experiments, the algorithm terminates when the absolute change in the objective value in \eqref{eq:reduced-objective} between two successive iterations falls below $10^{-6}$. 

\subsection{A One-Dimensional Stationary MFG Inverse Problem}
\label{1DMFGexample1}
In this subsection, we consider the inverse problem for the following one-dimensional stationary MFG posed on the torus \(\mathbb{T}\):
\begin{equation}
\begin{cases}
\frac{|P+D_xu|^2}{2}-V(x) = \lambda+\frac{1}{k} \ln m, 
& x \in \mathbb{T}, \\
-\!\left(m (P + D_xu)\right)_x=0, 
& x \in \mathbb{T},
\end{cases}
\label{eq:stationaryMFG_oneDim_inverse_effHam}
\end{equation}
where \(k=100\), \(P=2\), and the spatial cost is chosen as 
  $
  V(x)=\sin(2\pi x)\bigl(1+0.35\cos(4\pi x)\bigr).
  $ For this choice of data, we compute a reference solution \((u^*,m^*,\lambda^*)\) numerically with the HRF solver. In this experiment, our objective is to reconstruct \(u\), \(m\), \(\lambda\), and \(V\) in \eqref{eq:stationaryMFG_oneDim_inverse_effHam} from partial noisy observations of \(m\) and \(V\).

\textbf{Experimental Setup.}
  We identify \(\mathbb{T}\) with the interval \([0,1)\). We solve
  \eqref{eq:stationaryMFG_oneDim_inverse_effHam} using the HRF method of~\cite{gomes2020hessian} to compute
  the reference solution \((u^*,m^*,\lambda^*)\). The domain \(\mathbb{T}\) is discretized with grid spacing
  \(h_x=1/100\), yielding 100 grid points. We select 12 grid points as observation locations for \(m\) and
  12 grid points as observation locations for \(V\). The regularization parameters in
  \eqref{eq:OptGPProb_st} are set to \(\alpha=0.002\), \(\beta=2\), and \(\gamma=2\). We contaminate the
  observations with i.i.d.\ Gaussian noise \(\mathcal{N}(0,\eta^2 I)\) with \(\eta=10^{-3}\). To solve the
  inverse problem, we initialize \(u\equiv 0\), \(m\equiv 1\), and \(V\equiv 0\). We then apply the adjoint-based method and the GN method described in \Cref{sec:inverse-adjoint-weighted} and \Cref{sec:inverse-gn},
  respectively.
  

\textbf{Experimental Results.} 
Table~\ref{table:1D_ErgodicMFGeffHamilton} reports the HRF reference value of \(\lambda\) for \eqref{eq:stationaryMFG_oneDim_inverse_effHam} together with the values recovered by GD and GN. Figure~\ref{1D_ErgodicMFGwithEffectiveHamiltonianPlot} summarizes the numerical results: the HRF reference density \(m^*\) is shown in Figure~\ref{1D_ErgodicMFGwithEffectiveHamiltonianReferenceM}, with the GD and GN reconstructions in Figures~\ref{1D_ErgodicMFGwithEffectiveHamiltonianRecoverdM_GD} and~\ref{1D_ErgodicMFGwithEffectiveHamiltonianRecoverdM_GN}. The same figure also includes the corresponding reference and reconstructed profiles for \(u\) and \(V\). Reconstruction errors are reported as pointwise absolute errors, e.g., \(|m-m^*|\), with analogous curves for \(u\) and \(V\). Finally, Figure~\ref{1D_ErgodicMFGwithEffectiveHamiltonianLoss} compares the outer-level objective in \eqref{eq:OptGPProb_st} and shows that GN converges in substantially fewer iterations than GD.

\begin{table}[h]
    \centering
    \caption{Numerical results for \(\lambda\) in the first-order stationary MFG
  \eqref{eq:stationaryMFG_oneDim_inverse_effHam} using different methods. The reference solution is computed
  by the HRF \cite{gomes2020hessian}, and the recovered solution is obtained by the GD method and the GN
  method.}
    \label{table:1D_ErgodicMFGeffHamilton}
    \begin{tabular}{|c|c|c|c|}
      \hline
      Method & Reference & GD & GN \\
      \hline
      \(\lambda\) & 2.04463428454 & 2.04405032726 & 2.04406397895\\
      \hline
    \end{tabular}
  \end{table}

\begin{figure}[!htbp]
    \centering
    \begin{subfigure}[b]{0.23\textwidth}
        \centering
        \includegraphics[width=\linewidth]{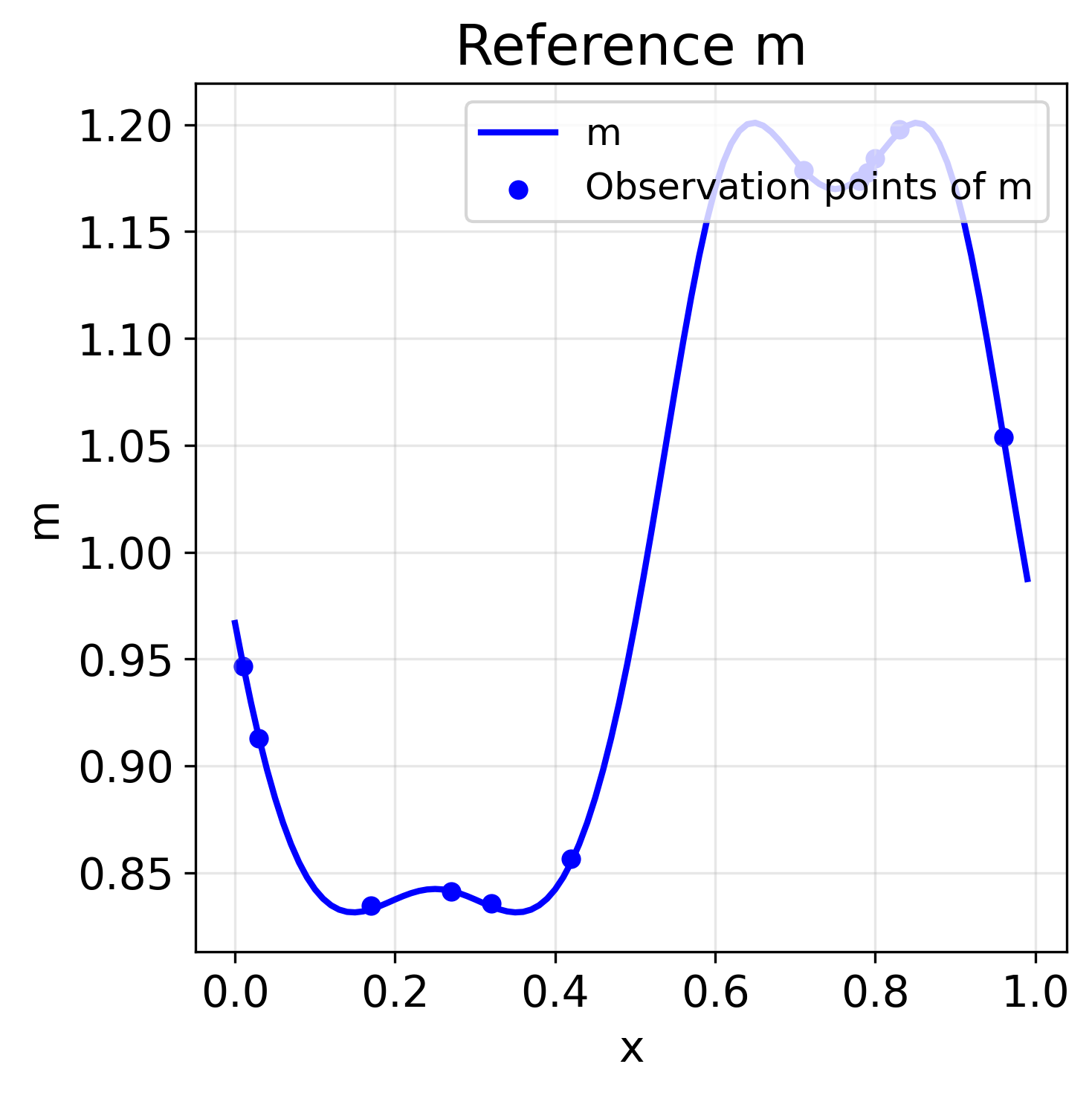}
        \caption{$m$ reference}
        \label{1D_ErgodicMFGwithEffectiveHamiltonianReferenceM}
    \end{subfigure}%
    \hspace{1mm}
    \begin{subfigure}[b]{0.23\textwidth}
        \centering
        \includegraphics[width=\linewidth]{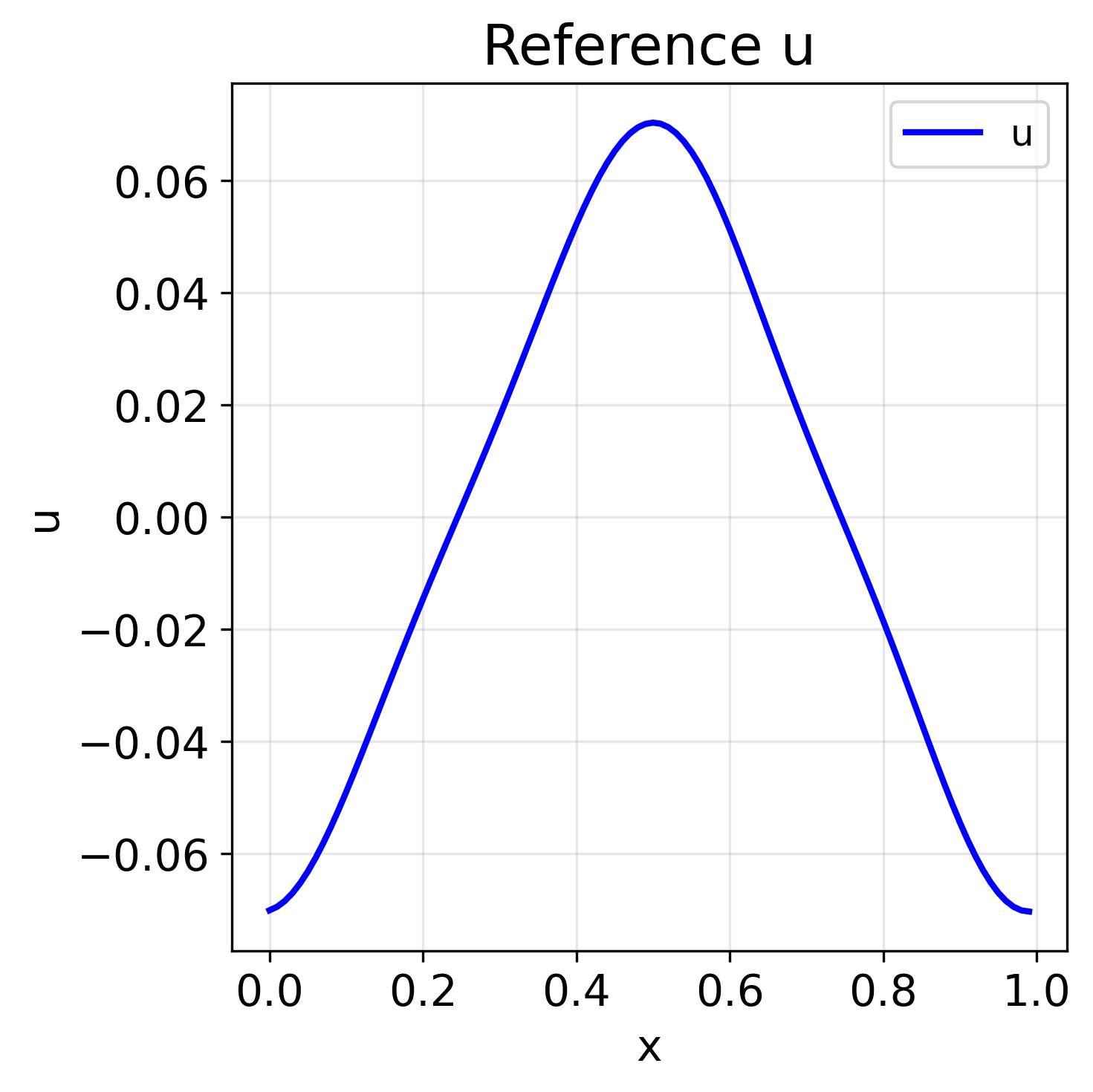}
        \caption{$u$ reference}
        \label{1D_ErgodicMFGwithEffectiveHamiltonianReferenceU}
    \end{subfigure}%
    \hspace{1mm}
    \begin{subfigure}[b]{0.23\textwidth}
        \centering
        \includegraphics[width=\linewidth]{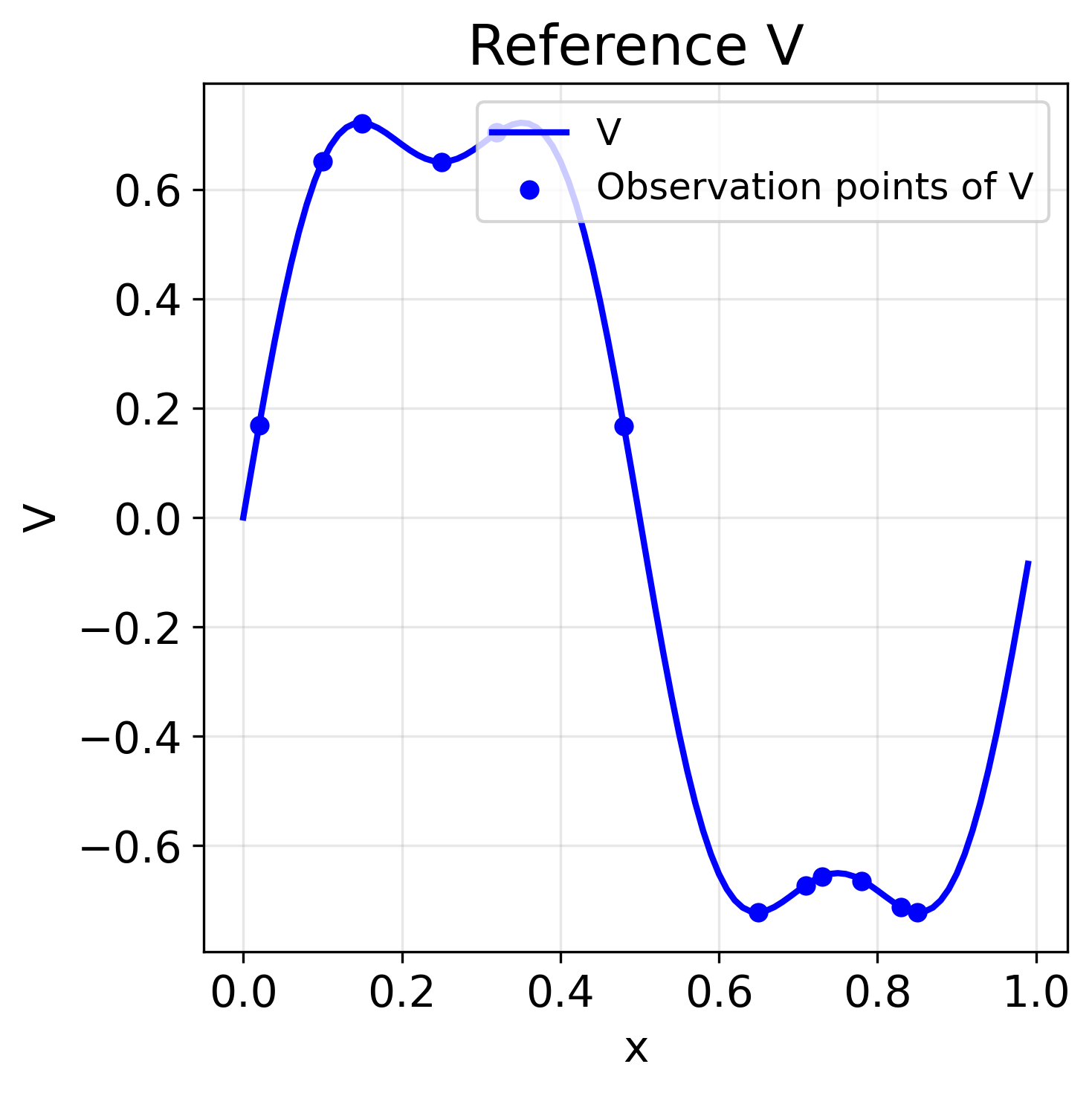}
        \caption{$V$ reference}
        \label{1D_ErgodicMFGwithEffectiveHamiltonianReferenceV}
    \end{subfigure}%
    \hspace{1mm}
    \begin{subfigure}[b]{0.23\textwidth}
        \centering
        \includegraphics[width=\linewidth]{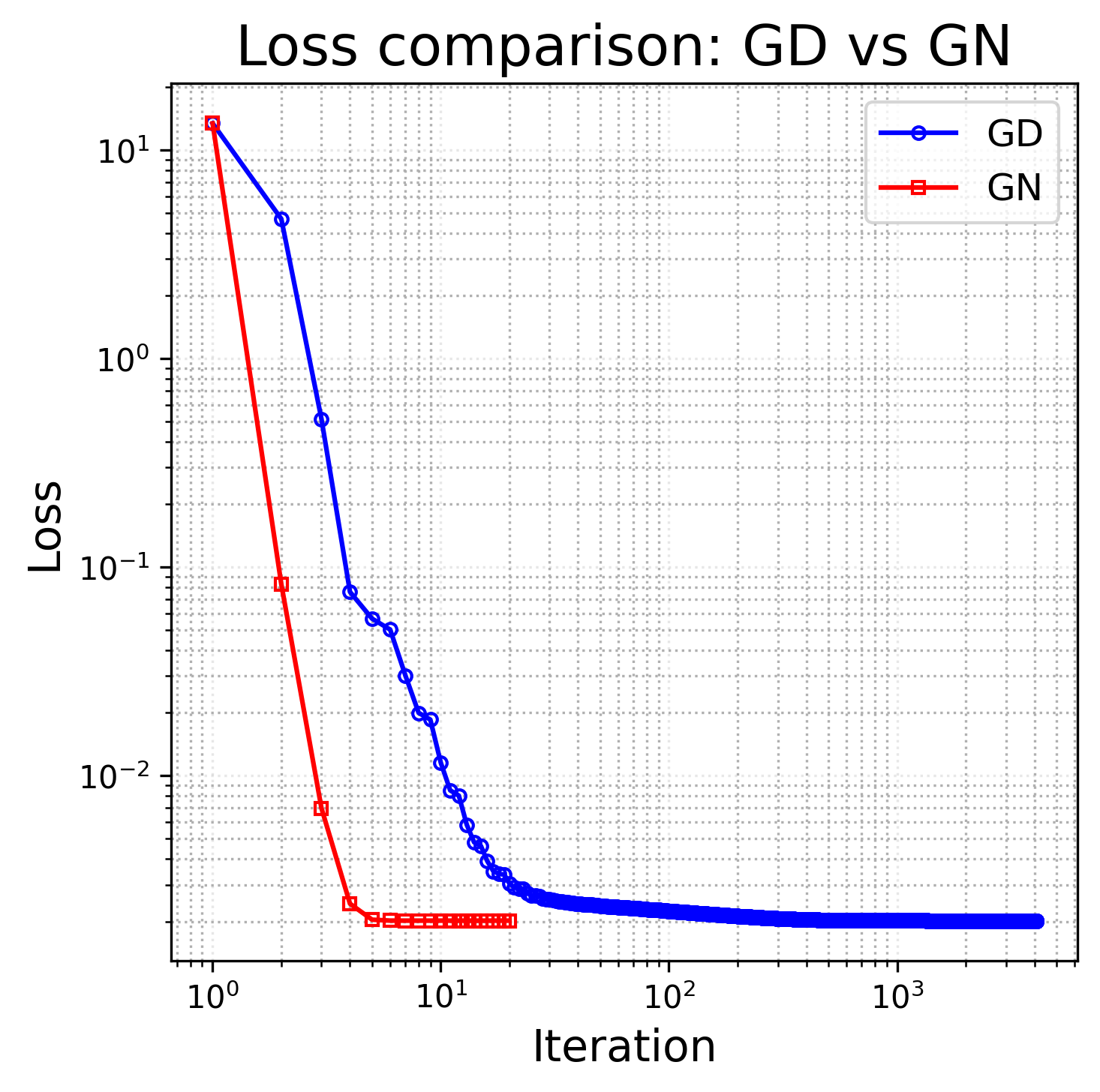}
        \caption{Loss comparison}
        \label{1D_ErgodicMFGwithEffectiveHamiltonianLoss}
    \end{subfigure}
    
    \begin{subfigure}[b]{0.23\textwidth}
        \centering
        \includegraphics[width=\linewidth]{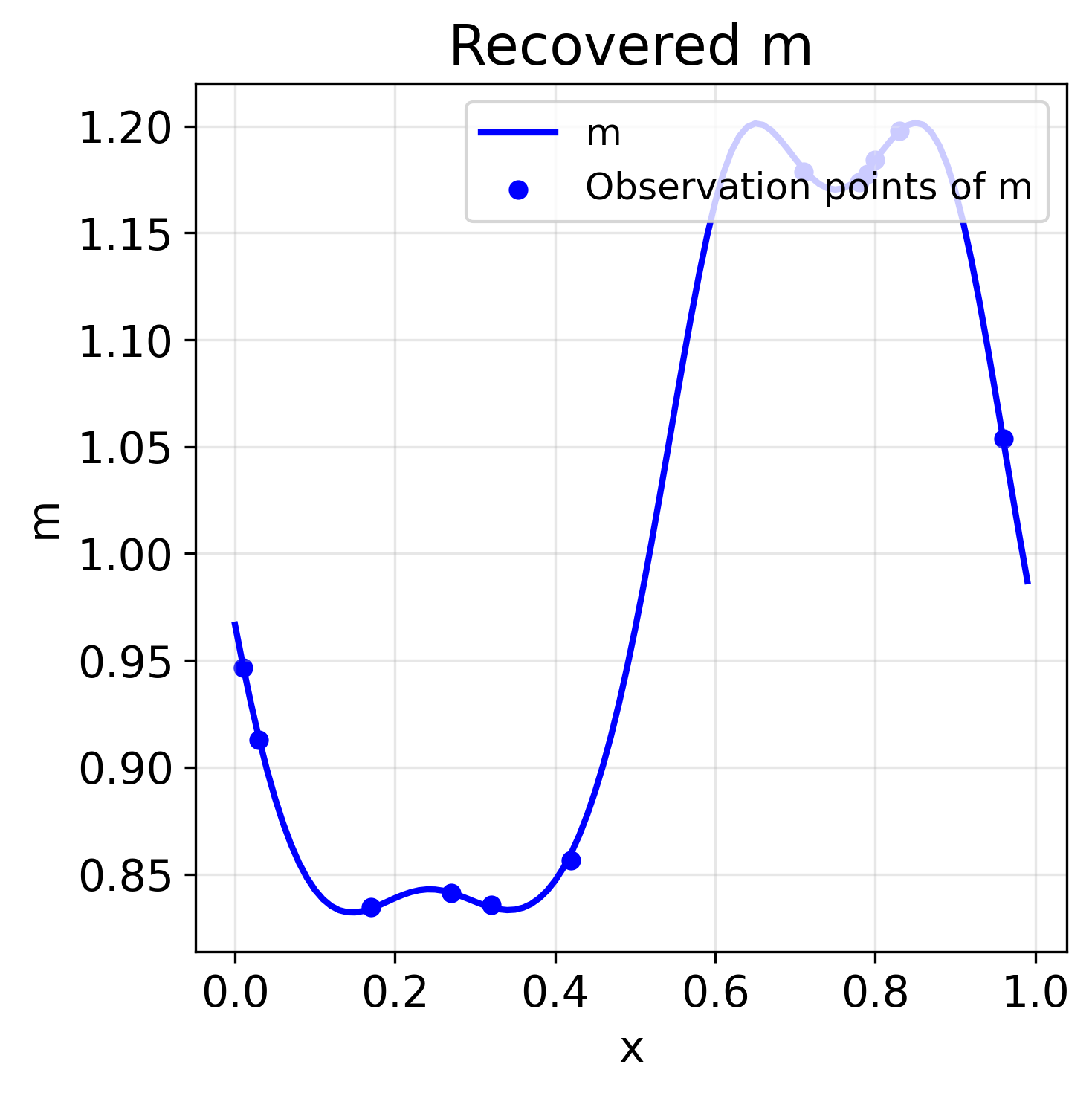}
        \caption{Recovered $m$ via GD}
        \label{1D_ErgodicMFGwithEffectiveHamiltonianRecoverdM_GD}
    \end{subfigure}%
    \hspace{1mm}
    \begin{subfigure}[b]{0.23\textwidth}
        \centering
        \includegraphics[width=\linewidth]{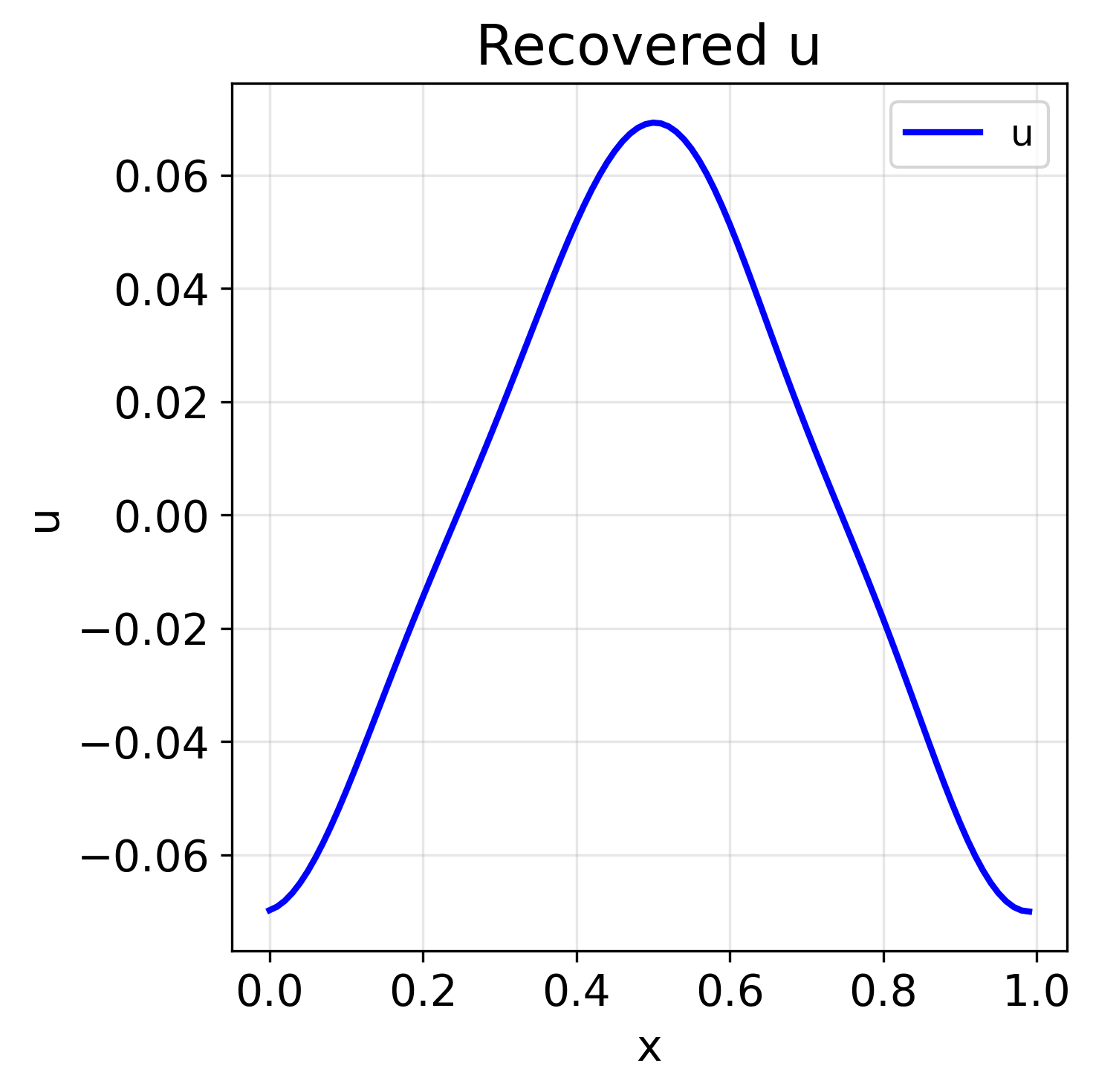}
        \caption{Recovered $u$ via GD}
        \label{1D_ErgodicMFGwithEffectiveHamiltonianRecoverdU_GD}
    \end{subfigure}%
    \hspace{1mm}
    \begin{subfigure}[b]{0.23\textwidth}
        \centering
        \includegraphics[width=\linewidth]{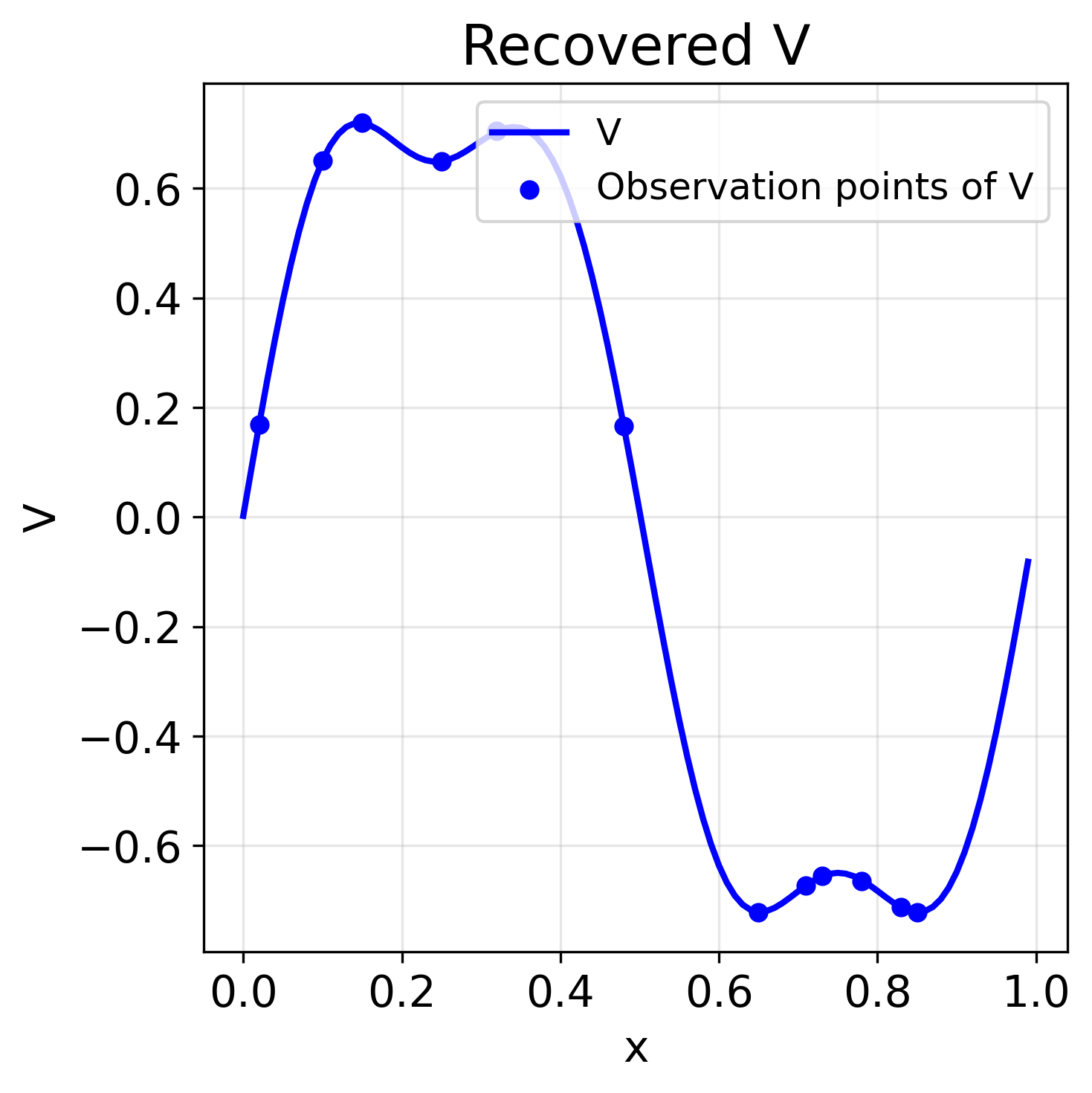}
        \caption{Recovered $V$ via GD}
        \label{1D_ErgodicMFGwithEffectiveHamiltonianRecoverdV_GD}
    \end{subfigure}%
    \hspace{1mm}
    \begin{subfigure}[b]{0.23\textwidth}
        \centering
        \includegraphics[width=\linewidth]{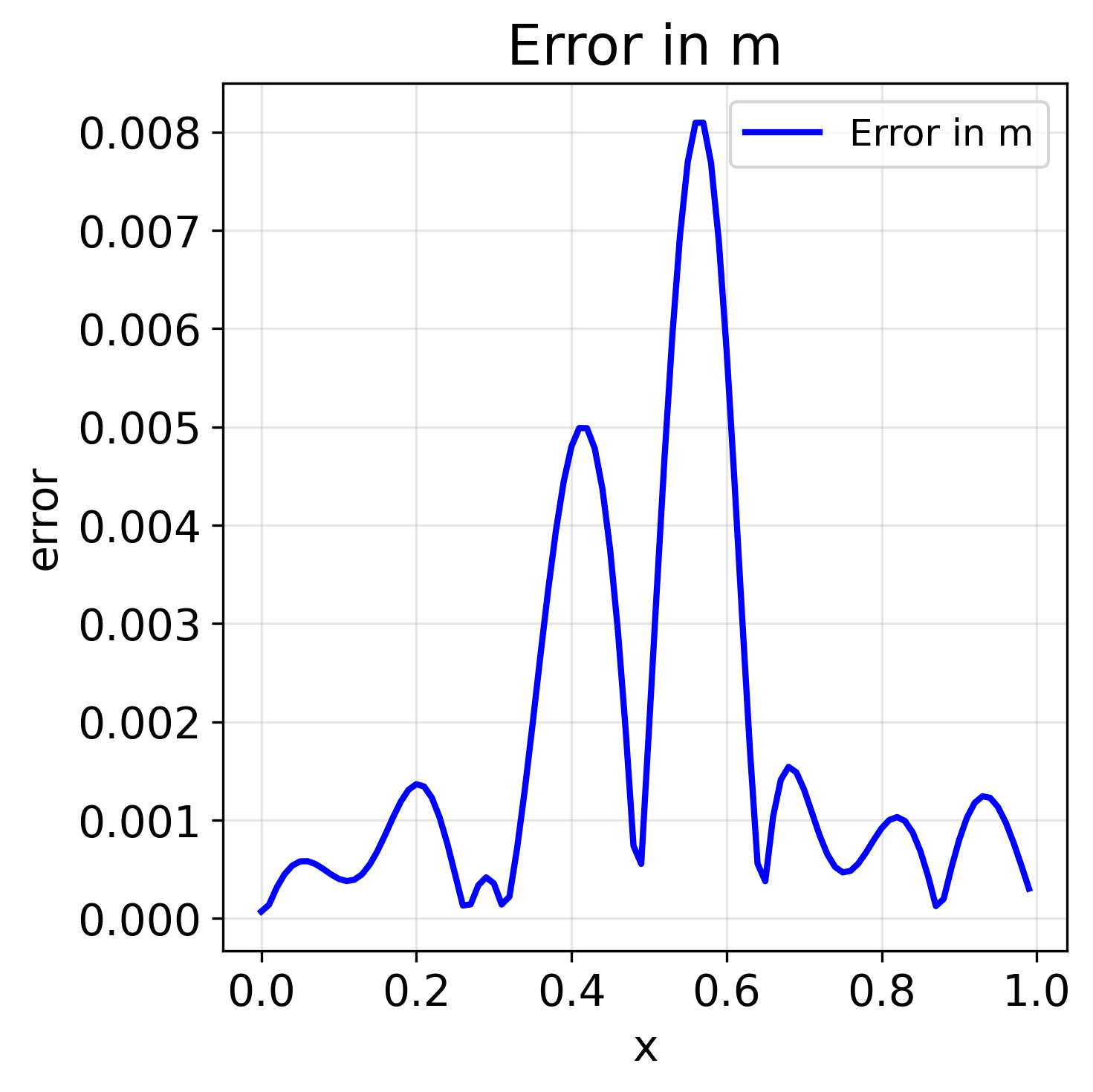}
        \caption{Error of $m$ via GD}
        \label{1D_ErgodicMFGwithEffectiveHamiltonianErrorM_GD}
    \end{subfigure}
    
    \begin{subfigure}[b]{0.23\textwidth}
        \centering
        \includegraphics[width=\linewidth]{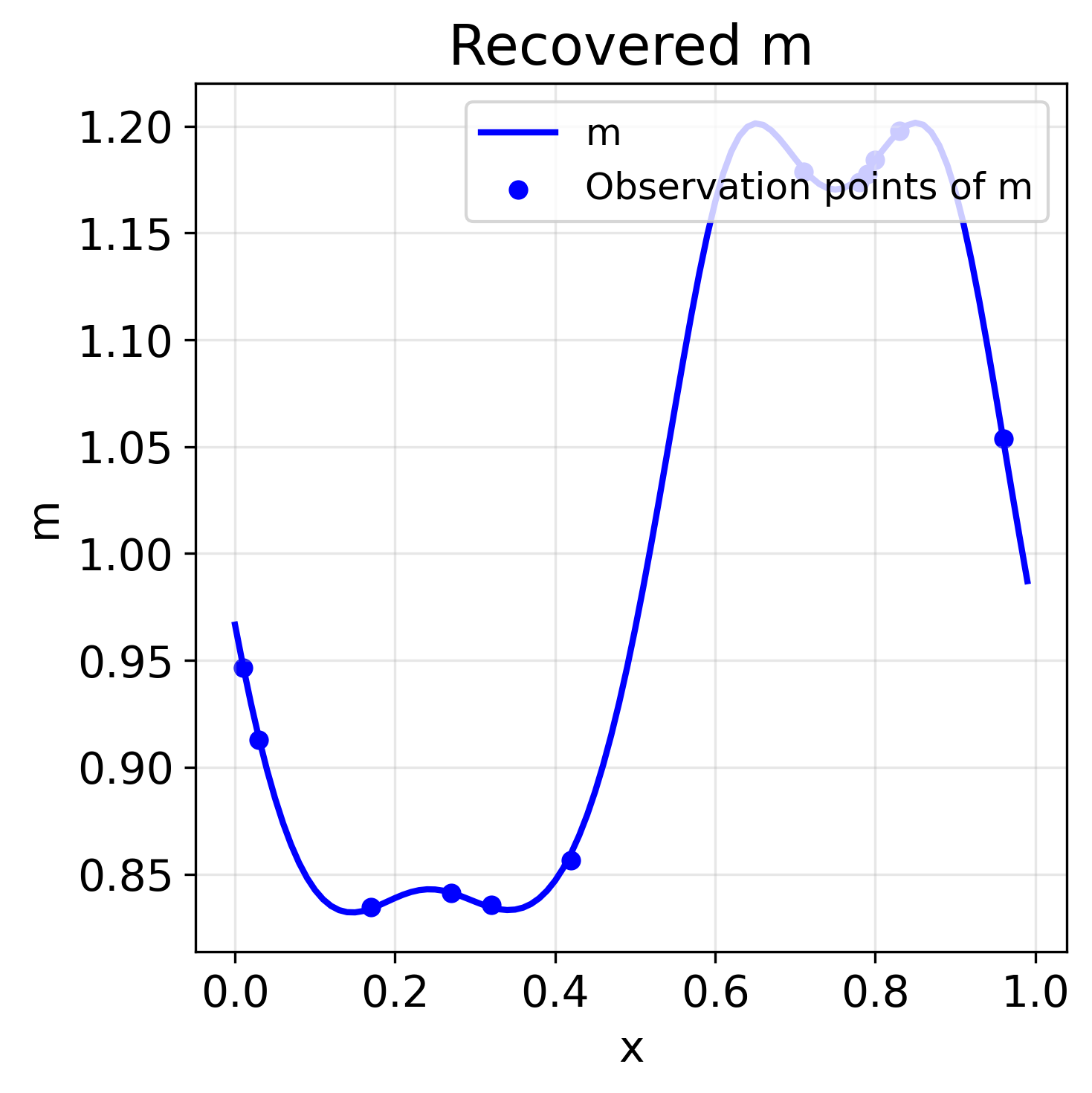}
        \caption{Recovered $m$ via GN}
        \label{1D_ErgodicMFGwithEffectiveHamiltonianRecoverdM_GN}
    \end{subfigure}%
    \hspace{1mm}
    \begin{subfigure}[b]{0.23\textwidth}
        \centering
        \includegraphics[width=\linewidth]{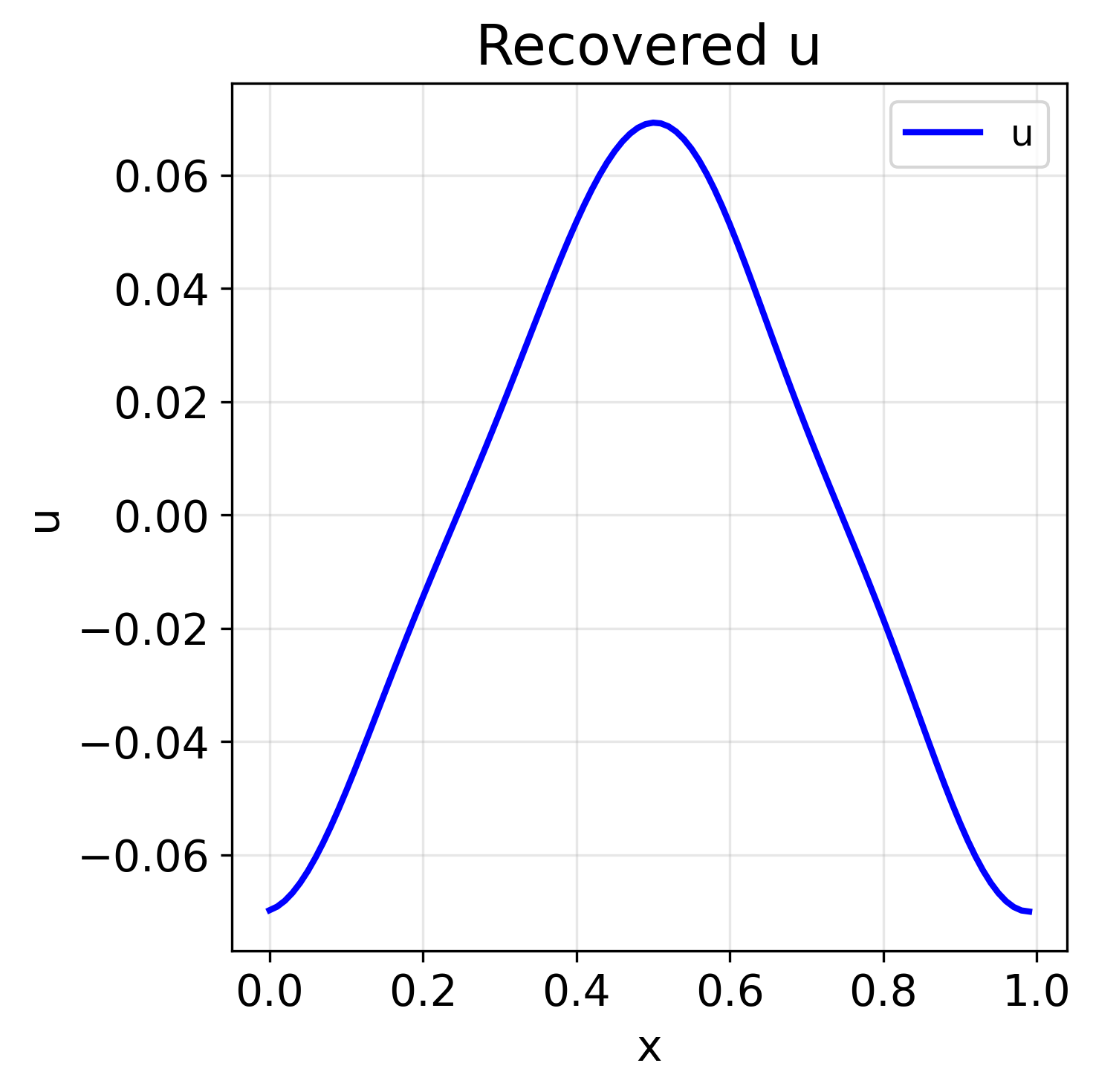}
        \caption{Recovered $u$ via GN}
        \label{1D_ErgodicMFGwithEffectiveHamiltonianRecoverdU_GN}
    \end{subfigure}%
    \hspace{1mm}
    \begin{subfigure}[b]{0.23\textwidth}
        \centering
        \includegraphics[width=\linewidth]{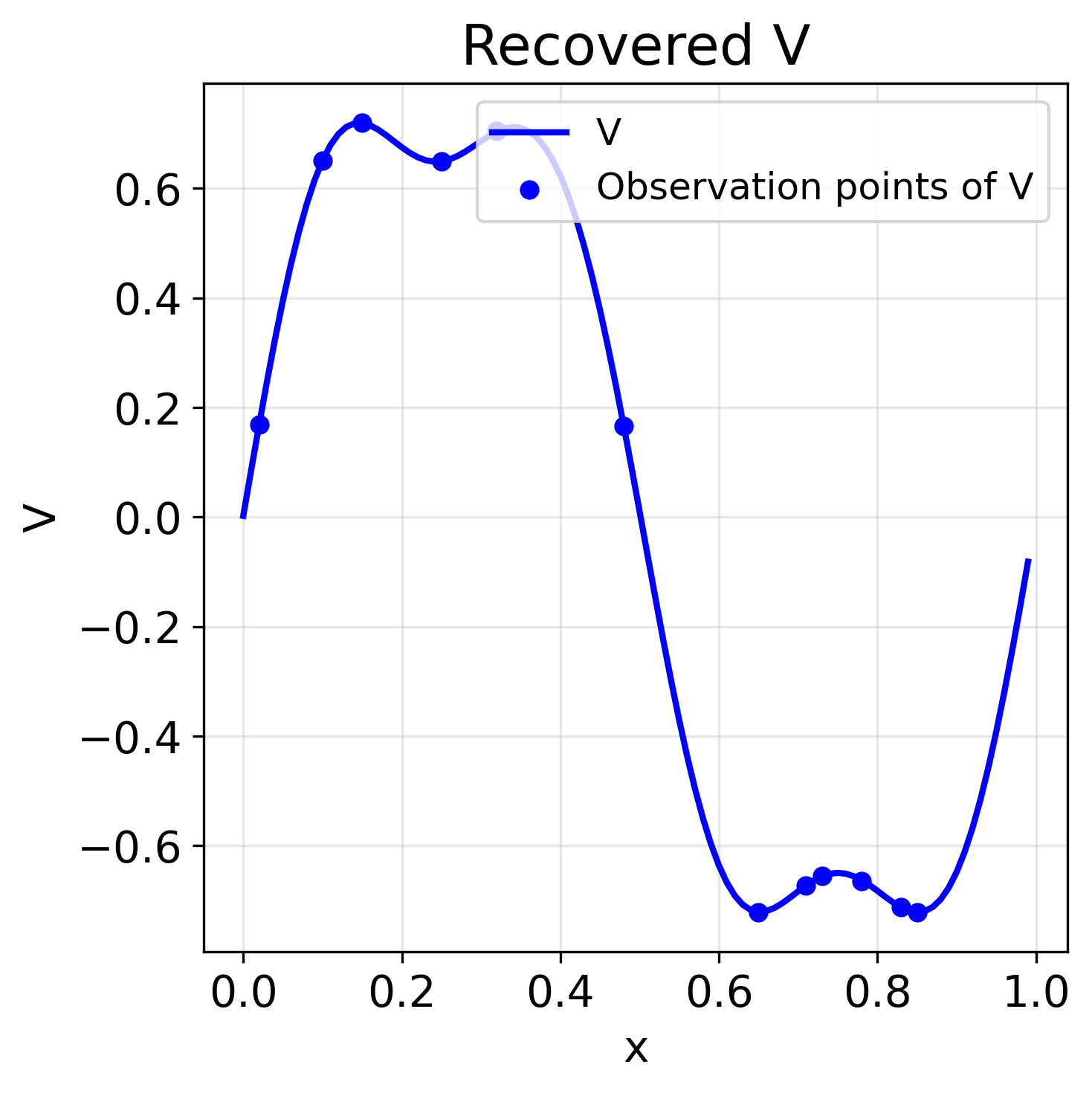}
        \caption{Recovered $V$ via GN}
        \label{1D_ErgodicMFGwithEffectiveHamiltonianRecoverdV_GN}
    \end{subfigure}%
    \hspace{1mm}
    \begin{subfigure}[b]{0.23\textwidth}
        \centering
        \includegraphics[width=\linewidth]{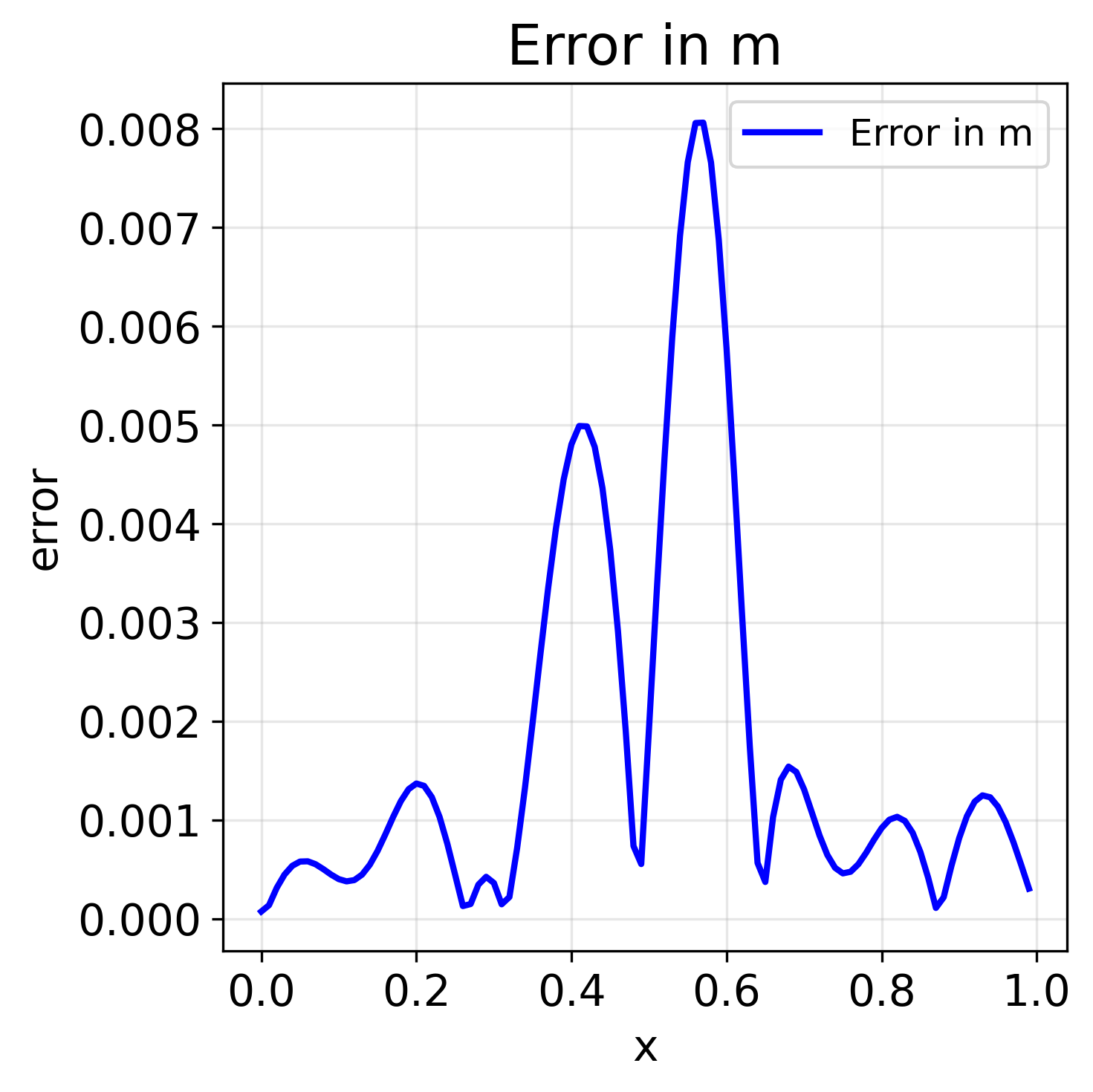}
        \caption{Error of $m$ via GN}
        \label{1D_ErgodicMFGwithEffectiveHamiltonianErrorM_GN}
    \end{subfigure}
    
    \begin{subfigure}[b]{0.23\textwidth}
        \centering
        \includegraphics[width=\linewidth]{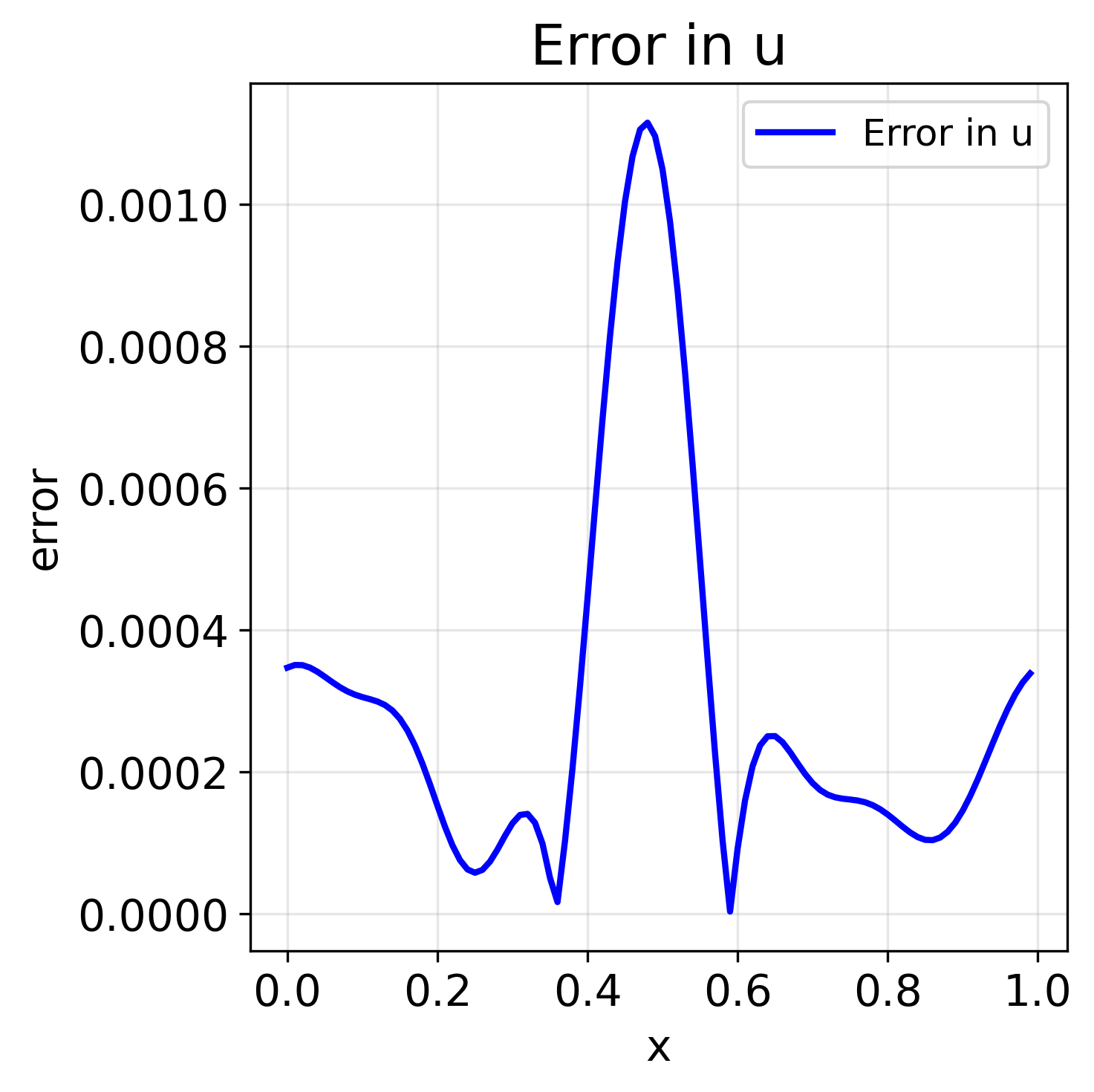}
        \caption{Error of $u$ via GD}
        \label{1D_ErgodicMFGwithEffectiveHamiltonianErrorU_GD}
    \end{subfigure}
    \hspace{1mm}%
    \begin{subfigure}[b]{0.23\textwidth}
        \centering
        \includegraphics[width=\linewidth]{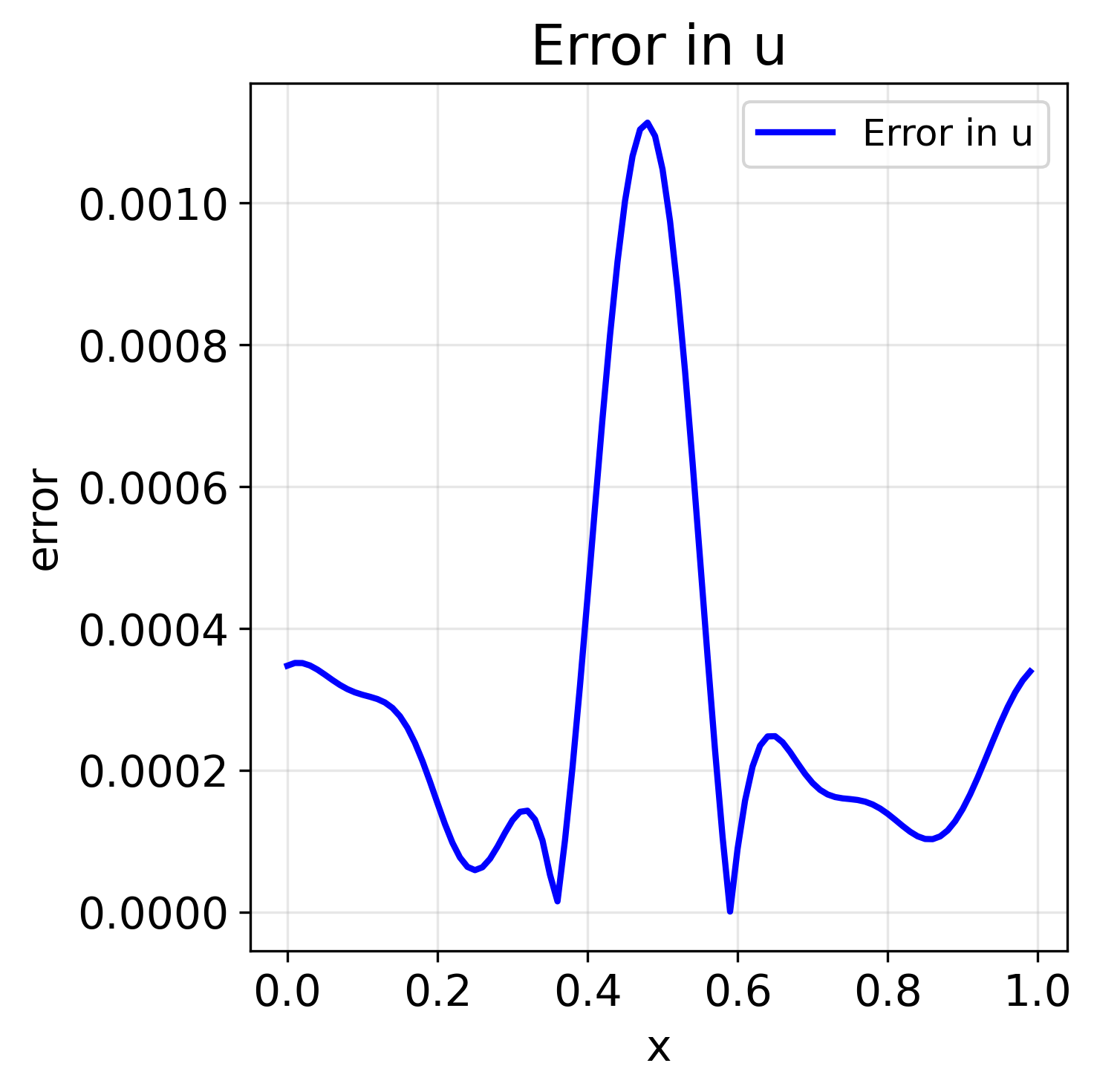}
        \caption{Error of $u$ via GN}
        \label{1D_ErgodicMFGwithEffectiveHamiltonianErrorU_GN}
    \end{subfigure}
    \hspace{1mm}
    \begin{subfigure}[b]{0.23\textwidth}
        \centering
        \includegraphics[width=\linewidth]{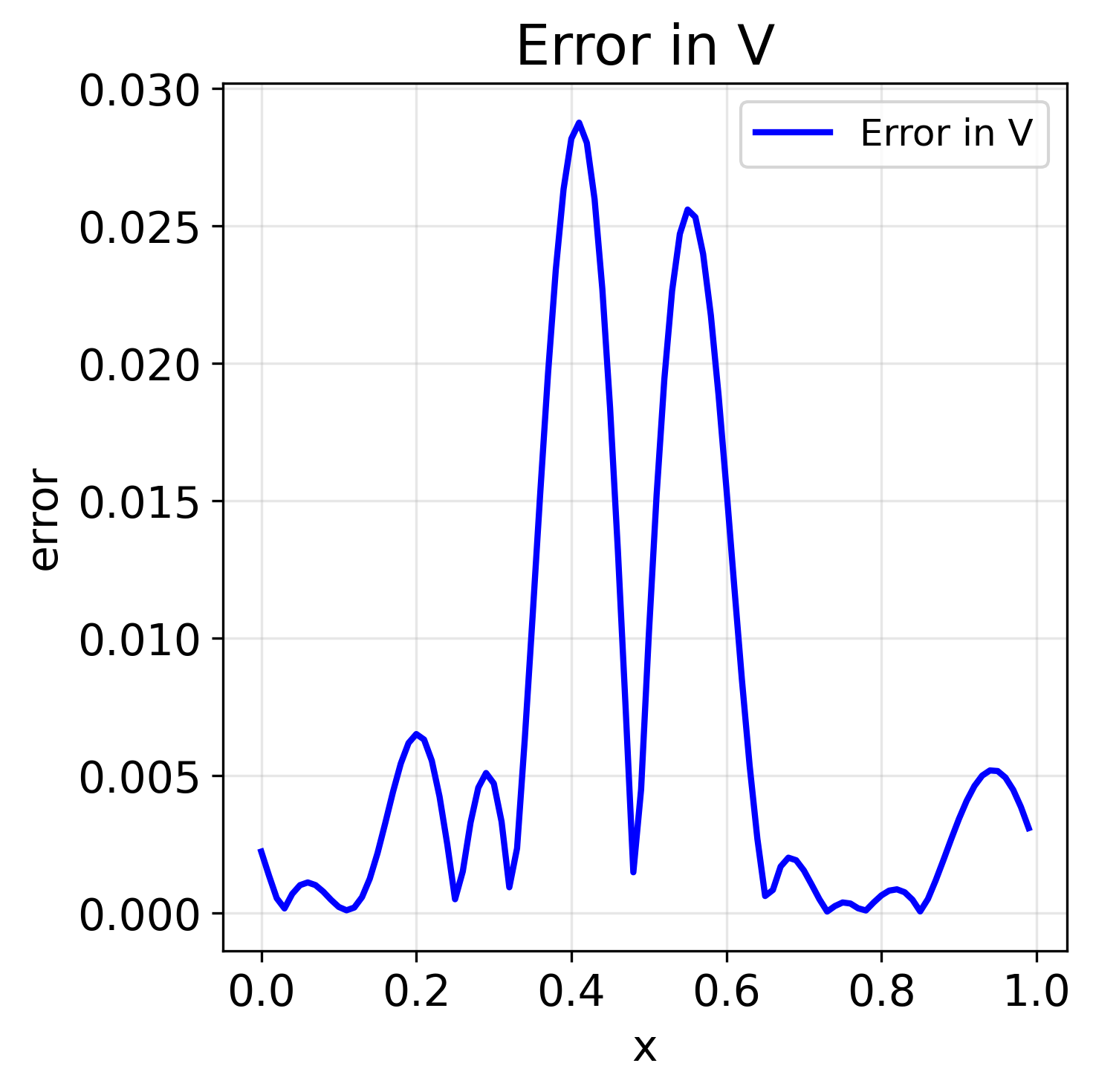}
        \caption{Error of $V$ via GD}   \label{1D_ErgodicMFGwithEffectiveHamiltonianErrorV_GD}
    \end{subfigure}%
    \hspace{1mm}%
    \begin{subfigure}[b]{0.23\textwidth}
        \centering
        \includegraphics[width=\linewidth]{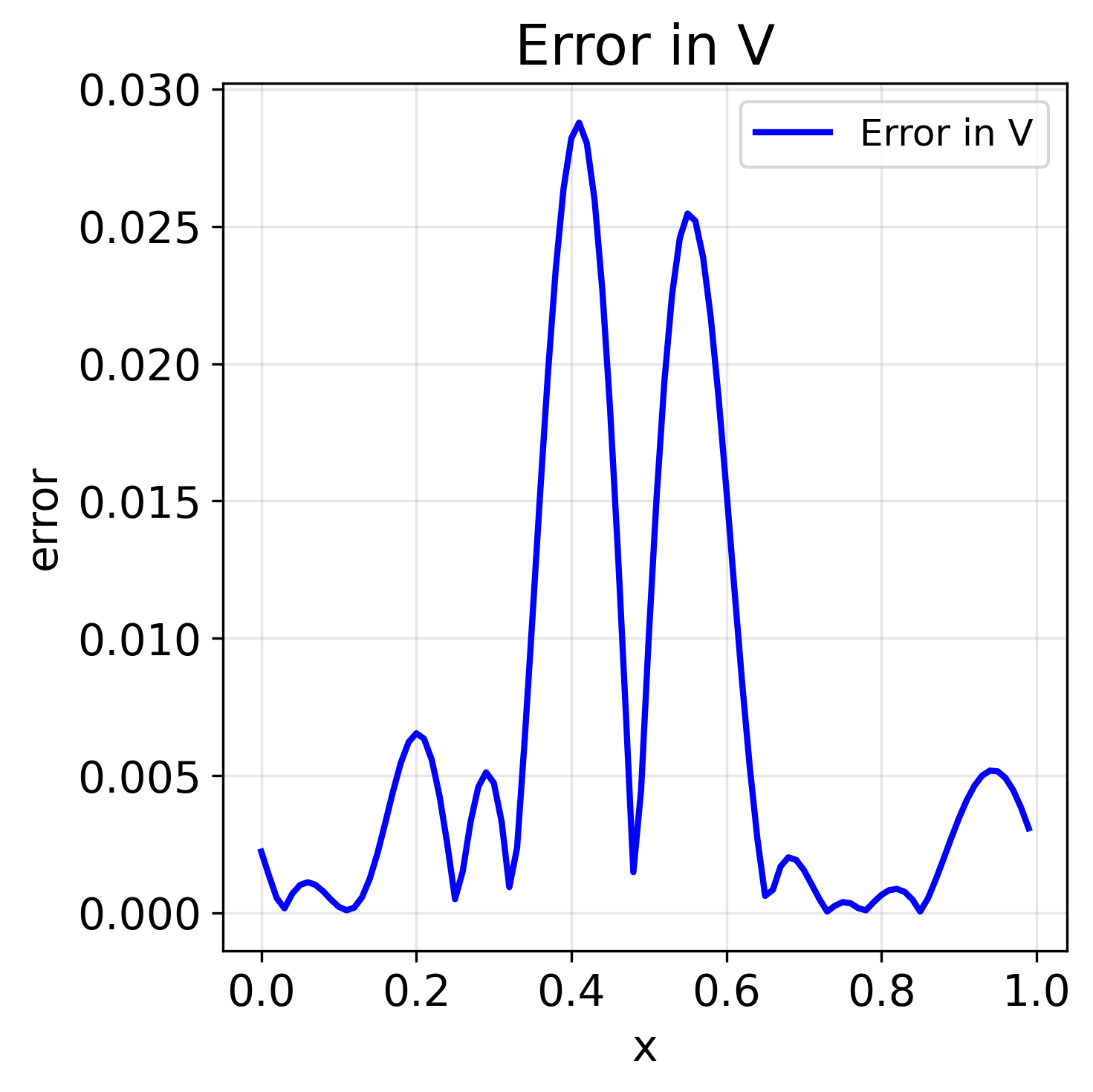}
        \caption{Error of $V$ via GN}
        \label{1D_ErgodicMFGwithEffectiveHamiltonianErrorV_GN}
    \end{subfigure}%

    \caption{Numerical results for the inverse problem of the one-dimensional stationary  MFG in \eqref{eq:stationaryMFG_oneDim_inverse_effHam}.  (a), (b), (c) are references for $m,u,V$; (d) log-log plot comparing the GD and GN losses across iterations; (e), (f), (g) recovered $m,u,V$ via GD; (h), (m), (o) errors of $m,u,V$ via GD; (i), (j), (k) recovered $m,u,V$ via GN; (l), (n), (p) errors of $m,u,V$ via GN.}
    \label{1D_ErgodicMFGwithEffectiveHamiltonianPlot}
\end{figure}

\subsection{Two-Dimensional Stationary MFG Inverse Problem}
In this study, we consider inverse problems for two stationary MFGs: a first-order model with congestion and a non-potential second-order model. We also examine the robustness of the solver-agnostic formulation by varying the inner forward solver used in the inverse problem.

\subsubsection{A First-Order Stationary MFG with Congestion}
\label{2DMFGexample1}
Here, we consider the first-order stationary MFG with congestion
\begin{align}
\label{eq:first_order_MFGS_congestion}
\begin{cases}
    \mathcal H(\nabla u,m)  = \lambda + m^3 + V(x, y), & (x,y)\in\mathbb{T}^2, \\
    -\div\!\bigl(D_p\mathcal{H}(\nabla u,m)\,m\bigr) = 0, & (x,y)\in\mathbb{T}^2, \\
    \displaystyle \int_{\mathbb{T}^2} u\,\mathrm{d}x\,\mathrm{d}y = 0,\qquad 
    \displaystyle \int_{\mathbb{T}^2} m\,\mathrm{d}x\,\mathrm{d}y = 1,
\end{cases}
\end{align}
where \(\lambda\in\mathbb{R}\).  The convergence theorem in Section~\ref{sec:tdmfg} is proved for separable
Hamiltonians. The congestion Hamiltonian considered here is
nonseparable because it depends on both \(p\) and \(m\), and is therefore not a
direct instance of the theorem. 
This example shows that the HRF method and the inverse pipeline remain applicable beyond the separable case. A convergence proof for this
nonseparable model would require a separate analysis of the corresponding
monotonicity and dissipation identities, and we keep the present experiment
focused on the numerical behavior of the method.  We take the congestion Hamiltonian $
\mathcal H(p,m)=\frac{|Q+p|^b}{b\,m^a}$, $b=2.0$,  $a=1.5$, $Q=(1,3)$, and 
$V(x,y)=2\sin\!\left(2\pi\Bigl(x+\tfrac{1}{4}\Bigr)\right)\cos\!\left(2\pi\Bigl(y+\tfrac{1}{4}\Bigr)\right).$
For this choice of coefficients, we compute a reference state \((u^*,m^*,\lambda^*)\) numerically with the HRF method. In this example, we study the inverse problem of recovering \(u\), \(m\), \(V\), and \(\lambda\) from partial noisy observations of \(m\) and \(V\).


\textbf{Experimental Setup.}
We identify \(\mathbb{T}^2\) with \([0,1)^2\) and discretize it with a uniform grid of spacing \(h_x=h_y=1/40\), yielding \(1600\) grid points. We then sample \(128\) observation locations for \(m\) and \(320\) observation locations for \(V\) independently and uniformly at random from \([0,1)^2\).
The regularization parameters are \(\alpha=0.04\), \(\beta=2\), and \(\gamma=2\). We add i.i.d.\ Gaussian noise \(\mathcal{N}(0,\eta^2 I)\) with \(\eta=10^{-3}\) to the observations. For the inverse problem, we initialize \(u\equiv 0\), \(m\equiv 1\), and \(V\equiv 0\). The reference solution \((u^*,m^*,\lambda^*)\) is computed by solving \eqref{eq:first_order_MFGS_congestion} with the HRF method. We then solve the inverse problem using GD and GN.

\textbf{Experimental Results.}
Figures~\ref{2D_ErgodicMFGwithCongestionSamplesM}--\ref{2D_ErgodicMFGwithCongestionSamplesV} visualize the collocation grid and the selected observation locations for \(m\) and \(V\). Figure~\ref{fig:L2_vs_obs_congestion_N40} reports reconstruction accuracy versus the number of observed spatial points, showing systematic decreases in the discrete \(L^2\) errors for \(m\), \(u\), and \(V\) as observation density increases. Table~\ref{table:2D_ErgodicMFGwithCongestionHbar} compares the HRF reference value of \(\lambda\) with the values recovered by GD and GN. Figure~\ref{2D_ErgodicMFGwithCongestionPlot} summarizes the inverse-recovery results for \eqref{eq:first_order_MFGS_congestion}: it includes the HRF reference fields, the GD and GN reconstructions, and the corresponding pointwise absolute error maps for \(m\), \(u\), and \(V\). Finally, Figure~\ref{2D_ErgodicMFGwithCongestionLoss} shows that GN converges faster than GD.
\begin{figure}[!htbp]
    \centering%
    \begin{subfigure}[b]{0.23\textwidth}
        \centering
        \includegraphics[width=\linewidth]{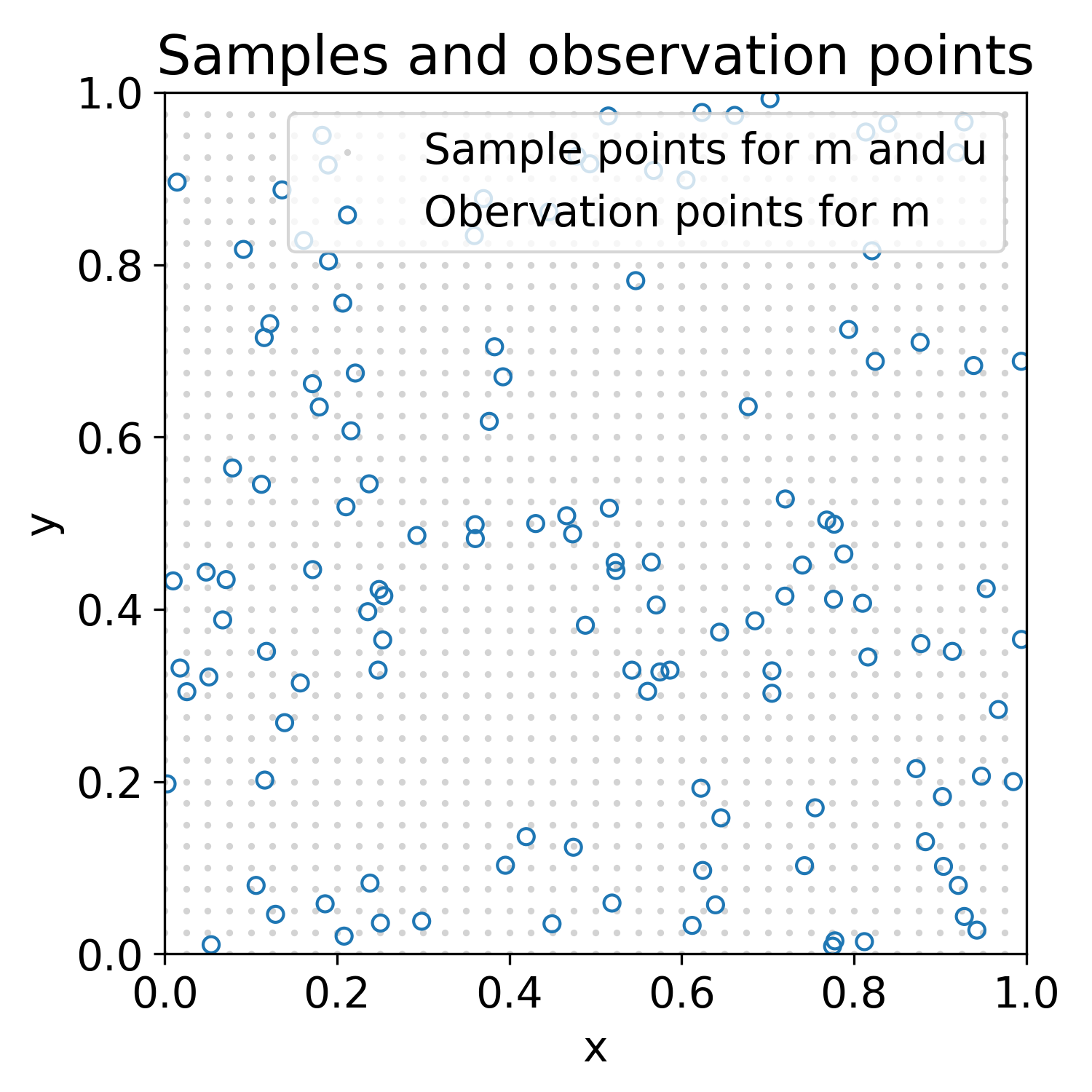}
        \caption{Samples and observations for $m$}
        \label{2D_ErgodicMFGwithCongestionSamplesM}
    \end{subfigure}
    \hspace{1mm}  
    \begin{subfigure}[b]{0.23\textwidth}
        \centering
        \includegraphics[width=\linewidth]{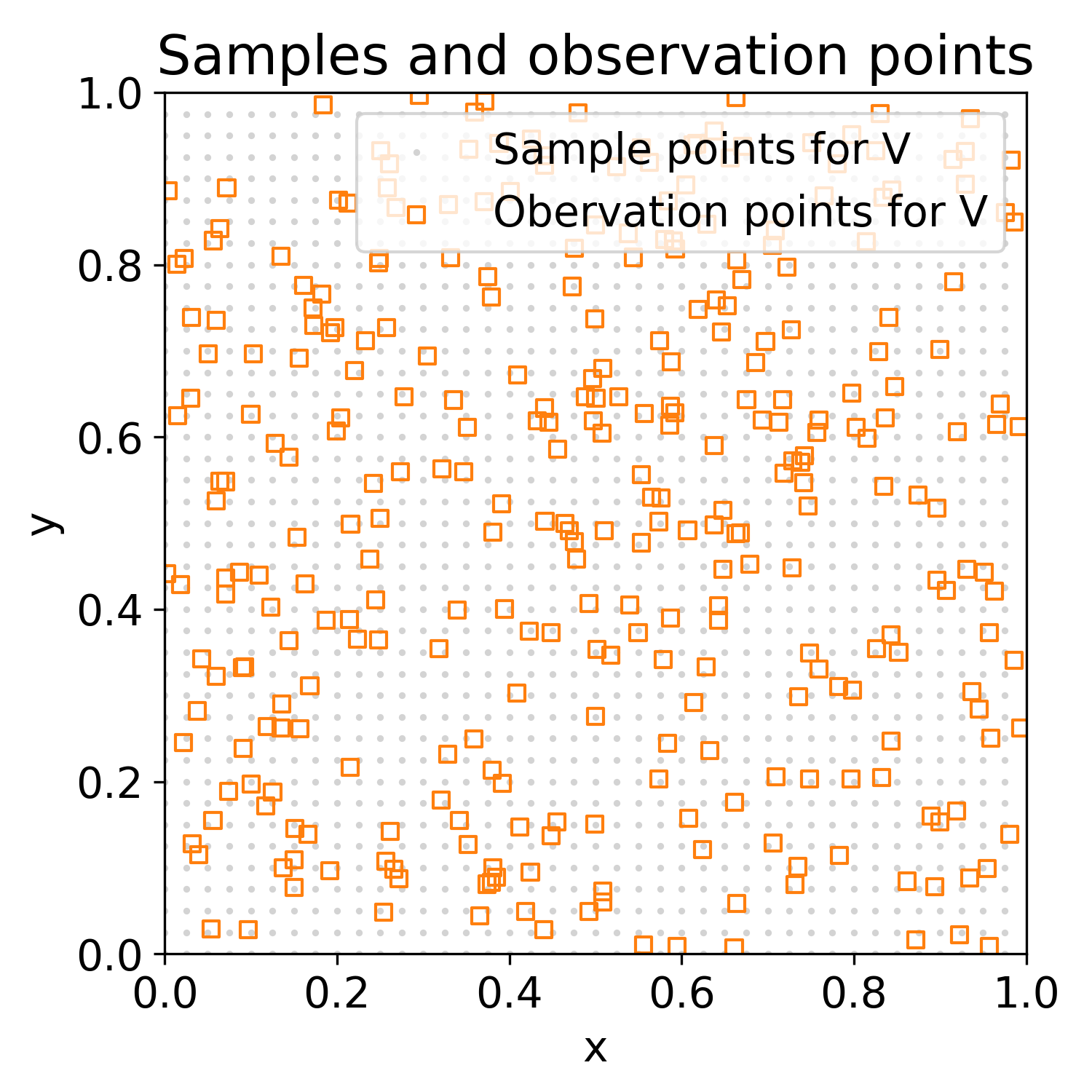}
        \caption{Samples and observations for $V$}    \label{2D_ErgodicMFGwithCongestionSamplesV}
    \end{subfigure}
    \caption{The inverse problem of the stationary MFG  \eqref{eq:first_order_MFGS_congestion}: samples for $m$ and $V$ and corresponding observation points.}
    \label{fig:samples_observationsmulti_2D_ErgodicMFGwithCongestion}
\end{figure}

\begin{figure}[!htbp]
    \centering  %
    \begin{subfigure}[b]{0.23\textwidth}
        \centering
        \includegraphics[width=\linewidth]{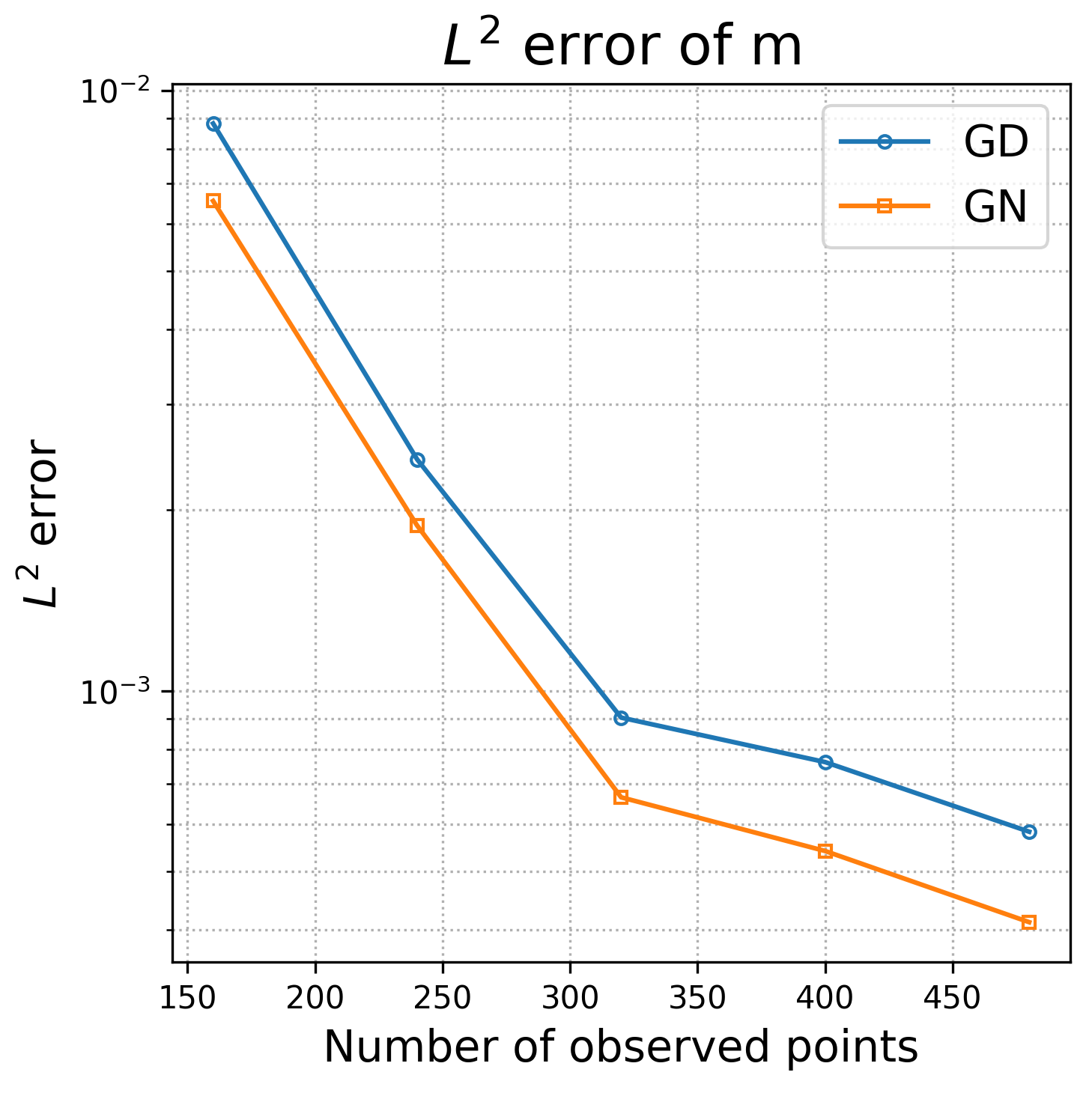}
        \caption{$L^2$ error of $m$}
        \label{fig:congestion_m_L2_vs_obs_N40}
    \end{subfigure}
    \hspace{1mm}
    \begin{subfigure}[b]{0.23\textwidth}
        \centering
        \includegraphics[width=\linewidth]{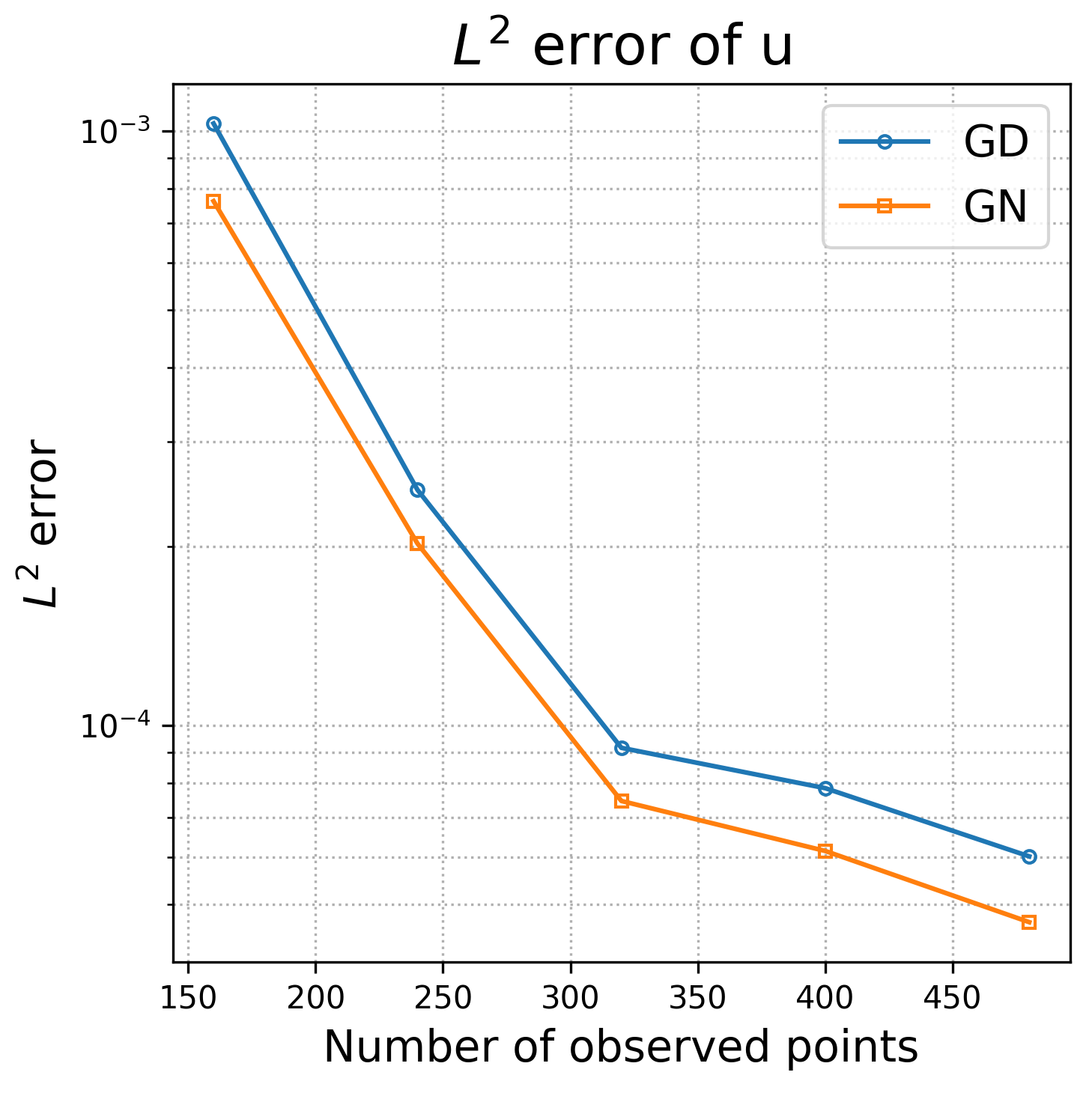}
        \caption{$L^2$ error of $u$}
    \label{fig:congestion_u_L2_vs_obs_N40}
    \end{subfigure}
    \hspace{1mm}
    \begin{subfigure}[b]{0.23\textwidth}
        \centering
        \includegraphics[width=\linewidth]{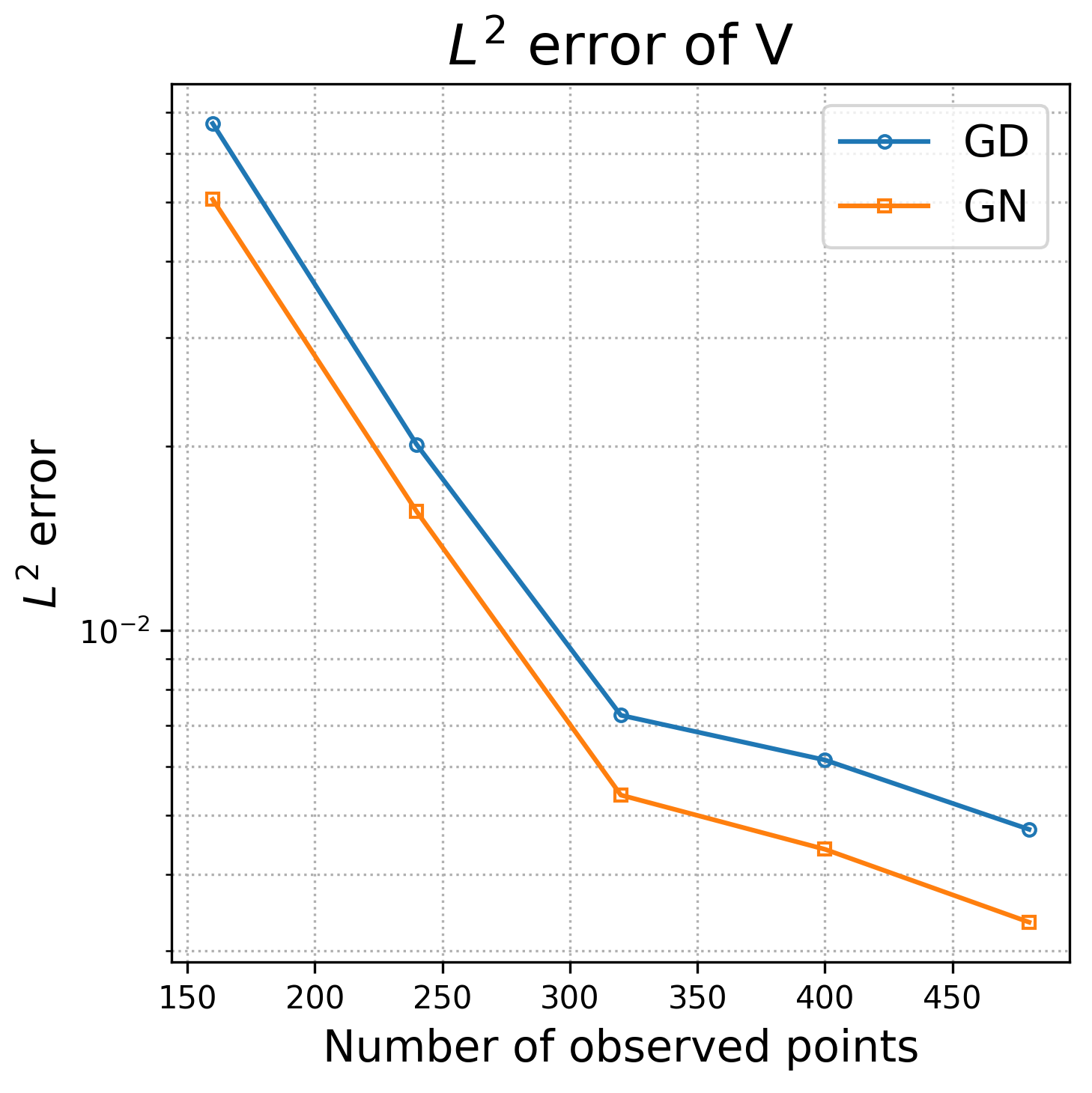}
        \caption{$L^2$ error of $V$}
        \label{fig:congestion_V_L2_vs_obs_N40}
    \end{subfigure}
    \caption{Numerical results for solving the inverse problem of the MFG system in \eqref{eq:first_order_MFGS_congestion} using the adjoint-based GD method and the GN method: (a)-(c) show the $L^2$ errors for $m$, $u$, and $V$, respectively, versus their reference counterparts as the number of observation points increases.
    }
    \label{fig:L2_vs_obs_congestion_N40}
\end{figure}

\begin{table}[!htbp]
  \centering
  \caption{Numerical results for \(\lambda\) in the first-order stationary MFG with congestion problem \eqref{eq:first_order_MFGS_congestion} using different methods. The reference solution is computed by the HRF-based solver, and the recovered solution is obtained by the GD method and the GN method.}
  \label{table:2D_ErgodicMFGwithCongestionHbar}
  \begin{tabular}{|c|c|c|c|}
    \hline
    Method & Reference & GD & GN \\
    \hline
    \(\lambda\) & 4.04456433468 & 4.04563955831 & 4.04507371305\\
    \hline
  \end{tabular}
\end{table}

\begin{figure}[!htbp]
    \centering
    \begin{subfigure}[b]{0.23\textwidth}
        \centering
        \includegraphics[width=\linewidth]{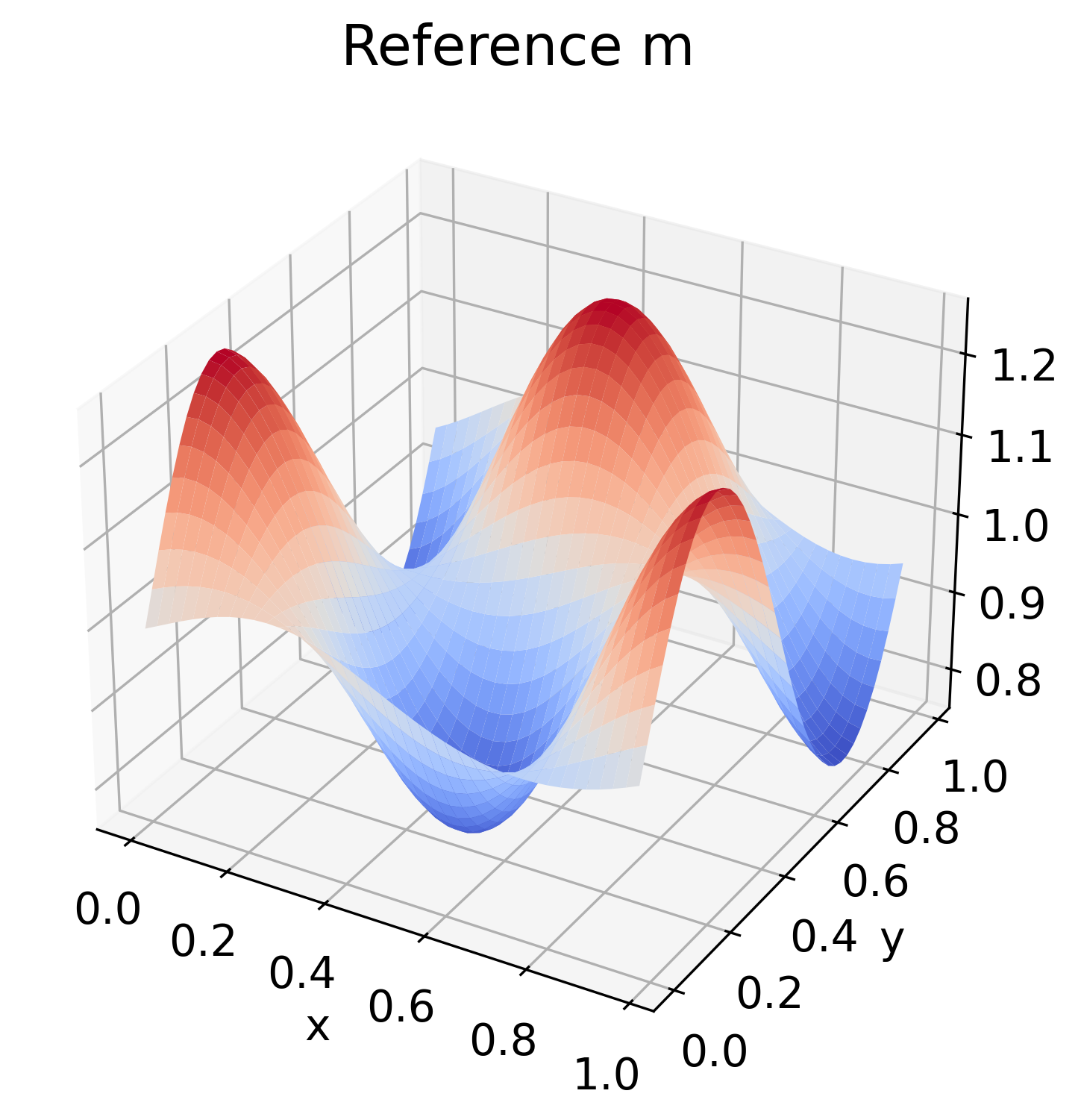}
        \caption{$m$ reference}
        \label{2D_ErgodicMFGwithCongestionReferenceM}
    \end{subfigure}%
    \hspace{1mm}
    \begin{subfigure}[b]{0.23\textwidth}
        \centering
        \includegraphics[width=\linewidth]{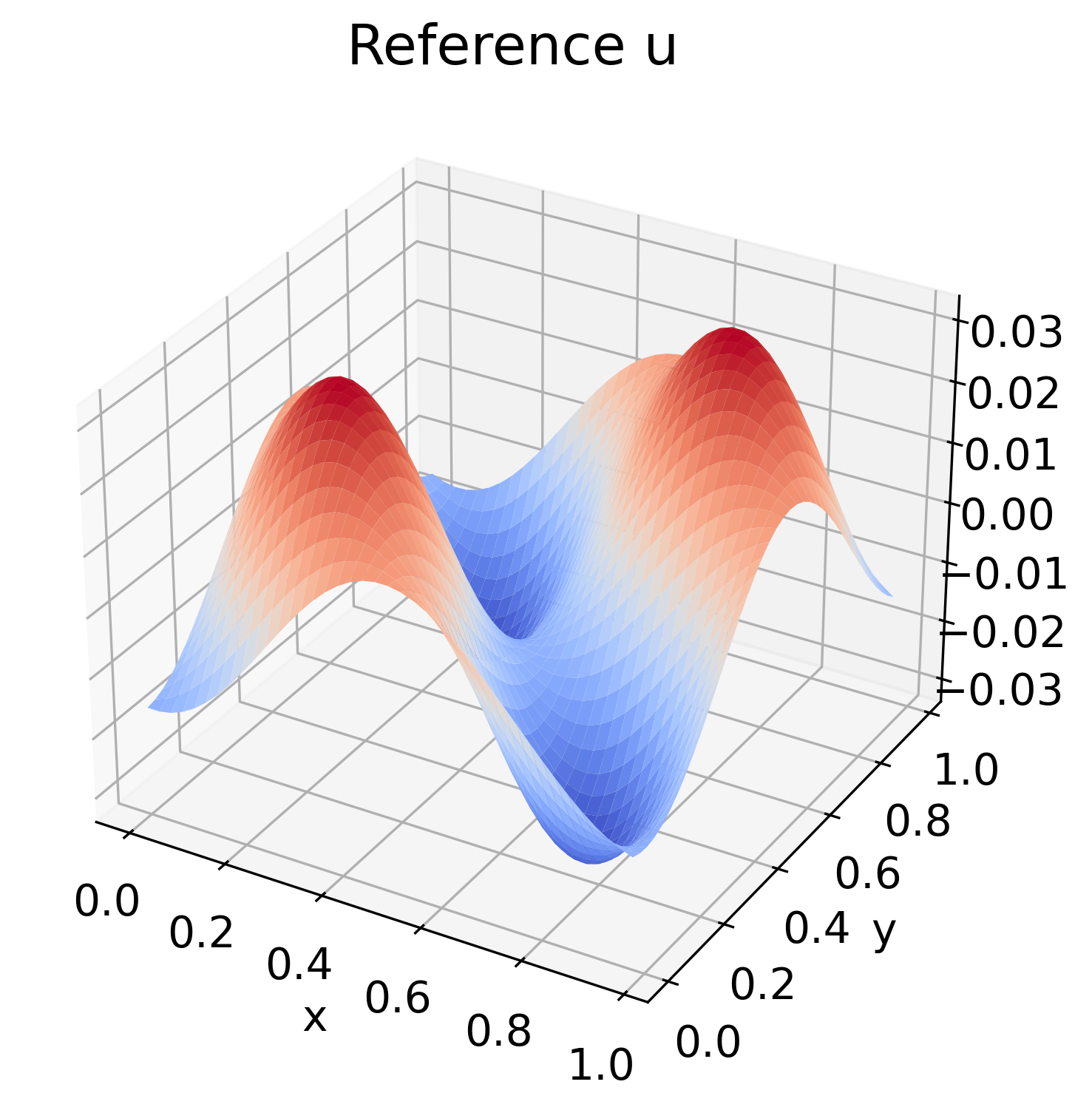}
        \caption{$u$ reference}
        \label{2D_ErgodicMFGwithCongestionReferenceU}
    \end{subfigure}%
    \hspace{1mm}
    \begin{subfigure}[b]{0.23\textwidth}
        \centering
        \includegraphics[width=\linewidth]{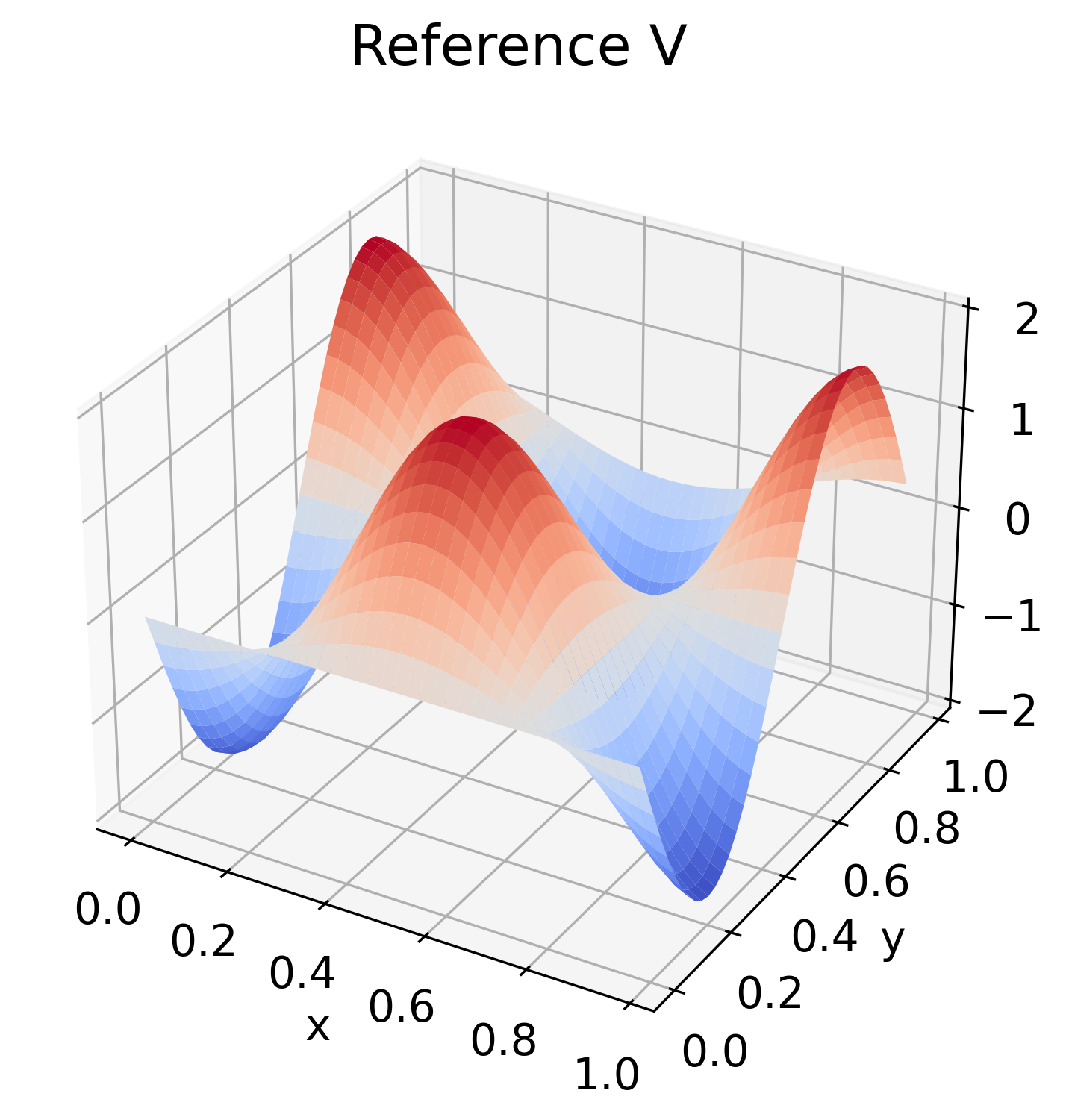}
        \caption{$V$ reference}
        \label{2D_ErgodicMFGwithCongestionReferenceV}
    \end{subfigure}%
    \hspace{1mm}
    \begin{subfigure}[b]{0.23\textwidth}
        \centering
        \includegraphics[width=\linewidth]{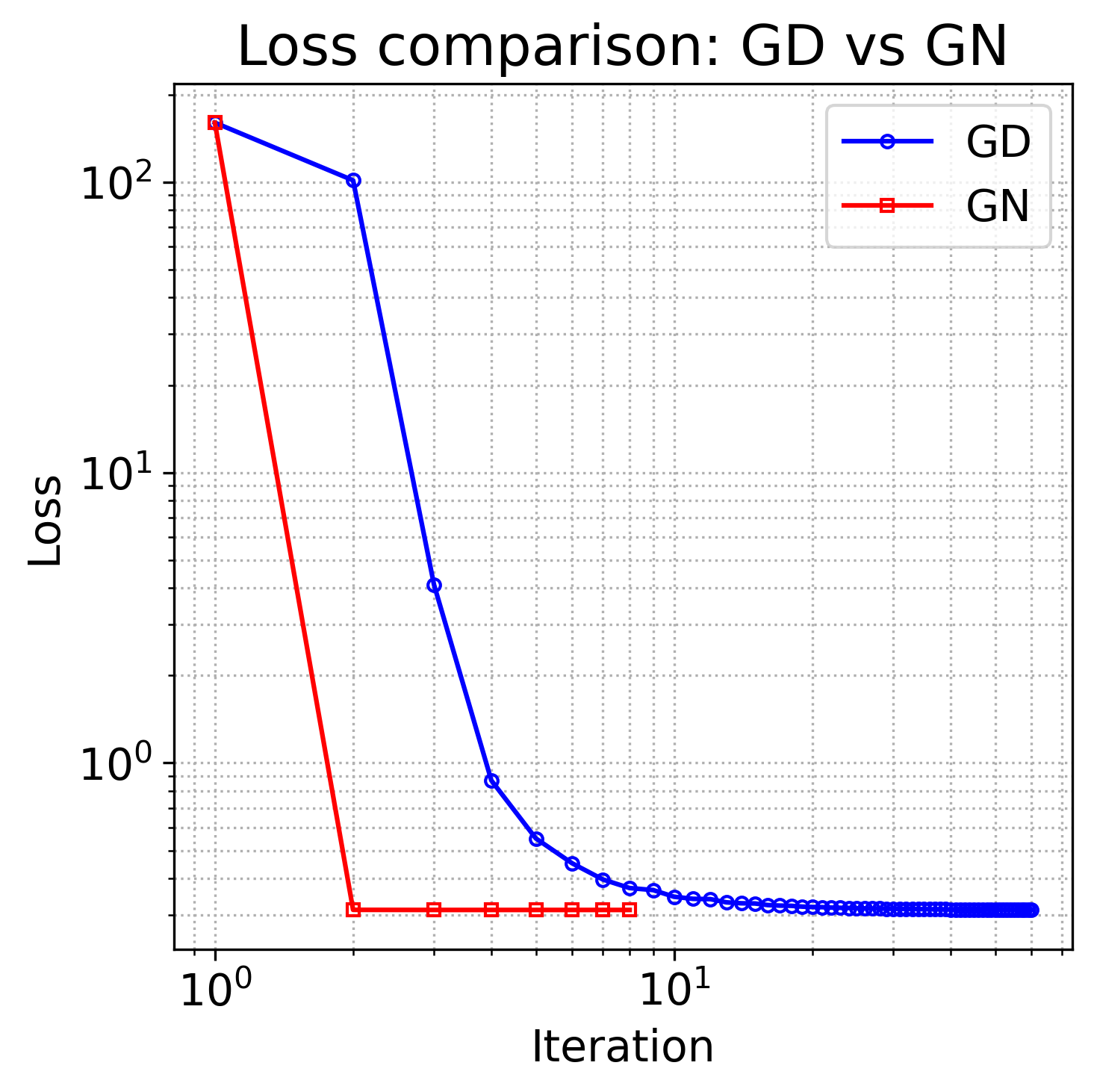}
        \caption{Loss comparison}
        \label{2D_ErgodicMFGwithCongestionLoss}
    \end{subfigure}
    
    \begin{subfigure}[b]{0.23\textwidth}
        \centering
        \includegraphics[width=\linewidth]{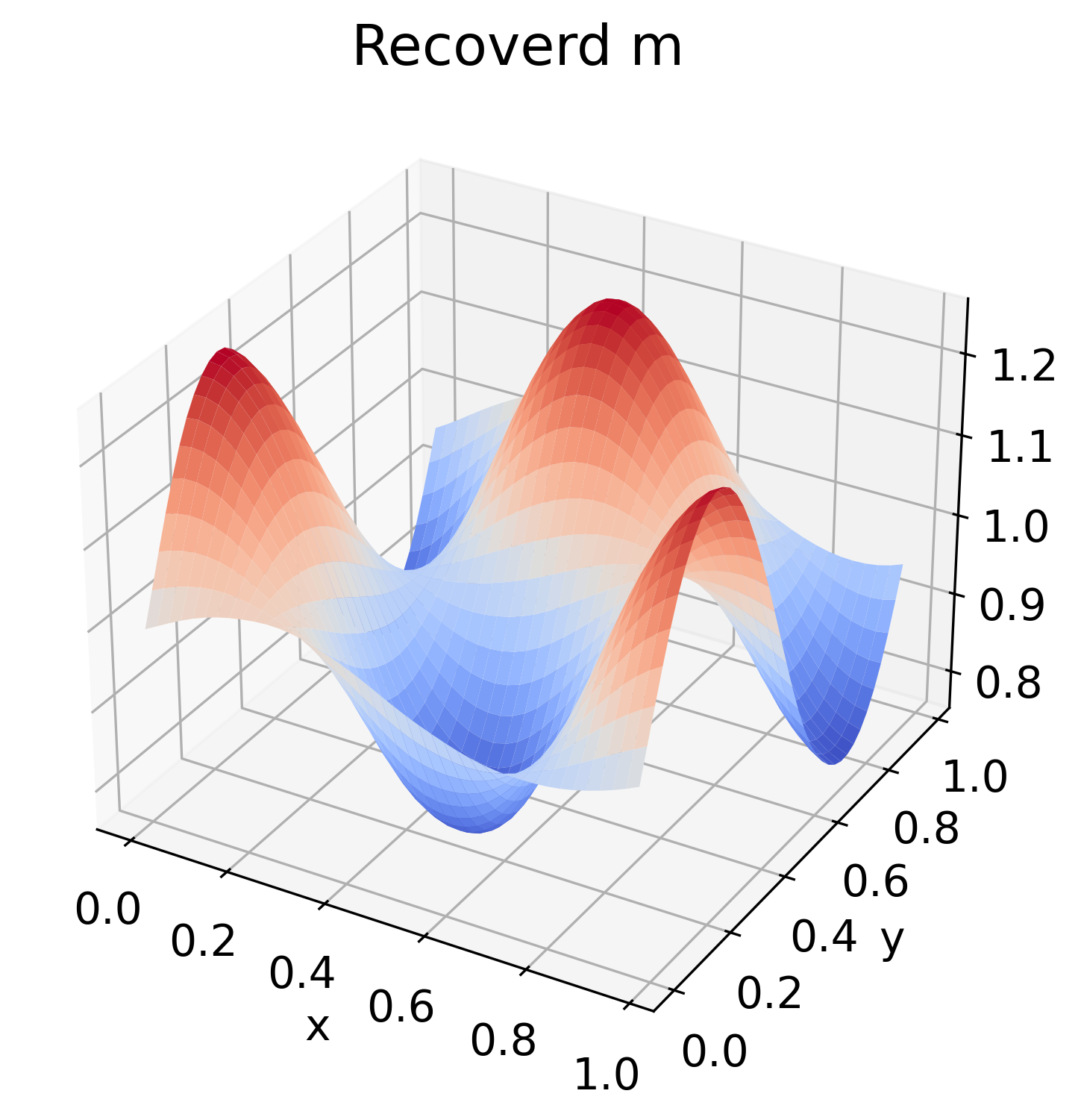}
        \caption{Recovered $m$ via GD}
        \label{2D_ErgodicMFGwithCongestionRecoverdM_GD}
    \end{subfigure}%
    \hspace{1mm}
    \begin{subfigure}[b]{0.23\textwidth}
        \centering
        \includegraphics[width=\linewidth]{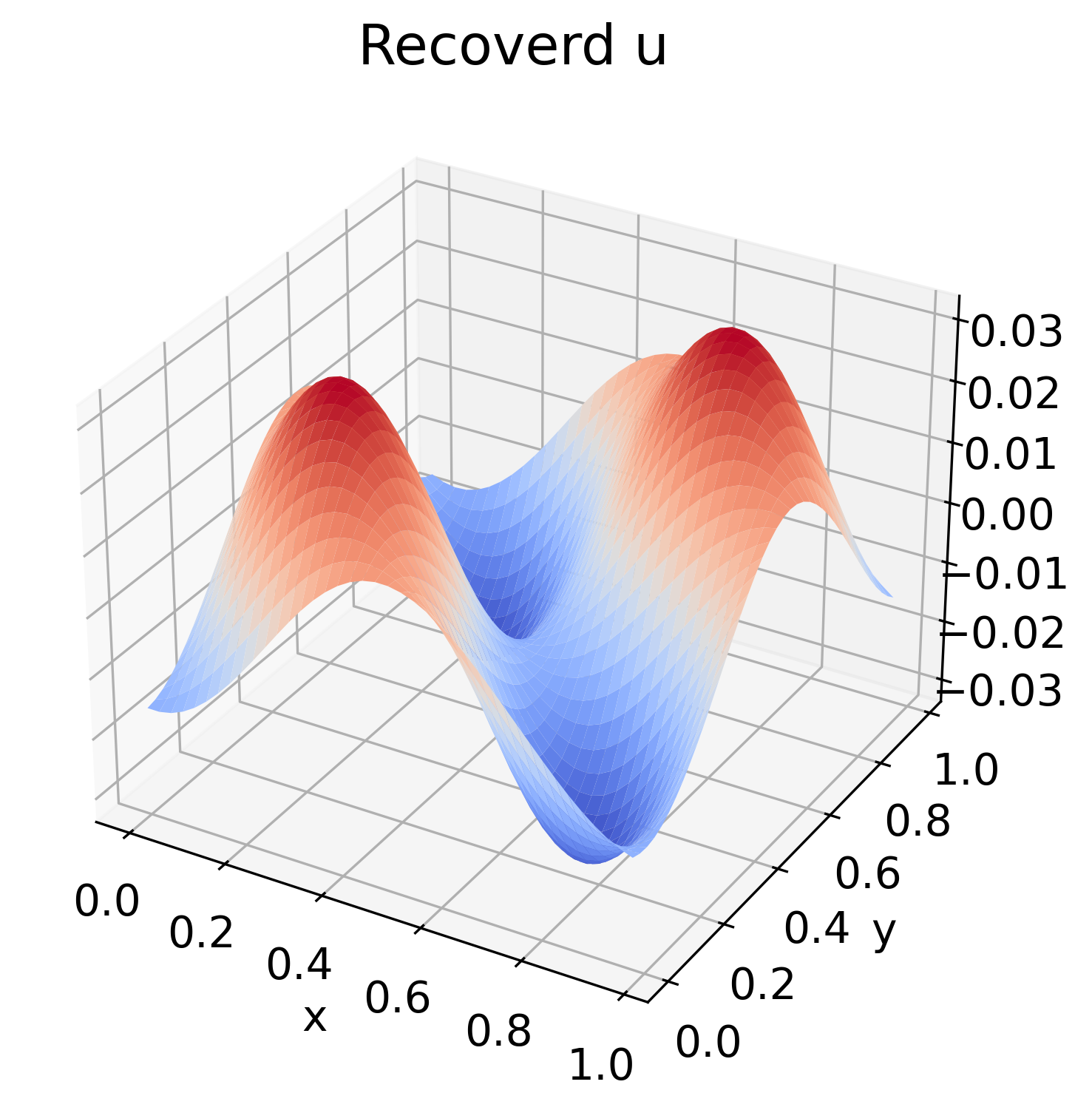}
        \caption{Recovered $u$ via GD}
        \label{2D_ErgodicMFGwithCongestionRecoverdU_GD}
    \end{subfigure}%
    \hspace{1mm}
    \begin{subfigure}[b]{0.23\textwidth}
        \centering
        \includegraphics[width=\linewidth]{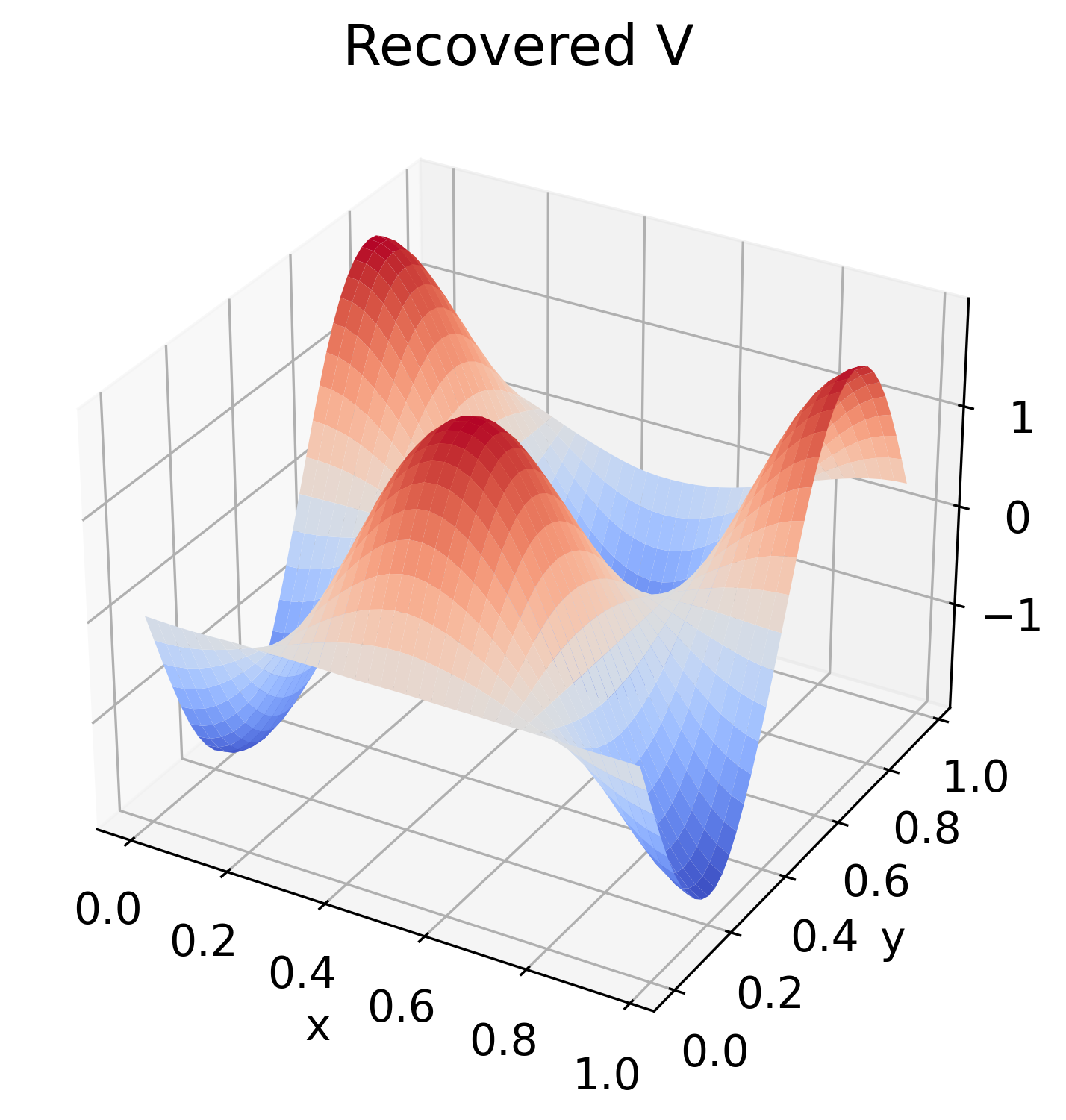}
        \caption{Recovered $V$ via GD}
        \label{2D_ErgodicMFGwithCongestionRecoverdV_GD}
    \end{subfigure}%
    \hspace{1mm}
    \begin{subfigure}[b]{0.23\textwidth}
        \centering
        \includegraphics[width=\linewidth]{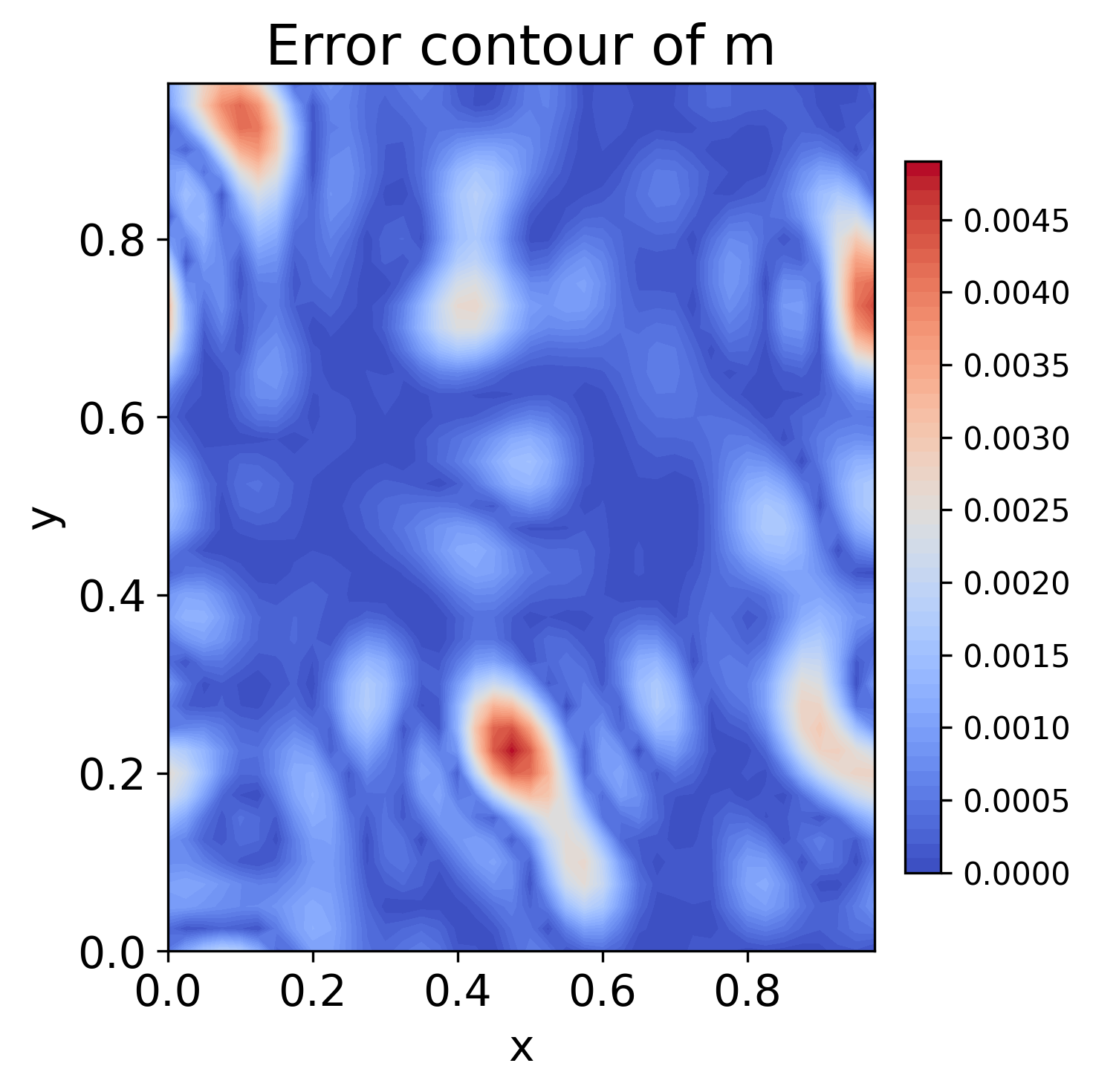}
        \caption{Error of $m$ via GD}
        \label{2D_ErgodicMFGwithCongestionErrorM_GD}
    \end{subfigure}
    
    \begin{subfigure}[b]{0.23\textwidth}
        \centering
        \includegraphics[width=\linewidth]{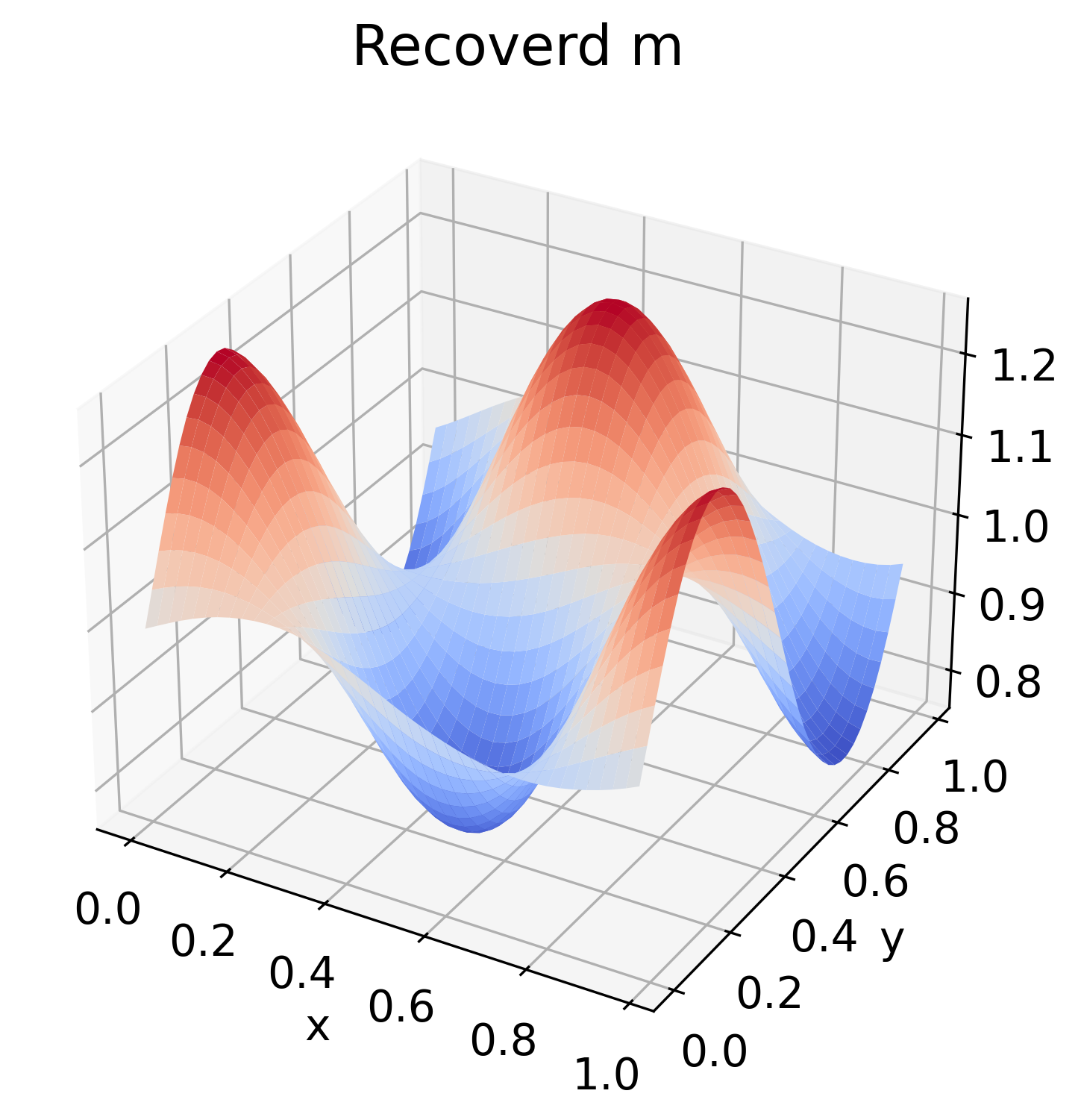}
        \caption{Recovered $m$ via GN}
        \label{2D_ErgodicMFGwithCongestionRecoverdM_GN}
    \end{subfigure}%
    \hspace{1mm}
    \begin{subfigure}[b]{0.23\textwidth}
        \centering
        \includegraphics[width=\linewidth]{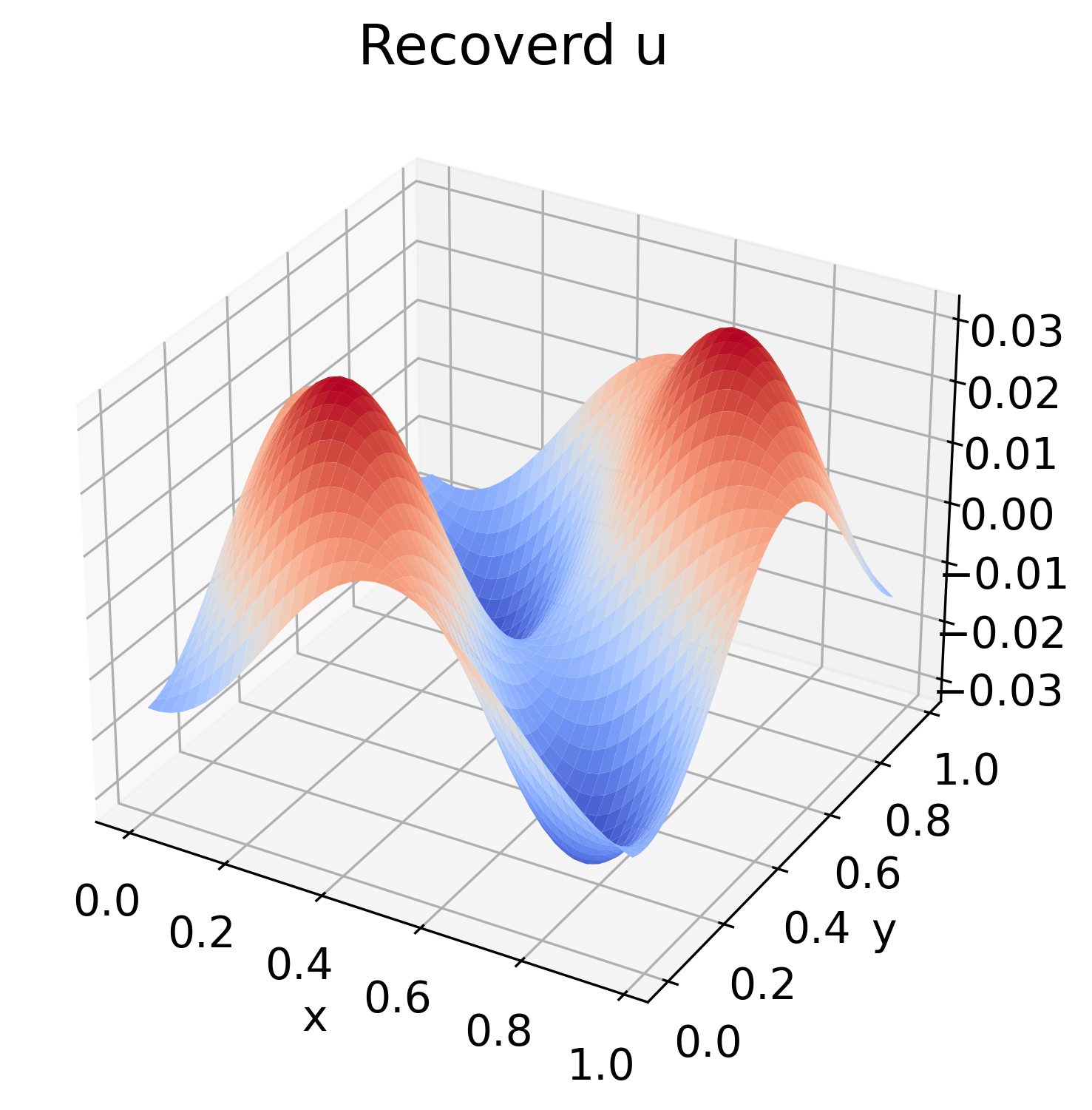}
        \caption{Recovered $u$ via GN}
        \label{2D_ErgodicMFGwithCongestionRecoverdU_GN}
    \end{subfigure}%
    \hspace{1mm}
    \begin{subfigure}[b]{0.23\textwidth}
        \centering
        \includegraphics[width=\linewidth]{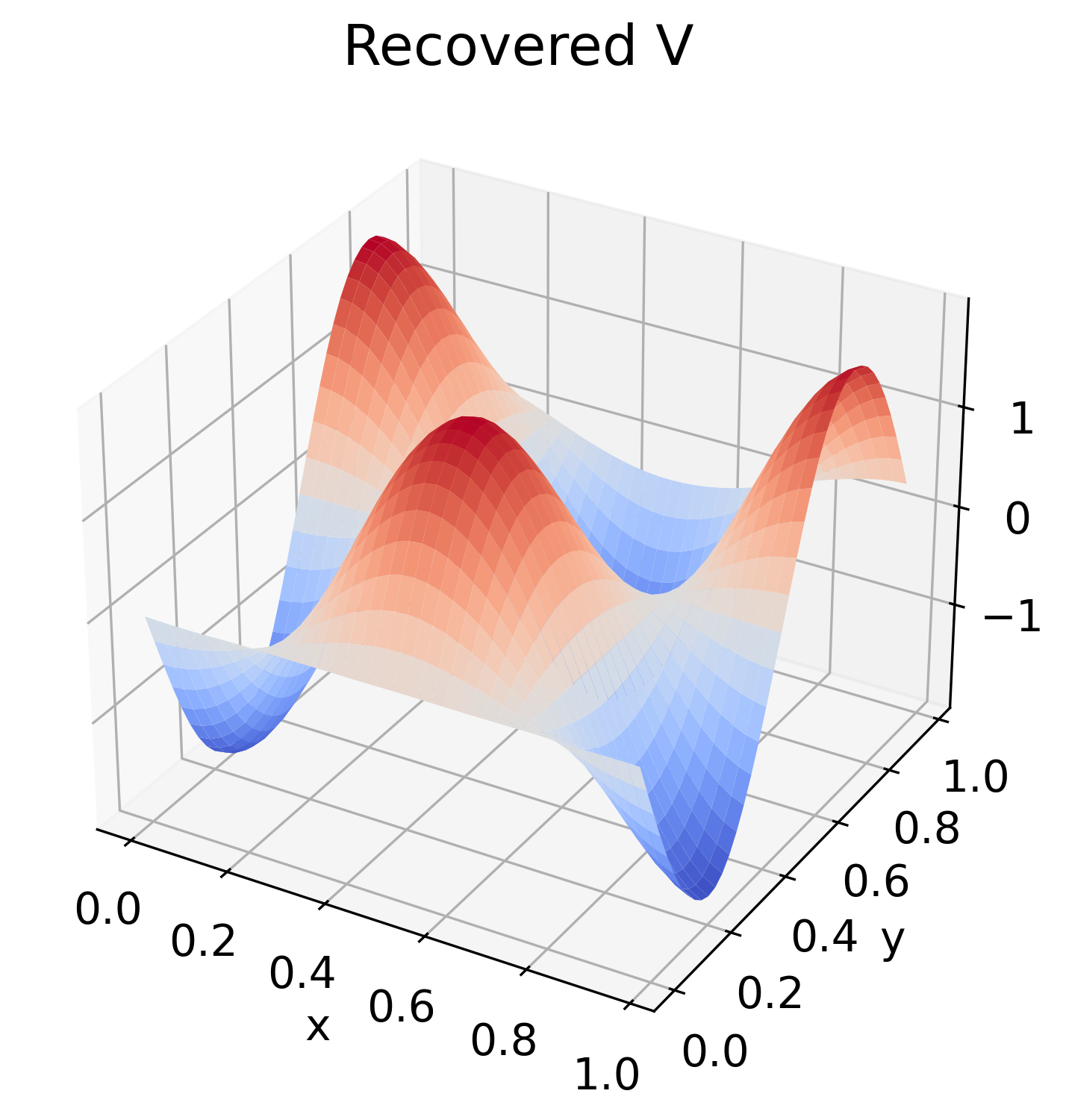}
        \caption{Recovered $V$ via GN}
        \label{2D_ErgodicMFGwithCongestionRecoverdV_GN}
    \end{subfigure}%
    \hspace{1mm}
    \begin{subfigure}[b]{0.23\textwidth}
        \centering
        \includegraphics[width=\linewidth]{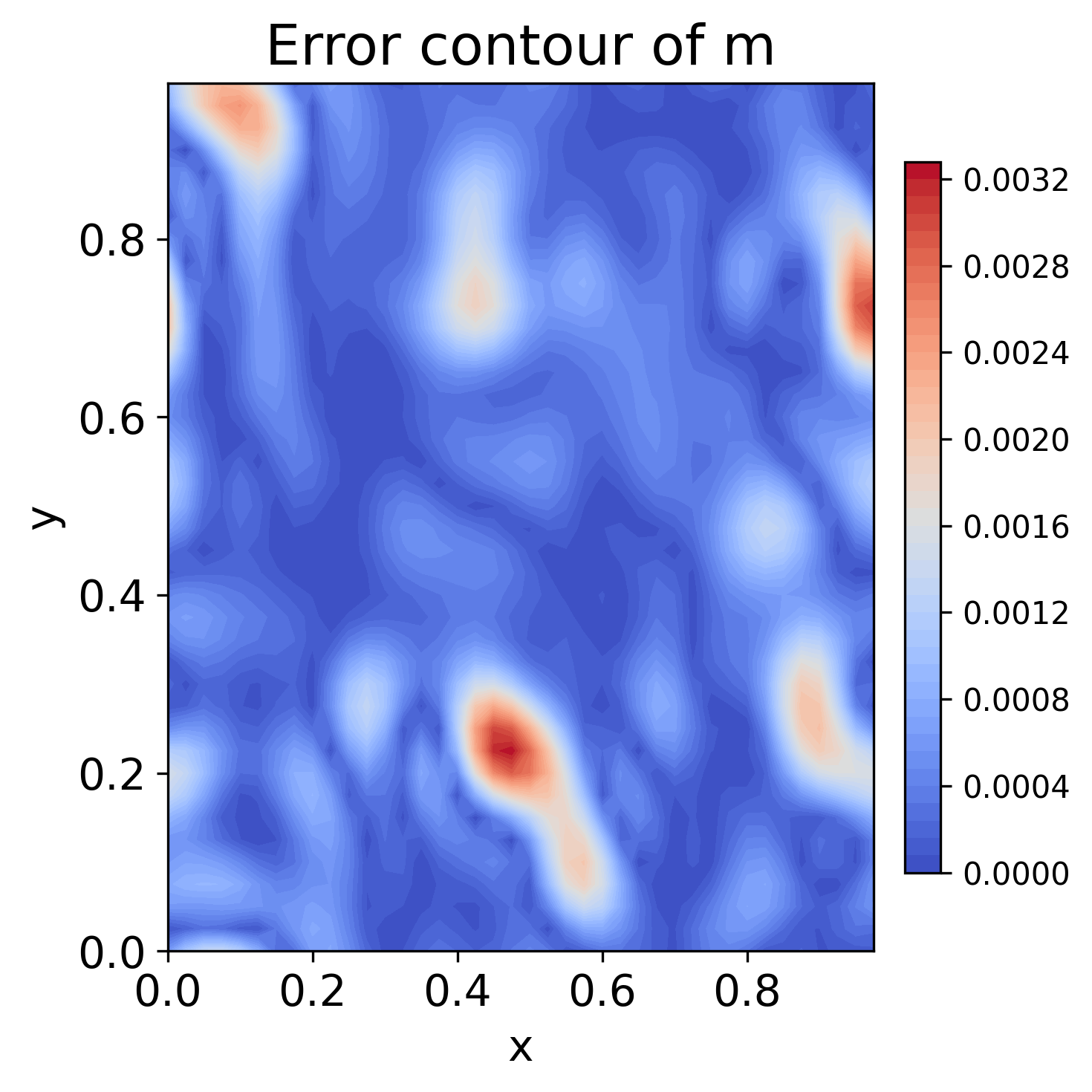}
        \caption{Error of $m$ via GN}
        \label{2D_ErgodicMFGwithCongestionErrorM_GN}
    \end{subfigure}
    
    \begin{subfigure}[b]{0.23\textwidth}
        \centering
        \includegraphics[width=\linewidth]{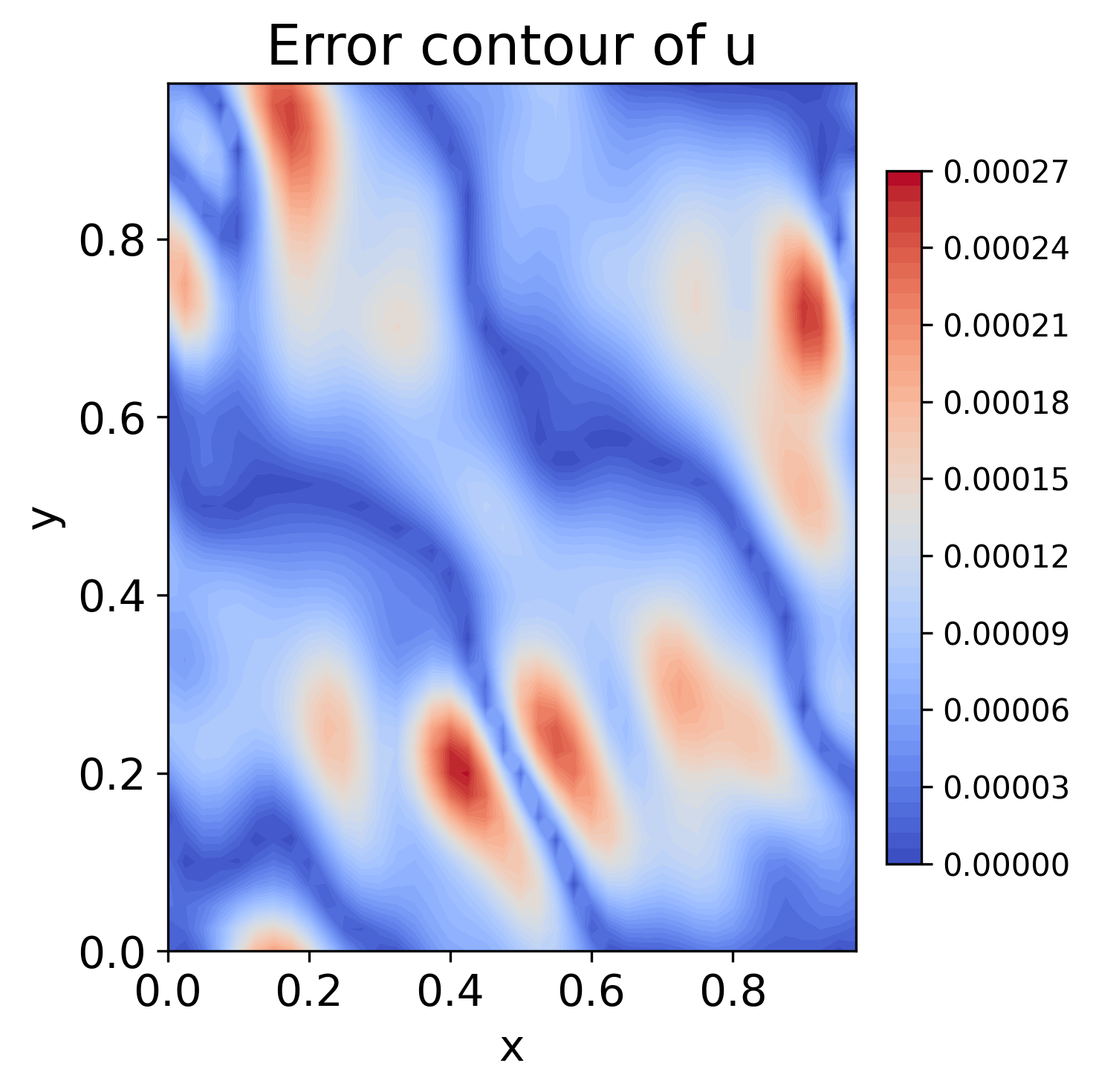}
        \caption{Error of $u$ via GD}
        \label{2D_ErgodicMFGwithCongestionErrorU_GD}
    \end{subfigure}
    \hspace{1mm}
    \begin{subfigure}[b]{0.23\textwidth}
        \centering
        \includegraphics[width=\linewidth]{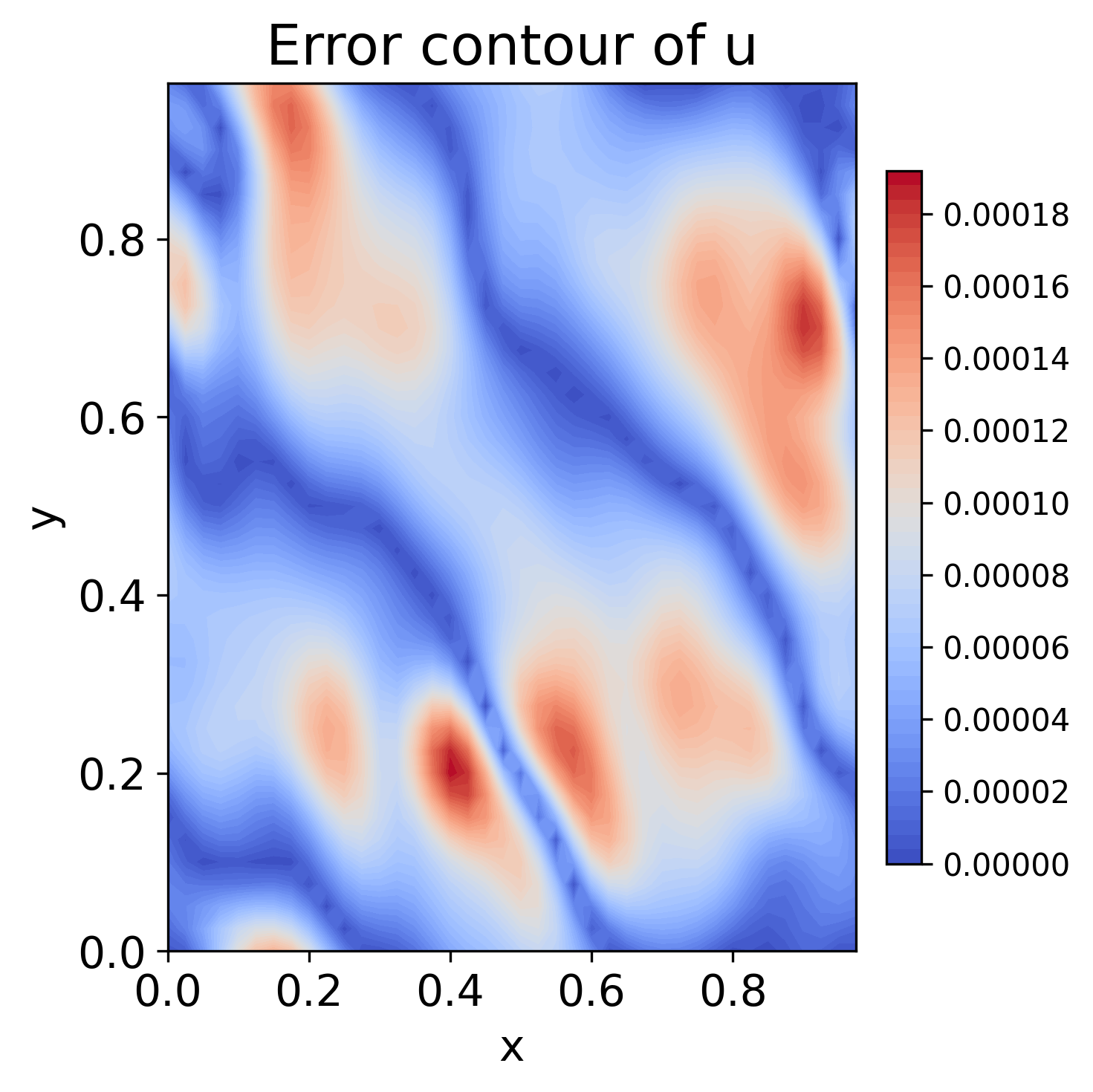}
        \caption{Error of $u$ via GN}
        \label{2D_ErgodicMFGwithCongestionErrorU_GN}
    \end{subfigure}
    \hspace{1mm}
    \begin{subfigure}[b]{0.23\textwidth}
        \centering
        \includegraphics[width=\linewidth]{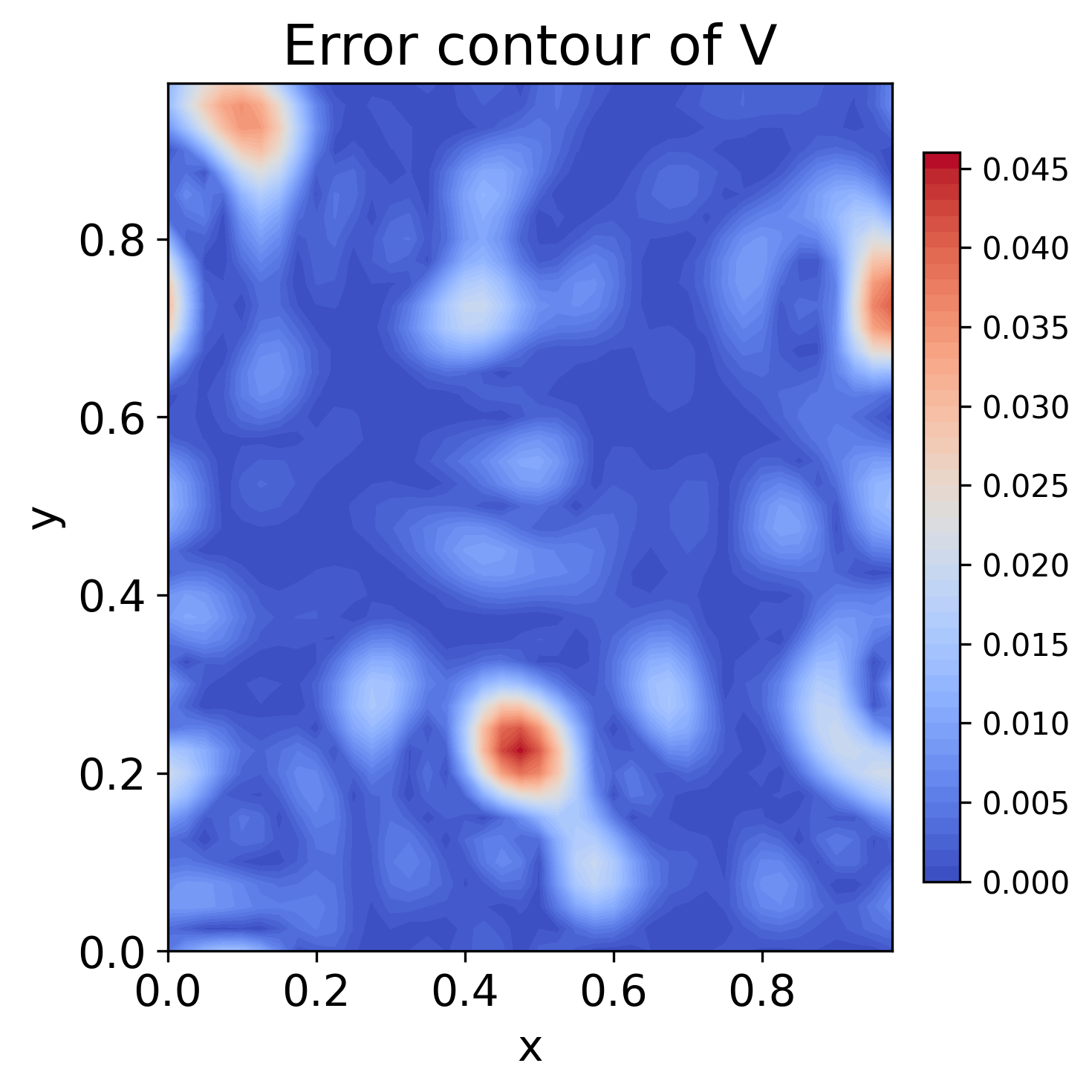}
        \caption{Error of $V$ via GD}
        \label{2D_ErgodicMFGwithCongestionErrorV_GD}
    \end{subfigure}%
    \hspace{1mm}
    \begin{subfigure}[b]{0.23\textwidth}
        \centering
        \includegraphics[width=\linewidth]{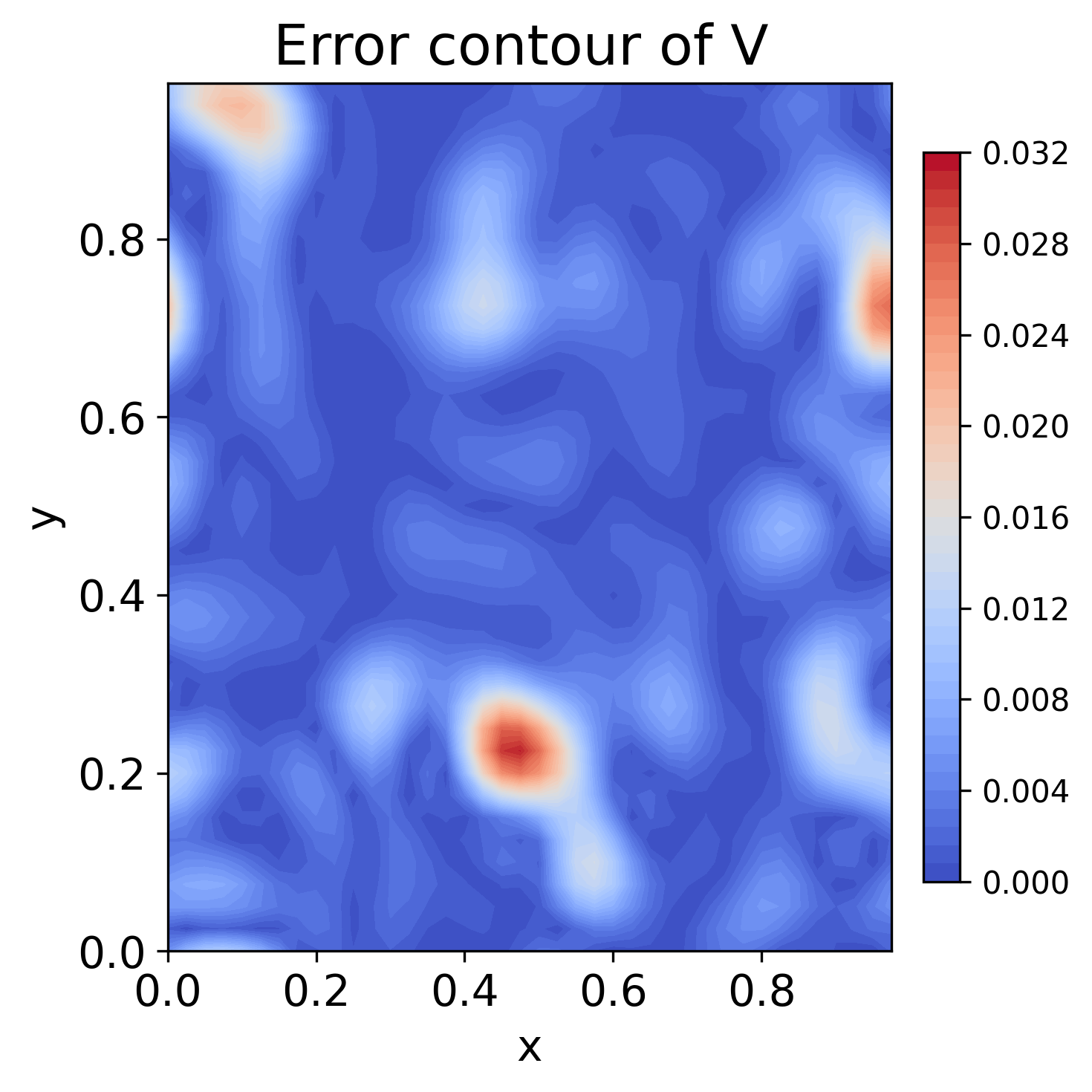}
        \caption{Error of $V$ via GN}
        \label{2D_ErgodicMFGwithCongestionErrorV_GN}
    \end{subfigure}

    \caption{Numerical results for the inverse problem of the stationary MFG in \eqref{eq:first_order_MFGS_congestion}.  (a), (b), (c) are references for $m,u,V$; (d) log-log plot of the loss comparison for GD and GN across iterations; (e), (f), (g) recovered $m,u,V$ via GD; (h), (m), (o) errors of $m,u,V$ via GD; (i), (j), (k) recovered $m,u,V$ via GN; (l), (n), (p) errors of $m,u,V$ via GN.}
    \label{2D_ErgodicMFGwithCongestionPlot}
\end{figure}

\subsubsection{A Non-Potential  Second-Order Stationary MFG }
\label{2DMFGexample2}
Here, we consider the following non-potential stationary MFG on \(\mathbb{T}^2\):
\begin{align}
\label{eq:2DstationaryMFGinvSmall_nu}
\begin{cases}
    -\nu \Delta u + \frac{|\nabla u|^2}{2} - V(x,y) = \lambda + f[m], & (x,y)\in\mathbb{T}^2, \\
    -\nu \Delta m - \div\!\bigl(m \nabla u\bigr) = 0, & (x,y)\in\mathbb{T}^2, \\
    \displaystyle \int_{\mathbb{T}^2} u\,\mathrm{d}x\,\mathrm{d}y = 0,\qquad 
    \displaystyle \int_{\mathbb{T}^2} m\,\mathrm{d}x\,\mathrm{d}y = 1.
\end{cases}
\end{align}
We take \(\nu=0.1\),  $f[m]= m^2 +  b(x, y)\cdot\nabla m$, $b(x,y) = (-\sin(2\pi y),\, \sin(2\pi x))$, for which $\operatorname{div}(b)=0$, and the spatial cost $
V(x,y)= -\bigl(\sin(2\pi x)+\cos(4\pi x)+\sin(2\pi y)\bigr)$. We compute a reference solution \((u^*,m^*,\lambda^*)\) numerically for \eqref{eq:2DstationaryMFGinvSmall_nu}. The inverse problem is to recover \((u,m,\lambda,V)\) from partial noisy observations of \(m\) and \(V\). This MFG still satisfies the Lasry--Lions monotonicity condition. However,
because of the derivative term, the model is not a potential MFG and does not
admit the variational structure used in \cite{briceno2018proximal}. Therefore,
the globally convergent optimization algorithm in \cite{briceno2018proximal}
does not apply here. In contrast, the HRF method developed in this paper can be
applied to this monotone non-potential setting, and we use this example to test
the inverse framework beyond the potential MFG case.


\textbf{Experimental Setup.} We identify \(\mathbb{T}^2\) with \([0,1)^2\) and discretize it with \(h_x=h_y=\tfrac{1}{30}\), yielding 900 grid points. The observation points are sampled at random: 72 locations for $m$ and 180 locations for $V$ are selected uniformly from $[0,1)^2$.  The regularization parameters are set to \(\alpha = 0.04\), \(\beta =1 \), \(\gamma = 1\). Gaussian noise \(\mathcal{N}(0,\eta^2 I)\) with \(\eta = 10^{-3}\) is added to the observations. The system is initialized with $u\equiv 0$, $m\equiv 1$, and $V\equiv 0$ for the numerical simulation. We compute the reference solution by the HRF and then solve the inverse problem with GD and GN.

\textbf{Experimental Results.}
Figures~\ref{2D_ErgodicMFGwithSmallViscositySamplesM}--\ref{2D_ErgodicMFGwithSmallViscositySamplesV} show the collocation grid and the observation locations for \(m\) and \(V\). Table~\ref{table:2D_ErgodicMFGwithSmallViscosityHbar} compares the HRF reference value of \(\lambda\) with the values recovered by GD and GN. Figure~\ref{2D_ErgodicMFGwithSmallViscosityPlot} summarizes the numerical results for \eqref{eq:2DstationaryMFGinvSmall_nu}: it includes the HRF reference fields, the GD and GN reconstructions (e.g., \(m\) in Panels~\ref{2D_ErgodicMFGwithSmallViscosityReferenceM}, \ref{2D_ErgodicMFGwithSmallViscosityRecoverdM_GD}, and~\ref{2D_ErgodicMFGwithSmallViscosityRecoverdM_GN}), and the corresponding pointwise absolute error maps for \(m\), \(u\), and \(V\). Finally, Figure~\ref{2D_ErgodicMFGwithSmallViscosityLoss} shows that GN reaches a comparable optimum in fewer iterations than GD and yields more accurate reconstructions.

\begin{figure}[!htbp]
    \centering
    \begin{subfigure}[b]{0.23\textwidth}
        \centering
        \includegraphics[width=\linewidth]{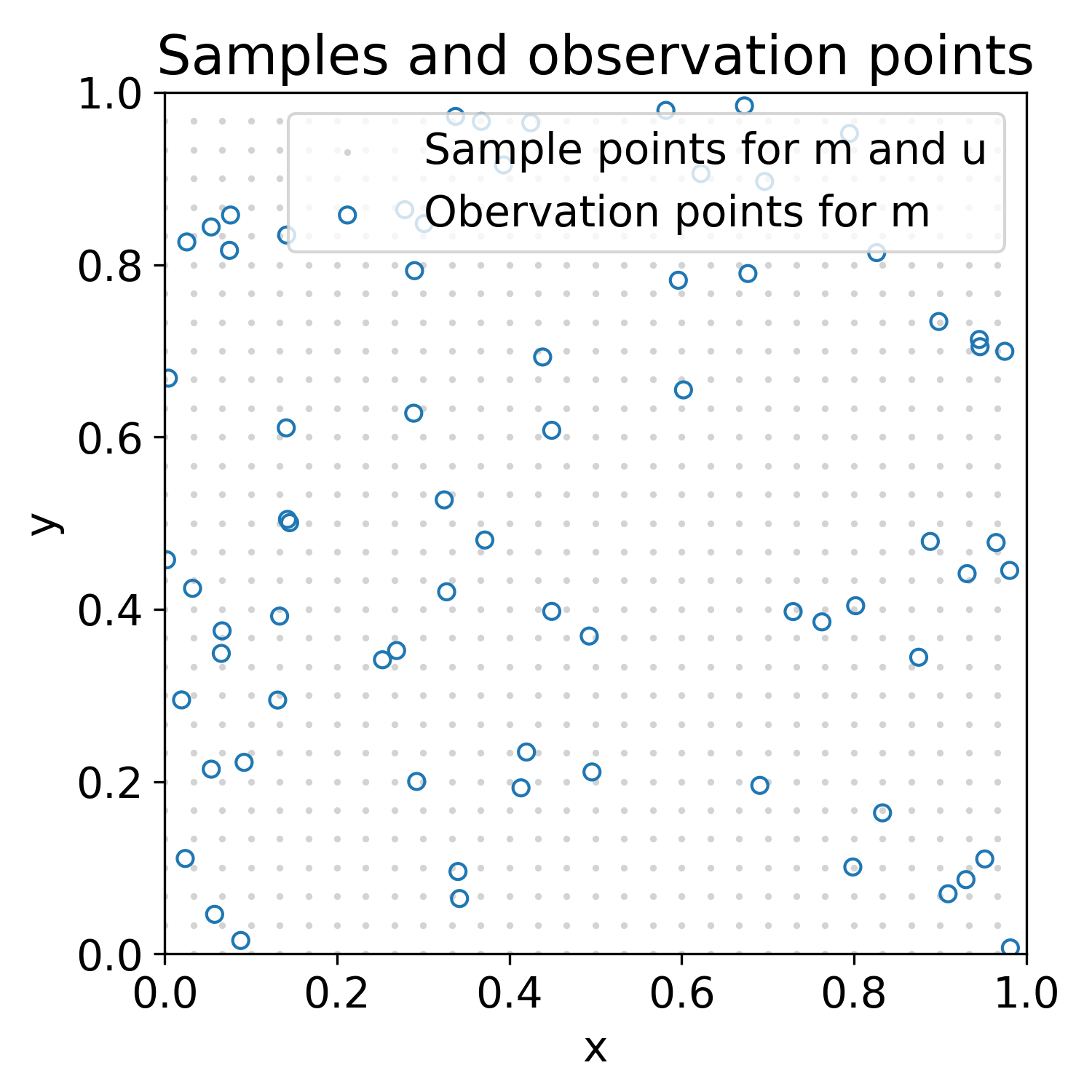}
        \caption{Samples and observations for $m$}
        \label{2D_ErgodicMFGwithSmallViscositySamplesM}
    \end{subfigure}
    \hspace{1mm}  
    \begin{subfigure}[b]{0.23\textwidth}
        \centering
        \includegraphics[width=\linewidth]{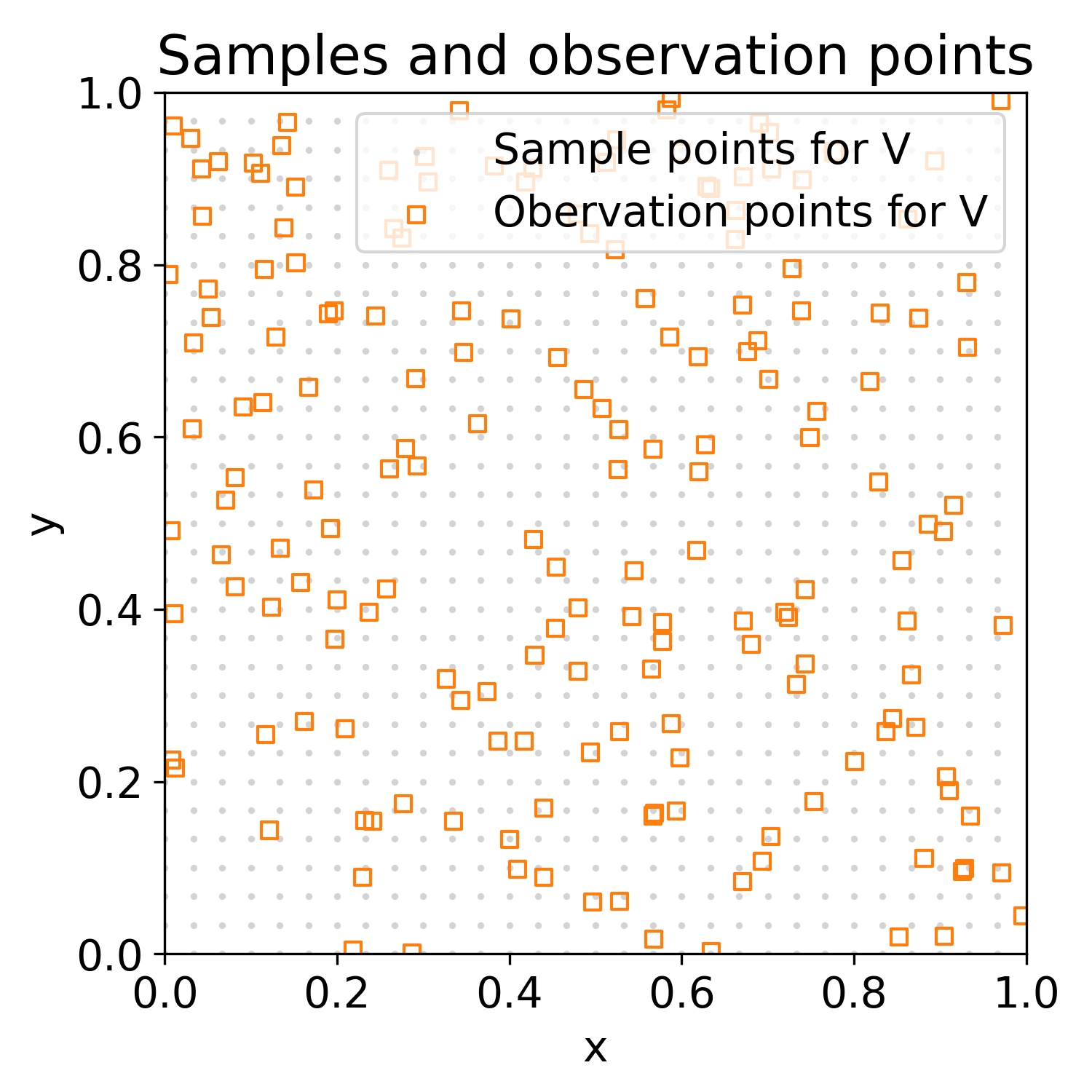}
        \caption{Samples and observations for $V$}      \label{2D_ErgodicMFGwithSmallViscositySamplesV}
    \end{subfigure}
    \caption{The inverse problem of the stationary MFG \eqref{eq:2DstationaryMFGinvSmall_nu}: samples for $m$ and $V$ and corresponding observation points.}
    \label{fig:samples_observationsmulti_2D_ErgodicMFGwithSmallViscosity}
\end{figure}

\begin{table}[!htbp]
  \centering
  \caption{Numerical results for \(\lambda\) in the stationary MFG  \eqref{eq:2DstationaryMFGinvSmall_nu} using different methods. The reference solution is computed by the HRF, and the recovered solution is obtained by the GD method and the GN method.}
  \label{table:2D_ErgodicMFGwithSmallViscosityHbar}
  \begin{tabular}{|c|c|c|c|}
    \hline
    Method & Reference & GD & GN \\
    \hline
    \(\lambda\) & -0.538470584730 & -0.543568902915 & -0.543101122058
\\
    \hline
  \end{tabular}
\end{table}

\begin{figure}[!htbp]
    \centering
    \begin{subfigure}[b]{0.23\textwidth}
        \centering
        \includegraphics[width=\linewidth]{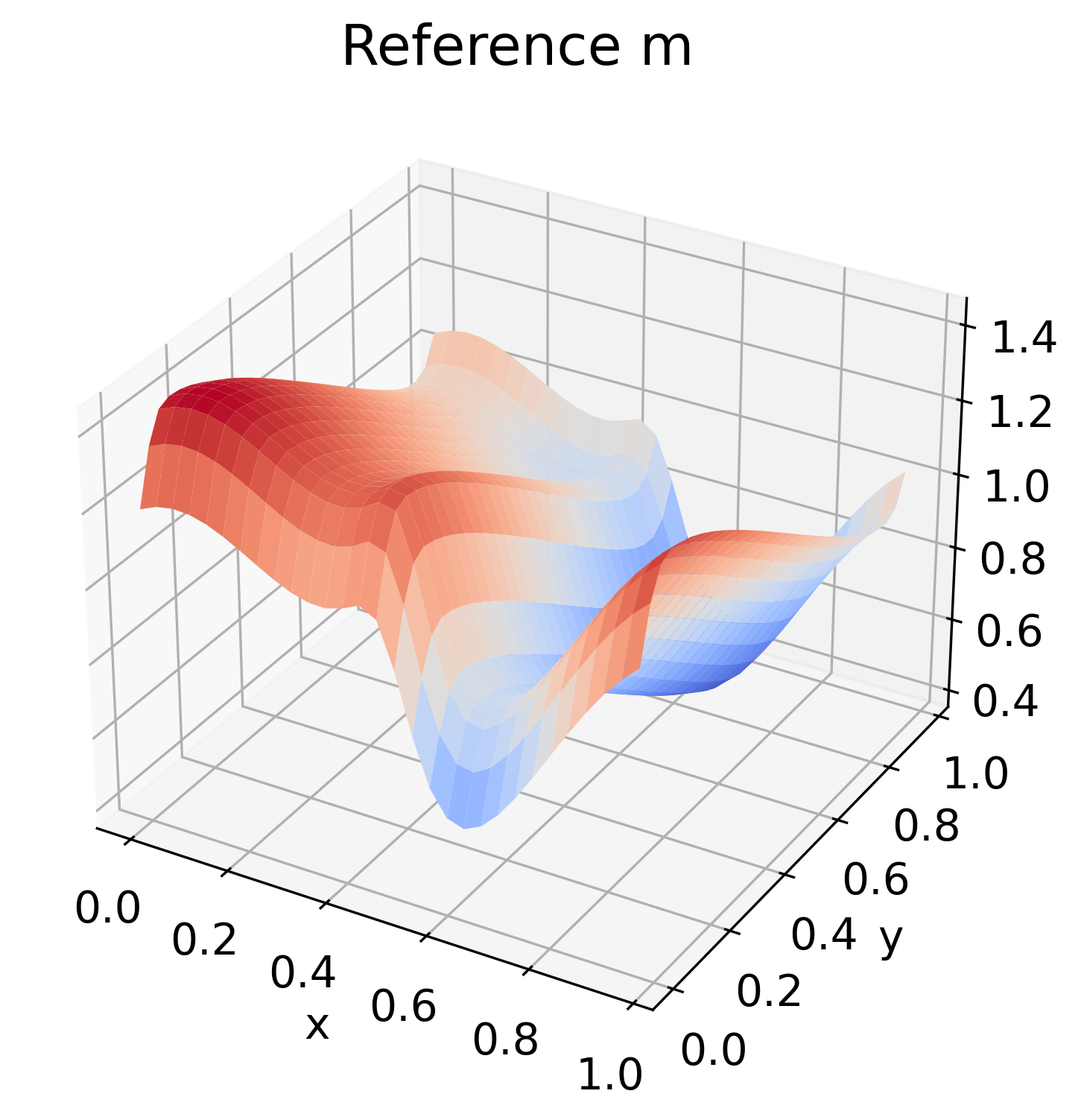}
        \caption{$m$ reference}
        \label{2D_ErgodicMFGwithSmallViscosityReferenceM}
    \end{subfigure}%
    \hspace{1mm}
    \begin{subfigure}[b]{0.23\textwidth}
        \centering
        \includegraphics[width=\linewidth]{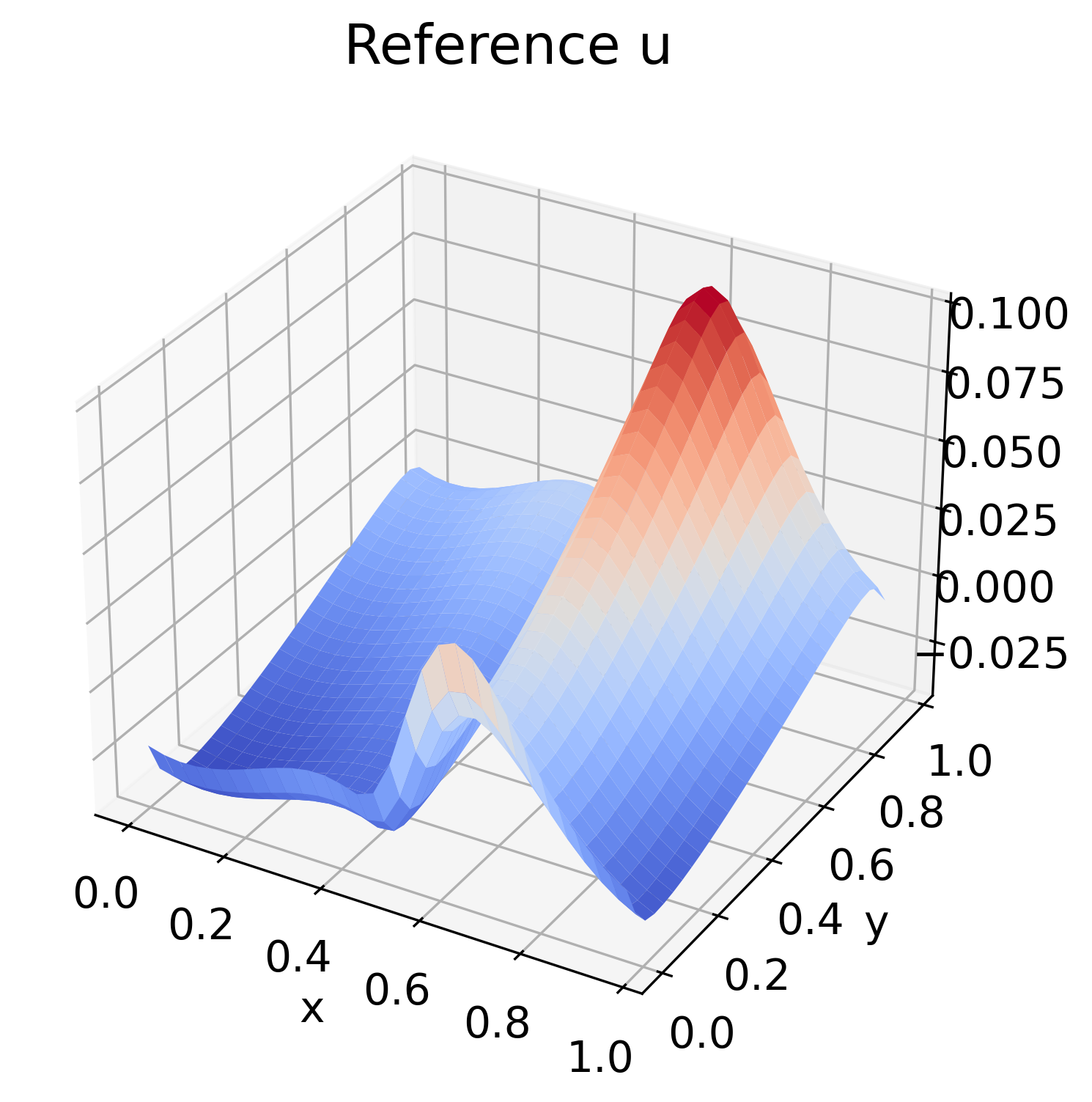}
        \caption{$u$ reference}
        \label{2D_ErgodicMFGwithSmallViscosityReferenceU}
    \end{subfigure}%
    \hspace{1mm}
    \begin{subfigure}[b]{0.23\textwidth}
        \centering
        \includegraphics[width=\linewidth]{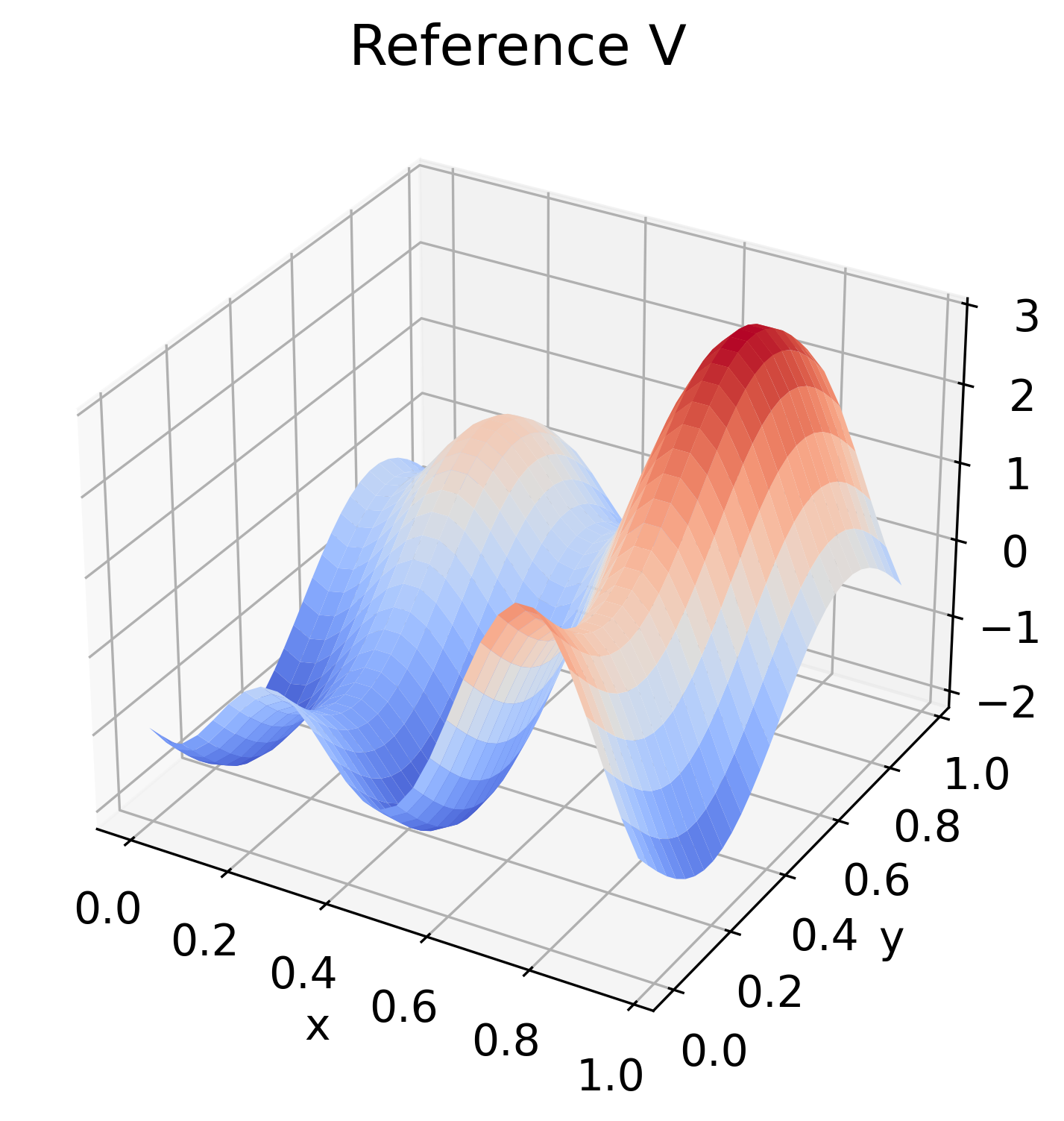}
        \caption{$V$ reference}
        \label{2D_ErgodicMFGwithSmallViscosityReferenceV}
    \end{subfigure}%
    \hspace{1mm}
    \begin{subfigure}[b]{0.23\textwidth}
        \centering
        \includegraphics[width=\linewidth]{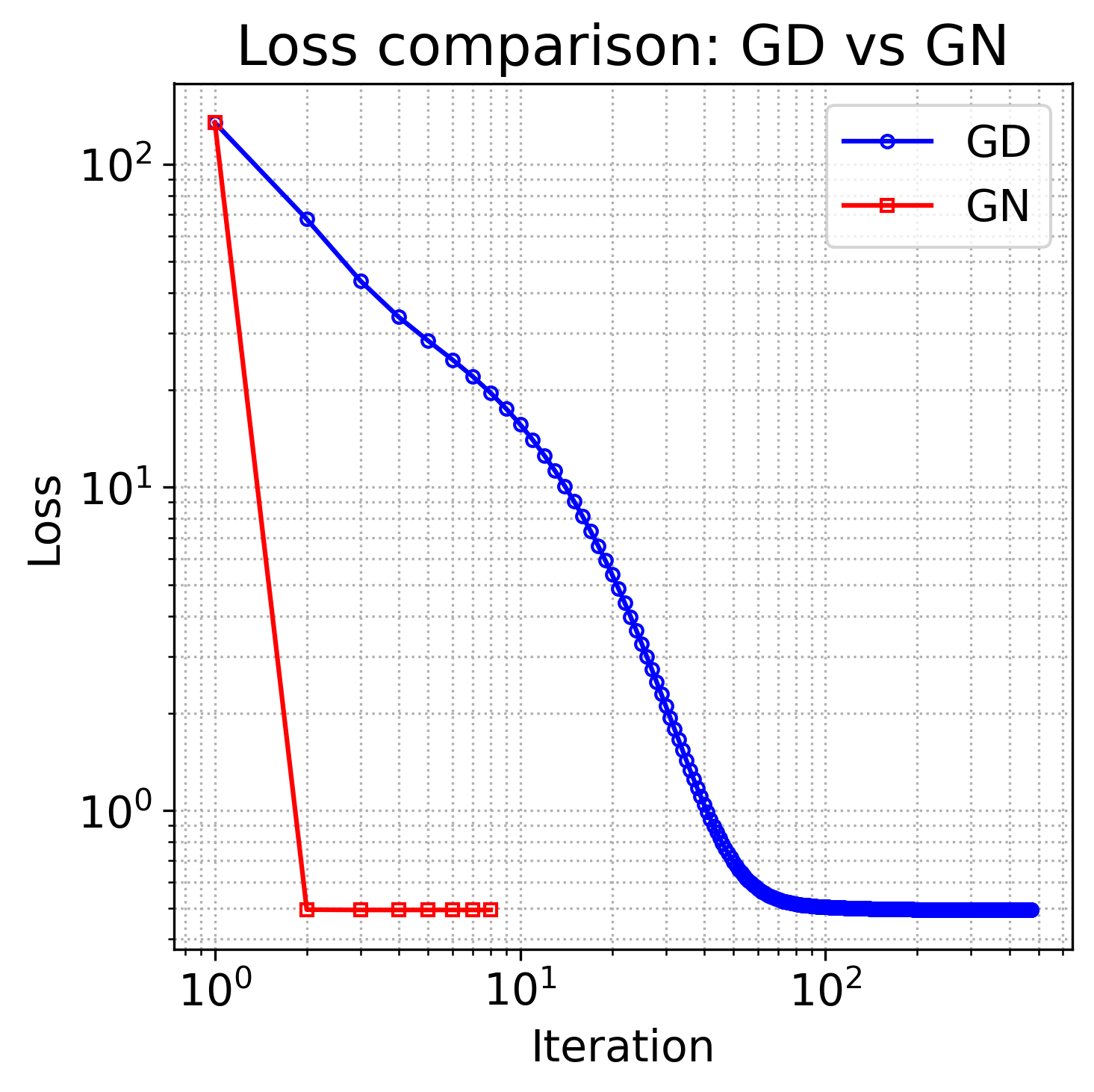}
        \caption{Loss comparison}
        \label{2D_ErgodicMFGwithSmallViscosityLoss}
    \end{subfigure}
    
    \begin{subfigure}[b]{0.23\textwidth}
        \centering
        \includegraphics[width=\linewidth]{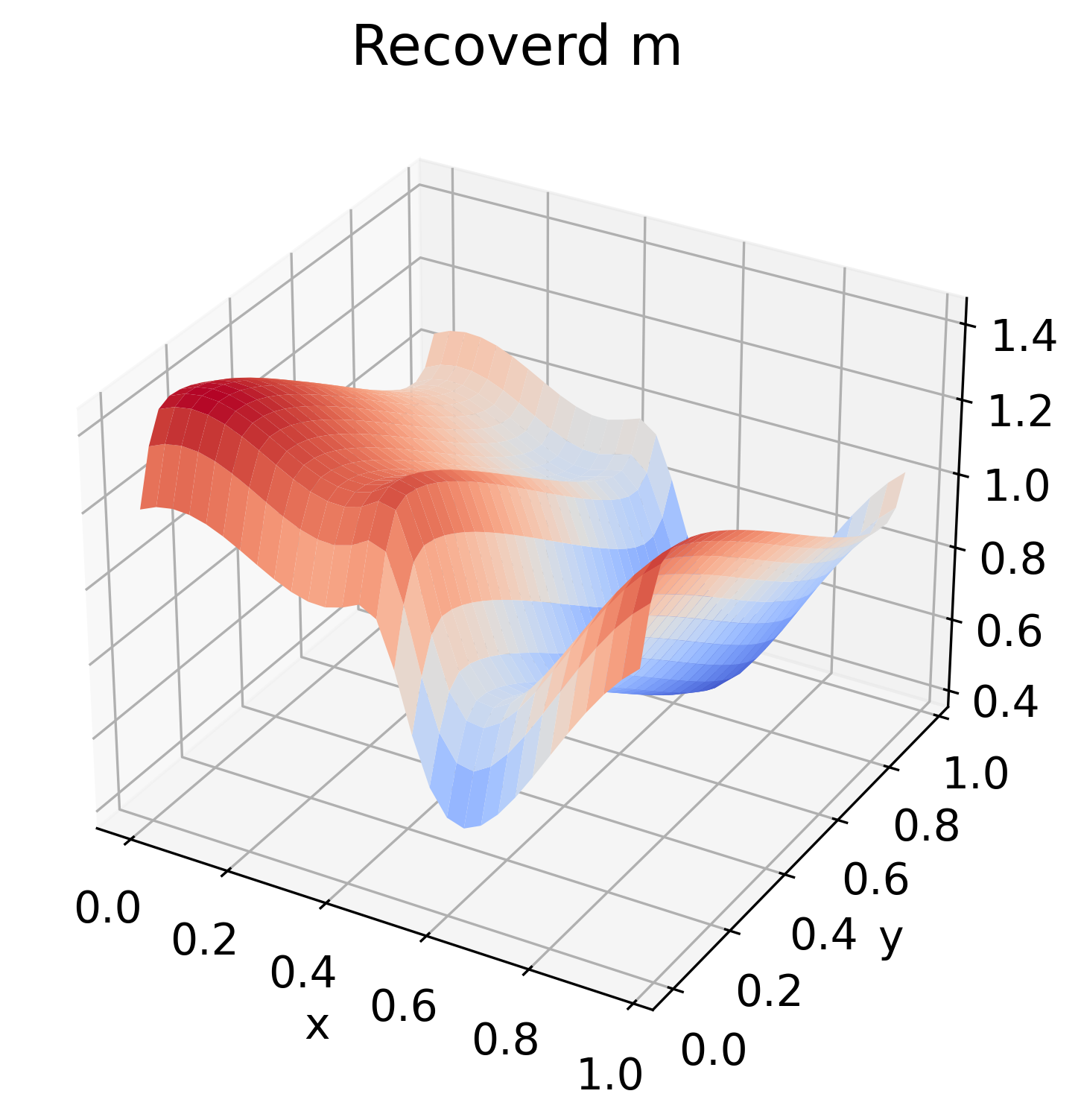}
        \caption{Recovered $m$ via GD}
        \label{2D_ErgodicMFGwithSmallViscosityRecoverdM_GD}
    \end{subfigure}%
    \hspace{1mm}
    \begin{subfigure}[b]{0.23\textwidth}
        \centering
        \includegraphics[width=\linewidth]{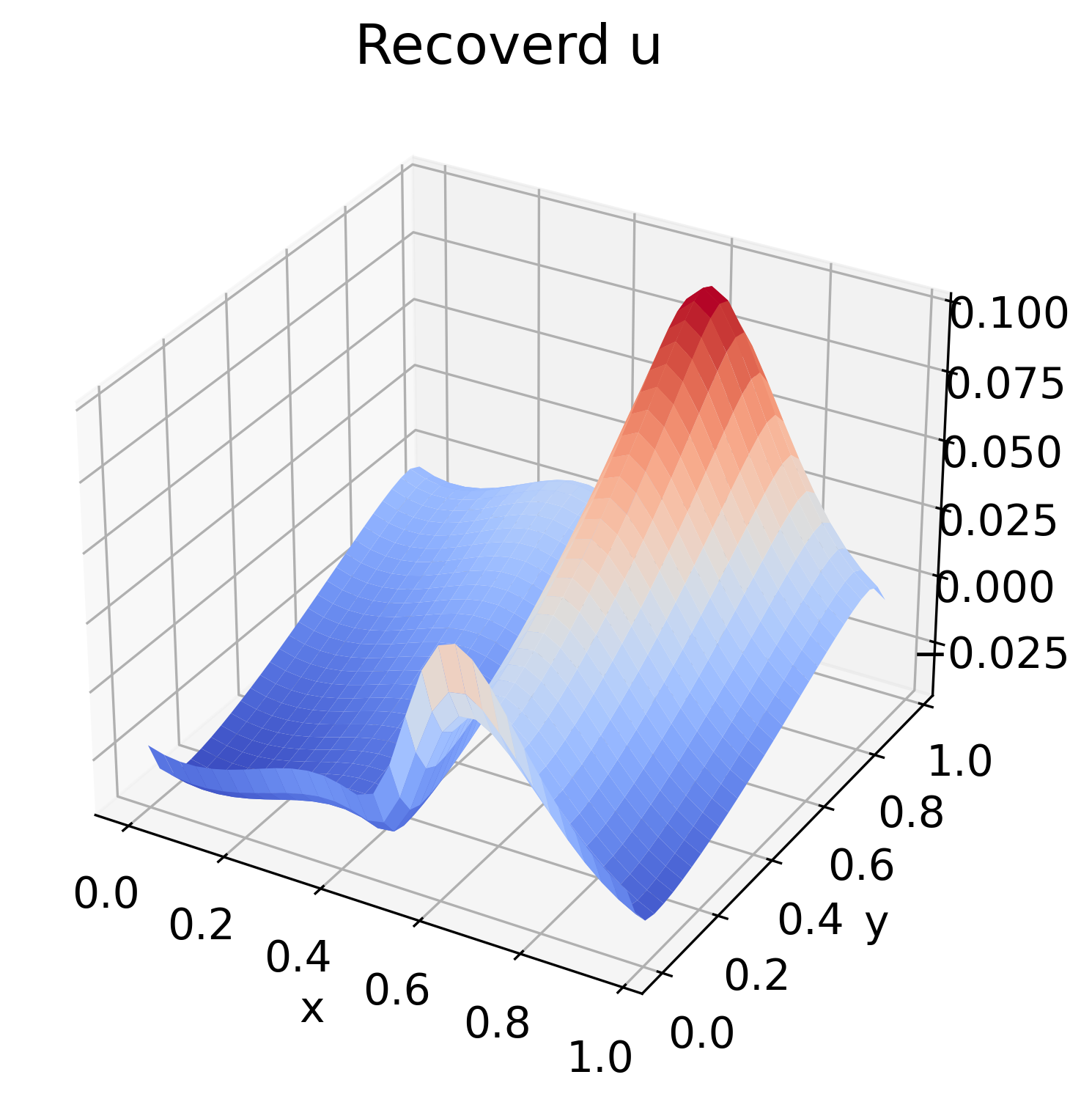}
        \caption{Recovered $u$ via GD}
        \label{2D_ErgodicMFGwithSmallViscosityRecoverdU_GD}
    \end{subfigure}%
    \hspace{1mm}
    \begin{subfigure}[b]{0.23\textwidth}
        \centering
        \includegraphics[width=\linewidth]{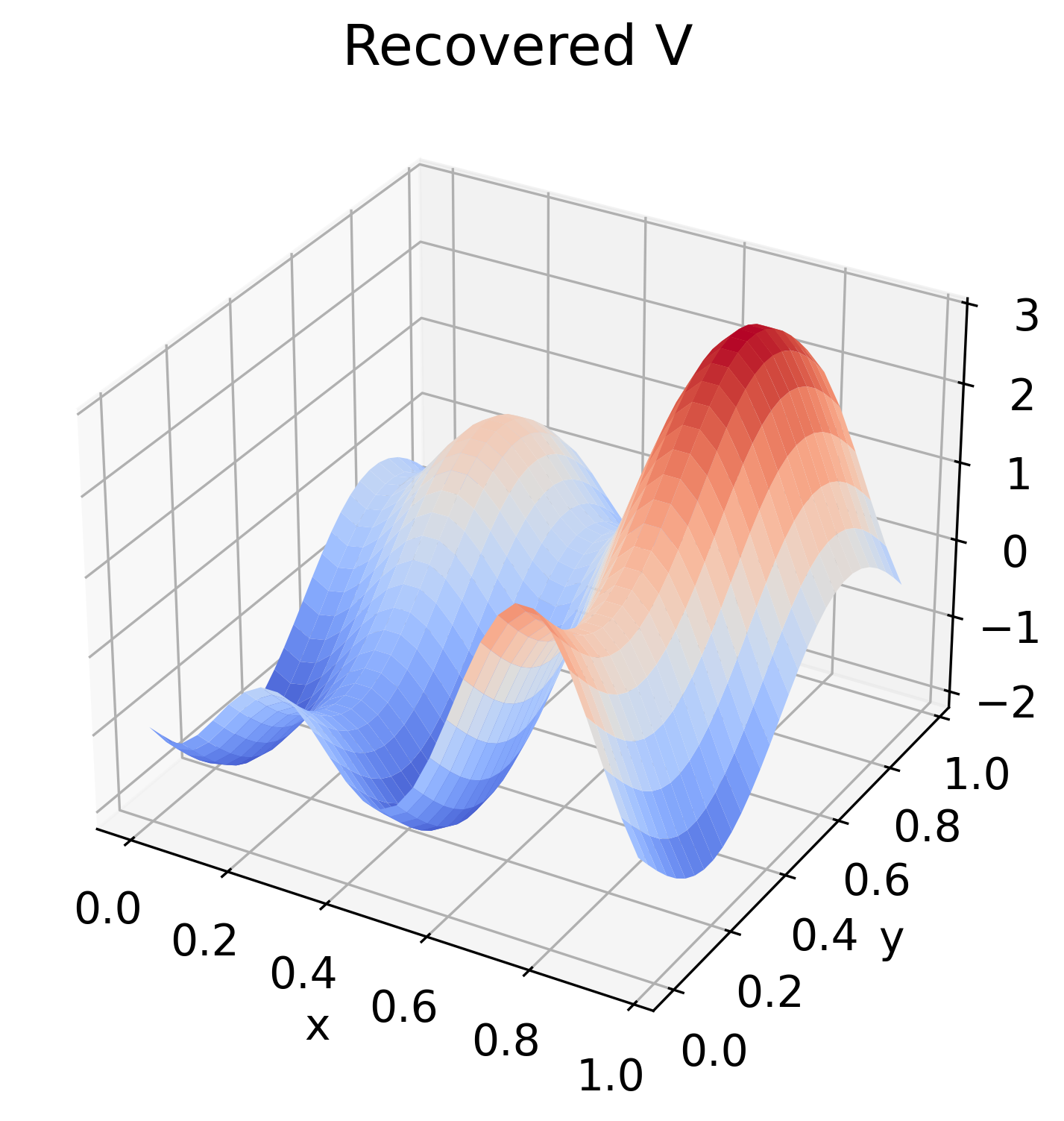}
        \caption{Recovered $V$ via GD}
        \label{2D_ErgodicMFGwithSmallViscosityRecoverdV_GD}
    \end{subfigure}%
    \hspace{1mm}
    \begin{subfigure}[b]{0.23\textwidth}
        \centering
        \includegraphics[width=\linewidth]{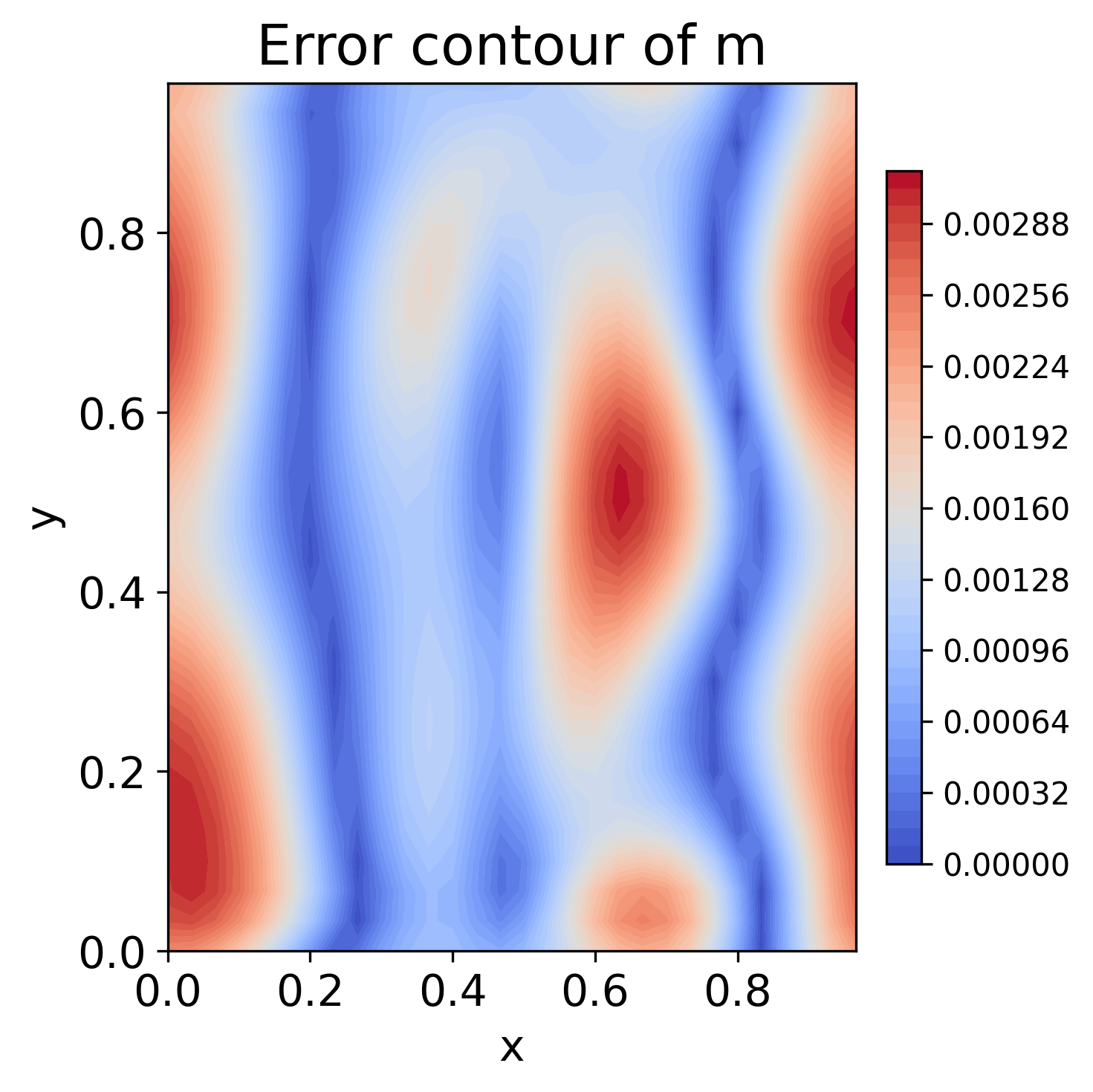}
        \caption{Error of $m$ via GD}
        \label{2D_ErgodicMFGwithSmallViscosityErrorM_GD}
    \end{subfigure}
    
    \begin{subfigure}[b]{0.23\textwidth}
        \centering
        \includegraphics[width=\linewidth]{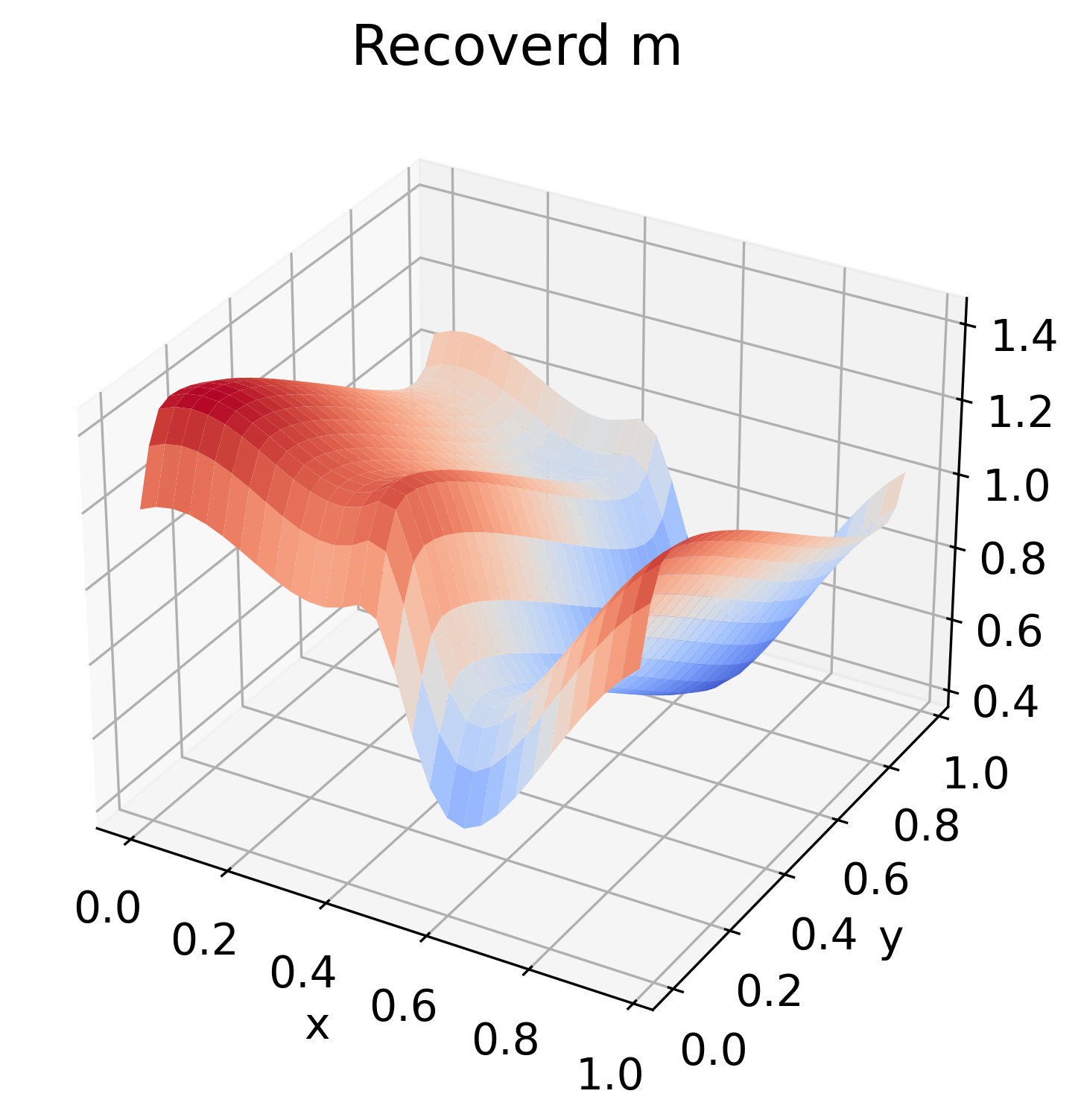}
        \caption{Recovered $m$ via GN}
        \label{2D_ErgodicMFGwithSmallViscosityRecoverdM_GN}
    \end{subfigure}%
    \hspace{1mm}
    \begin{subfigure}[b]{0.23\textwidth}
        \centering
        \includegraphics[width=\linewidth]{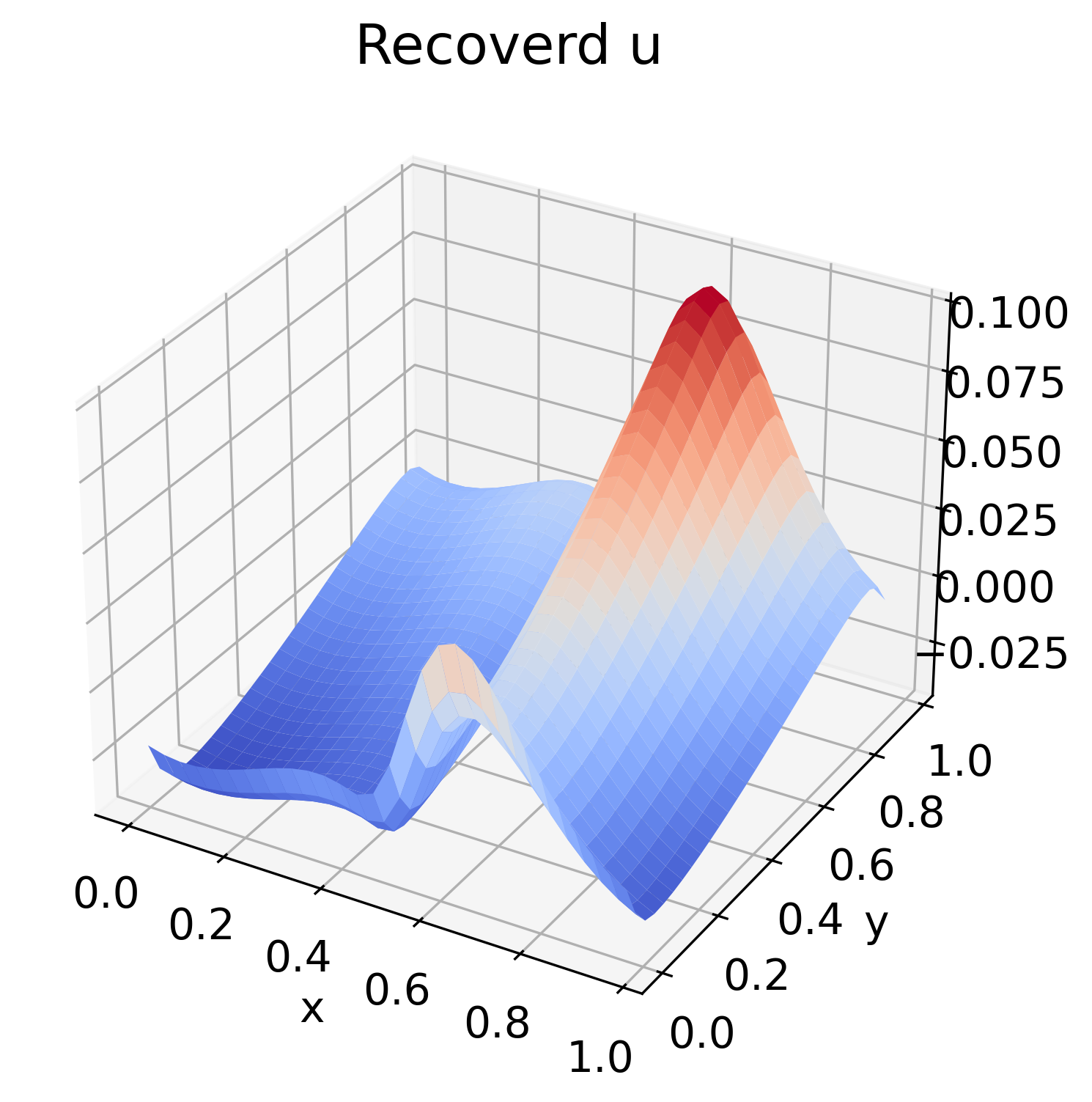}
        \caption{Recovered $u$ via GN}
        \label{2D_ErgodicMFGwithSmallViscosityRecoverdU_GN}
    \end{subfigure}%
    \hspace{1mm}
    \begin{subfigure}[b]{0.23\textwidth}
        \centering
        \includegraphics[width=\linewidth]{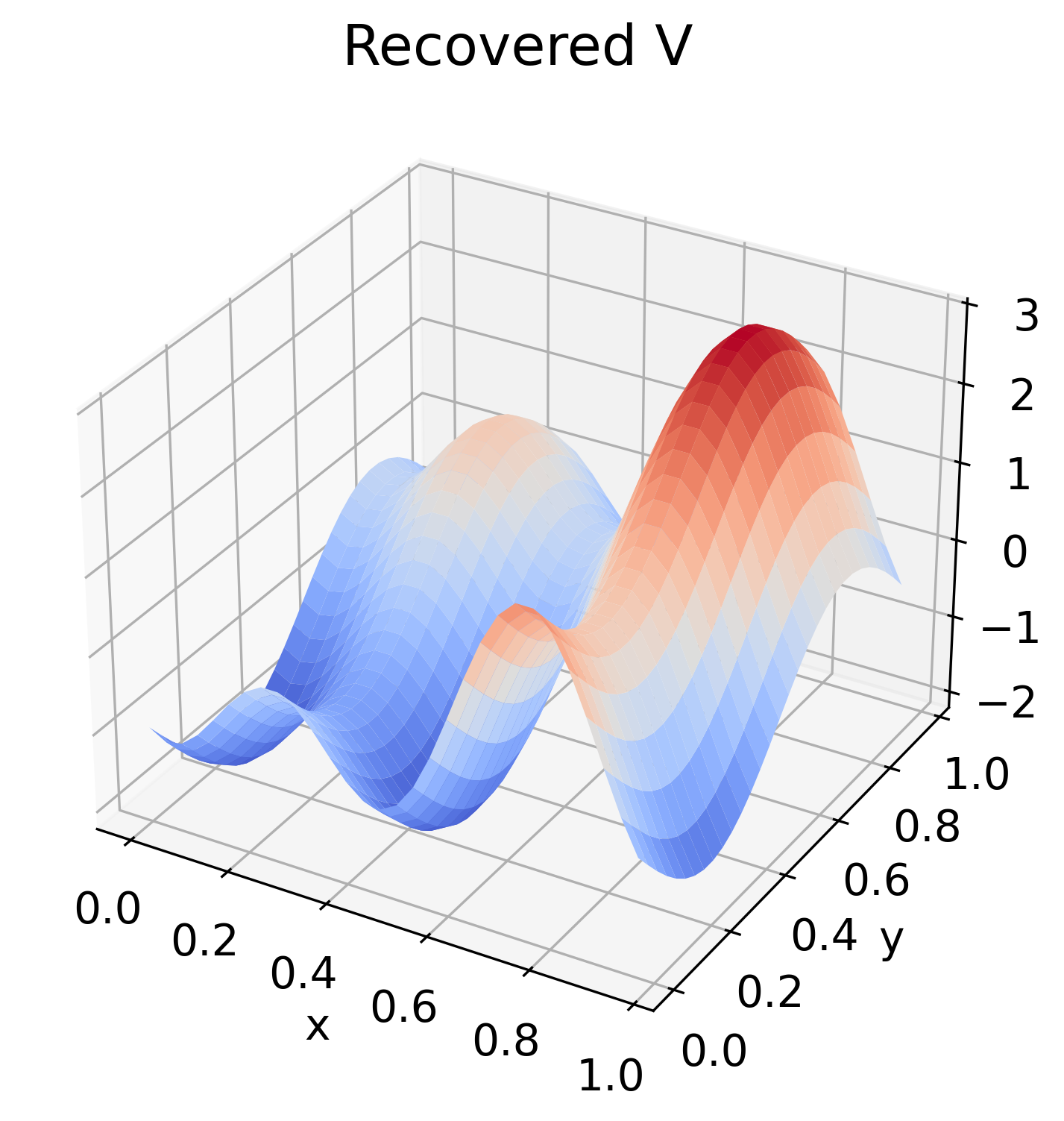}
        \caption{Recovered $V$ via GN}
        \label{2D_ErgodicMFGwithSmallViscosityRecoverdV_GN}
    \end{subfigure}%
    \hspace{1mm}
    \begin{subfigure}[b]{0.23\textwidth}
        \centering
        \includegraphics[width=\linewidth]{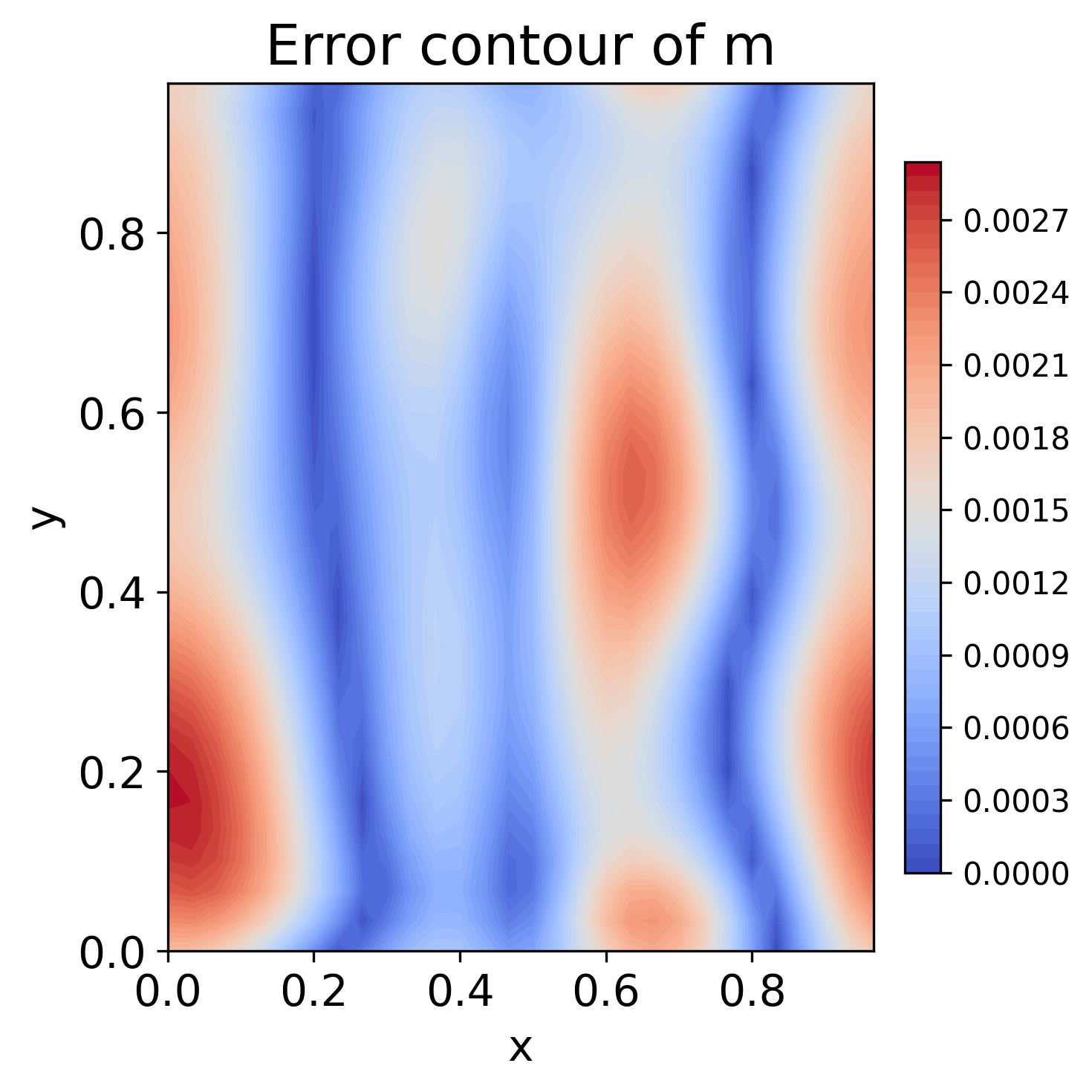}
        \caption{Error of $m$ via GN}
        \label{2D_ErgodicMFGwithSmallViscosityErrorM_GN}
    \end{subfigure}
    
    \begin{subfigure}[b]{0.23\textwidth}
        \centering
        \includegraphics[width=\linewidth]{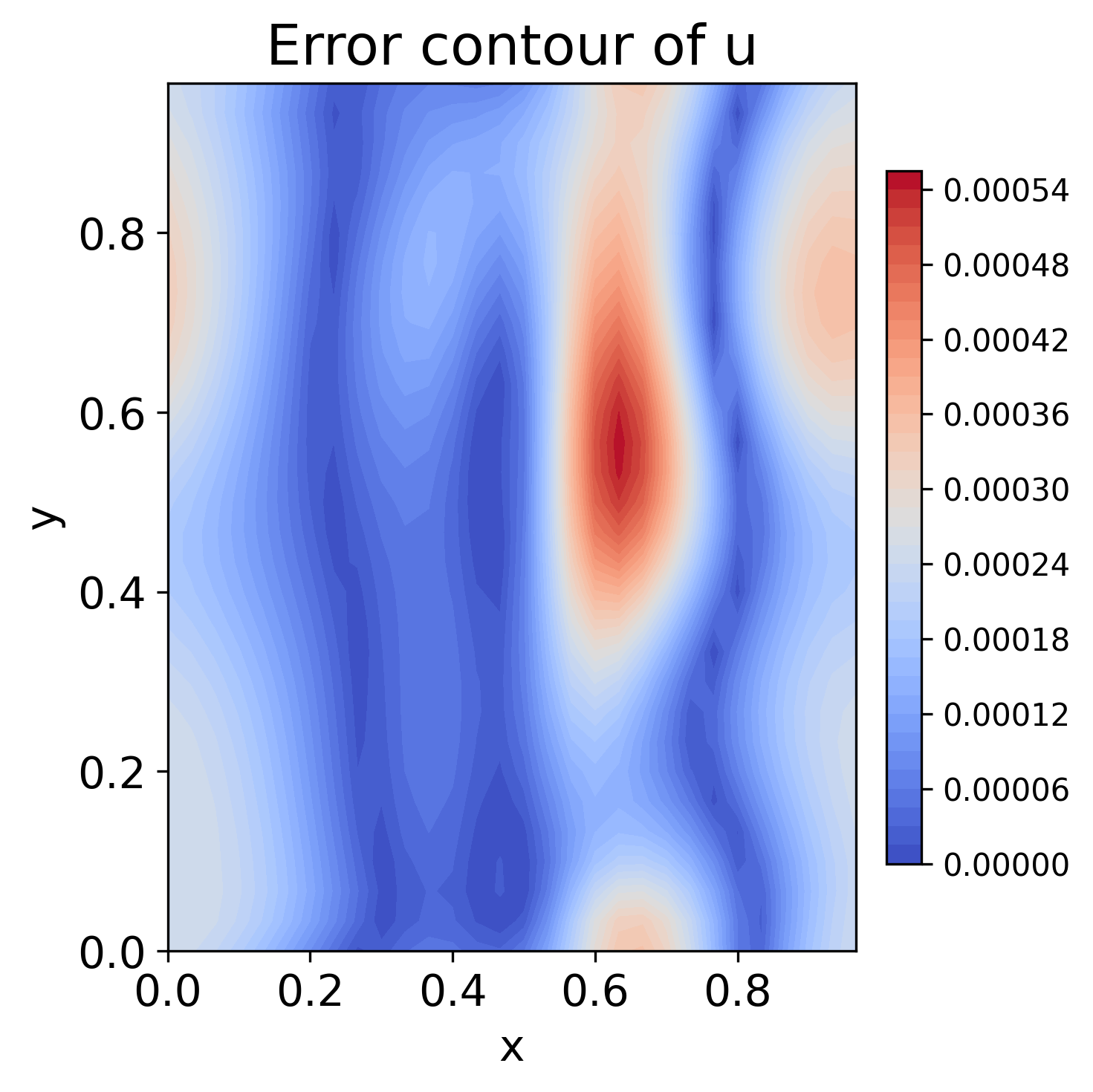}
        \caption{Error of $u$ via GD}
        \label{2D_ErgodicMFGwithSmallViscosityErrorU_GD}
    \end{subfigure}
    \hspace{1mm}
    \begin{subfigure}[b]{0.23\textwidth}
        \centering
        \includegraphics[width=\linewidth]{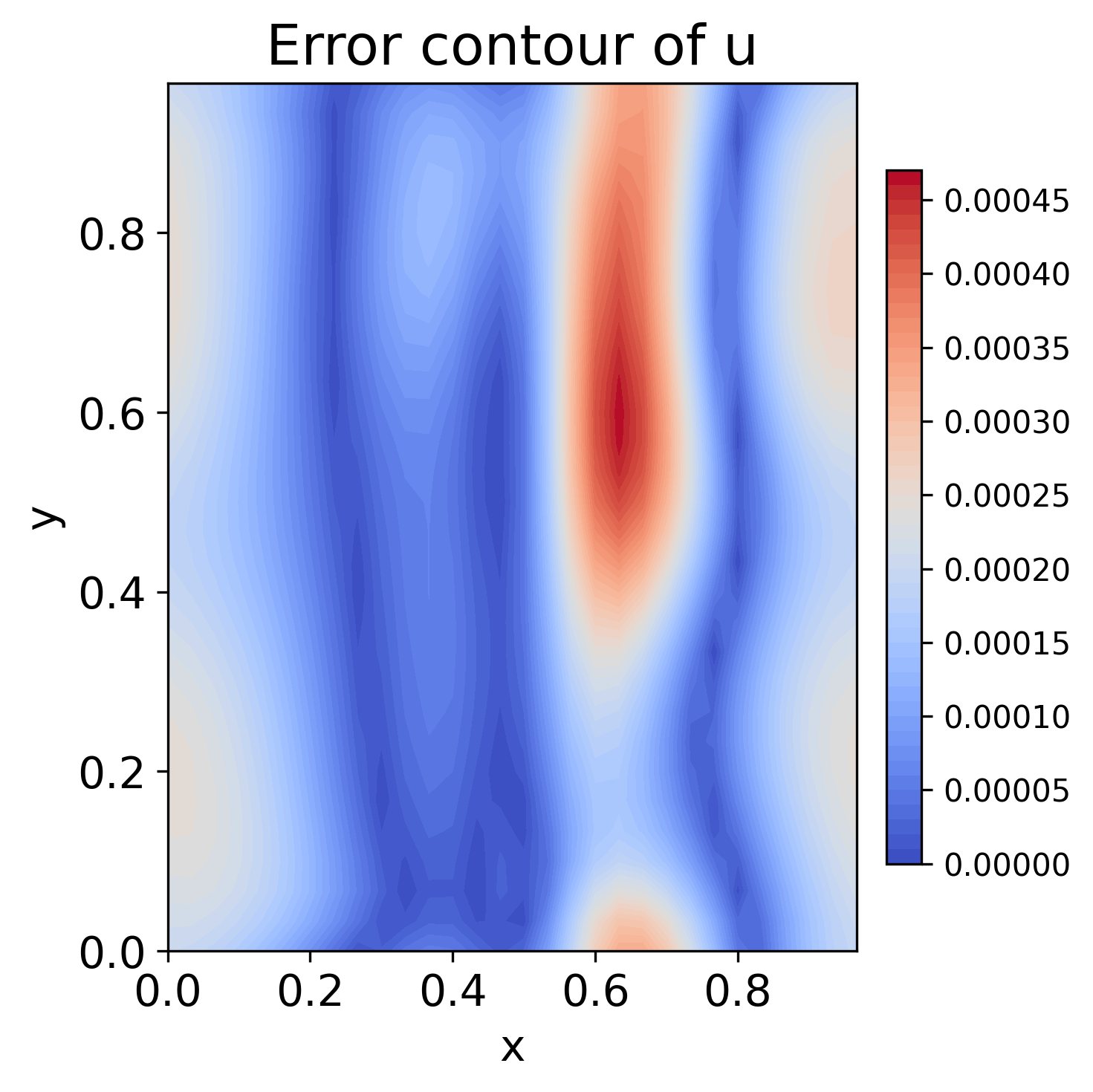}
        \caption{Error of $u$ via GN}
        \label{2D_ErgodicMFGwithSmallViscosityErrorU_GN}
    \end{subfigure}
    \hspace{1mm}
    \begin{subfigure}[b]{0.23\textwidth}
        \centering
        \includegraphics[width=\linewidth]{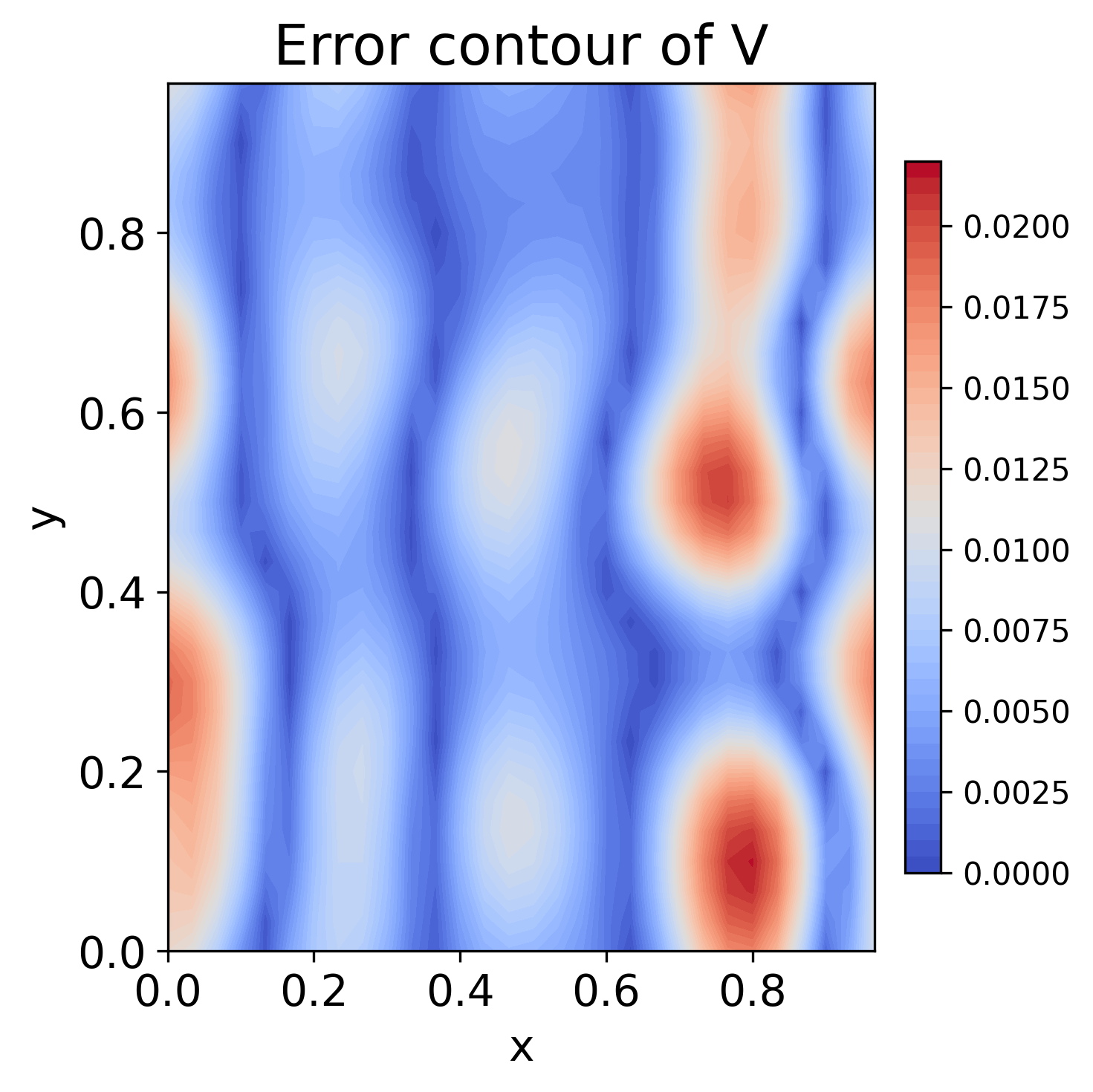}
        \caption{Error of $V$ via GD}
        \label{2D_ErgodicMFGwithSmallViscosityErrorV_GD}
    \end{subfigure}%
    \hspace{1mm}
    \begin{subfigure}[b]{0.23\textwidth}
        \centering
        \includegraphics[width=\linewidth]{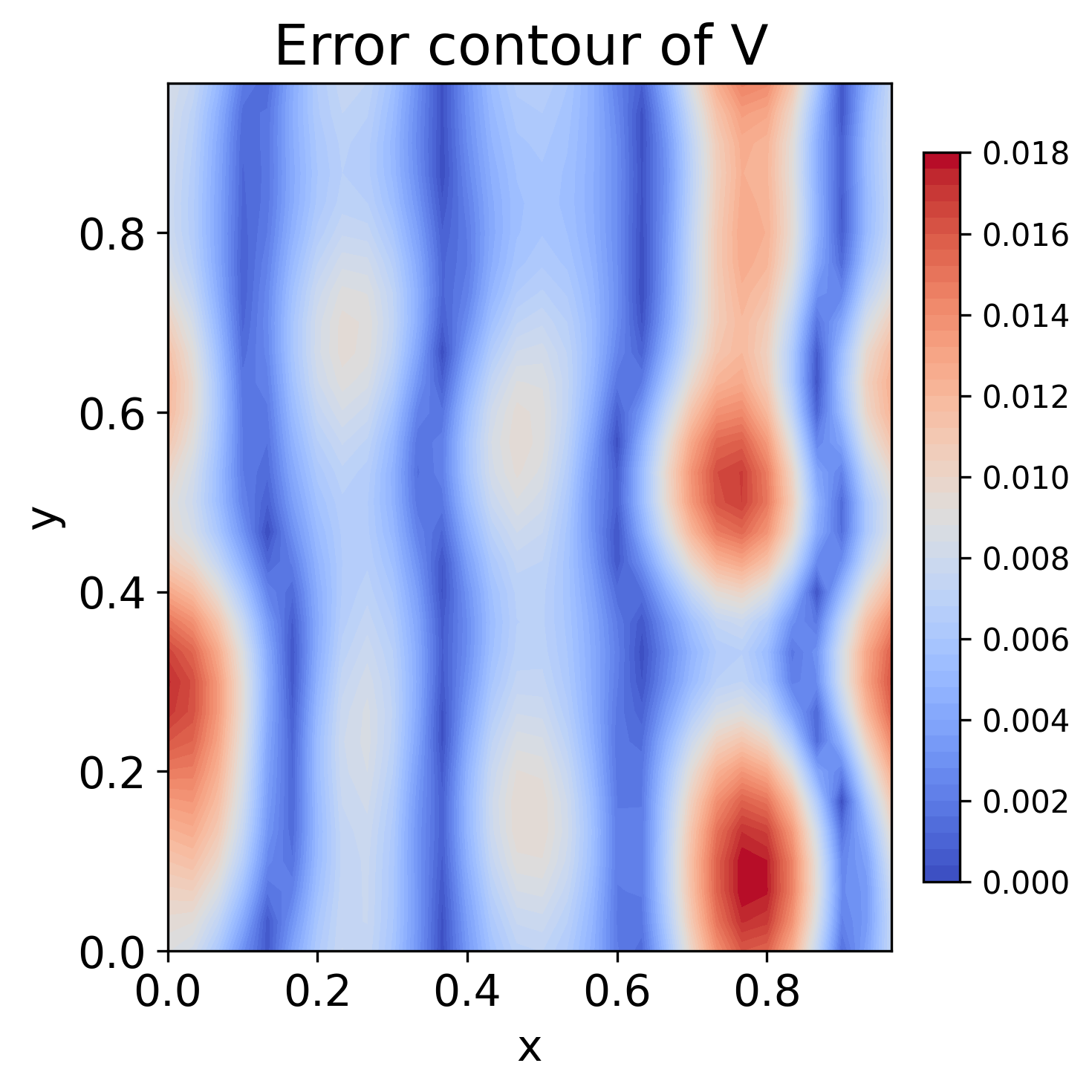}
        \caption{Error of $V$ via GN}
        \label{2D_ErgodicMFGwithSmallViscosityErrorV_GN}
    \end{subfigure}%

    \caption{Numerical results for the inverse problem of the second-order stationary MFG in \eqref{eq:2DstationaryMFGinvSmall_nu}.  (a), (b), (c) are references for $m,u,V$; (d) log-log plot of the loss comparison for GD and GN across iterations; (e), (f), (g) recovered $m,u,V$ via GD; (h), (m), (o) errors of $m,u,V$ via GD; (i), (j), (k) recovered $m,u,V$ via GN; (l), (n), (p) errors of $m,u,V$ via GN.}
    \label{2D_ErgodicMFGwithSmallViscosityPlot}
\end{figure}

\subsubsection{Solver-Agnostic Property of the Inverse Problem Framework}
\label{2DMFG_solver_free}
To illustrate the solver-agnostic nature of the framework in \Cref{sec:inverse_stationary}, we compare three different inner solvers for \eqref{eq:2DstationaryMFGinvSmall_nu}: the HRF~\cite{gomes2020hessian}, the Newton-based method~\cite{achdou2010mean}, and the policy iteration method~\cite{cacace2021policy}. We compare reconstruction accuracy and convergence behavior across these implementations to validate the robustness of the proposed inverse framework.  In this comparison, the three inner solvers are used to solve the same
discretized stationary MFG system in
\eqref{eq:st_forward_problem_discretized}. They differ only in how they compute
the state for a given parameter. After the inner solver has converged, the outer
solver forms the adjoint and sensitivity equations from the same discrete
residual. Thus, the outer derivative is taken with respect
to the discretized MFG equations, not through the iteration history of
the inner solver. 

\textbf{Experimental Setup.}
We consider \eqref{eq:2DstationaryMFGinvSmall_nu} with coupling \(f[m]=50(1-\Delta)^{-1}(1-\Delta)^{-1}m\) and spatial cost
$V(x,y)=\sin(2\pi x)+\cos(2\pi x)+\sin(4\pi y)$. We set \(\nu=0.2\) and discretize \(\mathbb{T}^2\) with \(h_x=h_y=1/30\). We use 72 observation points for \(m\) and 180 for \(V\). The regularization parameters are \(\alpha=0.04\), \(\beta=2\), and \(\gamma=2\), and we add i.i.d.\ Gaussian noise with level \(\eta=10^{-3}\) to the observations. This nonlocal coupling is not directly covered by the convergence theorem
for the HRF method in Section~\ref{sec:tdmfg}, which is proved for local
couplings. Nevertheless, the MFG still satisfies the Lasry--Lions monotonicity
condition, and the HRF method can be applied to the corresponding discrete
system. This example, therefore, also illustrates the use of the solver beyond the local-coupling setting covered by the theorem.


\textbf{Experimental Results.}
Table~\ref{table:Comparison_diff_solver} reports the discrete \(L^2\) errors for the recovered fields \(m,u,V\), together with the scalar error \(|\lambda-\lambda^*|\), for three inner MFG solvers and for both outer schemes (GD and GN). Figure~\ref{fig:solverfree_loss_GD_GN} shows the corresponding log-log plots of the loss versus iteration under each solver. The comparable accuracy and convergence behavior across all three implementations support the solver-agnostic nature of our framework. In particular, the inner solver can be chosen based on computational resources or implementation constraints without materially affecting reconstruction quality.

\begin{table}[htbp]
\centering
\caption{Recovery errors for the second-order stationary MFG in Section~\ref{2DMFG_solver_free}, using different inner MFG solvers. The columns for $m,u,V$ report discrete $L^2$ errors, while the last column reports the scalar error $|\lambda-\lambda^*|$.}
\label{table:Comparison_diff_solver}
\begin{tabular}{llcccc}
\toprule
\multirow{2}{*}{Inner MFG Solver} & \multirow{2}{*}{Inverse Solver} & \multicolumn{3}{c}{Discrete $L^2$ errors} & Scalar error \\
\cmidrule(lr){3-5} \cmidrule(lr){6-6}
& & $m$ & $u$ & $V$ & $|\lambda-\lambda^*|$ \\
\midrule
\multirow{2}{*}{HRF-based solver} & GD & 1.305434e-03 & 2.946326e-04 & 7.805980e-03 & 1.103745e-03 \\
 & GN & 1.262210e-03 & 2.821916e-04 & 7.552066e-03 & 1.126094e-03 \\
\addlinespace
\multirow{2}{*}{Newton-based Method} & GD & 1.411298e-03 & 3.200899e-04 & 8.291286e-03 & 1.136011e-03 \\
 & GN & 1.262438e-03 & 2.822591e-04 & 7.552039e-03 & 1.126042e-03 \\
\addlinespace
\multirow{2}{*}{Policy Iteration} & GD & 1.081458e-03 & 2.259894e-04 & 6.784161e-03 & 1.228694e-03 \\
 & GN & 1.024403e-03 & 2.122809e-04 & 6.232931e-03 & 1.062980e-03 \\
\bottomrule
\end{tabular}
\end{table}

\begin{figure}[!htbp]
    \centering
    \begin{subfigure}[b]{0.23\textwidth}
        \centering
        \includegraphics[width=\linewidth]{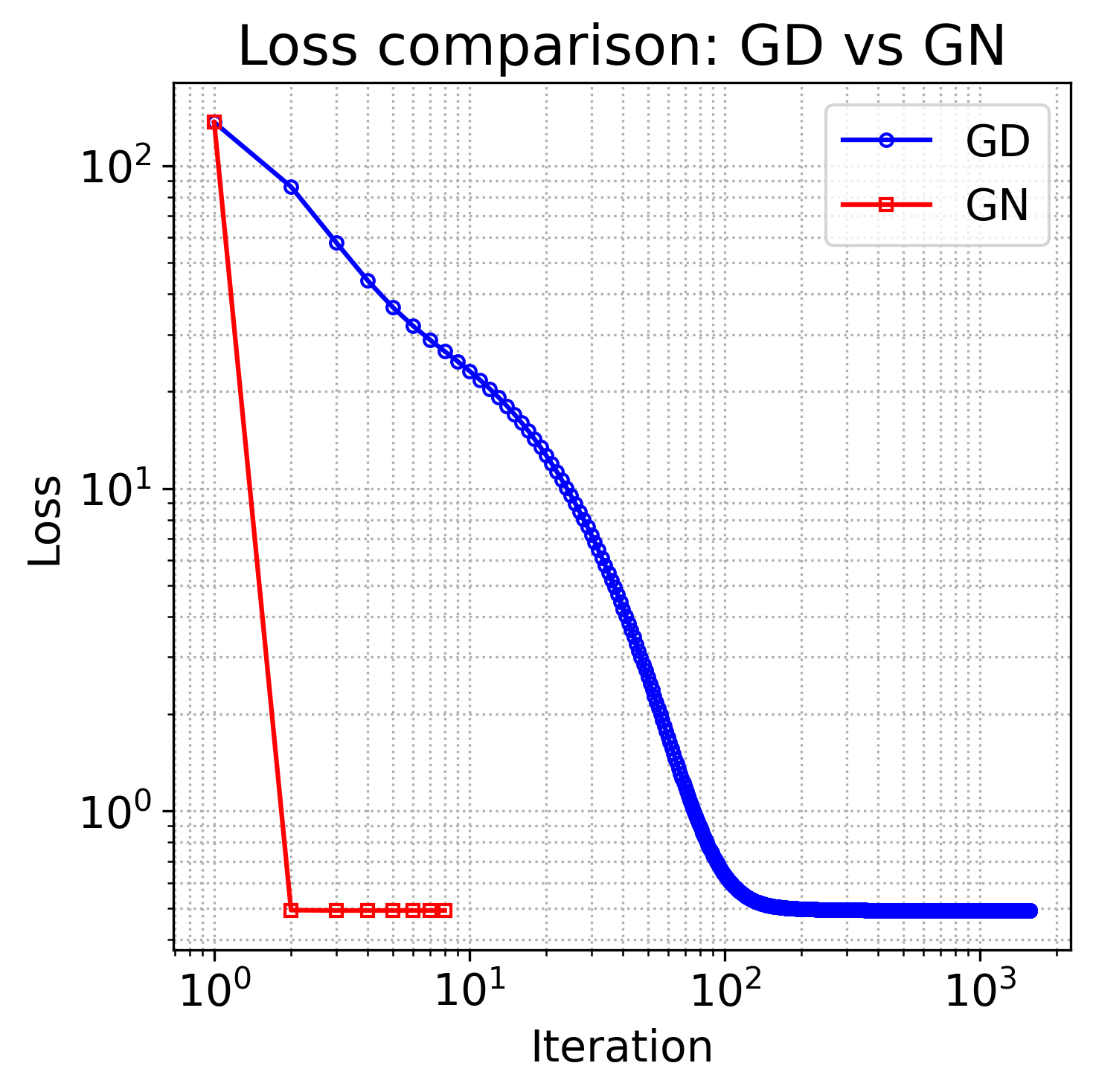}
        \caption{Loss comparison for the HRF}
        \label{fig:solverfree_loss_monotone}
    \end{subfigure}
    \hspace{1mm}
    \begin{subfigure}[b]{0.23\textwidth}
        \centering
        \includegraphics[width=\linewidth]{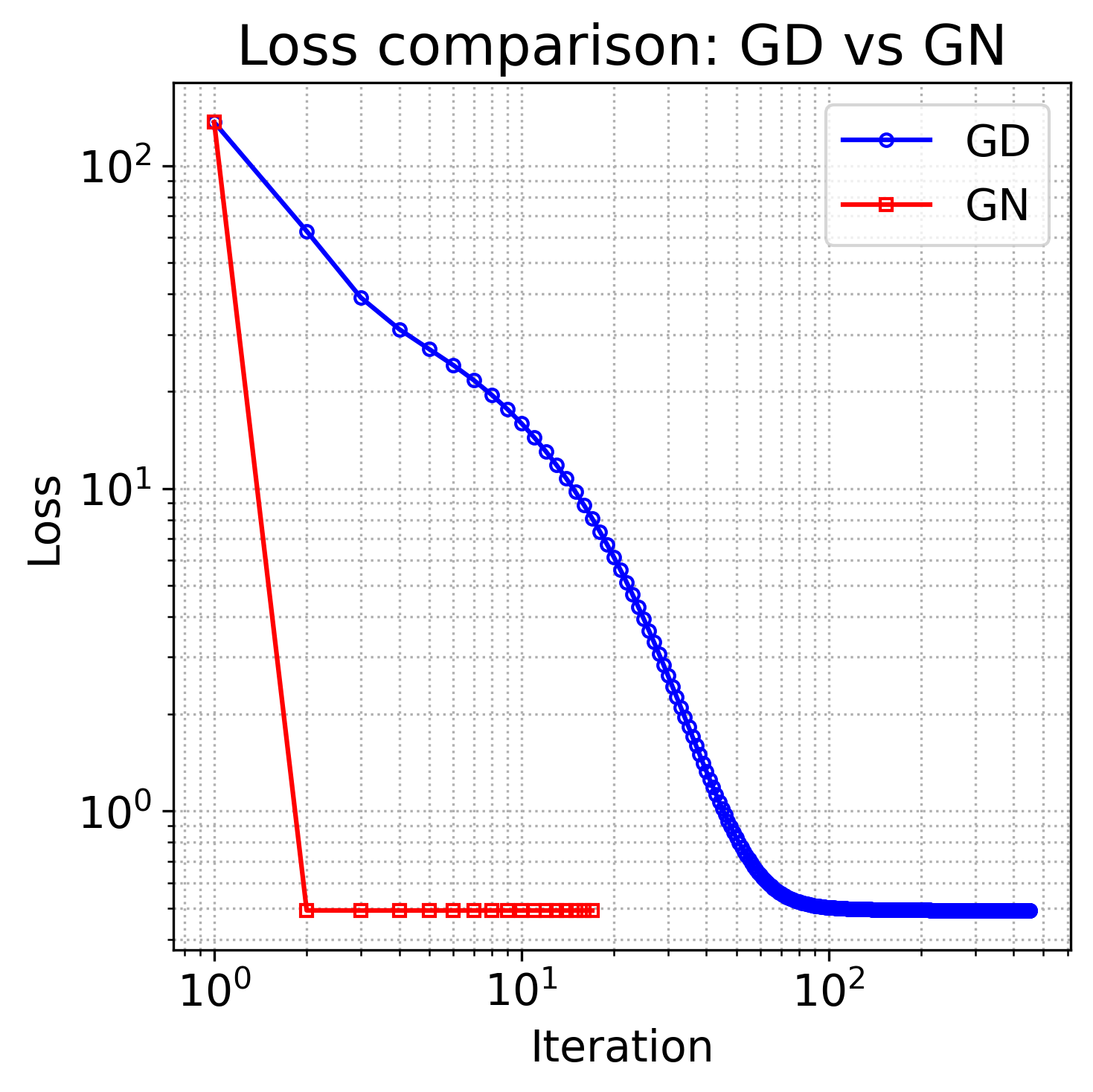}
        \caption{Loss comparison for the Newton-based method}
        \label{fig:solverfree_loss_newton}
    \end{subfigure}
    \hspace{1mm}
    \begin{subfigure}[b]{0.23\textwidth}
        \centering
        \includegraphics[width=\linewidth]{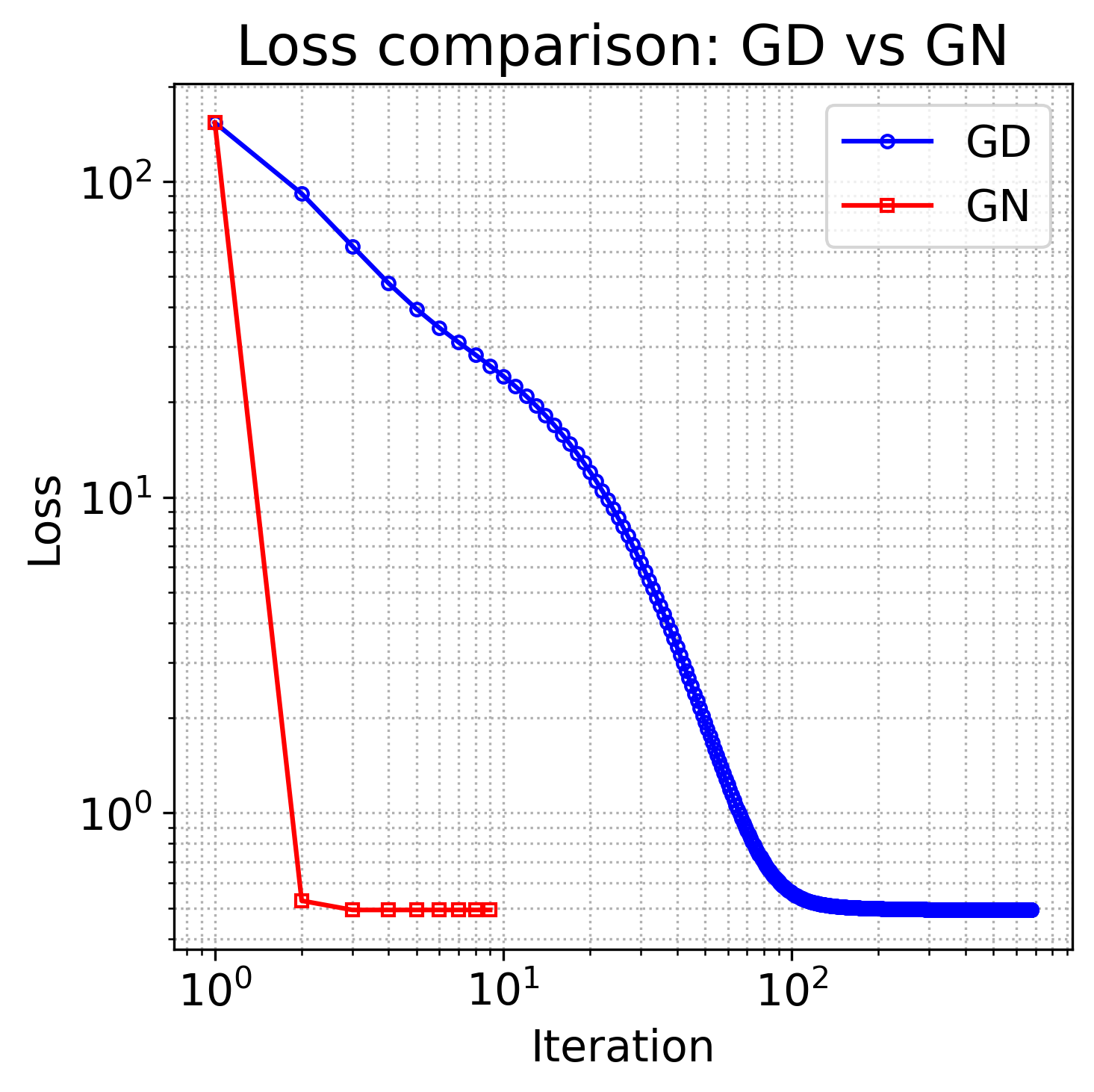}
        \caption{Loss comparison for the policy iteration}
        \label{fig:solverfree_loss_policy}
    \end{subfigure}
    \caption{Log-log plots of the loss versus iteration for GD and GN applied to the inverse problem in Section~\ref{2DMFG_solver_free}, using different inner MFG solvers. Panels (a)--(c) correspond to HRF, Newton-based method, and policy iteration, respectively.}
    \label{fig:solverfree_loss_GD_GN}
\end{figure}

\subsection{Time-Dependent MFG Inverse Problems}
In this subsection, we consider inverse problems for time-dependent MFGs. Fixing \(\nu \ge 0\), we restrict attention to choices of \(H\), \(f\), \(V\), \(m_0\), and \(u_T\) for which the MFG system below is well-posed:
\begin{equation}
\begin{cases}
-\partial_t u - \nu \Delta u + H(x,D_xu) - V(x)= f[m], & \forall (x, t) \in \mathbb{T}^d \times (0, T), \\ 
\partial_t m - \nu\Delta m - \operatorname{div}\!\left( D_p H(x,D_xu)\, m \right) = 0, & \forall (x, t) \in \mathbb{T}^d \times (0, T), \\
m(x, 0) = m_0(x), \quad u(x, T) = u_T(x).
\end{cases}
\label{eqtimedependinv}
\end{equation}

\subsubsection{One-Dimensional Time-Dependent MFG Inverse Problem}
\label{timedependenton1Dexample1}
In this example, we set \(d=1\) and specialize \eqref{eqtimedependinv} to the one-dimensional time-dependent MFG setting. We choose the Hamiltonian $H(p)=\tfrac{1}{2}\lvert p\rvert^2$. We identify \(\mathbb{T}\) with \([0,1)\), set \(T=1\), and take \(m_0(\cdot)\equiv 1\) and \(u_T(\cdot)\equiv 0\). The coupling is chosen as $f[m]=(1-\Delta)^{-1}(1-\Delta)^{-1}m$,
and the spatial cost is $V(x)=\sin(2\pi x)+\cos(4\pi x)$. 
We set the viscosity coefficient to \(\nu=0.1\). 

\textbf{Experimental Setup.}
We use \(h_x=h_t=\tfrac{1}{40}\). Each reduced block \(M\) or \(U\) has \(40\times 40=1600\) nodal values. After including the fixed boundary slices \(M_0\) and \(U_{N_T}\), each
full field has \(40\times 41=1640\) nodal values. The endpoint data are imposed at \(t=0\) for \(m\) and at \(t=T\) for \(u\). We sample 96 observations of \(m\). We use 10 spatial observation points for \(V\). The regularization parameters are set to \(\alpha=0.04\), \(\beta=2\), and \(\gamma=2\). Gaussian noise \(\mathcal{N}(0,\eta^2 I)\) with \(\eta=10^{-3}\) is added to the observations. The initial guesses are \(u\equiv 0\), \(m\equiv 1\), and \(V\equiv 0\). We solve \eqref{eq:td_forward_problem_discretized} numerically using the
HRF method and use the computed solution as the reference \((u^*,m^*)\). We then consider the inverse problem of recovering \(u\), \(m\), and \(V\) from noisy and partial observations of \(m\) and \(V\), using the adjoint algorithm and the GN method described in \Cref{sec:inverse_stationary}.

\textbf{Experimental Results.} 
Figure~\ref{fig:samples_observationsmulti} depicts the sampling pattern for the one-dimensional space-time fields: Panel~(a) shows the sample points for \(m\) together with the \(m\)-observation locations, while Panel~(b) shows the sample points for \(u\). The spatial observation locations for \(V\) are not displayed in this figure.
Figure~\ref{timedependV} summarizes the reconstruction performance by displaying the reference fields, the reconstructions, and the corresponding pointwise error contours for \(m\), \(u\), and \(V\).
Finally, the loss curves in Figure~\ref{Timedependentinversefig5} show that the GN method converges markedly faster than adjoint-based GD, attaining comparable or smaller reconstruction errors in far fewer iterations.
\begin{figure}[!htbp]
    \centering
    \raisebox{-3.8mm}{%
    \begin{subfigure}[b]{0.23\textwidth}
        \centering
        \includegraphics[width=\linewidth]{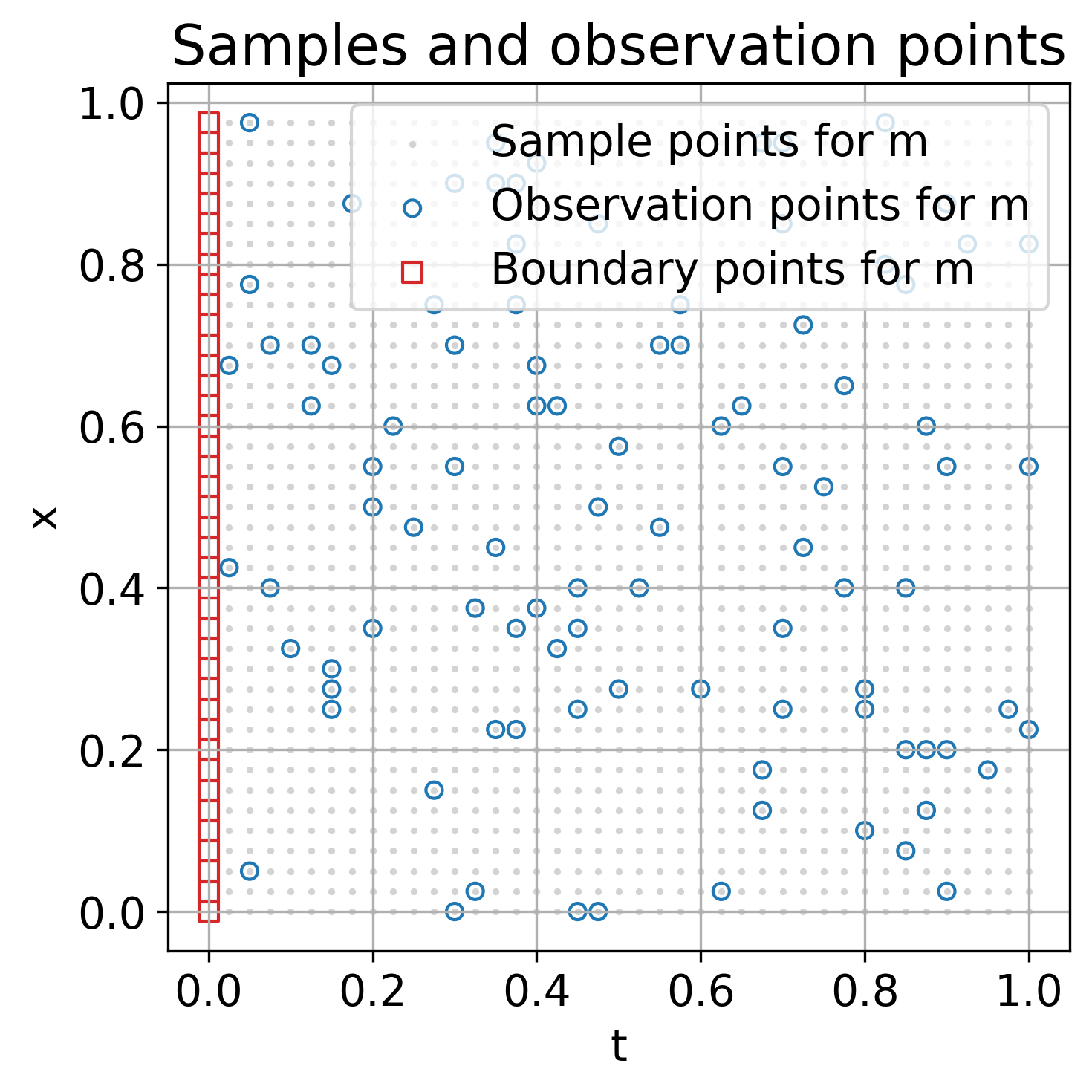}
        \caption{Samples and observations for $m$}
        \label{Multidimensioninversefig4}
    \end{subfigure}}
    \hspace{1mm}  
    \begin{subfigure}[b]{0.23\textwidth}
        \centering
        \includegraphics[width=\linewidth]{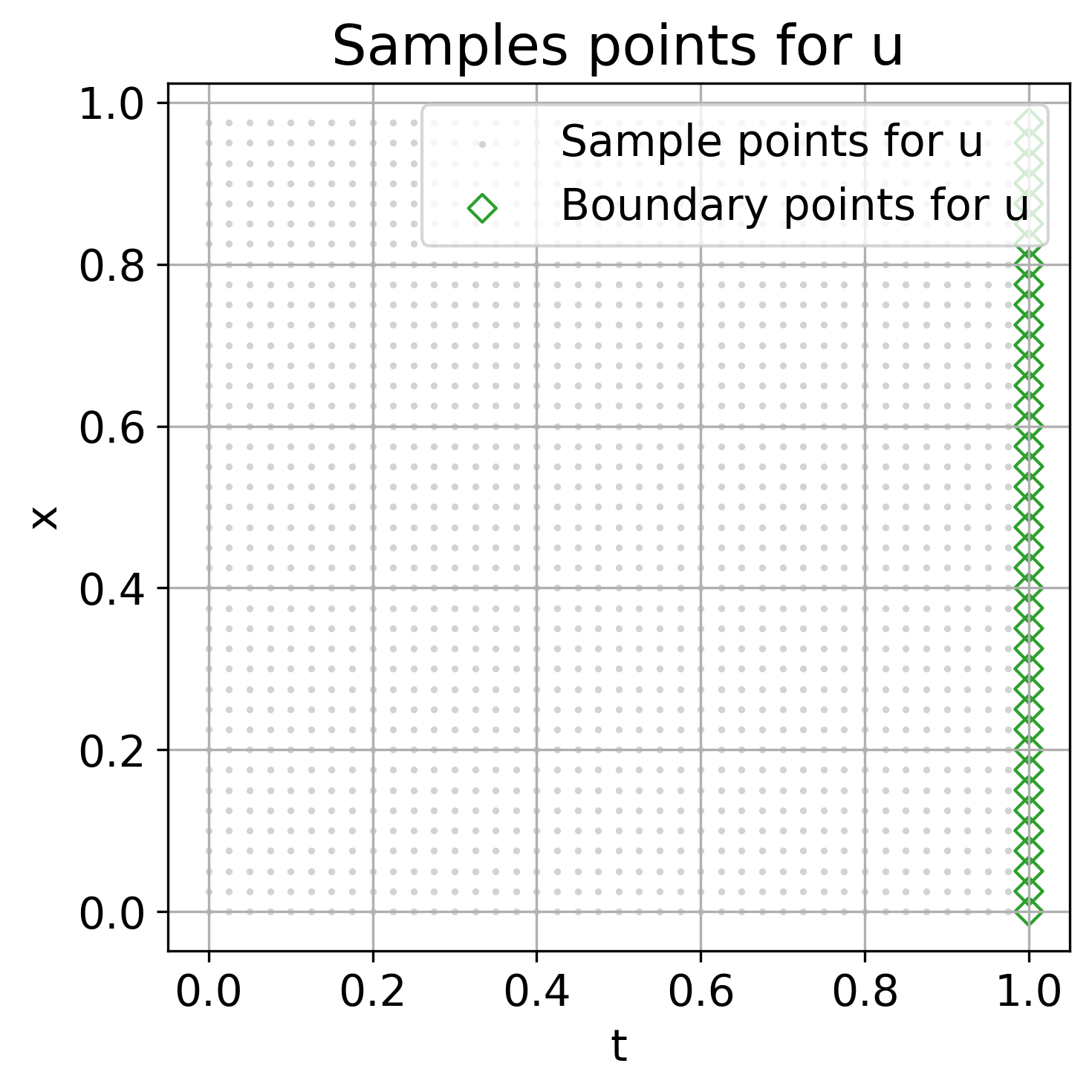}
        \caption{Samples for $u$}
        \label{Multidimensioninversefig0}
    \end{subfigure}
    \caption{The inverse problem of the time-dependent MFG in Section~\ref{timedependenton1Dexample1}: sample points for $m$ (with $m$-observation locations) and sample points for $u$.}
    \label{fig:samples_observationsmulti}
\end{figure}

\begin{figure}[!htbp]
    \centering
    \begin{subfigure}[b]{0.23\textwidth}
        \centering
        \includegraphics[width=\linewidth]{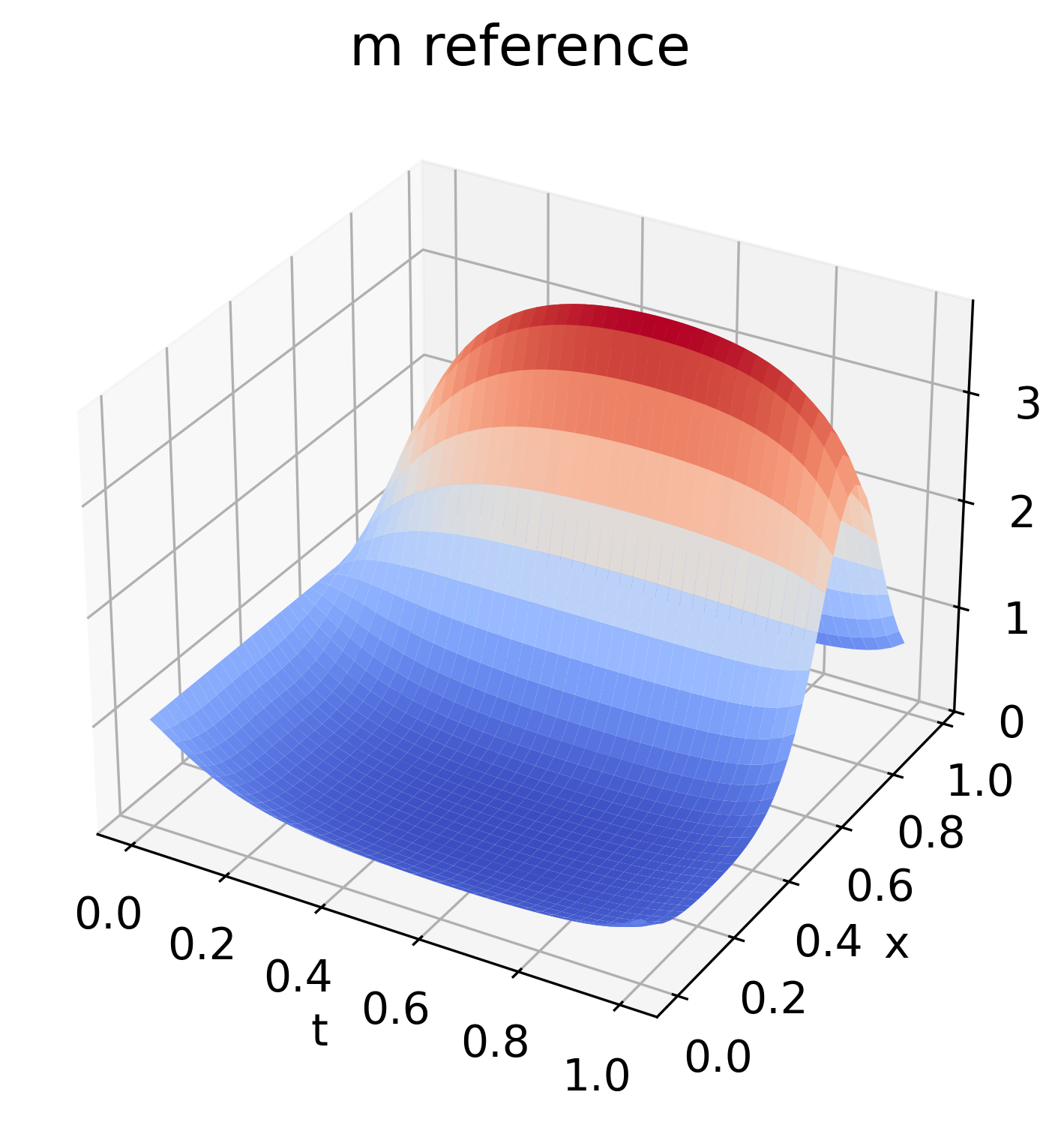}
        \caption{$m$ reference}
        \label{Timedependentinversefig1}
    \end{subfigure}%
    \hspace{1mm}
    \begin{subfigure}[b]{0.23\textwidth}
        \centering
        \includegraphics[width=\linewidth]{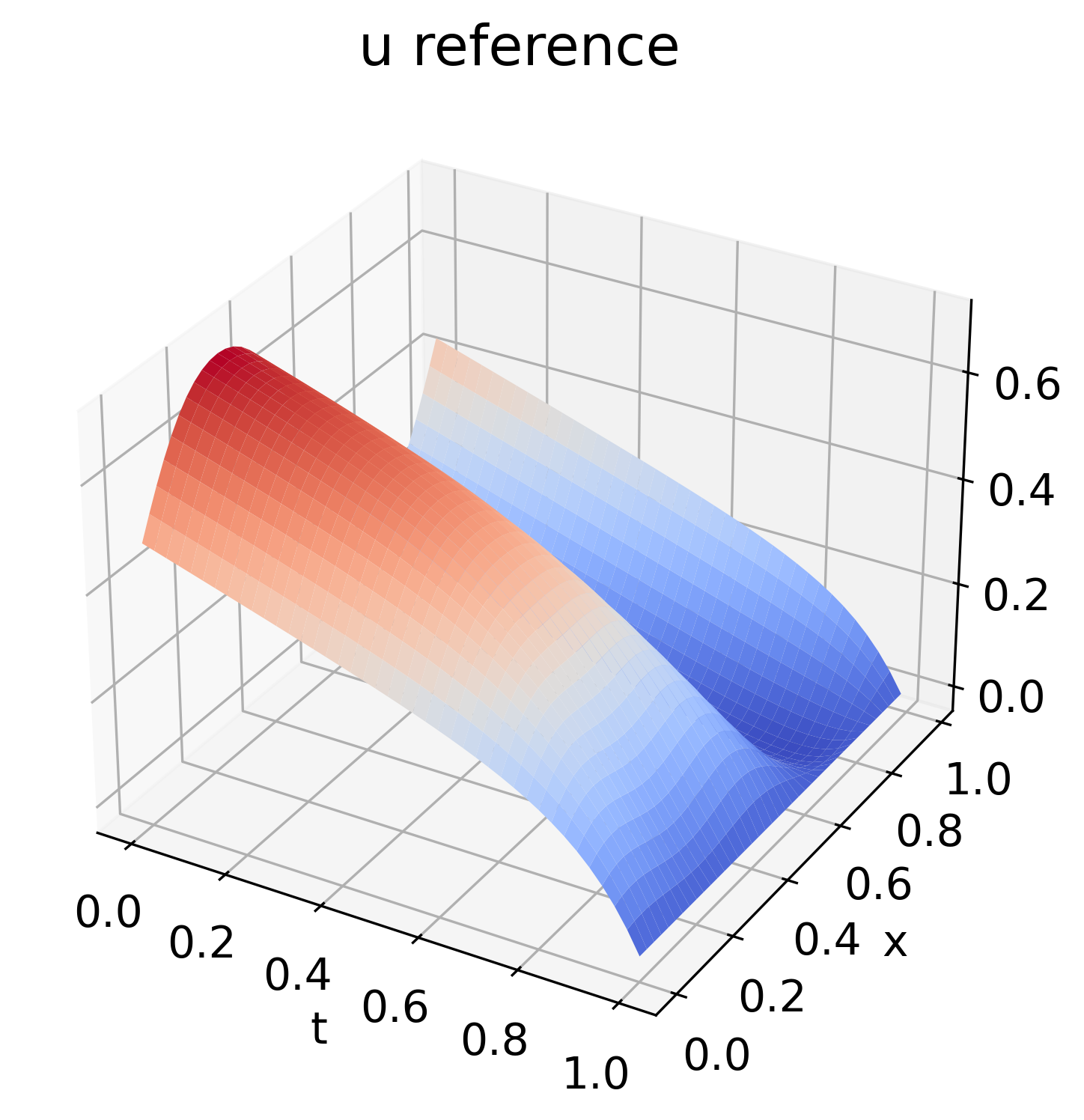}
        \caption{$u$ reference}
        \label{Timedependentinversefig2}
    \end{subfigure}%
    \hspace{1mm}
    \begin{subfigure}[b]{0.23\textwidth}
        \centering
        \includegraphics[width=\linewidth]{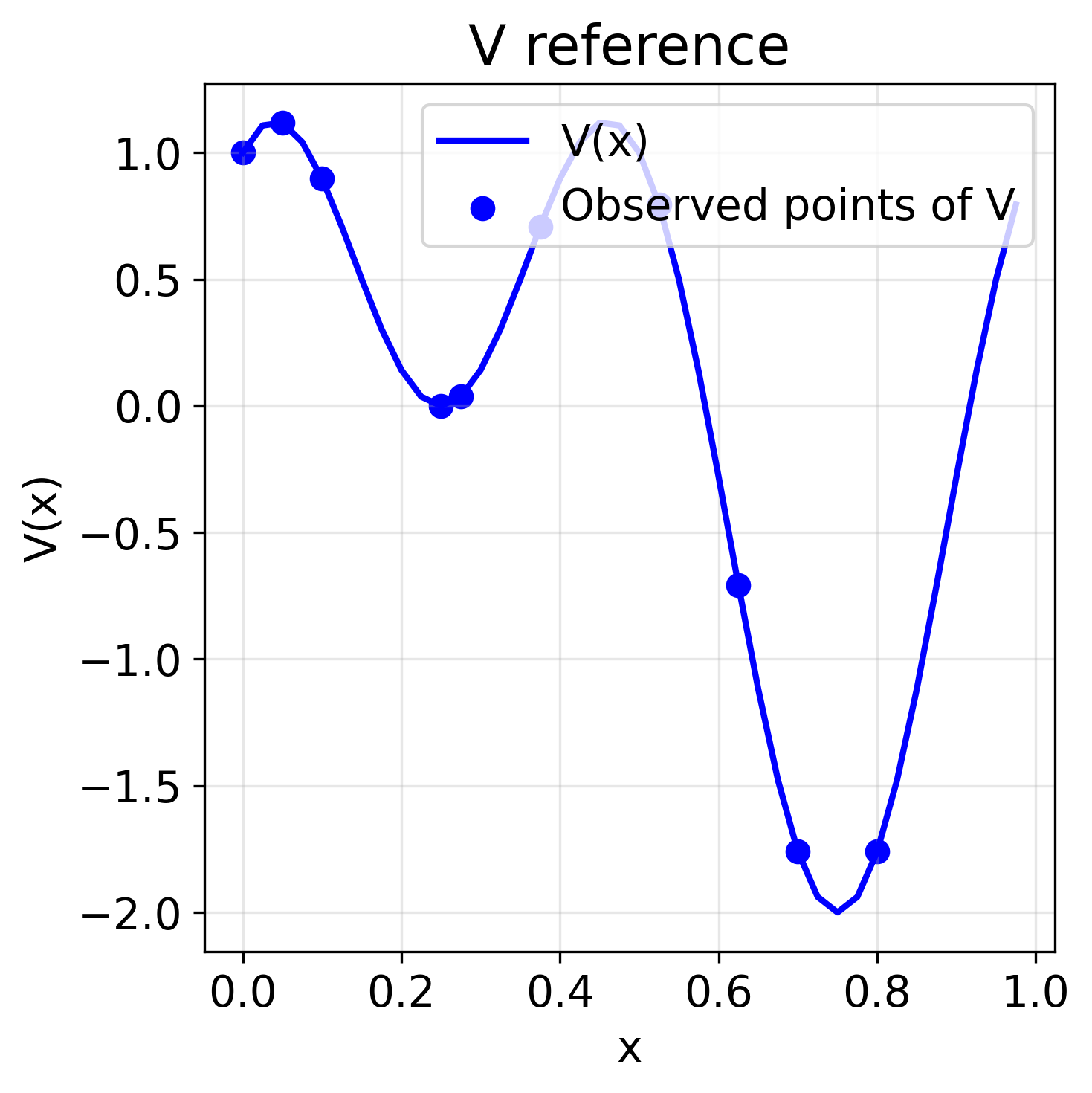}
        \caption{$V$ reference}
        \label{Timedependentinversefig3}
    \end{subfigure}%
    \hspace{1mm}
    \begin{subfigure}[b]{0.23\textwidth}
        \centering
        \includegraphics[width=\linewidth]{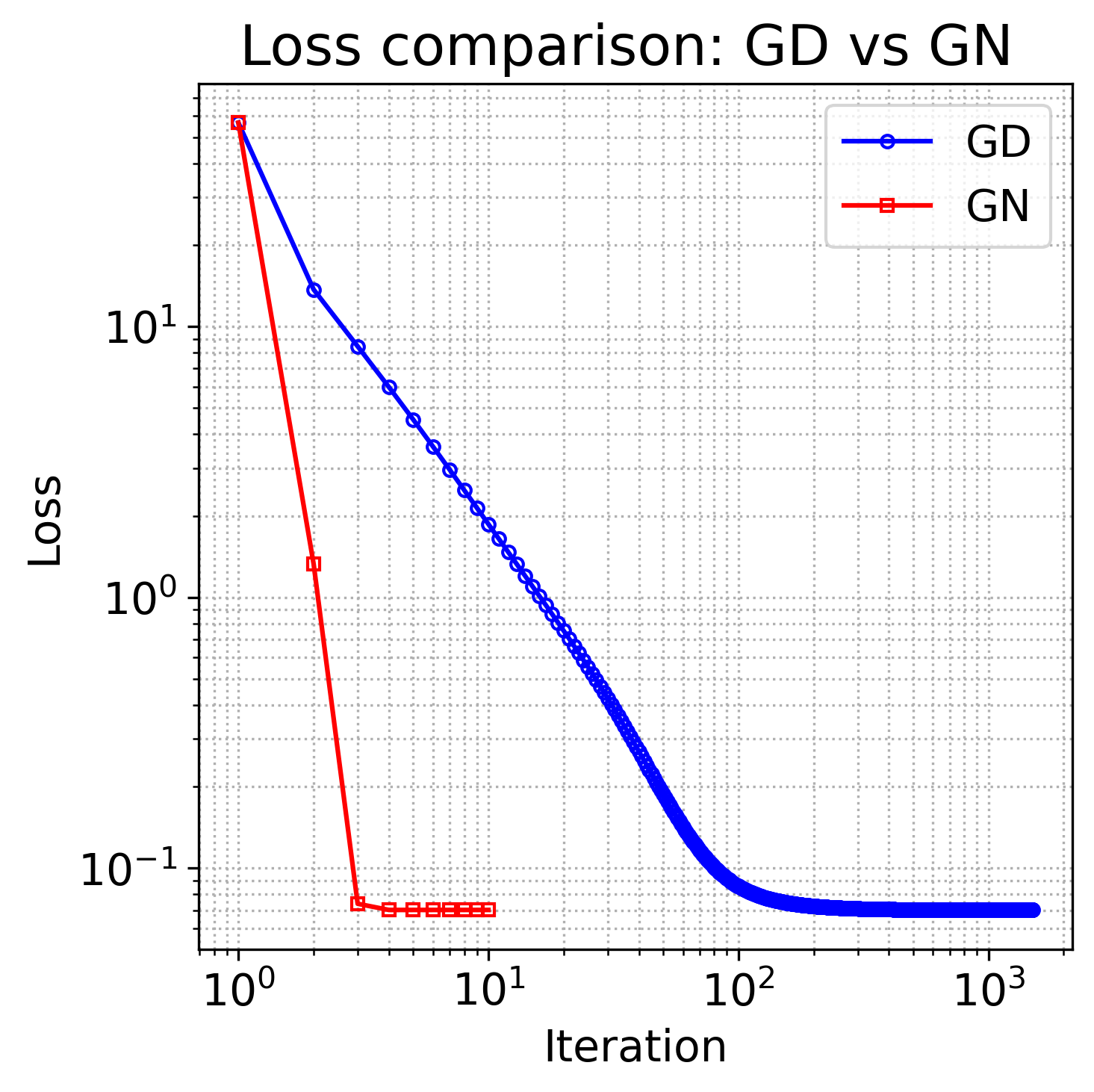}
        \caption{Loss comparison}
        \label{Timedependentinversefig5}
    \end{subfigure}
    
    \begin{subfigure}[b]{0.23\textwidth}
        \centering
        \includegraphics[width=\linewidth]{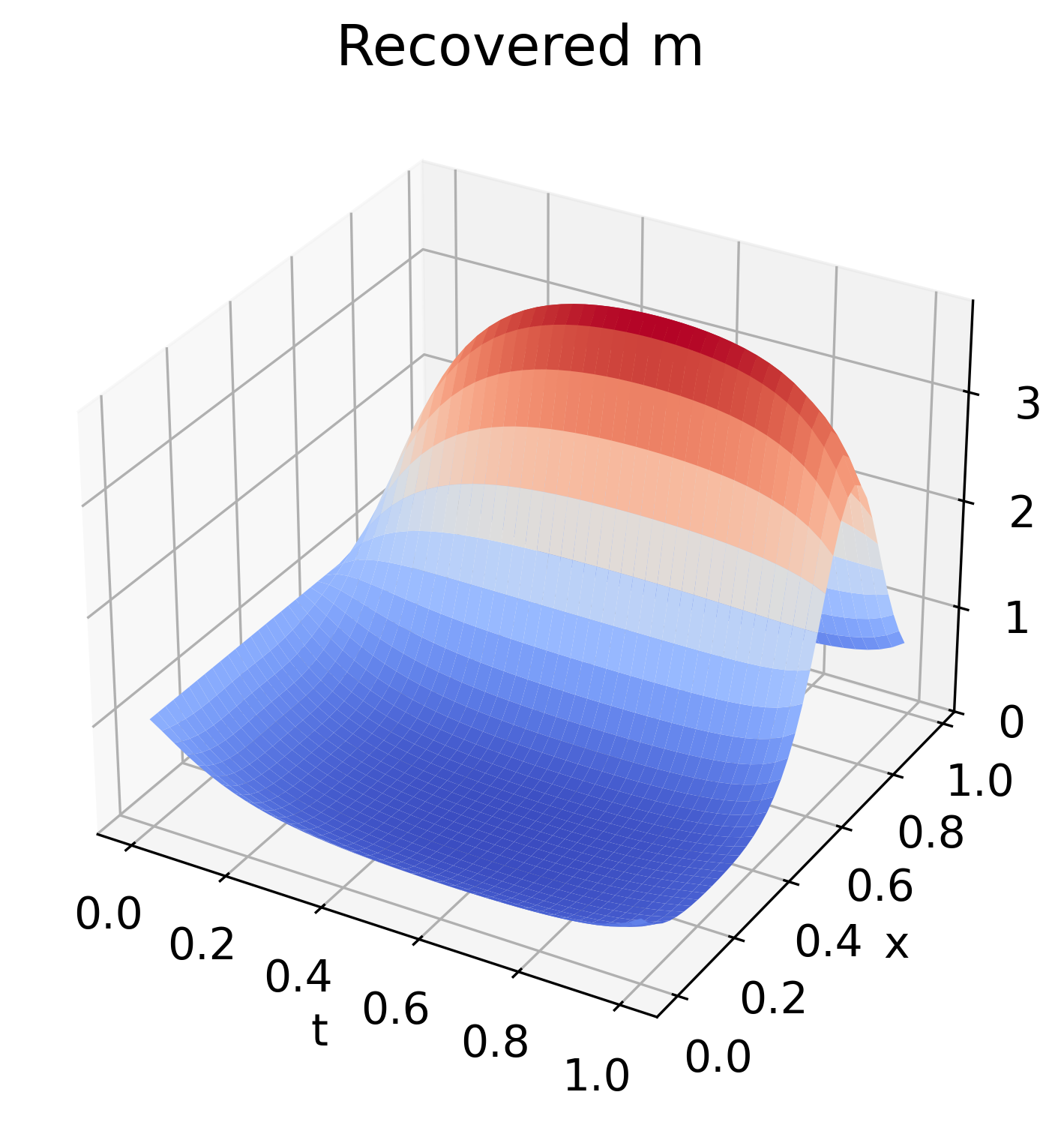}
        \caption{Recovered $m$ via GD}
        \label{Timedependentinversefig6}
    \end{subfigure}%
    \hspace{1mm}
    \begin{subfigure}[b]{0.23\textwidth}
        \centering
        \includegraphics[width=\linewidth]{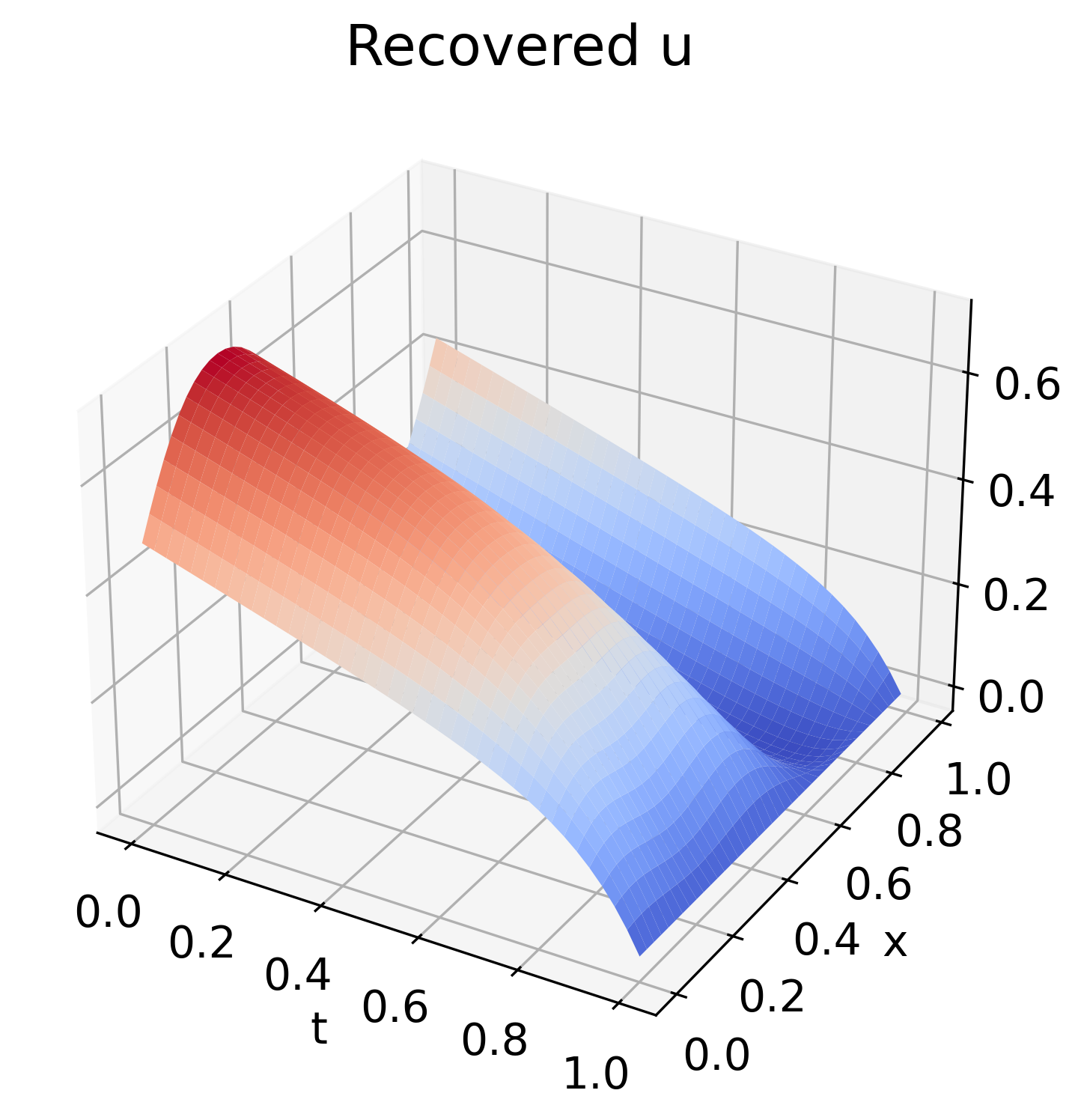}
        \caption{Recovered $u$ via GD}
        \label{Timedependentinversefig7}
    \end{subfigure}%
    \hspace{1mm}
    \begin{subfigure}[b]{0.23\textwidth}
        \centering
        \includegraphics[width=\linewidth]{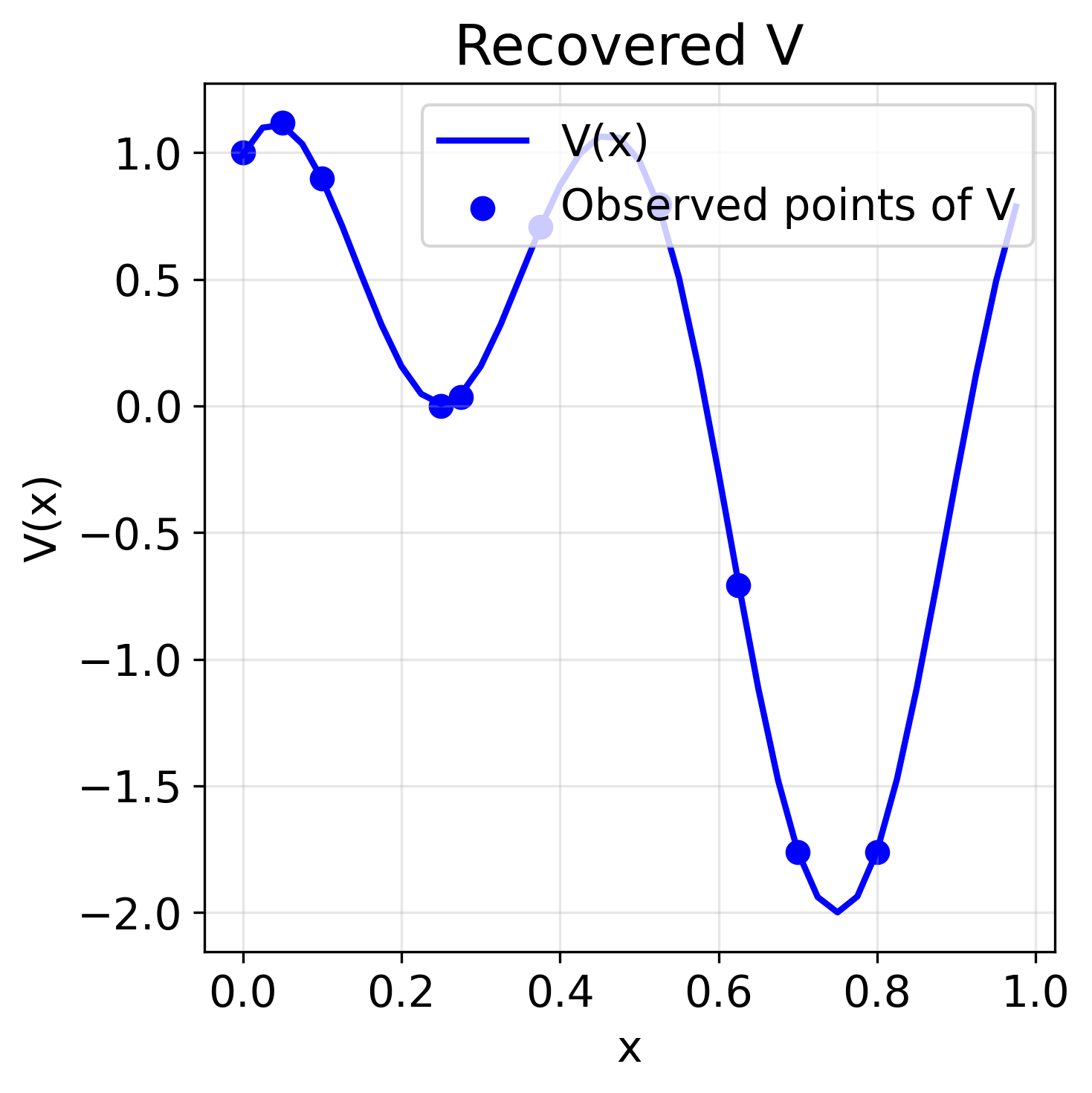}
        \caption{Recovered $V$ via GD}
        \label{Timedependentinversefig8}
    \end{subfigure}%
    \hspace{1mm}
    \begin{subfigure}[b]{0.23\textwidth}
        \centering
        \includegraphics[width=\linewidth]{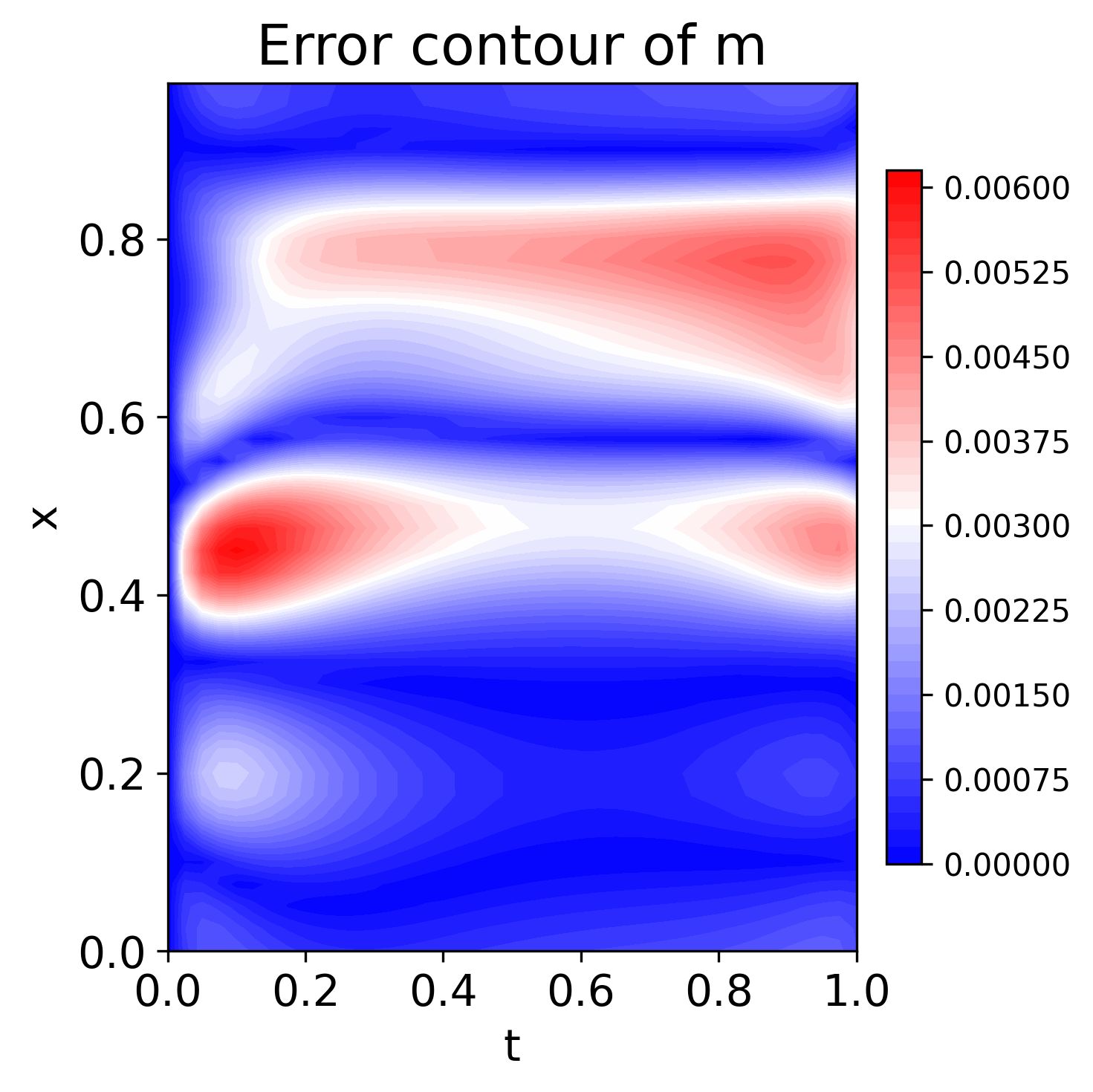}
        \caption{Error of $m$ via GD}
        \label{Timedependentinversefig9}
    \end{subfigure}
    
    \begin{subfigure}[b]{0.23\textwidth}
        \centering
        \includegraphics[width=\linewidth]{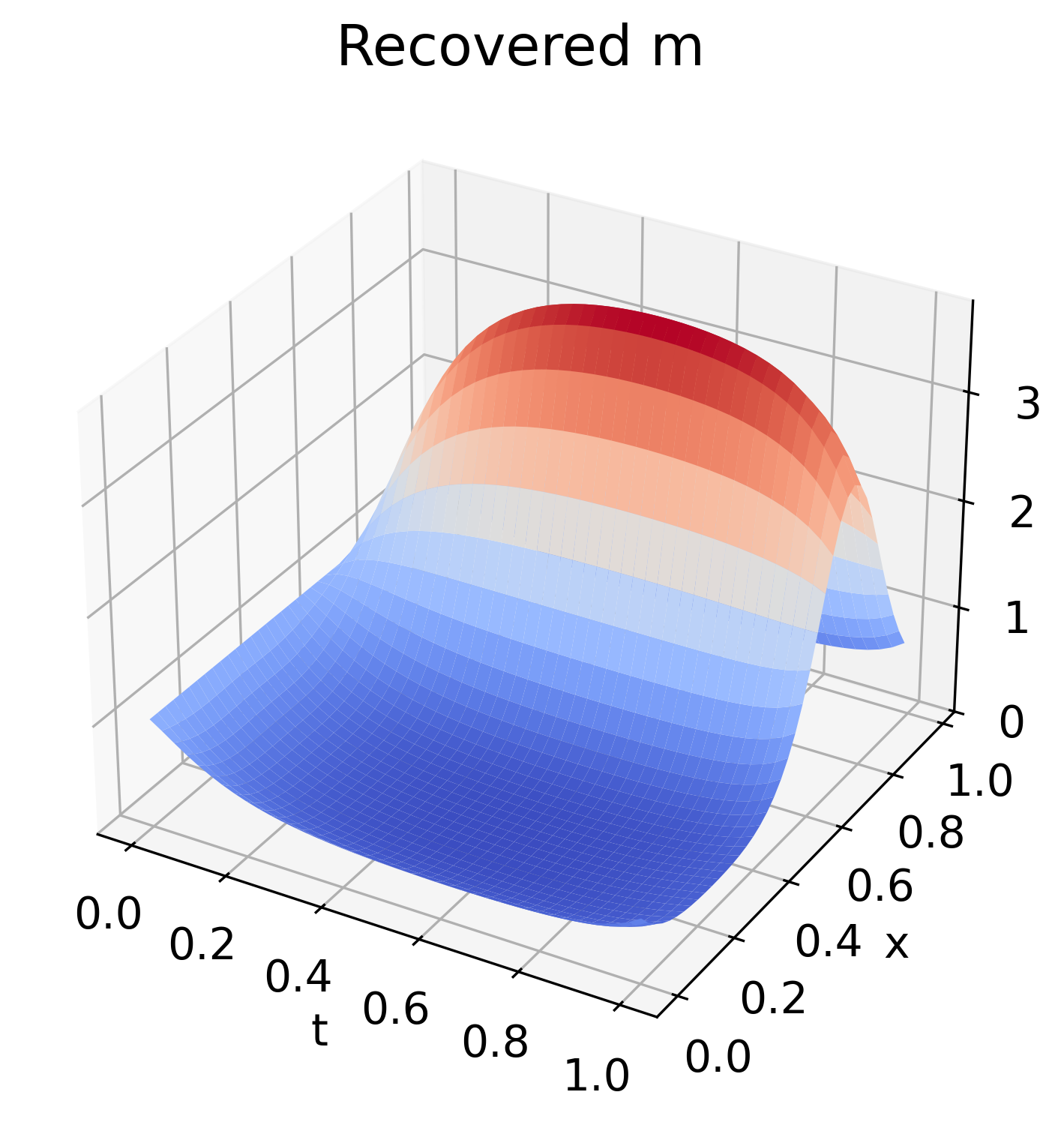}
        \caption{Recovered $m$ via GN}
        \label{Timedependentinversefig11}
    \end{subfigure}%
    \hspace{1mm}
    \begin{subfigure}[b]{0.23\textwidth}
        \centering
        \includegraphics[width=\linewidth]{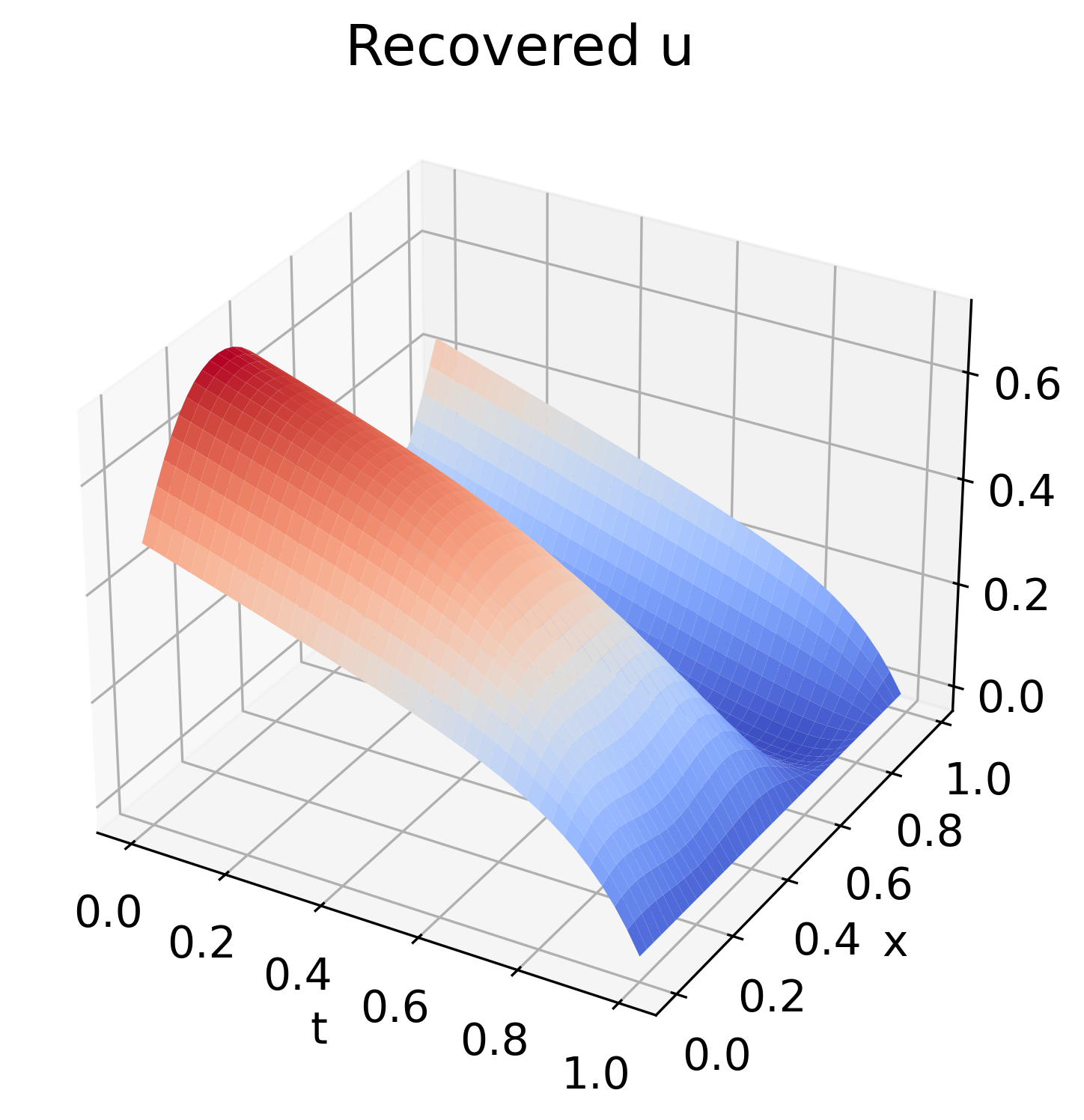}
        \caption{Recovered $u$ via GN}
        \label{Timedependentinversefig12}
    \end{subfigure}%
    \hspace{1mm}
    \begin{subfigure}[b]{0.23\textwidth}
        \centering
        \includegraphics[width=\linewidth]{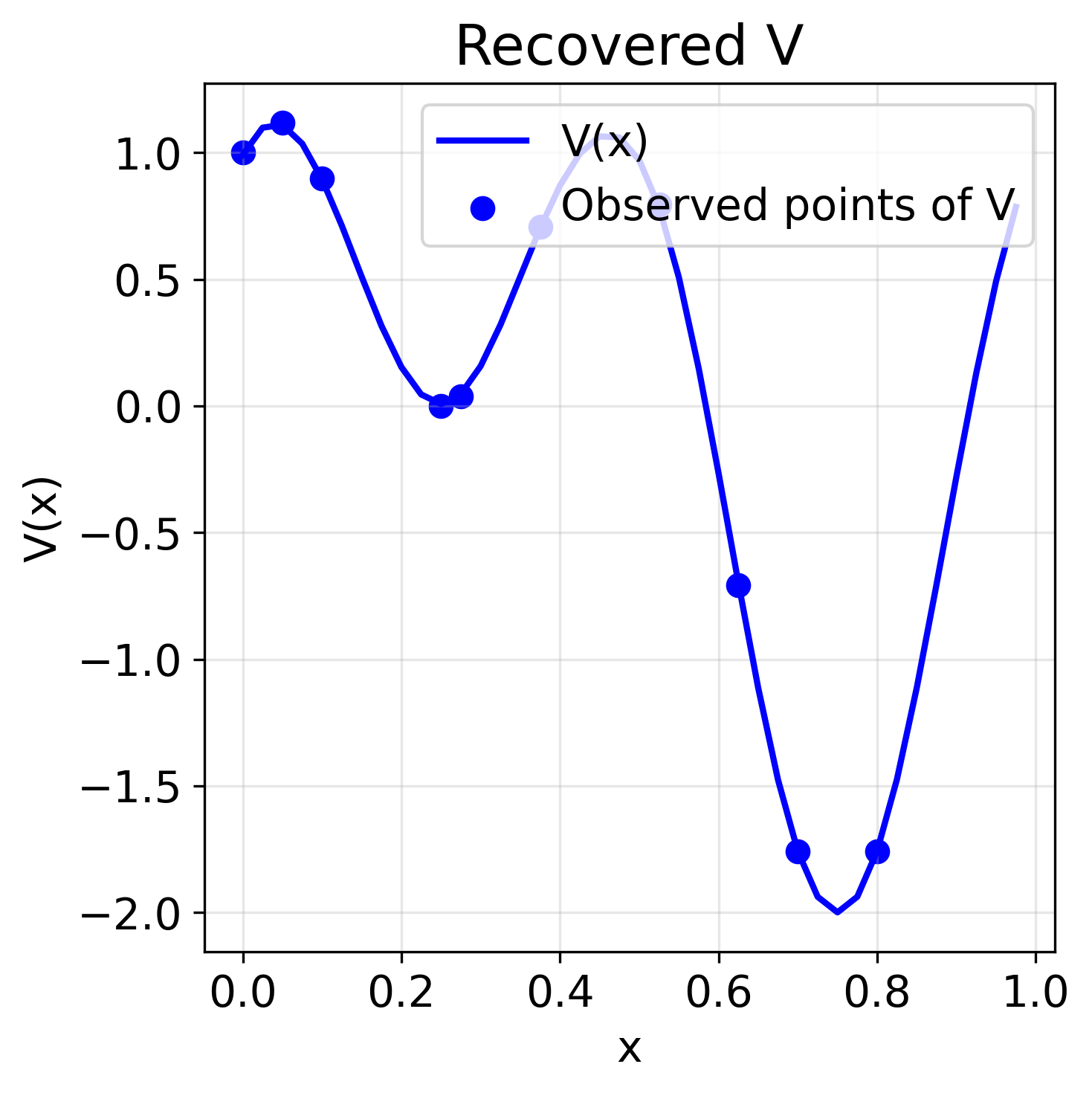}
        \caption{Recovered $V$ via GN}
        \label{Timedependentinversefig13}
    \end{subfigure}%
    \hspace{1mm}
    \begin{subfigure}[b]{0.23\textwidth}
        \centering
        \includegraphics[width=\linewidth]{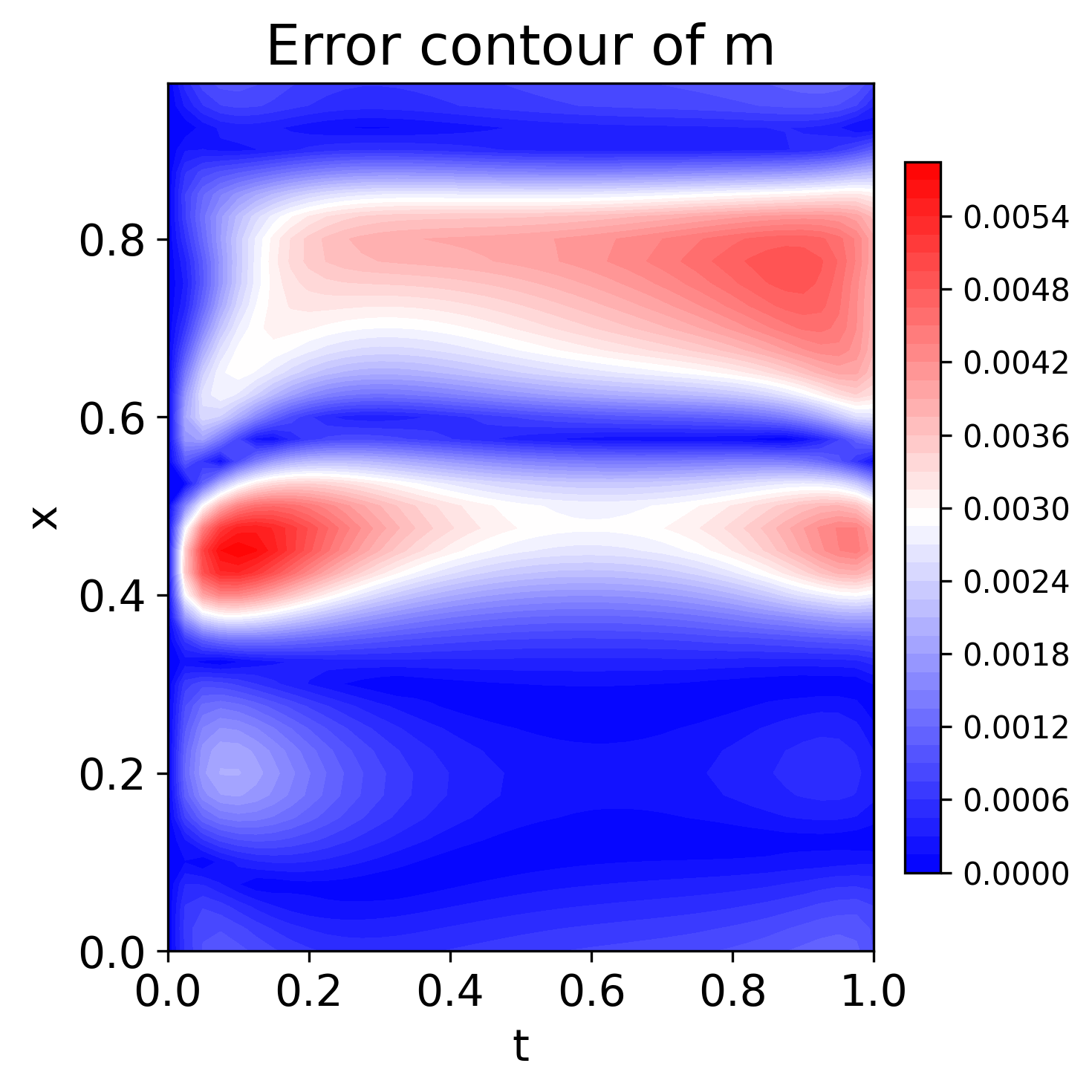}
        \caption{Error of $m$ via GN}
        \label{Timedependentinversefig14}
    \end{subfigure}
    
    \begin{subfigure}[b]{0.23\textwidth}
        \centering
        \includegraphics[width=\linewidth]{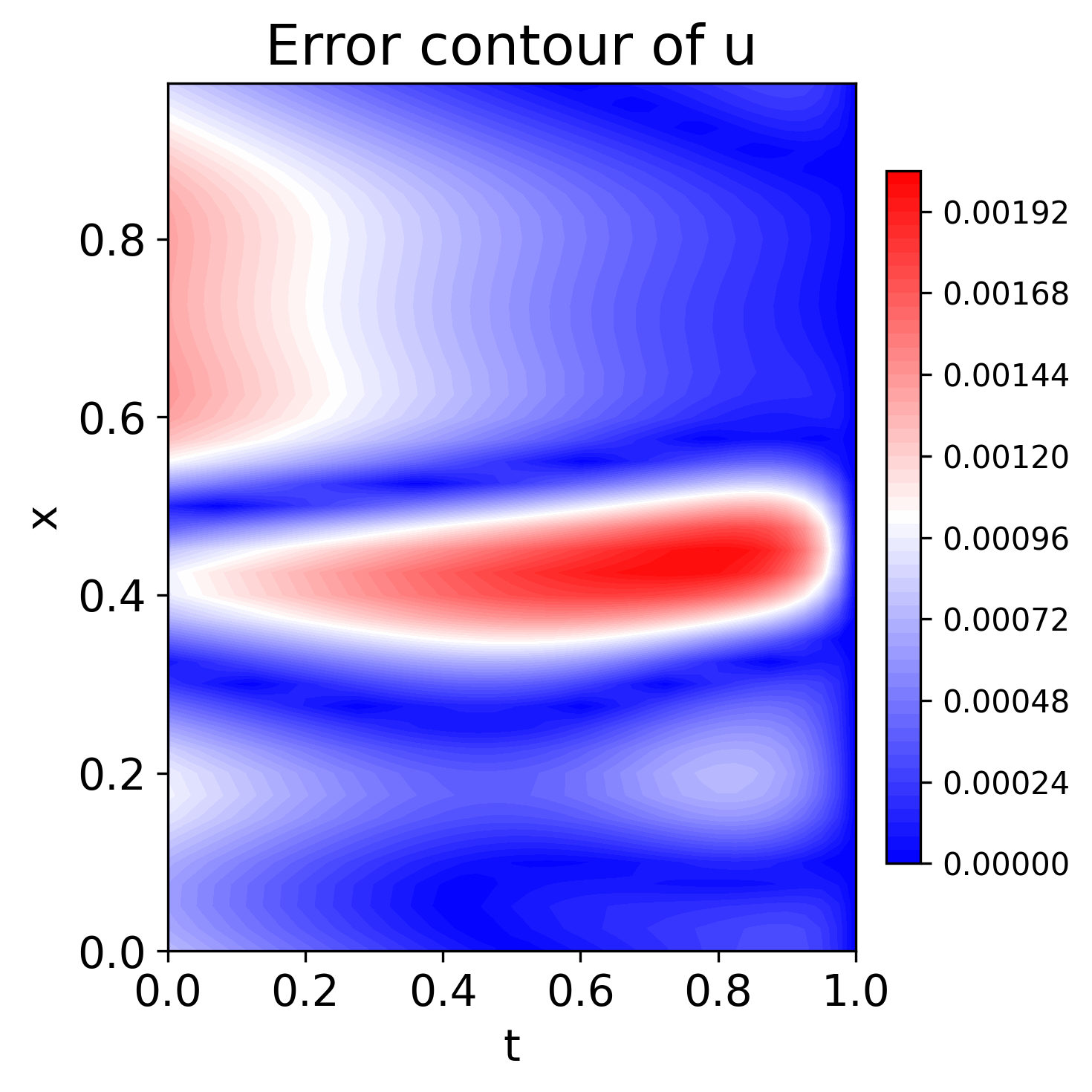}
        \caption{Error of $u$ via GD}
        \label{Timedependentinversefig10}
    \end{subfigure}
    \hspace{1mm}
    \begin{subfigure}[b]{0.23\textwidth}
        \centering
        \includegraphics[width=\linewidth]{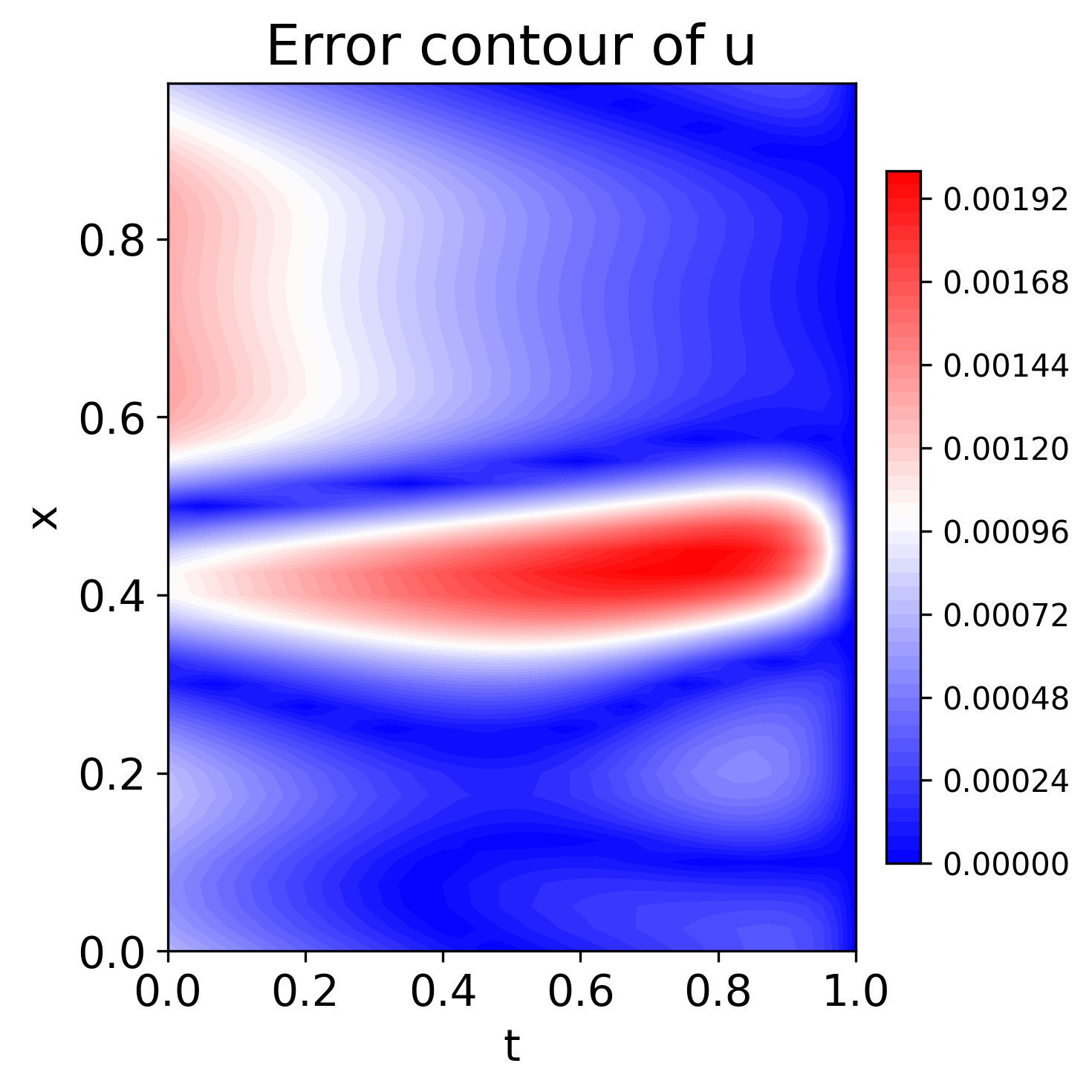}
        \caption{Error of $u$ via GN}
        \label{Timedependentinversefig15}
    \end{subfigure}
    \hspace{1mm}
    \begin{subfigure}[b]{0.23\textwidth}
        \centering
        \includegraphics[width=\linewidth]{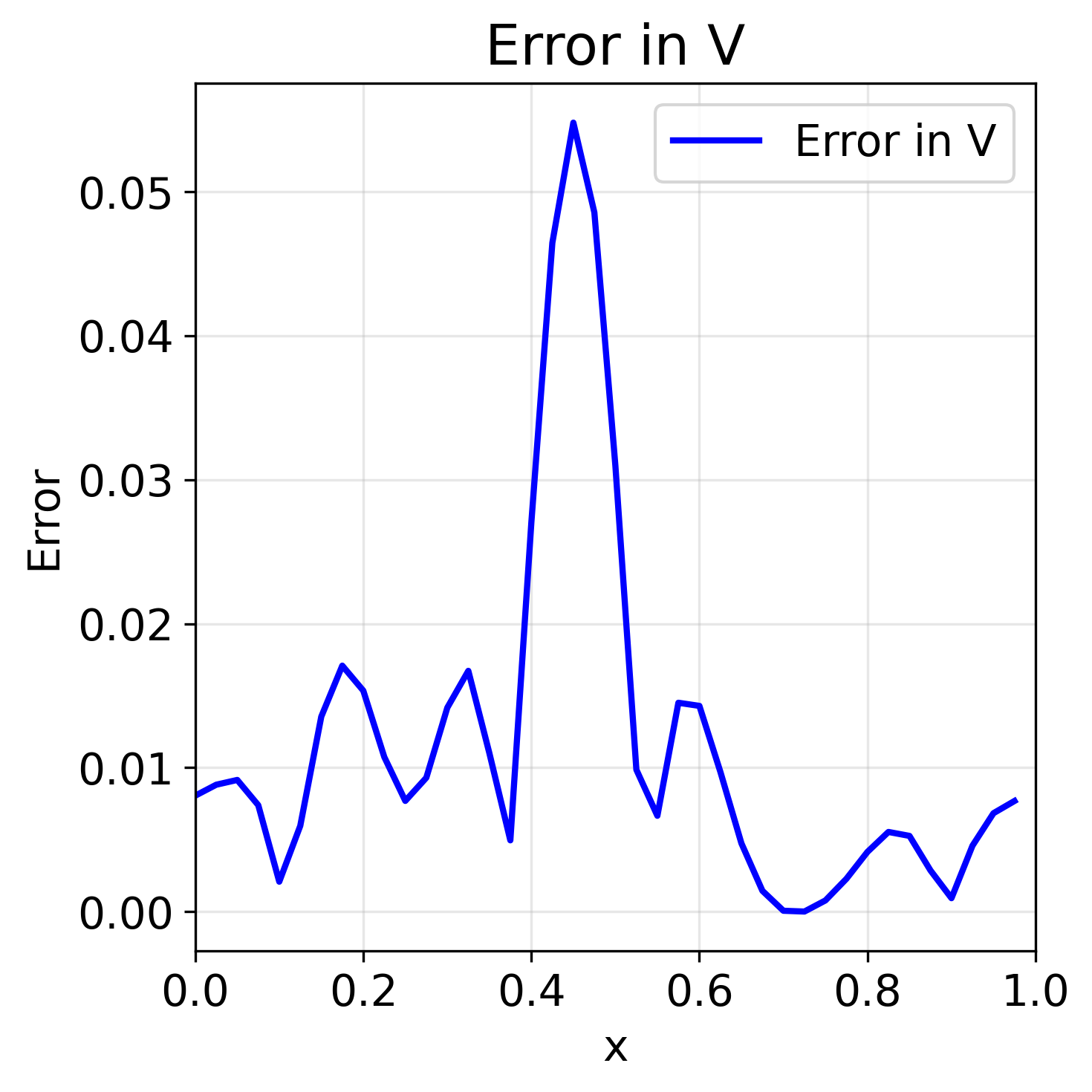}
        \caption{Error of $V$ via GD}
        \label{Timedependentinversefig16}
    \end{subfigure}%
    \hspace{1mm}
    \begin{subfigure}[b]{0.23\textwidth}
        \centering
        \includegraphics[width=\linewidth]{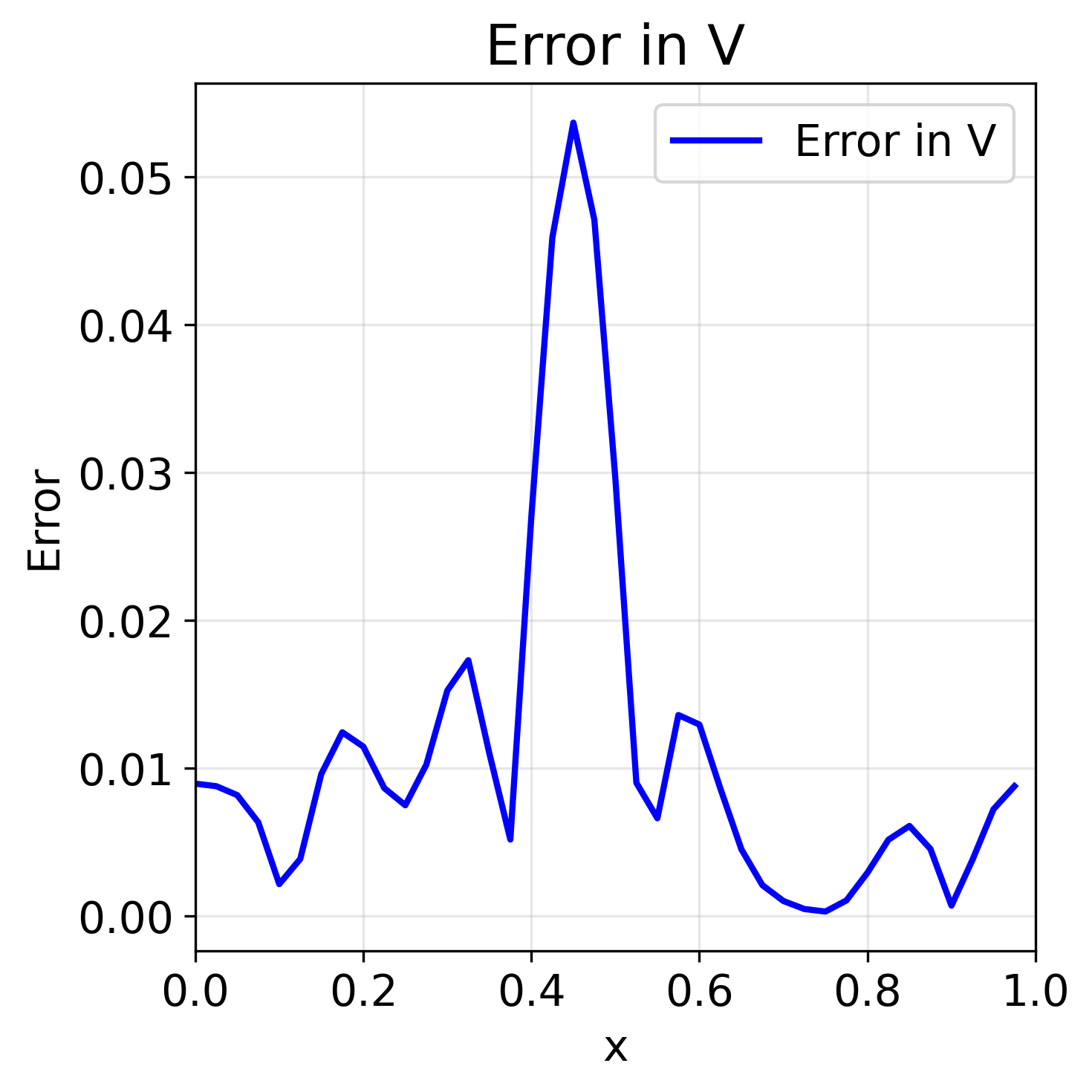}
        \caption{Error of $V$ via GN}
        \label{Timedependentinversefig17}
    \end{subfigure}
    \caption{Numerical results for the inverse problem of the time-dependent MFG in Section~\ref{timedependenton1Dexample1}.  (a), (b), (c) are references for $m,u,V$; (d) log-log plot of the loss comparison for GD and GN across iterations; (e), (f), (g) recovered $m,u,V$ via GD; (h), (m), (o) errors of $m,u,V$ via GD; (i), (j), (k) recovered $m,u,V$ via GN; (l), (n), (p) errors of $m,u,V$ via GN.}
    \label{timedependV}
\end{figure}

\begin{figure}[!htbp]
    \centering

    \begin{subfigure}[b]{0.23\textwidth}
        \centering
        \includegraphics[width=\linewidth]{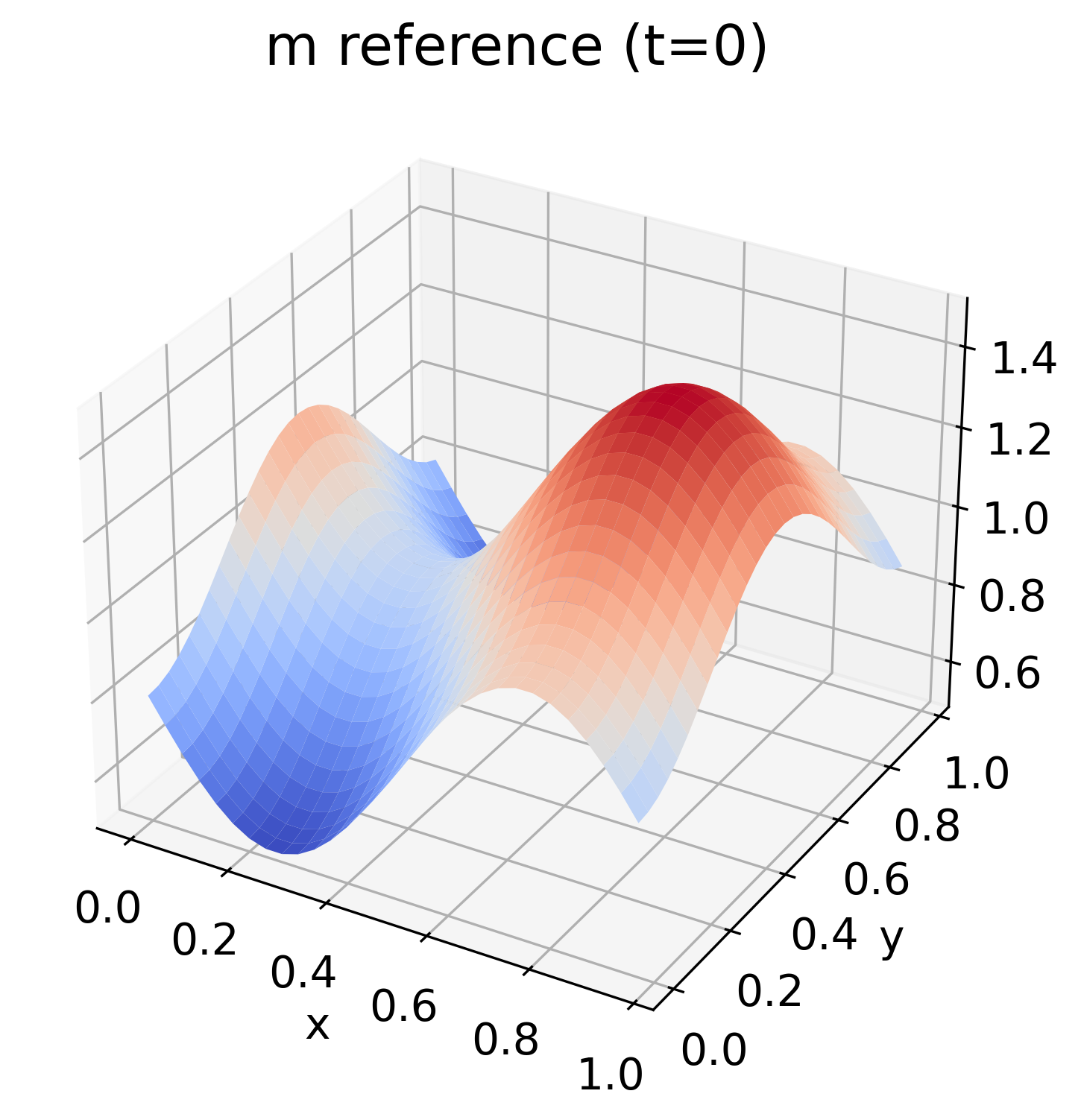}
        \caption{$m$ reference ($t=0$)}
        \label{fig:mref_t0}
    \end{subfigure}\hspace{1mm}
    \begin{subfigure}[b]{0.23\textwidth}
        \centering
        \includegraphics[width=\linewidth]{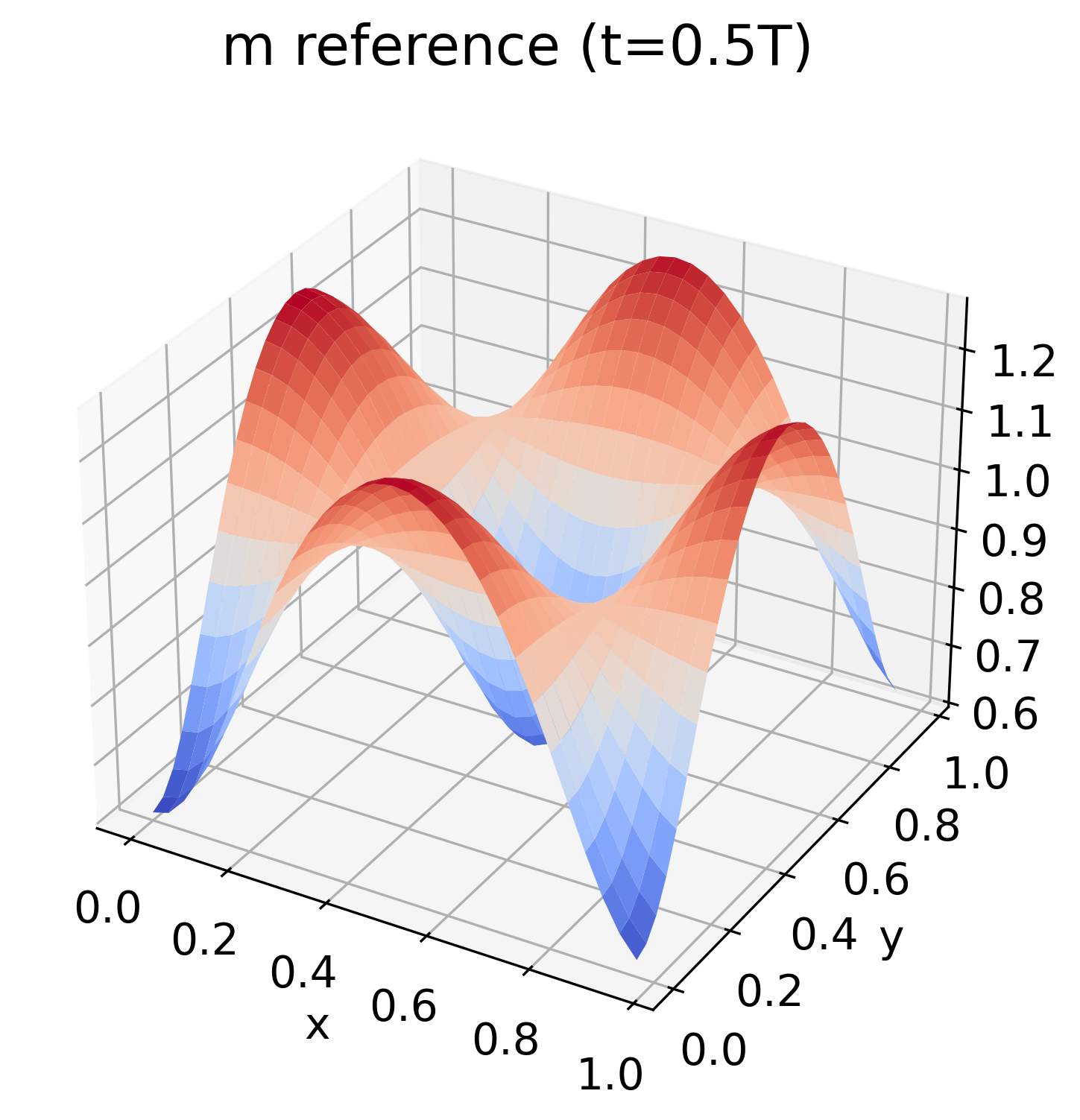}
        \caption{$m$ reference ($t=0.5T$)}
        \label{fig:mref_t05}
    \end{subfigure}\hspace{1mm}
    \begin{subfigure}[b]{0.23\textwidth}
        \centering
        \includegraphics[width=\linewidth]{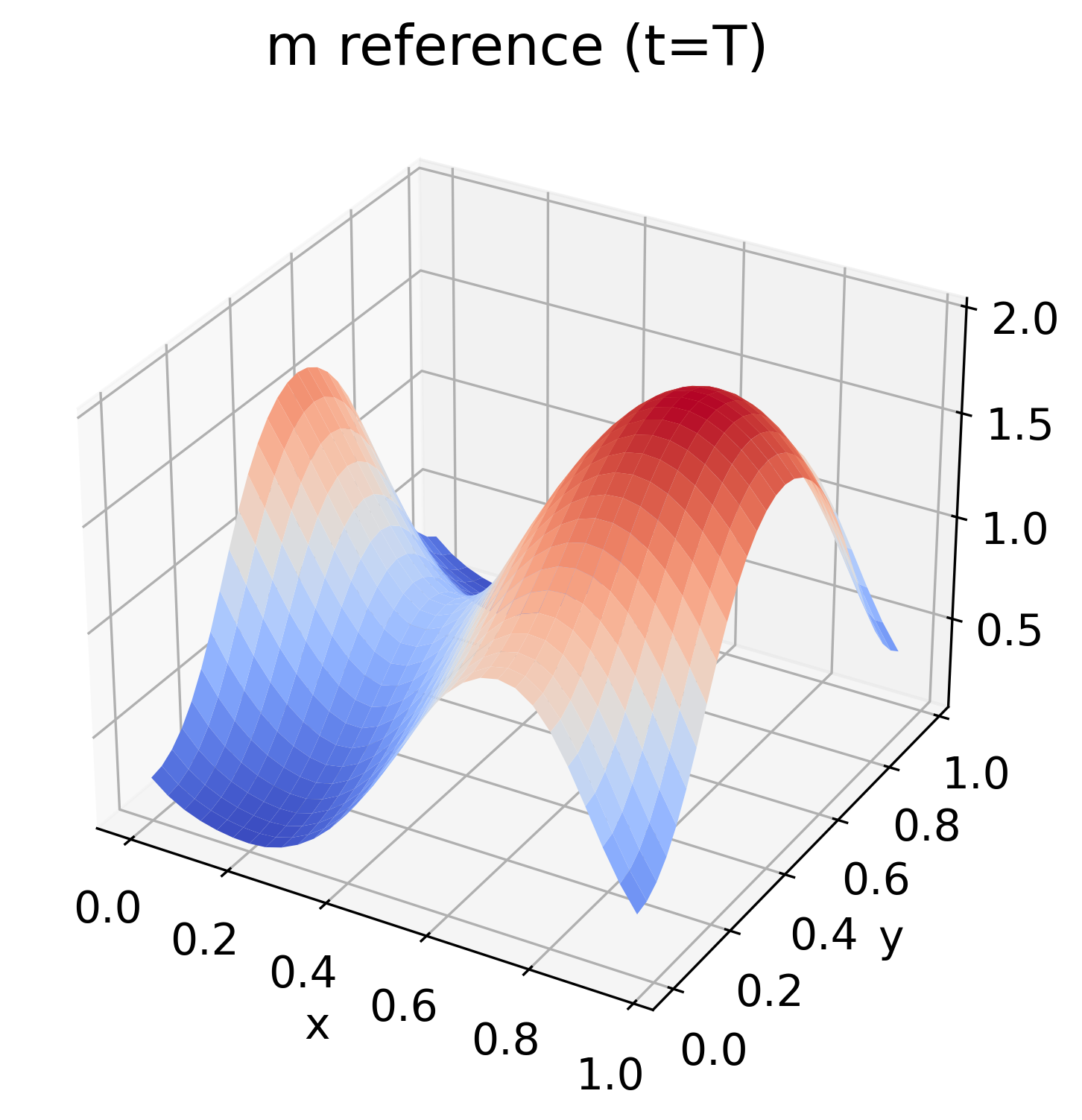}
        \caption{$m$ reference ($t=T$)}
        \label{fig:mref_t1}
    \end{subfigure}\hspace{1mm}
    \begin{subfigure}[b]{0.23\textwidth}
        \centering
        \includegraphics[width=\linewidth]{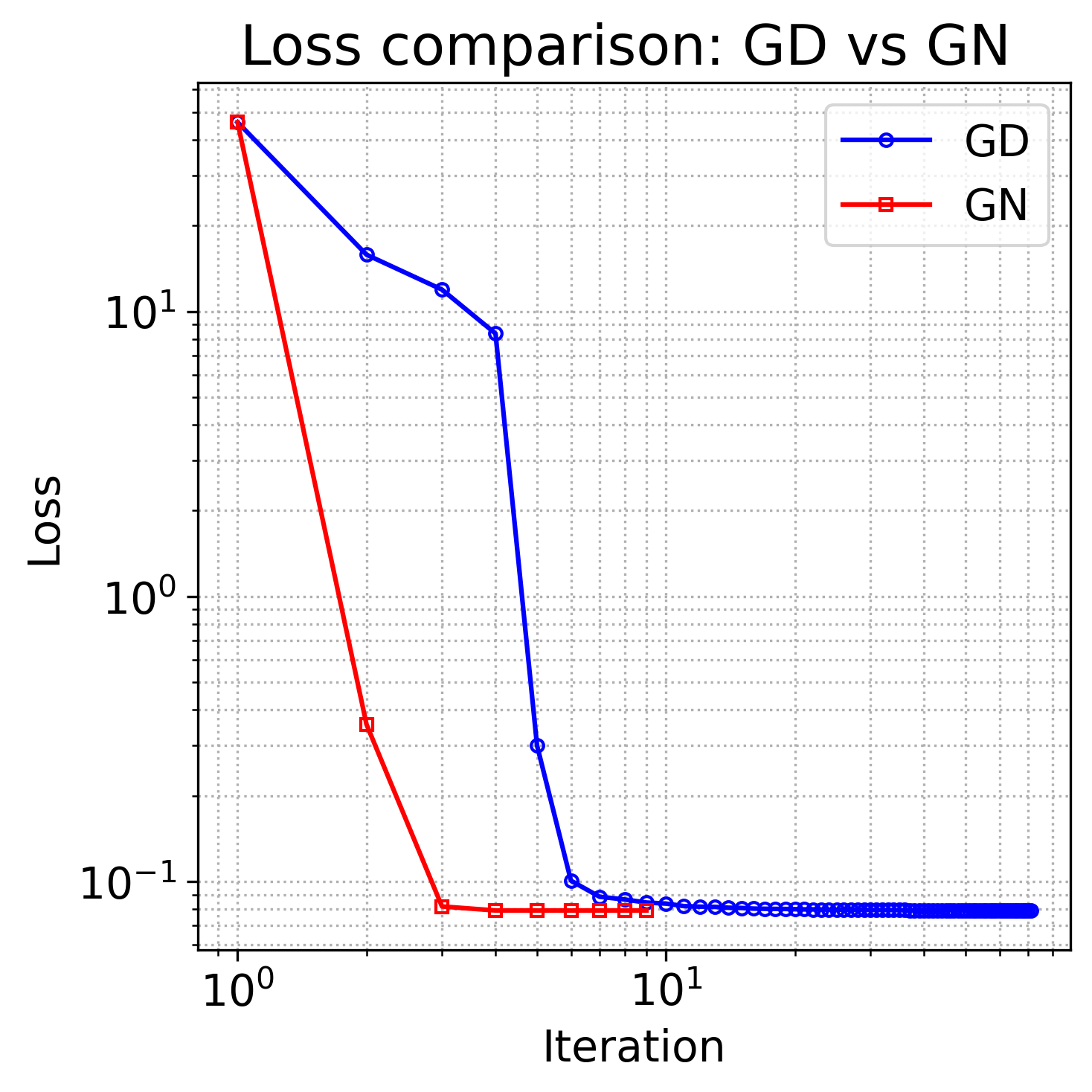}
        \caption{Loss comparison}
        \label{fig:loss}
    \end{subfigure}

    \begin{subfigure}[b]{0.23\textwidth}
        \centering
        \includegraphics[width=\linewidth]{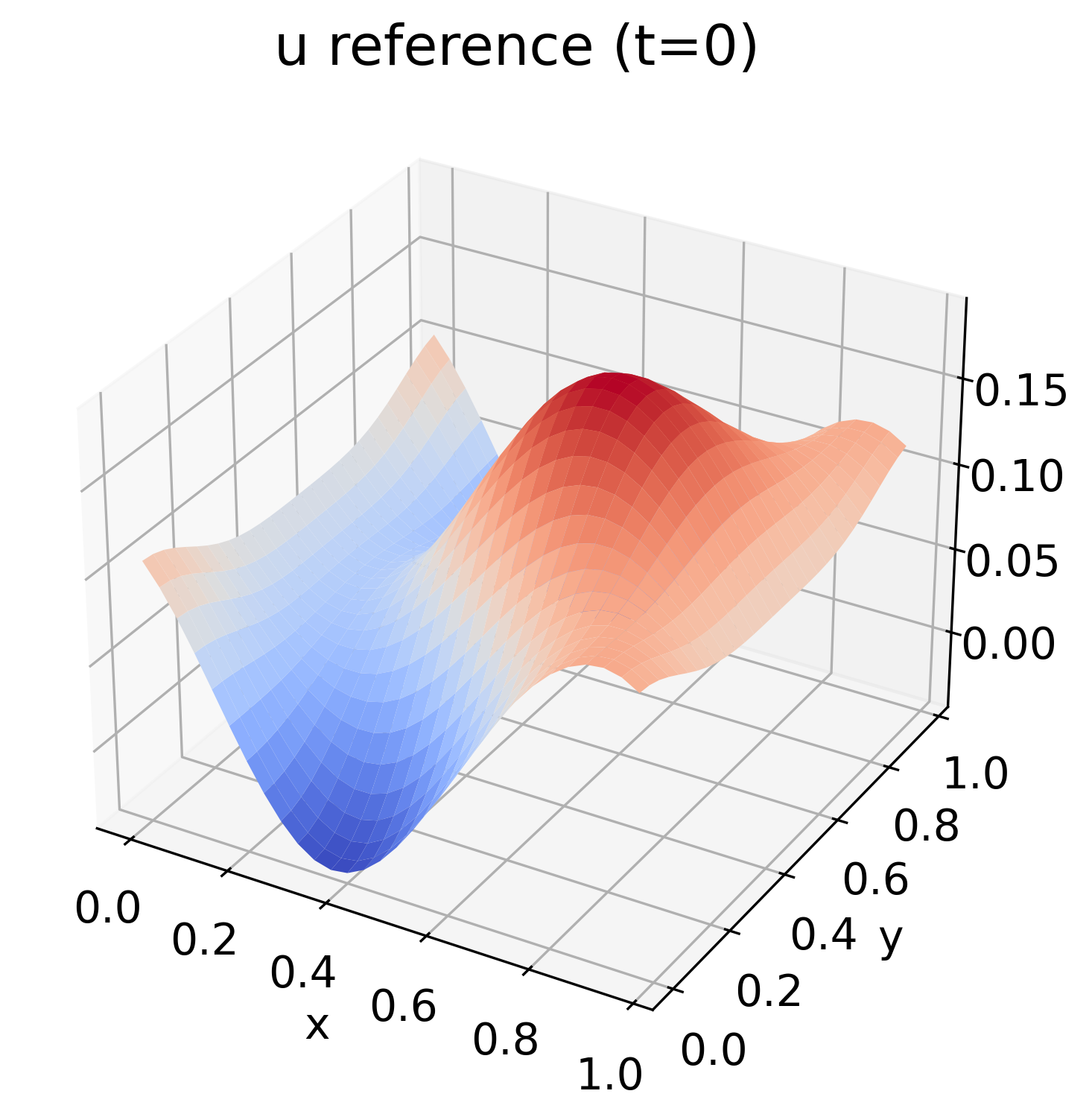}
        \caption{$u$ reference ($t=0$)}
        \label{fig:uref_t0}
    \end{subfigure}\hspace{1mm}
    \begin{subfigure}[b]{0.23\textwidth}
        \centering
        \includegraphics[width=\linewidth]{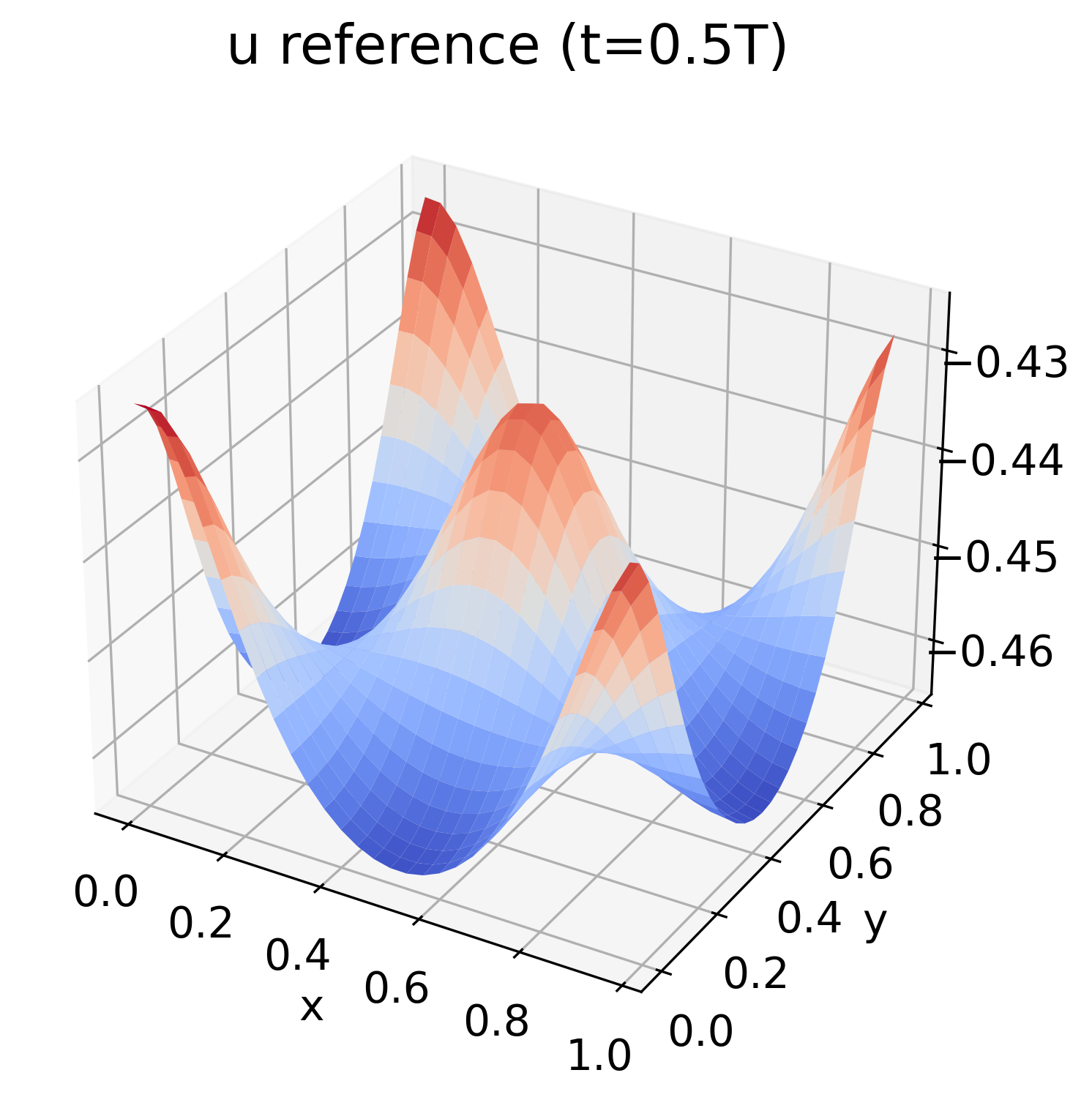}
        \caption{$u$ reference ($t=0.5T$)}
        \label{fig:uref_t05}
    \end{subfigure}\hspace{1mm}
    \begin{subfigure}[b]{0.23\textwidth}
        \centering
        \includegraphics[width=\linewidth]{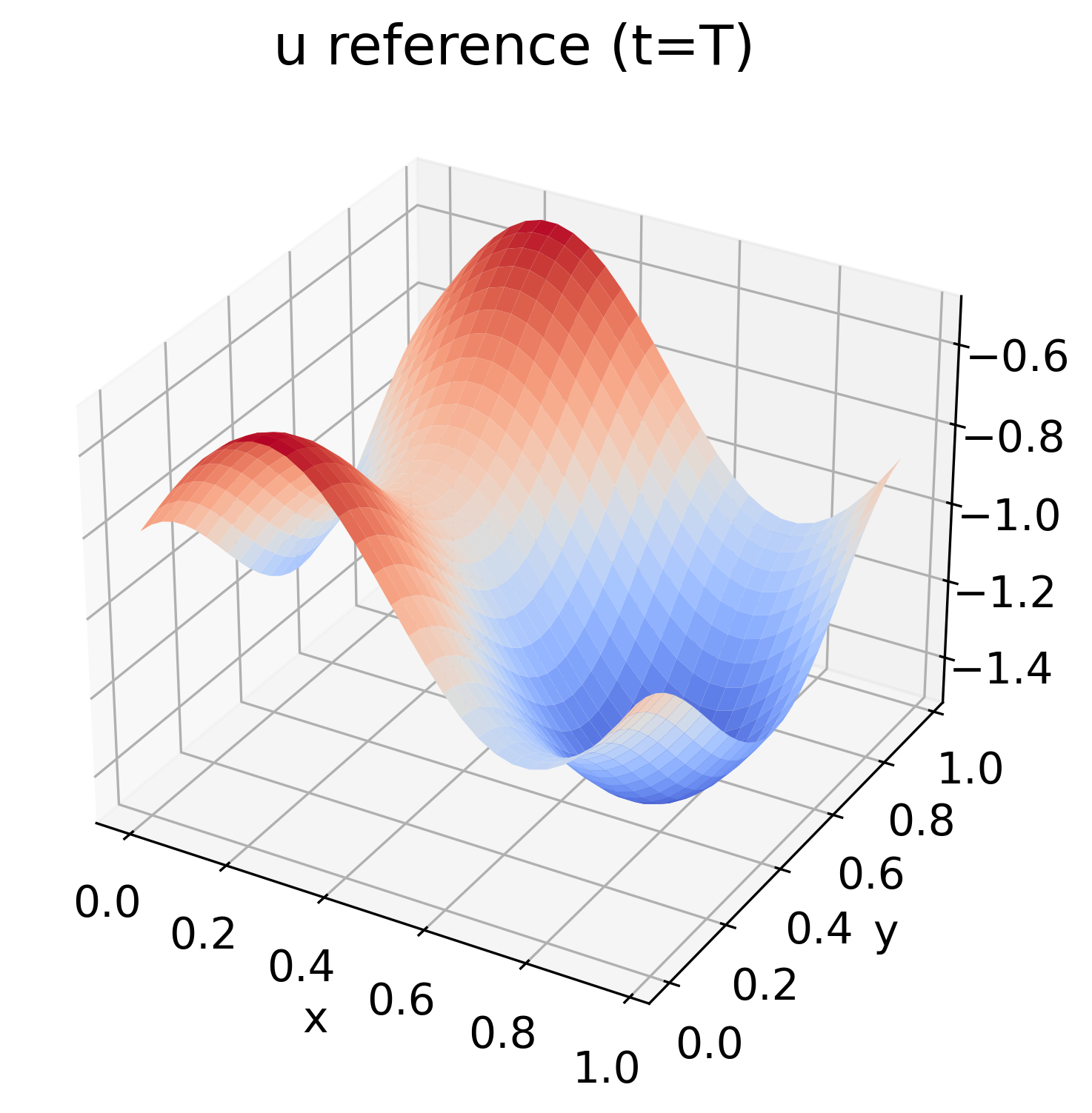}
        \caption{$u$ reference ($t=T$)}
        \label{fig:uref_t1}
    \end{subfigure}\hspace{1mm}
    \begin{subfigure}[b]{0.23\textwidth}
        \centering
        \includegraphics[width=\linewidth]{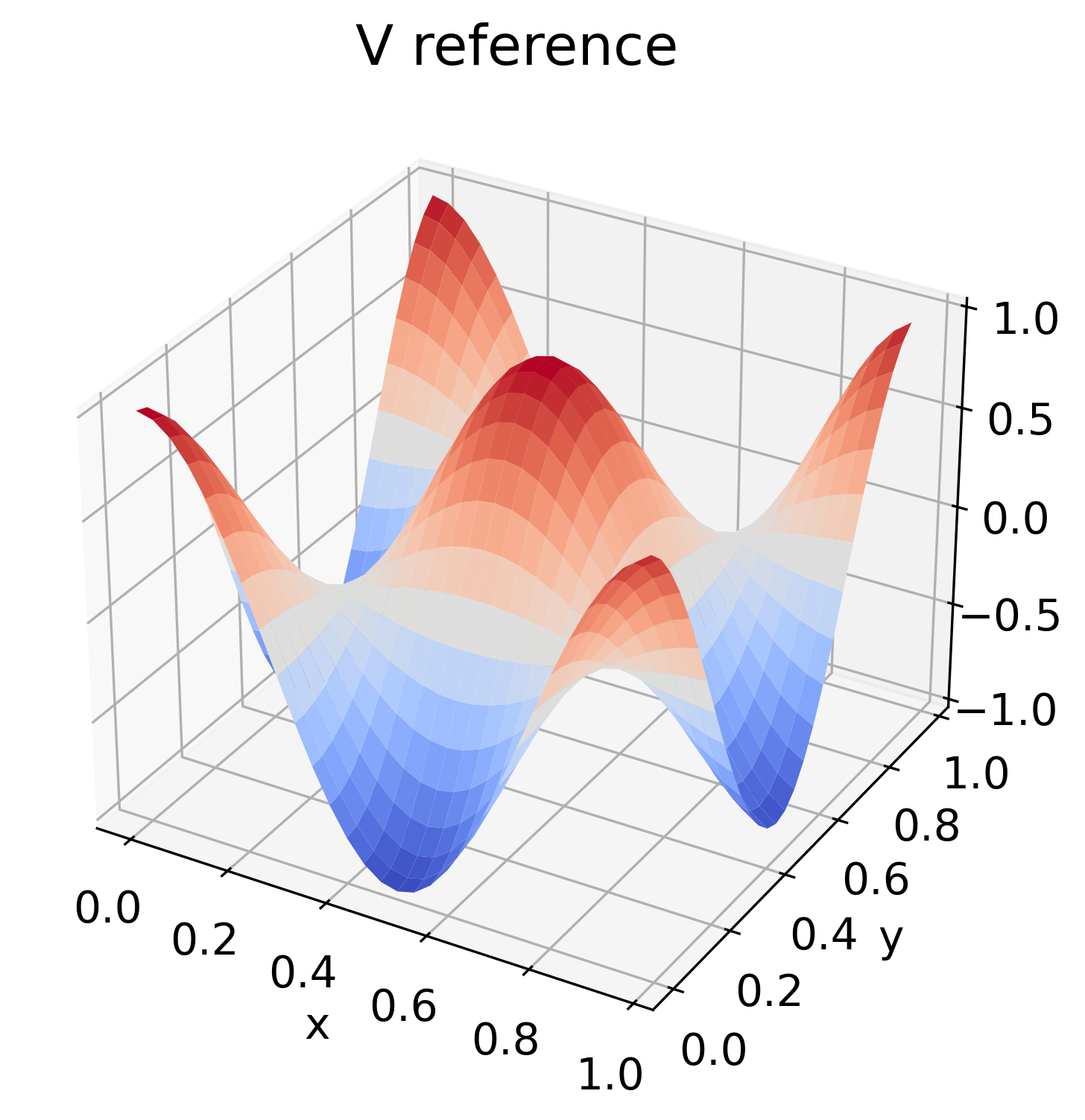}
        \caption{$V$ reference}
        \label{fig:Vref}
    \end{subfigure}

    \begin{subfigure}[b]{0.23\textwidth}
        \centering
        \includegraphics[width=\linewidth]{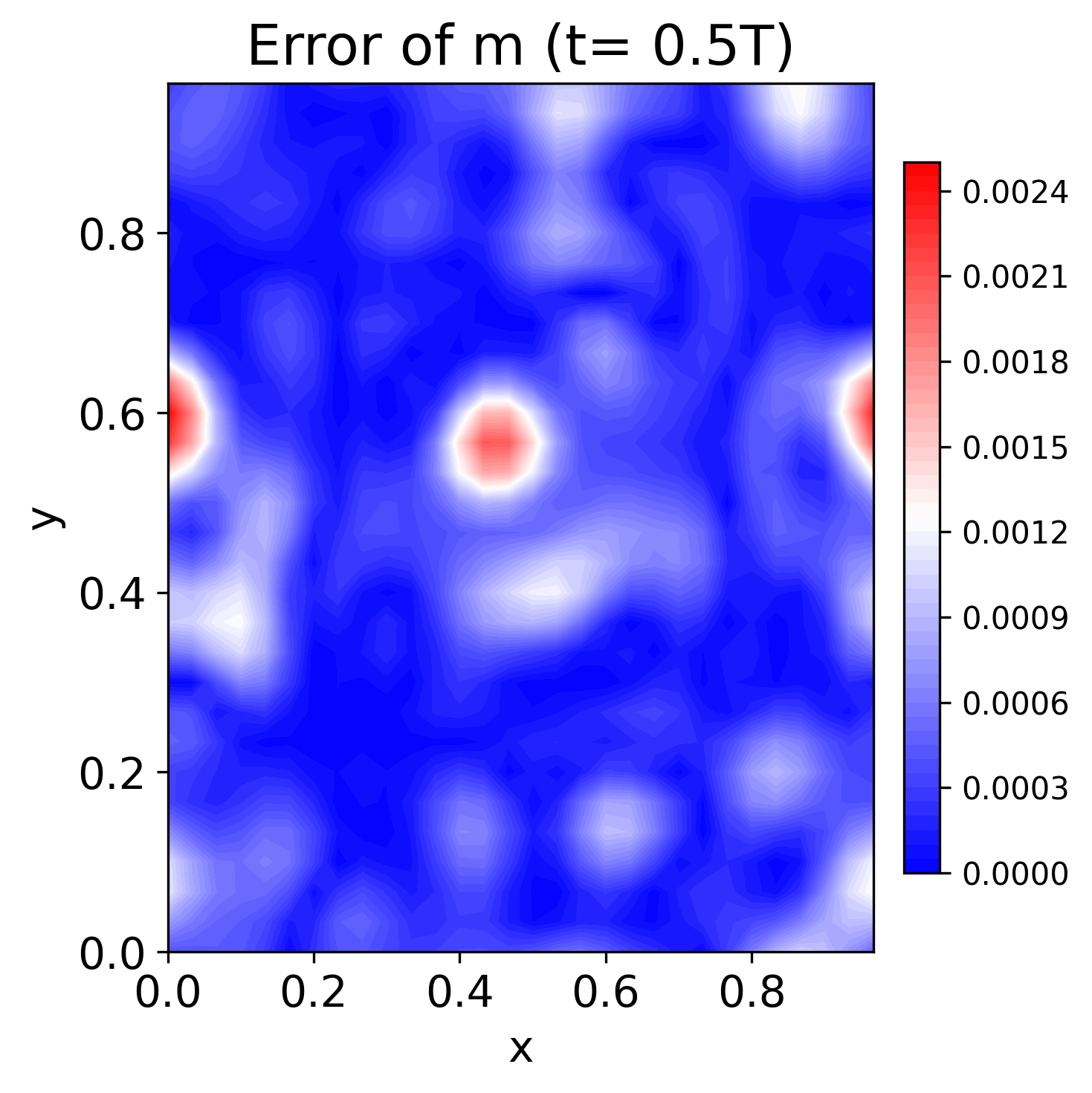}
        \caption{$m$ error (GD, $t=0.5T$)}
        \label{fig:merr_gd_t05}
    \end{subfigure}\hspace{1mm}
    \begin{subfigure}[b]{0.23\textwidth}
        \centering
        \includegraphics[width=\linewidth]{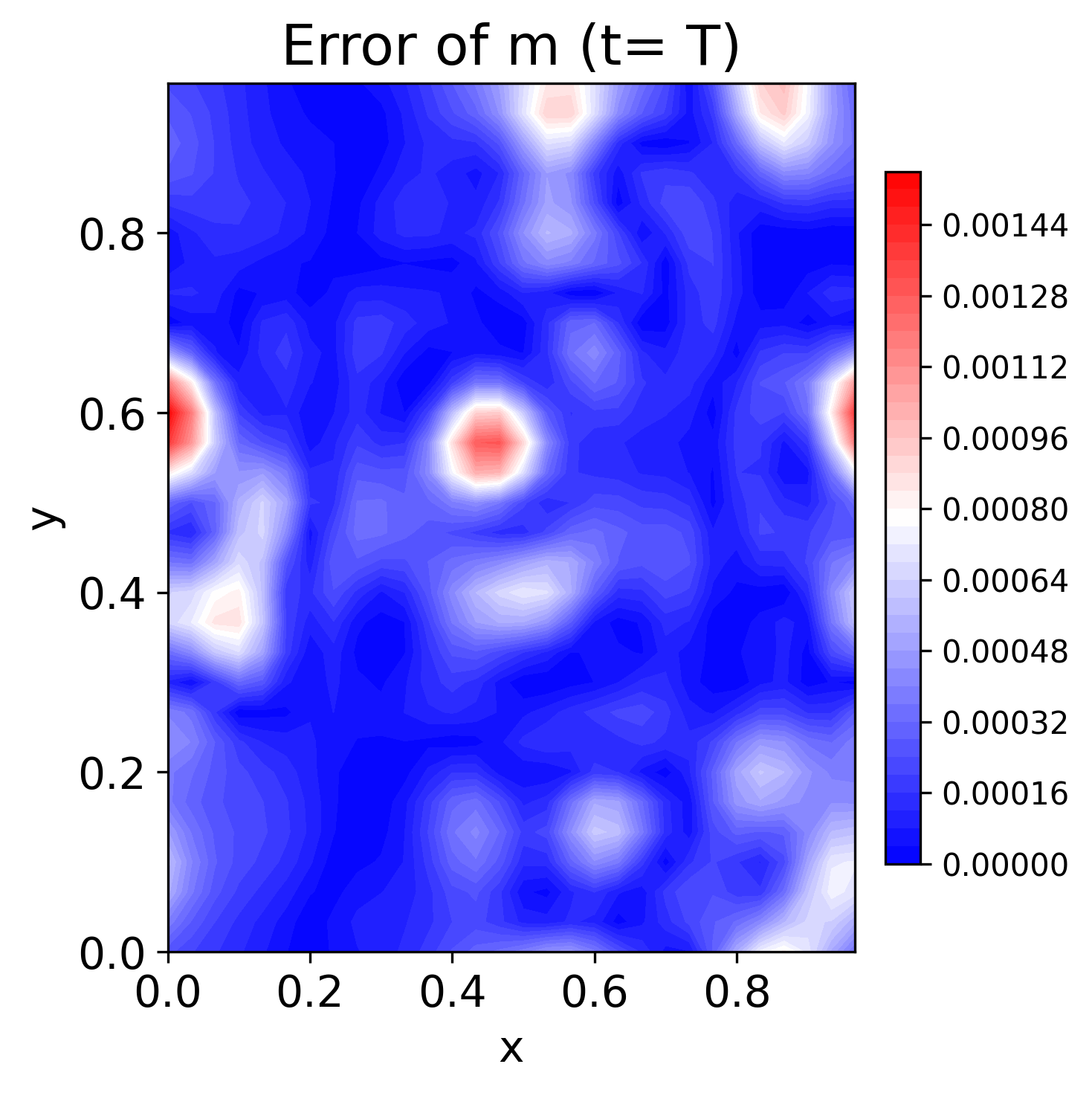}
        \caption{$m$ error (GD, $t=T$)}
        \label{fig:merr_gd_t1}
    \end{subfigure}\hspace{1mm}
    \begin{subfigure}[b]{0.23\textwidth}
        \centering
        \includegraphics[width=\linewidth]{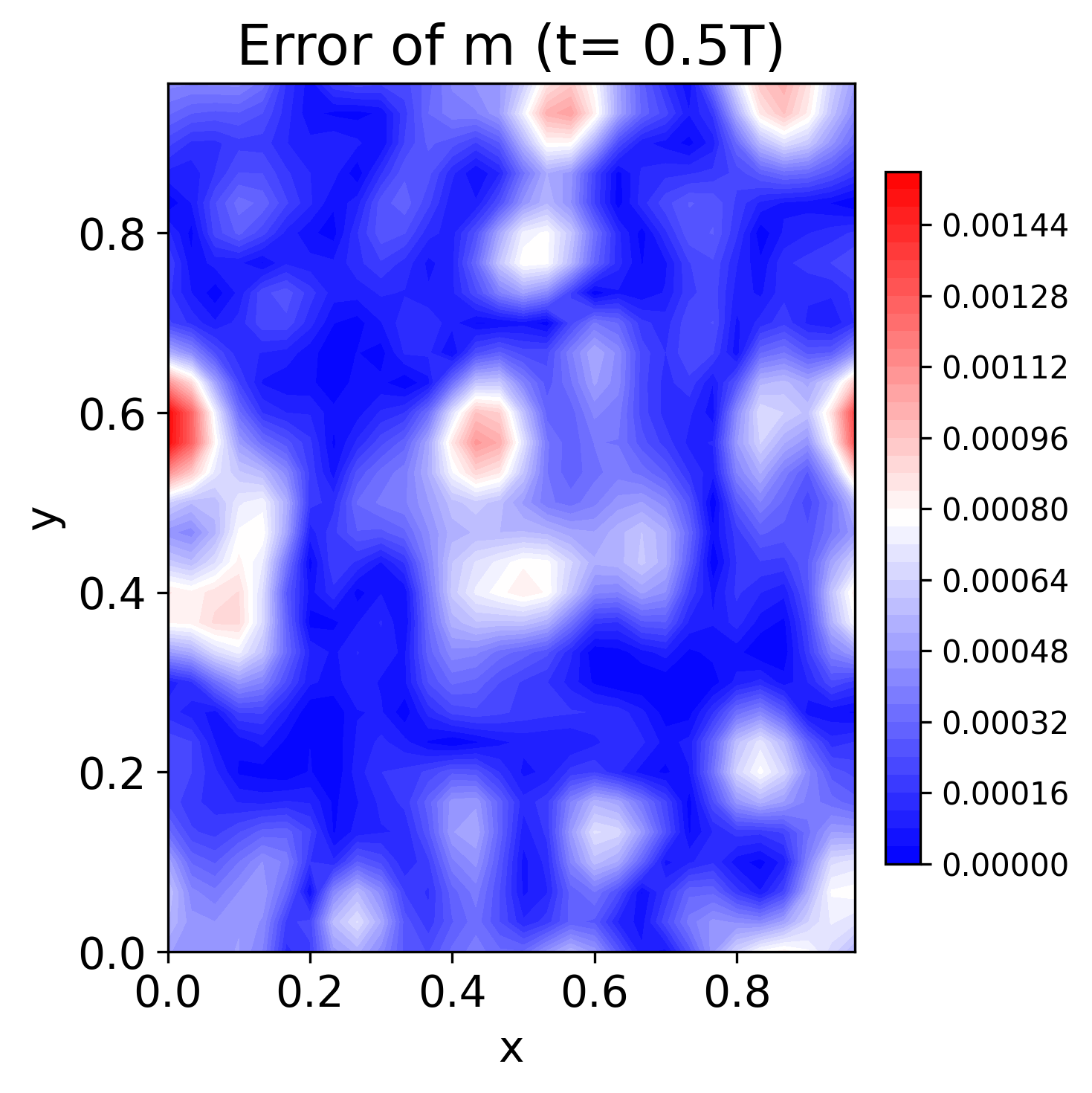}
        \caption{$m$ error (GN, $t=0.5T$)}
        \label{fig:merr_gn_t05}
    \end{subfigure}\hspace{1mm}
    \begin{subfigure}[b]{0.23\textwidth}
        \centering
        \includegraphics[width=\linewidth]{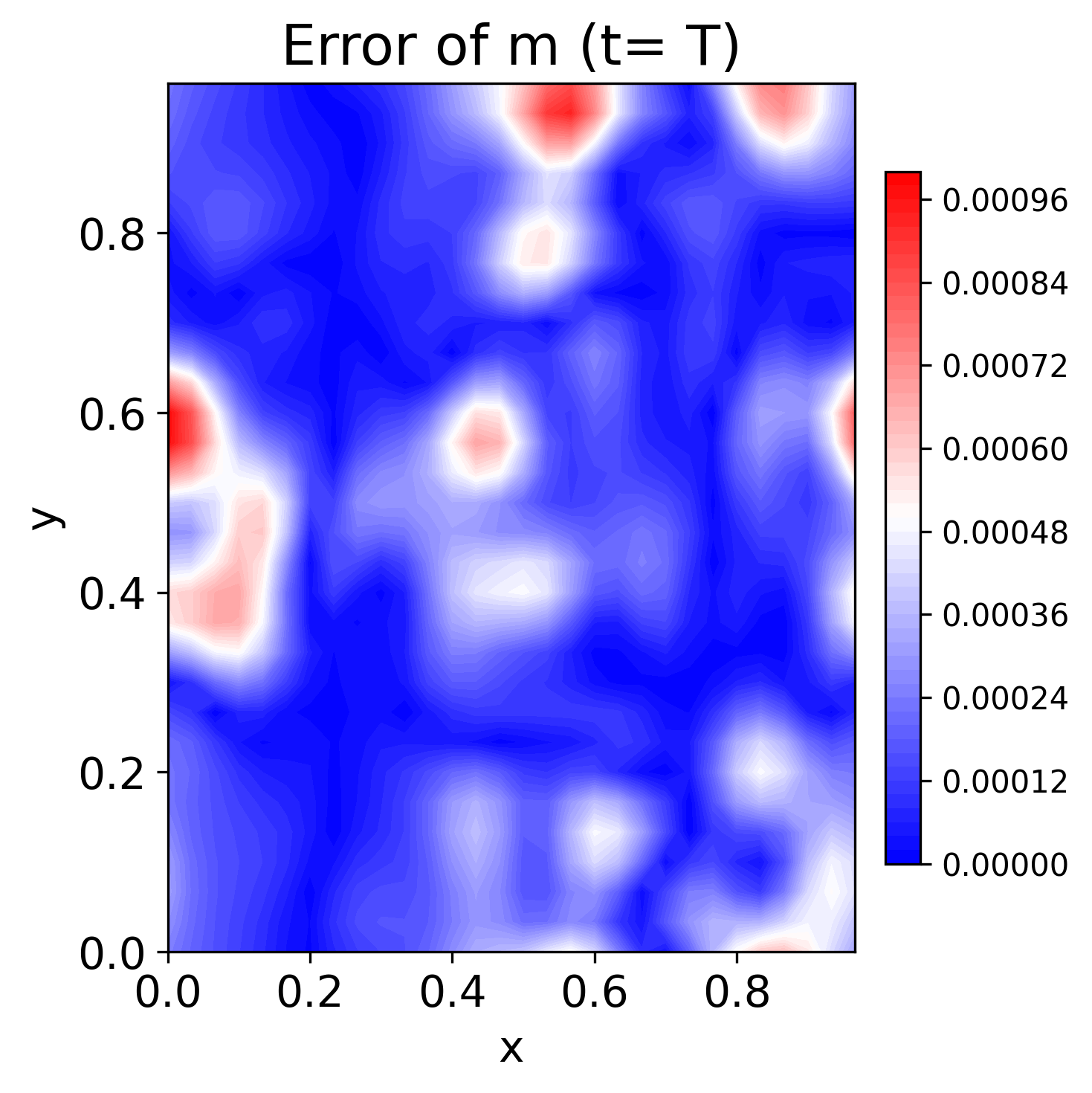}
        \caption{$m$ error (GN, $t=T$)}
        \label{fig:merr_gn_t1}
    \end{subfigure}

    \begin{subfigure}[b]{0.23\textwidth}
        \centering
        \includegraphics[width=\linewidth]{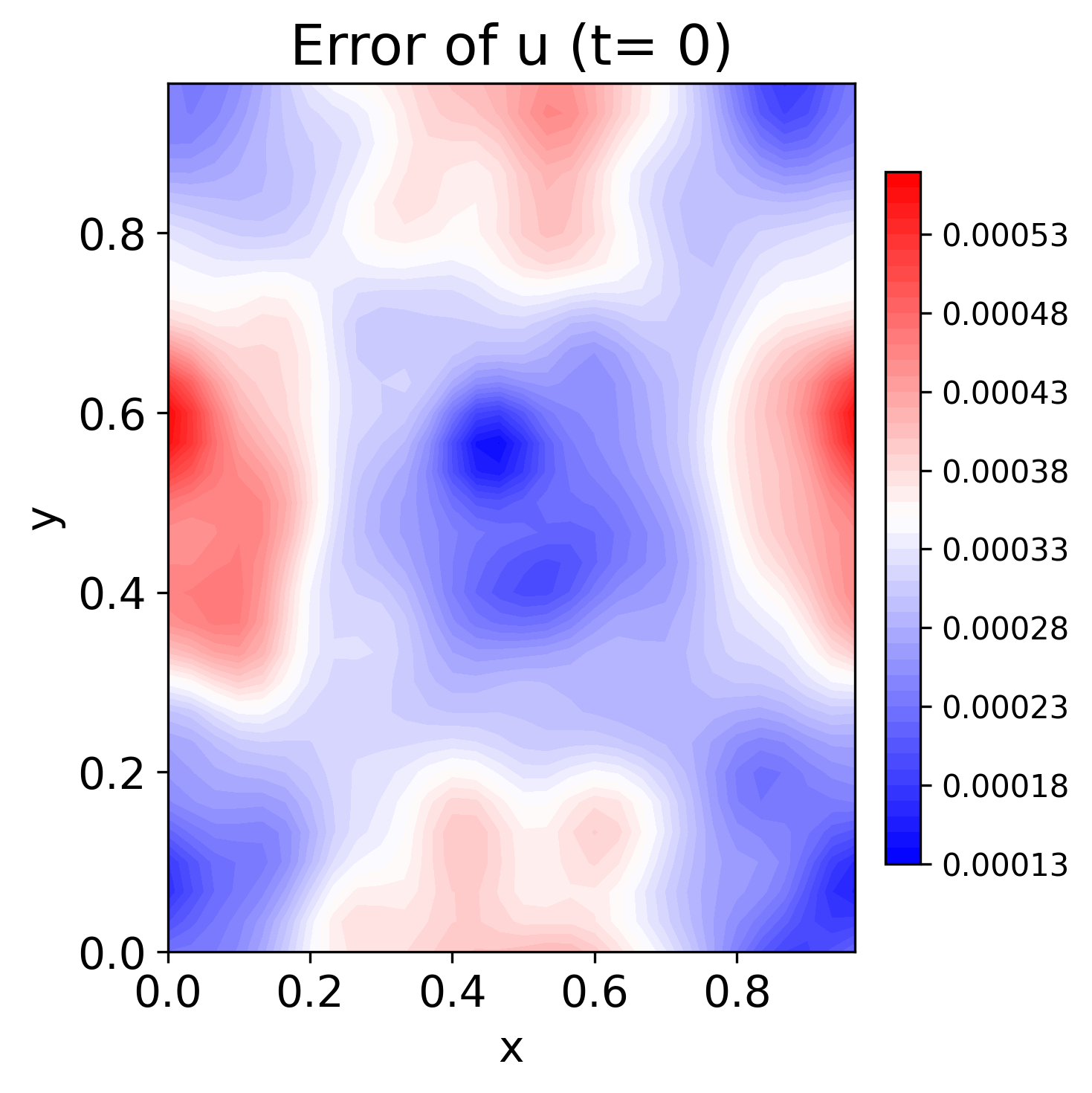}
        \caption{$u$ error (GD, $t=0$)}
        \label{fig:uerr_gd_t0}
    \end{subfigure}\hspace{1mm}
    \begin{subfigure}[b]{0.23\textwidth}
        \centering
        \includegraphics[width=\linewidth]{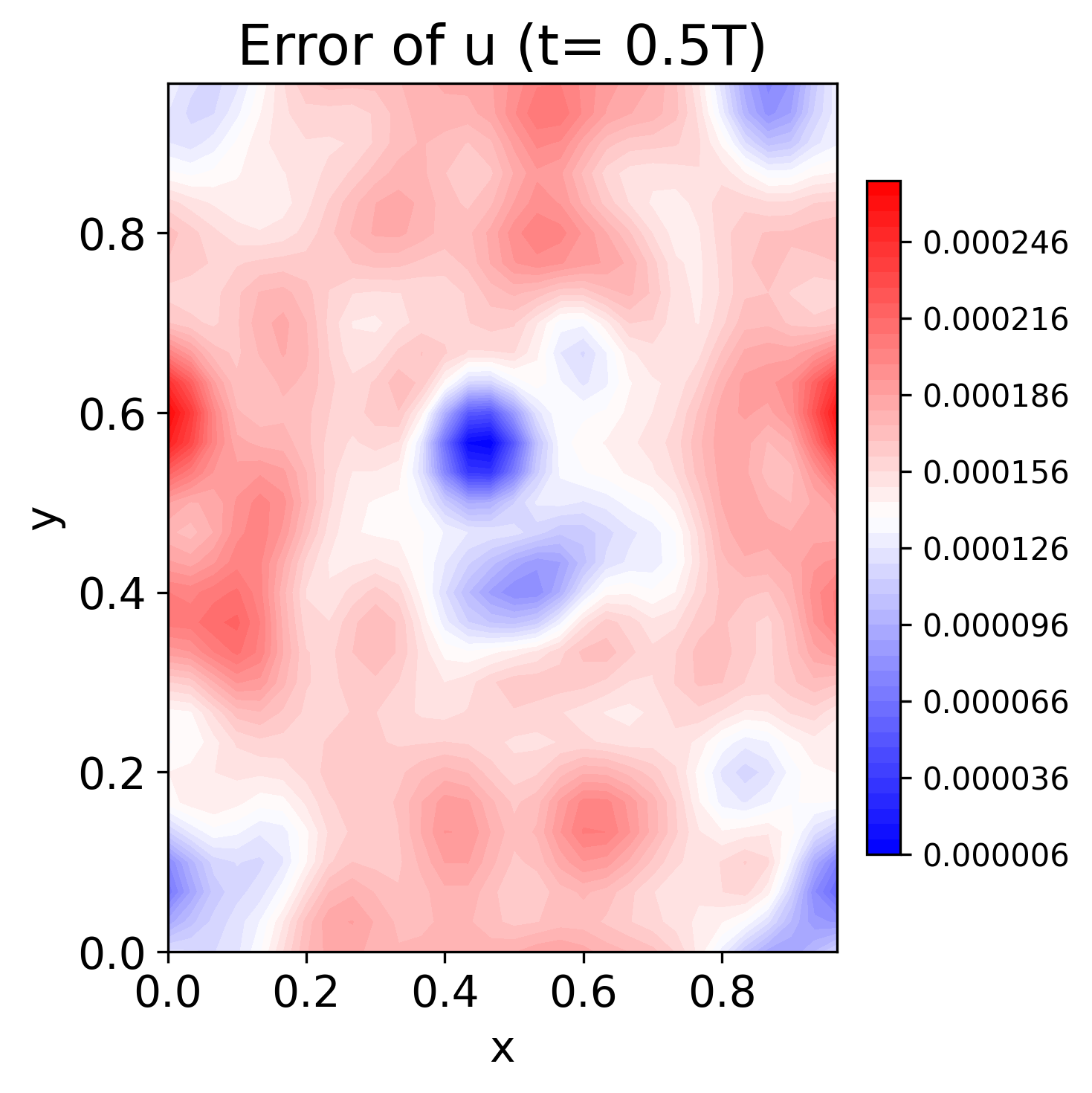}
        \caption{$u$ error (GD, $t=0.5T$)}
        \label{fig:uerr_gd_t05}
    \end{subfigure}\hspace{1mm}
    \begin{subfigure}[b]{0.23\textwidth}
        \centering
        \includegraphics[width=\linewidth]{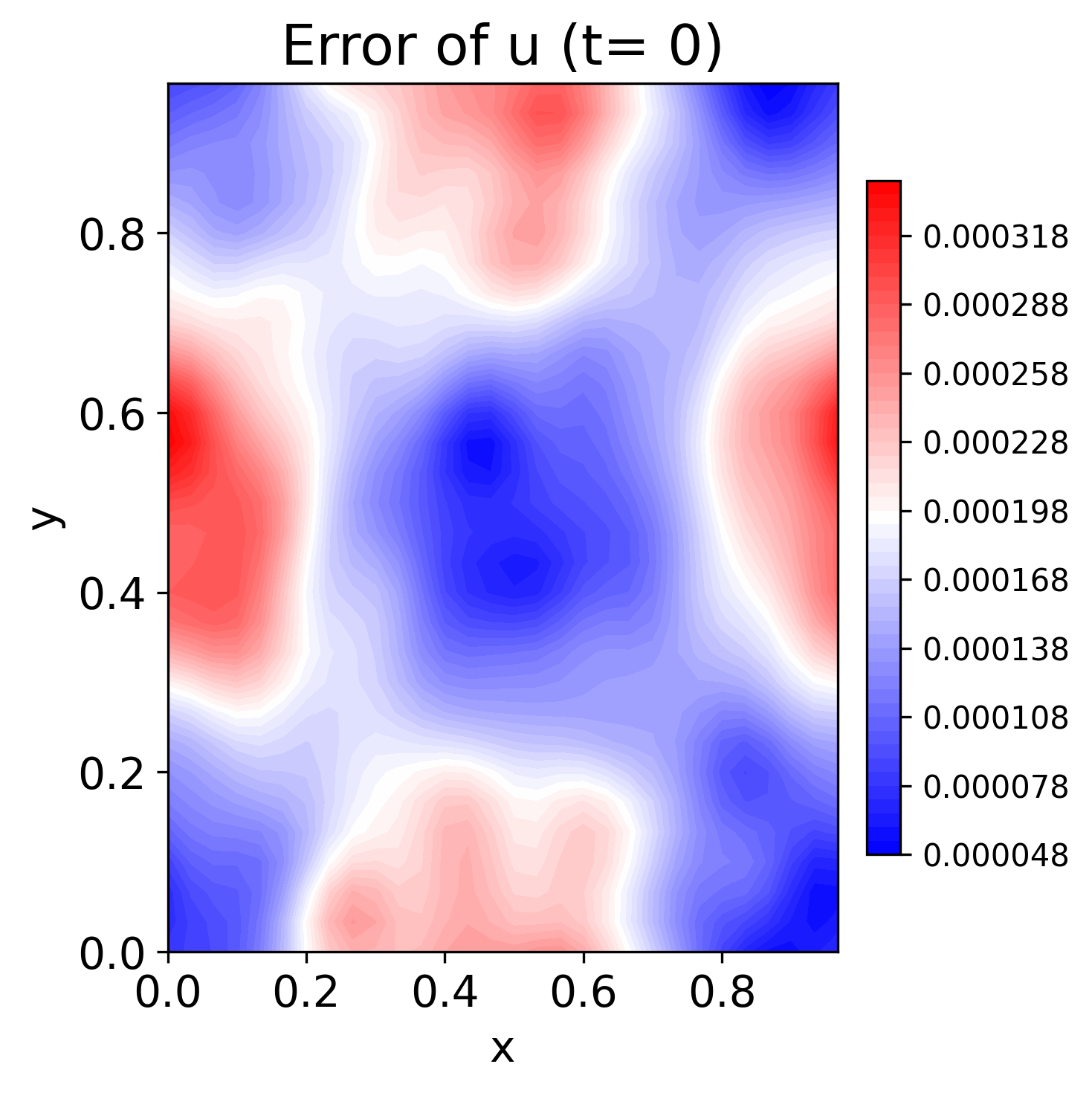}
        \caption{$u$ error (GN, $t=0$)}
        \label{fig:uerr_gn_t0}
    \end{subfigure}\hspace{1mm}
    \begin{subfigure}[b]{0.23\textwidth}
        \centering
        \includegraphics[width=\linewidth]{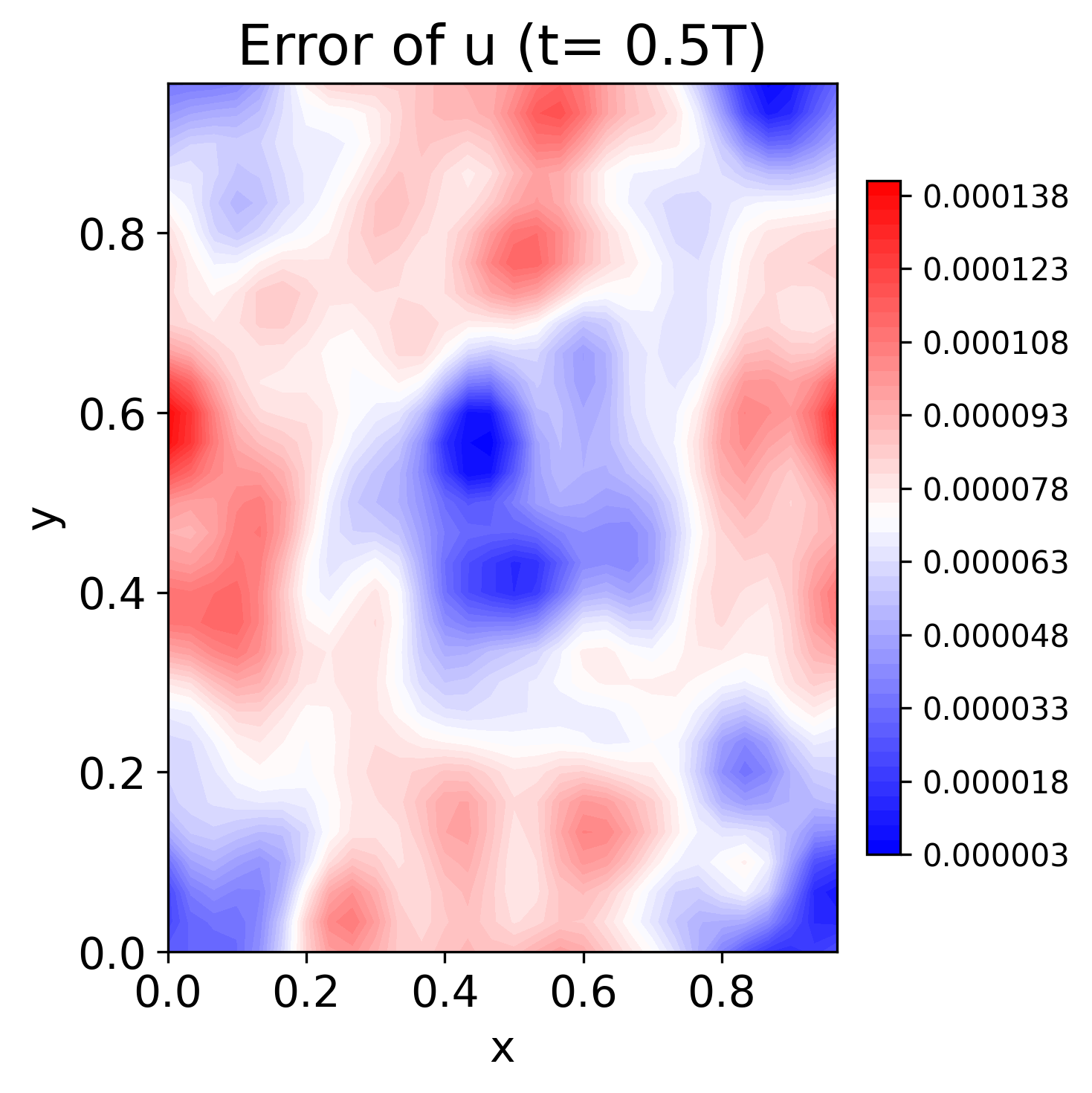}
        \caption{$u$ error (GN, $t=0.5T$)}
        \label{fig:uerr_gn_t05}
    \end{subfigure}

    \caption{Numerical results for the inverse problem of the 2D time-dependent MFG in \Cref{2Dtimedependentexample1}.  (a)-(c) reference $m$ at $t=0$, $t=0.5T$, and $t=T$; (d) log-log loss comparison of GD vs. GN; (e)-(g) reference $u$ at $t=0$, $t=0.5T$, and $t=T$; (h) time-independent reference $V$; (i)-(j) error of $m$ (GD) at $t=0.5T$ and $t=T$; (k)-(l) error of $m$ (GN) at $t=0.5T$ and $t=T$; (m)-(n) error of $u$ (GD) at $t=0$ and $t=0.5T$; (o)-(p) error of $u$ (GN) at $t=0$ and $t=0.5T$.}
    \label{fig:timedep_2d_inverse}
\end{figure}

\subsubsection{Two-Dimensional Time-Dependent MFG Inverse Problem}
\label{2Dtimedependentexample1}
In this example, we study the inverse problem for \eqref{eqtimedependinv} on the flat torus \(\mathbb{T}^2 \simeq [0,1)^2\) with \(T=1\). We take \(H(p)=\tfrac{1}{2}\lvert p\rvert^2\) and \(\nu=0.05\). The initial density is \(m_0(x,y)=1+a_1\cos\bigl(2\pi(x-\tfrac12)\bigr)+a_2\sin\bigl(2\pi(y-\tfrac12)\bigr)\) where $a_1=0.2$, $a_2=0.3$ and the terminal condition is \(u_T(x,y)=-m_0(x,y)\). We set \(f(m)=m^3\) and set the exact spatial cost \(V^*(x,y)=\cos(2\pi x)\cos(2\pi y)\). With \(H\), \(\nu\), \(f\), and \(V^*\) fixed, we compute the reference state \((u^*,m^*)\) numerically by running the implicit-Euler discretization of the HRF. We then run GD and GN to reconstruct \((u,m,V)\), and compare against \((u^*,m^*,V^*)\).

\textbf{Experimental Setup.}
 We discretize \(\mathbb{T}^2\simeq[0,1)^2\) on a uniform
\(30\times 30\) periodic grid with spacing \(h_x=h_y=1/30\), and discretize
\([0,1]\) into \(N_T=30\) time intervals with \(h_t=1/30\). Thus the full
space--time grid has \(N_T+1=31\) time levels and contains
\(30^2\times 31=27900\) nodes. In the HRF formulation, the fixed endpoint
slices \(M_0\) and \(U_{N_T}\) are reinstated when forming the full stacked
fields, while each reduced unknown block
\(M=(M_1,\ldots,M_{N_T})\) or \(U=(U_0,\ldots,U_{N_T-1})\) contains
\(30^2\times 30=27000\) free nodal values. We select
\(N_m=2160\) density observation samples
uniformly at random from the reduced density block \(M\), and we use \(180\)
spatial observation locations for \(V\). The regularization parameters are
\(\alpha=0.04\), \(\beta=2\), and \(\gamma=2\). We perturb the observations
with i.i.d. Gaussian noise \(\mathcal{N}(0,\eta^2I)\), with
\(\eta=10^{-3}\). The free unknown blocks are initialized by
\(U_k\equiv0\), \(M_k\equiv1\), and \(V\equiv0\), while the endpoint slices
\(M_0\) and \(U_{N_T}\) are kept fixed.

\textbf{Experimental Results.}
Figure~\ref{fig:timedep_2d_inverse} summarizes the two-dimensional, time-dependent MFG inverse problem. Panels~\ref{fig:mref_t0}--\ref{fig:mref_t1} show the reference density \(m^*\) at \(t=0\), \(t=0.5T\), and \(t=T\). Panels~\ref{fig:uref_t0}--\ref{fig:uref_t1} show the corresponding value function \(u^*\) at the same times. Panel~\ref{fig:Vref} shows the exact spatial cost \(V^*\).

To assess reconstruction quality, Panel~\ref{fig:loss} plots the loss history. Pointwise absolute error contours for \(m\) are reported at \(t=0.5T\) and at \(t=T\) for GD (Panels~\ref{fig:merr_gd_t05}--\ref{fig:merr_gd_t1}) and GN (Panels~\ref{fig:merr_gn_t05}--\ref{fig:merr_gn_t1}). Likewise, pointwise error maps for \(u\) are shown at \(t=0\) and \(t=0.5T\) for GD (Panels~\ref{fig:uerr_gd_t0}--\ref{fig:uerr_gd_t05}) and GN (Panels~\ref{fig:uerr_gn_t0}--\ref{fig:uerr_gn_t05}).

The reconstruction of \(V\) is presented separately in Figure~\ref{fig:V_only_panel}, which compares the recovered fields from GD and GN (Panels~\ref{fig:V_rec_gd} and \ref{fig:V_rec_gn}) together with their pointwise absolute errors (Panels~\ref{fig:V_err_gd} and \ref{fig:V_err_gn}). Overall, the loss history and the error fields across all reported time slices show that GN converges markedly faster than adjoint-based GD, attaining smaller errors in substantially fewer outer iterations.

\begin{figure}[!htbp]
    \centering

    \begin{subfigure}[b]{0.23\textwidth}
        \centering
        \includegraphics[width=\linewidth]{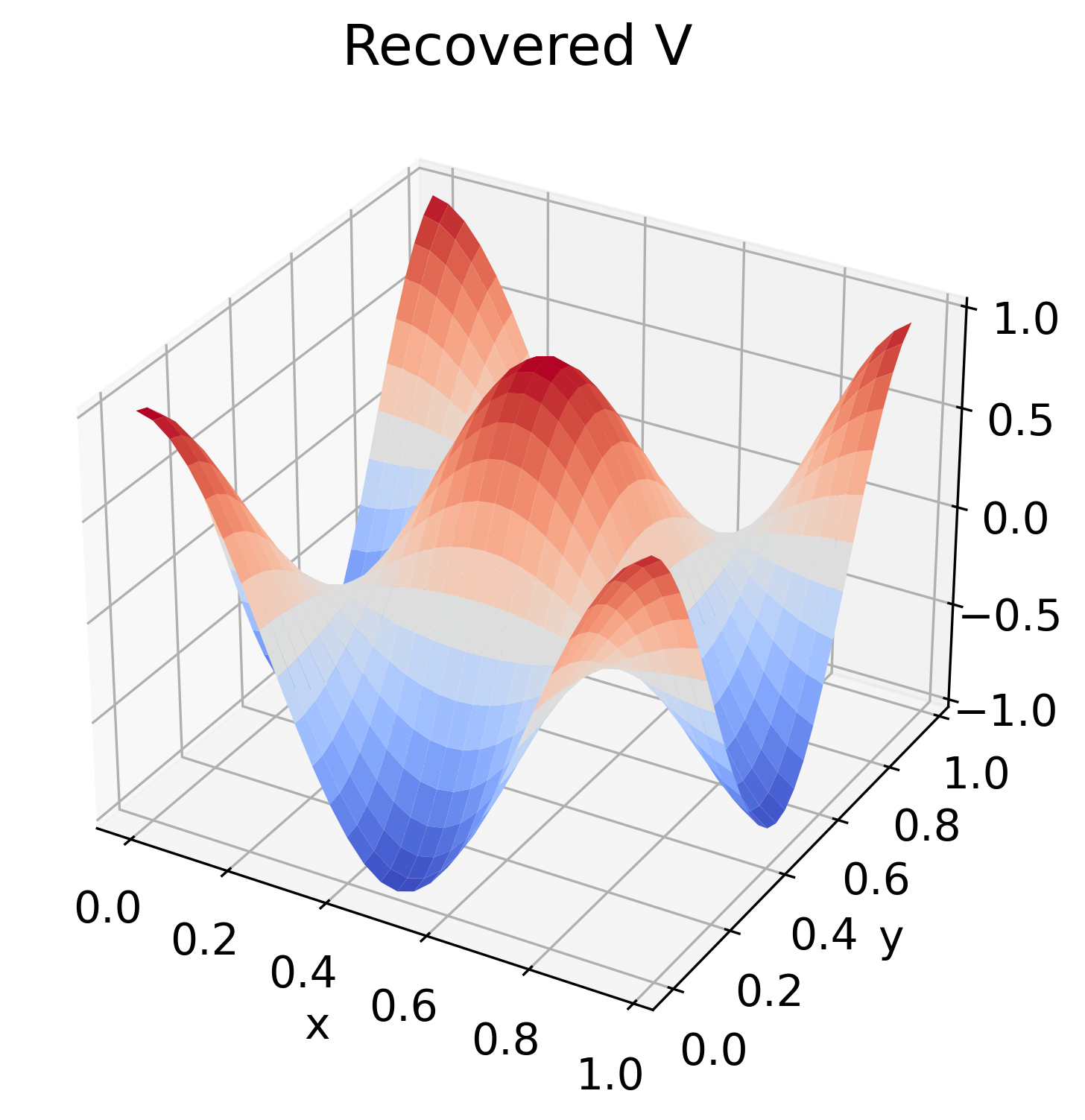}
        \caption{Recovered $V$ via GD}
        \label{fig:V_rec_gd}
    \end{subfigure}\hspace{1mm}
    \begin{subfigure}[b]{0.23\textwidth}
        \centering
        \includegraphics[width=\linewidth]{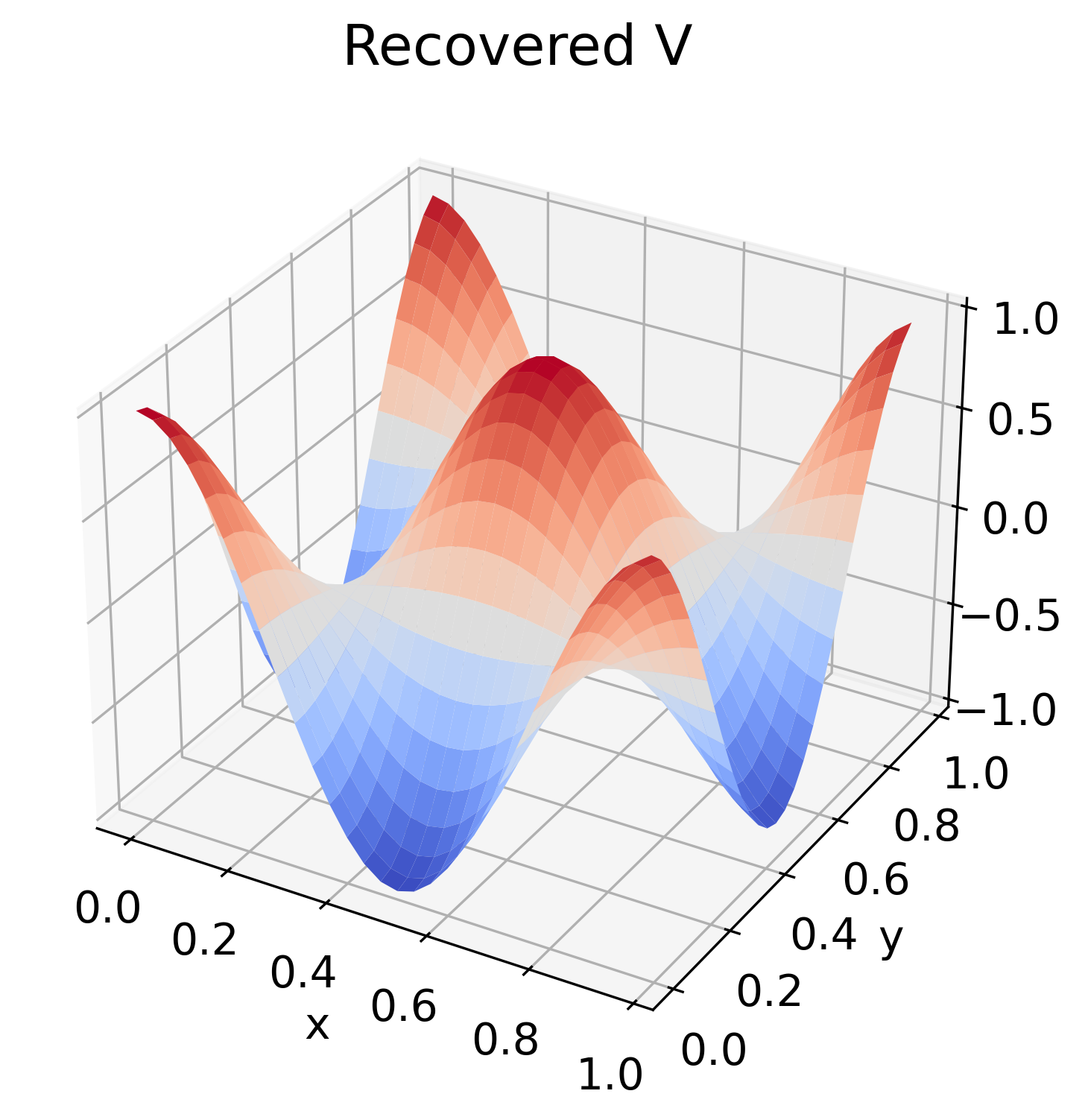}
        \caption{Recovered $V$ via GN}
        \label{fig:V_rec_gn}
    \end{subfigure}\hspace{1mm}
    \begin{subfigure}[b]{0.23\textwidth}
        \centering
        \includegraphics[width=\linewidth]{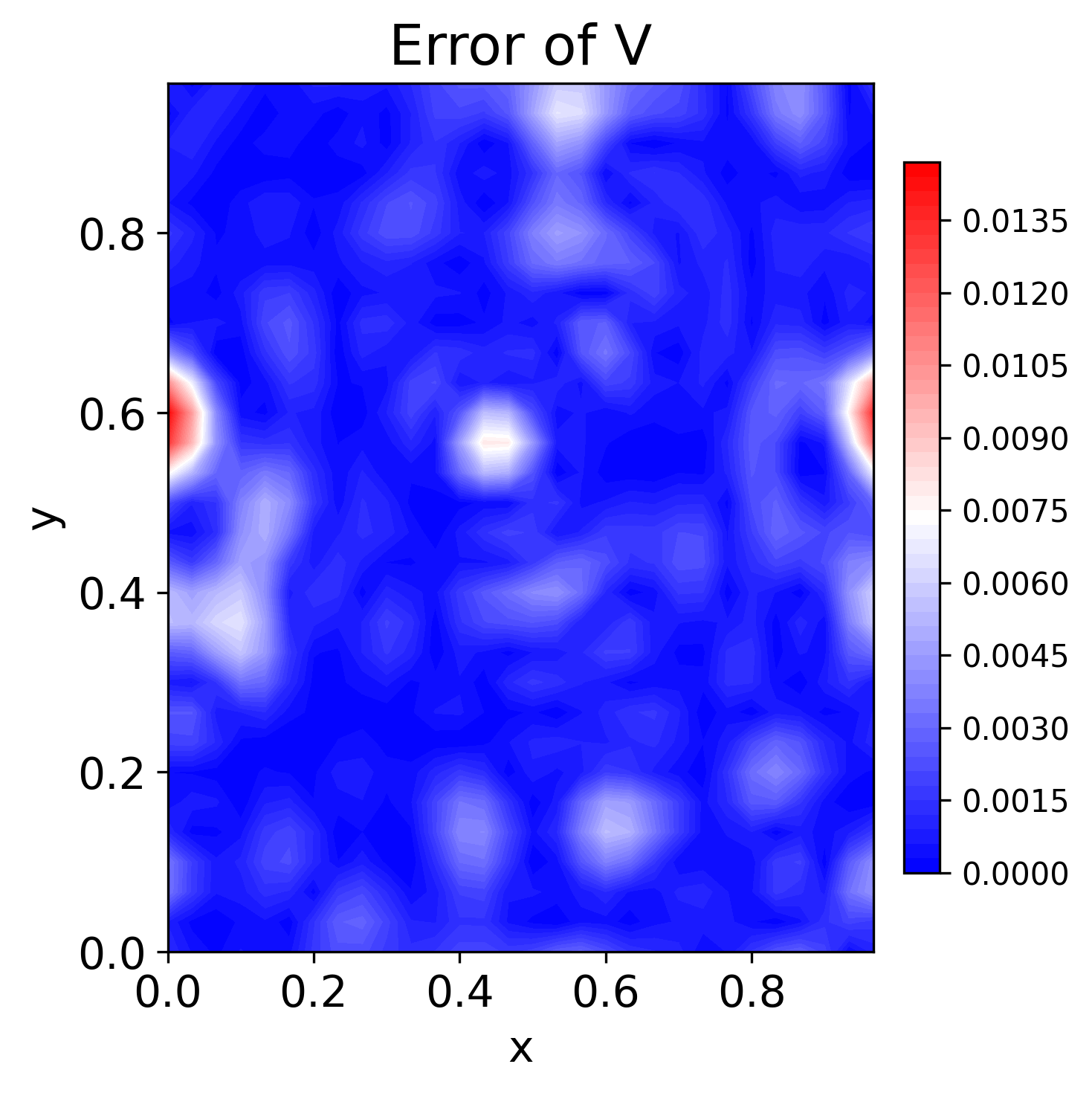}
        \caption{Error of $V$ via GD}
        \label{fig:V_err_gd}
    \end{subfigure}\hspace{1mm}
    \begin{subfigure}[b]{0.23\textwidth}
        \centering
        \includegraphics[width=\linewidth]{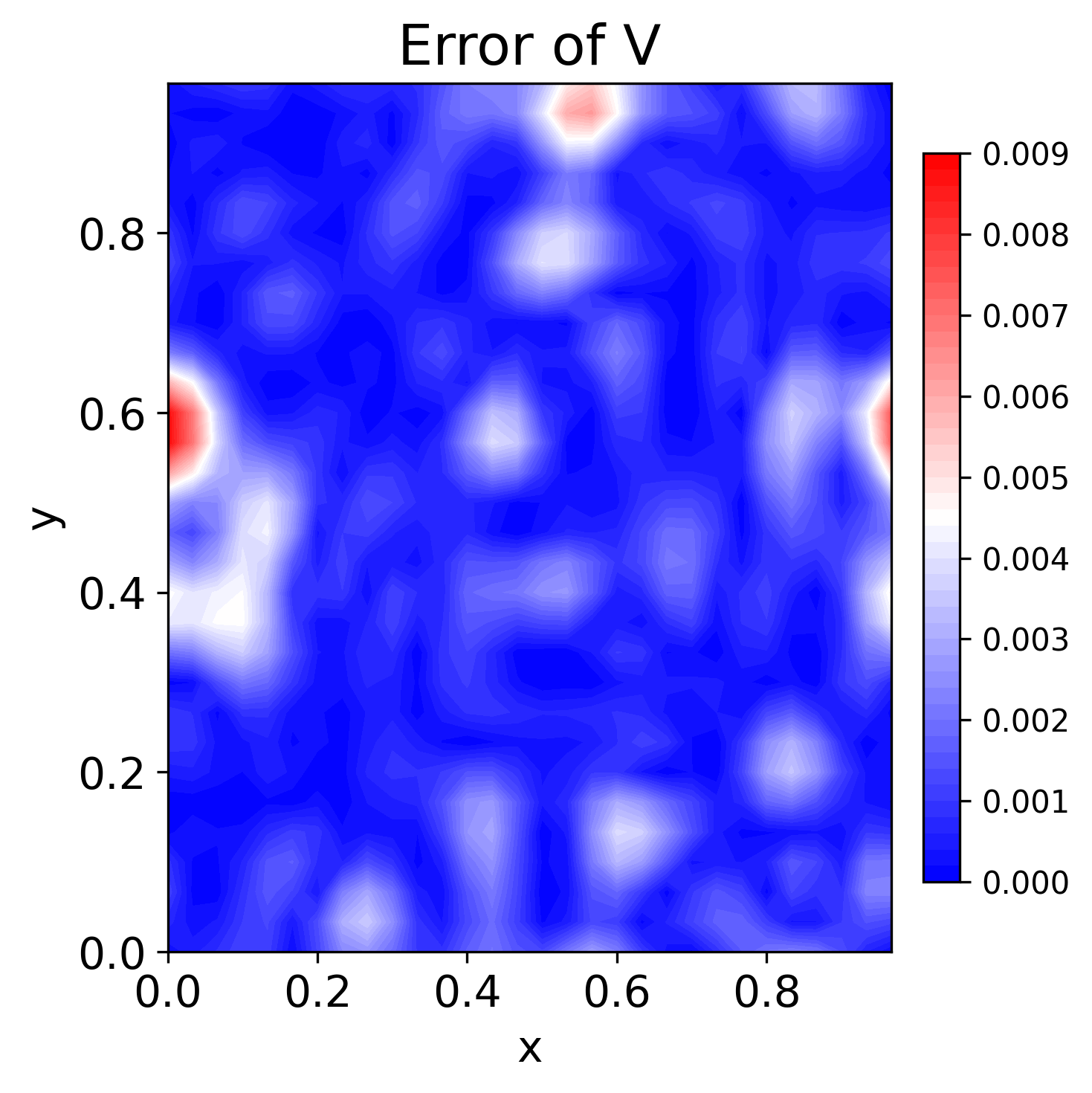}
        \caption{Error of $V$ via GN}
        \label{fig:V_err_gn}
    \end{subfigure}

    \caption{Spatial cost $V$ in 2D time-dependent MFG in \Cref{2Dtimedependentexample1}. (a) recovered $V$ via GD; (b) recovered $V$ via GN; (c) error of $V$ via GD; (d) error of $V$ via GN.}
    \label{fig:V_only_panel}
\end{figure}

\section{Conclusion}
\label{sec:conclusion}

In this paper, we developed a positivity-preserving flow method for
time-dependent MFGs and a solver-agnostic framework for MFG inverse problems.

First, we introduced a fully discrete HRF method for time-dependent MFG systems.
After discretizing the MFG system in space and time, we keep the boundary time
slices fixed in order to preserve the initial density and the terminal value
function. The flow then evolves only the remaining unknowns. The density
variables evolve on the positive probability simplex at each time slice, so
positivity and mass conservation are preserved by construction. Under the  monotonicity and convexity assumptions stated in
Section~\ref{sec:tdmfg}, the resulting flow is globally defined, preserves
positivity and mass, and converges to a state which solves the projected discretized MFG system. From this limit, one can recover a solution of the full discretized
MFG system by adding suitable spatial constants to the value-function slices, as
described in \Cref{rem:td-postprocess-full}. 

Second, we proposed a solver-agnostic framework for inverse problems in
stationary and time-dependent MFGs. The inverse problem is written as an outer
optimization problem over the unknown parameter, while the inner problem solves
the discretized MFG system for the current parameter value. The key point is
that the outer derivative is computed from the discrete equations satisfied by
the converged inner state, not by differentiating through the iterations of a
particular inner solver. This leads to an adjoint-based gradient formula and a GN acceleration. 

The numerical experiments show that the proposed framework can be used with
different inner solvers and with several stationary and time-dependent MFG
models. In particular, the GN method consistently reduces the number
of outer iterations compared with gradient descent. 

Several extensions are natural. The discrete HRF construction is not tied to the
particular uniform periodic grid used in this paper, and it would be useful to
develop the analysis for non-periodic boundary conditions, nonuniform meshes,
and broader classes of Hamiltonians and couplings. The same inverse framework
can also accommodate richer parameterizations of the unknown model components.
For example, beyond recovering the spatial cost, one may infer the Hamiltonian
or the coupling function from observations of equilibrium states or agent
trajectories.

\section*{Acknowledgments}
HY acknowledges support in part from Hong Kong ECS Grant 21302521, Hong Kong RGC Grant 11311422, and Hong Kong RGC Grant 11303223. XY acknowledges support from the Air Force Office of Scientific Research through the MURI award FA9550-20-1-0358 (Machine Learning and Physics-Based Modeling and Simulation). XY also acknowledges support from the Air Force Office of Scientific Research under MURI award FOA-AFRL-AFOSR-2023-0004 (Mathematics of Digital Twins), the Department of Energy under award DE-SC0023163 (SEA-CROGS: Scalable, Efficient, and Accelerated Causal Reasoning Operators, Graphs and Spikes for Earth and Embedded Systems), the National Science Foundation under award 2425909 (Discovering the Law of Stress Transfer and Earthquake Dynamics in a Fault Network using a Computational Graph Discovery Approach), and the VBFF under ONR-N000142512035. JZ acknowledges support from the NUS Risk Management Institute research scholarship and the IoTeX Foundation Industry Grant A-8001180-00-00. We wish to express our sincere gratitude to Professor Mou Chenchen from the City University of Hong Kong and Professor Zhou Chao from the National University of Singapore for their fruitful discussions.

\appendix
\crefalias{section}{appendix}
\Crefname{appendix}{Appendix}{Appendices}
\crefname{appendix}{appendix}{appendices}
\section{\texorpdfstring{Proofs of the Main Results}{Proofs of the Main Results}}
\label{app:appendix}
We first recall the classical result from \cite{achdou2010mean} that the discrete transport operator in the FP equation is the adjoint of the linearization of the discrete Hamiltonian.
\begin{lemma}[Discrete adjoint structure of the transport operator]
\label{lem:Bh_adjoint}
For any grid functions $U$, $M$, and $V$, one has
\begin{equation}
\label{eq:Bh_adjoint_identity}
\big\langle \mathcal{B}_h(U,M),\,V\big\rangle_h
=
\big\langle M\,\nabla_q g(\cdot,[D_hU]),\,[D_hV]\big\rangle_h, 
\end{equation}
where \(M\,\nabla_q g(\cdot,[D_hU])\) denotes componentwise multiplication. The right-hand side is understood with the vector-valued pairing introduced above.
Equivalently, for fixed \(U\), define the linearized Hamiltonian operator $L_U:\mathbb R^{N_x}\to\mathbb R^{N_x}$ by
\[
(L_UV)_{i,j}:=\nabla_q g_{i,j}([D_hU]_{i,j})\cdot [D_hV]_{i,j}.
\]
Then the map $M\mapsto \mathcal B_h(U,M)$ is the adjoint of $L_U$ with respect to \(\langle\cdot,\cdot\rangle_h\).
\end{lemma}
\begin{proof}
Fix $U$, $M$, and $V$. Write
\[
Q_{i,j}:=[D_hU]_{i,j}\in\mathbb{R}^4,
\qquad
\nabla_q g_{i,j}(Q_{i,j})
=
\Bigl(
\frac{\partial g_{i,j}}{\partial q_1}(Q_{i,j}),
\frac{\partial g_{i,j}}{\partial q_2}(Q_{i,j}),
\frac{\partial g_{i,j}}{\partial q_3}(Q_{i,j}),
\frac{\partial g_{i,j}}{\partial q_4}(Q_{i,j})
\Bigr).
\]
By the definition of $\mathcal B_h$ in \eqref{eq:Bh_def},
\begin{align*}
\big\langle \mathcal B_h(U,M),V\big\rangle_h
&=
\sum_{i,j} h^2\,\mathcal B_h(U,M)_{i,j}\,V_{i,j}.
\end{align*}
For arbitrary scalar grid functions $A=(A_{i,j})$ and $B=(B_{i,j})$ on the periodic grid, the one-sided difference operators satisfy the discrete summation-by-parts identities
\[
\langle -D_1^- A,B\rangle_h=\langle A,D_1^+B\rangle_h,
\qquad
\langle -D_1^+ A,B\rangle_h=\langle A,D_1^-B\rangle_h.
\]
The completely analogous identities with $D_2^\pm$ in place of $D_1^\pm$ are used for the last two flux terms.
Substituting the divergence form of $\mathcal B_h(U,M)$ and applying these identities yields
\begin{align*}
\big\langle \mathcal B_h(U,M),V\big\rangle_h
&=
\sum_{i,j} h^2\,M_{i,j}
\Bigl(
\frac{\partial g_{i,j}}{\partial q_1}(Q_{i,j})\,(D_1^+V)_{i,j}
+
\frac{\partial g_{i,j}}{\partial q_2}(Q_{i,j})\,(D_1^-V)_{i,j}
\\
&\hspace{7em}
+
\frac{\partial g_{i,j}}{\partial q_3}(Q_{i,j})\,(D_2^+V)_{i,j}
+
\frac{\partial g_{i,j}}{\partial q_4}(Q_{i,j})\,(D_2^-V)_{i,j}
\Bigr).
\end{align*}
By the definition of the stencil
\[
[D_hV]_{i,j}
=
\bigl((D_1^+V)_{i,j},(D_1^-V)_{i,j},(D_2^+V)_{i,j},(D_2^-V)_{i,j}\bigr),
\]
the right-hand side is exactly
\[
\sum_{i,j} h^2\,M_{i,j}\,
\nabla_q g_{i,j}(Q_{i,j})\cdot [D_hV]_{i,j},
\]
which is
\[
\big\langle M\,\nabla_q g(\cdot,[D_hU]),\,[D_hV]\big\rangle_h.
\]
This proves \eqref{eq:Bh_adjoint_identity}.

To state the adjoint relation without mixing the roles of $U$ and $M$, fix $U$ and define the linearization
\[
(L_UV)_{i,j}:=\nabla_q g_{i,j}([D_hU]_{i,j})\cdot [D_hV]_{i,j},
\qquad V\in\mathbb R^{N_x}.
\]
Then \eqref{eq:Bh_adjoint_identity} reads
\[
\langle \mathcal B_h(U,M),V\rangle_h=\langle M,L_UV\rangle_h
\qquad\text{for all }M,V\in\mathbb R^{N_x},
\]
which is exactly the statement that, for fixed $U$, the map $M\mapsto \mathcal B_h(U,M)$ is the adjoint of the linearized Hamiltonian operator $L_U$.
\end{proof}

We now verify that the rigidity condition in Assumption~\ref{ass:disc_g} is
satisfied by the Godunov-type fluxes used in the paper. The point is simple but
slightly subtle. A Godunov flux is not necessarily strictly convex in the full
variable \(q=(q_1,q_2,q_3,q_4)\), so \(D_g(x;q,\widetilde q)=0\) does not
immediately imply \(q=\widetilde q\). However, the flux depends on \(q\) only
through the upwind parts $\rho(q)=(q_1^-,q_2^+,q_3^-,q_4^+)$,
and it is strictly convex in these upwind variables. Therefore, the vanishing of
the Bregman term first gives equality of the upwind parts. On a periodic grid,
the missing positive or negative parts are recovered from neighboring nodes
because backward differences are shifted forward differences. This is enough to
recover all one-sided differences.
\begingroup
\begin{lemma}[Godunov-type fluxes satisfy the rigidity condition]
\label{lem:godunov_rigidity}
Let \(g\) be a numerical Hamiltonian of the form
\[
g(x,q)=\widehat g(x,q_1^-,q_2^+,q_3^-,q_4^+),
\]
where \(a^+:=\max\{a,0\}\) and \(a^-:=\max\{-a,0\}\). Assume that, for each
fixed \(x\), the map \(q\mapsto g(x,q)\) is convex and \(C^1\) on
\(\mathbb R^4\). Assume further that, for each fixed \(x\), the map
\(r\mapsto \widehat g(x,r)\) is \(C^1\), nondecreasing in each coordinate, and
strictly convex on \(\mathbb R_+^4\). Then \(g\) satisfies the rigidity
condition in Assumption~\ref{ass:disc_g}: for any periodic grid functions
\(W,\widetilde W\), if
\[
D_g\bigl(x_{ij};[D_hW]_{ij},[D_h\widetilde W]_{ij}\bigr)=0
\qquad
\text{for every node }(i,j),
\]
then
\[
[D_hW]=[D_h\widetilde W].
\]
\end{lemma}

\begin{proof}
We divide the proof into three steps.

\smallskip
\noindent\textbf{Step 1: A pointwise consequence of \(D_g(x;q,\widetilde q)=0\).}
Fix \(x\in\mathbb T^2\) and \(q,\widetilde q\in\mathbb R^4\). We first show that
\(D_g(x;q,\widetilde q)=0\) implies
\[
(q_1^-,q_2^+,q_3^-,q_4^+)
=
(\widetilde q_1^-,\widetilde q_2^+,\widetilde q_3^-,
 \widetilde q_4^+).
\]
Set
\[
\omega_\ell
:=
\partial_{r_\ell}\widehat g
\bigl(x,\widetilde q_1^-,\widetilde q_2^+,
        \widetilde q_3^-,\widetilde q_4^+\bigr),
\qquad \ell=1,\ldots,4.
\]
Since \(\widehat g(x,\cdot)\) is nondecreasing in each coordinate, we have
\(\omega_\ell\ge0\). Since \(a^-=\max\{-a,0\}\), we have
\((a^-)'=-\mathbf 1_{\{a<0\}}\) away from \(a=0\), while
\((a^+)'=\mathbf 1_{\{a>0\}}\). Therefore, using
\(g(x,q)=\widehat g(x,q_1^-,q_2^+,q_3^-,q_4^+)\), we obtain
\[
\nabla_q g(x,\widetilde q)
=
\bigl(
-\omega_1\mathbf 1_{\{\widetilde q_1<0\}},
\ \omega_2\mathbf 1_{\{\widetilde q_2>0\}},
\ -\omega_3\mathbf 1_{\{\widetilde q_3<0\}},
\ \omega_4\mathbf 1_{\{\widetilde q_4>0\}}
\bigr).
\]
At points where some \(\widetilde q_\ell=0\), the corresponding one-sided
factor \(q_\ell^+\) or \(q_\ell^-\) is not differentiable as a function of
\(q_\ell\). However, \(g(x,\cdot)\) is assumed to be \(C^1\). Therefore, the
above formula is understood by continuity of \(\nabla_q g\). 

By definition,
\begin{align}
D_g(x;q,\widetilde q)
&=
\widehat g(x,q_1^-,q_2^+,q_3^-,q_4^+)
-
\widehat g(x,\widetilde q_1^-,\widetilde q_2^+,
             \widetilde q_3^-,\widetilde q_4^+) 
-\nabla_q g(x,\widetilde q)\cdot(q-\widetilde q).
\label{eq:godunov_Dg_start}
\end{align}
Substituting the formula for \(\nabla_q g(x,\widetilde q)\), we get
\begin{align}
-\nabla_q g(x,\widetilde q)\cdot(q-\widetilde q)
&=
\omega_1\mathbf 1_{\{\widetilde q_1<0\}}(q_1-\widetilde q_1)
-\omega_2\mathbf 1_{\{\widetilde q_2>0\}}(q_2-\widetilde q_2) \notag\\
&\quad
+\omega_3\mathbf 1_{\{\widetilde q_3<0\}}(q_3-\widetilde q_3)
-\omega_4\mathbf 1_{\{\widetilde q_4>0\}}(q_4-\widetilde q_4).
\label{eq:godunov_full_linear_term}
\end{align}

We now use two elementary inequalities. For any \(a,b\in\mathbb R\),
\[
a^- - b^-+\mathbf 1_{\{b<0\}}(a-b)\ge0 
\]
and 
\[
a^+ - b^+-\mathbf 1_{\{b>0\}}(a-b)\ge0.
\]

Applying these two inequalities to the four terms in
\eqref{eq:godunov_full_linear_term}, and using \(\omega_\ell\ge0\), gives
\begin{align}
\omega_1\mathbf 1_{\{\widetilde q_1<0\}}(q_1-\widetilde q_1)
&\ge
-\omega_1(q_1^- - \widetilde q_1^-), \label{eq:godunov_comp1}\\
-\omega_2\mathbf 1_{\{\widetilde q_2>0\}}(q_2-\widetilde q_2)
&\ge
-\omega_2(q_2^+ - \widetilde q_2^+), \label{eq:godunov_comp2}\\
\omega_3\mathbf 1_{\{\widetilde q_3<0\}}(q_3-\widetilde q_3)
&\ge
-\omega_3(q_3^- - \widetilde q_3^-), \label{eq:godunov_comp3}\\
-\omega_4\mathbf 1_{\{\widetilde q_4>0\}}(q_4-\widetilde q_4)
&\ge
-\omega_4(q_4^+ - \widetilde q_4^+). \label{eq:godunov_comp4}
\end{align}
Combining \eqref{eq:godunov_Dg_start}--\eqref{eq:godunov_comp4}, we obtain
\begin{align}
D_g(x;q,\widetilde q)
&\ge
\widehat g(x,q_1^-,q_2^+,q_3^-,q_4^+)
-
\widehat g(x,\widetilde q_1^-,\widetilde q_2^+,
             \widetilde q_3^-,\widetilde q_4^+) \notag\\
&\quad
-\omega_1(q_1^- - \widetilde q_1^-)
-\omega_2(q_2^+ - \widetilde q_2^+)
-\omega_3(q_3^- - \widetilde q_3^-)
-\omega_4(q_4^+ - \widetilde q_4^+).
\label{eq:godunov_lower_bound}
\end{align}
The right-hand side of \eqref{eq:godunov_lower_bound} is the Bregman divergence
of the strictly convex function \(\widehat g(x,\cdot)\) between
\((q_1^-,q_2^+,q_3^-,q_4^+)\) and
\((\widetilde q_1^-,\widetilde q_2^+,\widetilde q_3^-,
\widetilde q_4^+)\). Hence, it is nonnegative and vanishes only when
\[
(q_1^-,q_2^+,q_3^-,q_4^+)
=
(\widetilde q_1^-,\widetilde q_2^+,\widetilde q_3^-,
 \widetilde q_4^+).
\]
Therefore, if \(D_g(x;q,\widetilde q)=0\), then the right-hand side of
\eqref{eq:godunov_lower_bound} must also be zero, and the desired equality
follows.

\smallskip
\noindent\textbf{Step 2: Apply the pointwise result at every grid node.}
Assume that
\[
D_g\bigl(x_{ij};[D_hW]_{ij},[D_h\widetilde W]_{ij}\bigr)=0
\qquad\text{for every node }(i,j).
\]
At each node, take \(q=[D_hW]_{ij}\) and
\(\widetilde q=[D_h\widetilde W]_{ij}\). By Step~1, for every \((i,j)\),
\[
(D_1^+W)^-_{ij}=(D_1^+\widetilde W)^-_{ij},
\qquad
(D_1^-W)^+_{ij}=(D_1^-\widetilde W)^+_{ij},
\]
and
\[
(D_2^+W)^-_{ij}=(D_2^+\widetilde W)^-_{ij},
\qquad
(D_2^-W)^+_{ij}=(D_2^-\widetilde W)^+_{ij}.
\]
Our goal is to prove that all four one-sided differences agree:
\[
D_1^+W=D_1^+\widetilde W,\qquad
D_1^-W=D_1^-\widetilde W,\qquad
D_2^+W=D_2^+\widetilde W,\qquad
D_2^-W=D_2^-\widetilde W.
\]
The identities above already give the negative part of \(D_1^+\) and \(D_2^+\),
and the positive part of \(D_1^-\) and \(D_2^-\). It remains to use periodicity
to obtain the positive parts of the forward differences. Once the forward
differences agree, the backward differences follow by shifting indices.

\smallskip
\noindent\textbf{Step 3: Use periodicity to recover the full stencil.}
We first prove \(D_1^+W=D_1^+\widetilde W\). From Step~2, we already know
\[
(D_1^+W)^-_{ij}=(D_1^+\widetilde W)^-_{ij}.
\]
It remains to prove the equality of the positive parts. By periodicity,
\[
(D_1^-W)_{i+1,j}=(D_1^+W)_{ij}.
\]
Therefore,
\[
(D_1^+W)^+_{ij}
=
(D_1^-W)^+_{i+1,j}
=
(D_1^-\widetilde W)^+_{i+1,j}
=
(D_1^+\widetilde W)^+_{ij}.
\]
Thus, both the positive and negative parts of \(D_1^+W\) and
\(D_1^+\widetilde W\) agree, so
\[
D_1^+W=D_1^+\widetilde W.
\]

The same argument in the second coordinate gives
\[
D_2^+W=D_2^+\widetilde W.
\]
Finally, on the periodic grid,
\[
(D_1^-W)_{ij}=(D_1^+W)_{i-1,j},
\qquad
(D_2^-W)_{ij}=(D_2^+W)_{i,j-1}.
\]
Since the forward differences agree, these identities imply
\[
D_1^-W=D_1^-\widetilde W,\qquad
D_2^-W=D_2^-\widetilde W.
\]
Hence, all four one-sided differences agree, namely
\[
[D_hW]=[D_h\widetilde W].
\]
\end{proof}
\endgroup

We next prove the monotonicity of the operator \(F_h^{\mathrm{td}}\) defined in \eqref{eq:Fh_full_def}, as formalized in \Cref{prop:Fh_monotone_MU}. 
\begin{proof}[Proof of \Cref{prop:Fh_monotone_MU}]
Write
\[
F_h^{\mathrm{td}}(Y)=(r_1(Y),\dots,r_{N_T}(Y),s_1(Y),\dots,s_{N_T}(Y)),
\]
where $r_k,s_k$ are given by \eqref{eq:rk_def_rewrite}--\eqref{eq:sk_def_rewrite}. By the definition of the space-time pairing,
\[
\big\langle F_h^{\mathrm{td}}(Y)-F_h^{\mathrm{td}}(\widetilde Y),\,Y-\widetilde Y\big\rangle_{h,\Delta t}
=
\Delta t\sum_{k=1}^{N_T}\langle r_k(Y)-r_k(\widetilde Y),\,M_k-\widetilde M_k\rangle_h
+
\Delta t\sum_{k=1}^{N_T}\langle s_k(Y)-s_k(\widetilde Y),\,U_{k-1}-\widetilde U_{k-1}\rangle_h .
\]

\medskip
\noindent\textbf{Step 1: Time-derivative terms.}
Collecting the two discrete time-difference contributions and applying a telescoping sum yields the boundary term
\[
\langle U_{N_T}-\widetilde U_{N_T},\, M_{N_T}-\widetilde M_{N_T}\rangle_h
-\langle U_0-\widetilde U_0,\, M_0-\widetilde M_0\rangle_h .
\]
Since $M_0=\widetilde M_0$ and $U_{N_T}=\widetilde U_{N_T}$ are fixed endpoint samples, this term equals $0$.

\medskip
\noindent\textbf{Step 2: Diffusion terms.}
On the periodic grid $\mathbb{T}^2_h$, the five-point Laplacian is self-adjoint with respect to $\langle\cdot,\cdot\rangle_h$, hence
\[
\langle \Delta_h A, B\rangle_h=\langle A,\Delta_h B\rangle_h .
\]
Therefore the diffusion contributions coming from $r_k$ and $s_k$ cancel:
\[
\sum_{k=1}^{N_T}\nu\langle \Delta_h(U_{k-1}-\widetilde U_{k-1}),\,M_k-\widetilde M_k\rangle_h
-
\sum_{k=1}^{N_T}\nu\langle \Delta_h(M_k-\widetilde M_k),\,U_{k-1}-\widetilde U_{k-1}\rangle_h
=0 .
\]

\medskip
\noindent\textbf{Step 3: Coupling terms.}
The coupling contribution equals
\[
\Delta t\sum_{k=1}^{N_T}\langle f(M_k)-f(\widetilde M_k),\,M_k-\widetilde M_k\rangle_h.
\]
By \Cref{ass:disc_f}, each summand is nonnegative, and it is strictly positive whenever $M_k\neq \widetilde M_k$. Hence
\begin{equation}
\label{eq:strict_coupling_term}
\Delta t\sum_{k=1}^{N_T}\langle f(M_k)-f(\widetilde M_k),\,M_k-\widetilde M_k\rangle_h \ge 0,
\end{equation}
with equality only if $M_k=\widetilde M_k$ for all $k$.

\medskip
\noindent\textbf{Step 4: Hamiltonian and transport terms.}
Fix $k\in\{1,\dots,N_T\}$ and set
\[
\delta U:=U_{k-1}-\widetilde U_{k-1},
\qquad
\delta M:=M_k-\widetilde M_k.
\]
Introduce the discrete Hamiltonian functional
\[
\mathcal G_h(U):=g(\cdot,[D_hU]).
\]
Using \Cref{lem:Bh_adjoint}, we obtain
\begin{equation}
\label{eq:Bh_adjoint_linearization}
\big\langle \mathcal B_h(U,M),\,V\big\rangle_h
=
\big\langle M\,\nabla_q g(\cdot,[D_hU]),\,[D_hV]\big\rangle_h
\qquad\text{for all }V.
\end{equation}
Equivalently, for fixed $M$, the map
\[
U\longmapsto \Phi_M(U):=\big\langle M,\mathcal G_h(U)\big\rangle_h
\]
has Fr\'echet derivative
\[
D\Phi_M(U)[V]
=
\big\langle \mathcal B_h(U,M),\,V\big\rangle_h.
\]

Since each $q\mapsto g_{i,j}(q)$ is convex by \Cref{ass:disc_g}, the functional $\Phi_M$ is convex for every nonnegative density $M$. Hence
\[
\Phi_{\widetilde M_k}(U_{k-1})-\Phi_{\widetilde M_k}(\widetilde U_{k-1})
-
D\Phi_{\widetilde M_k}(\widetilde U_{k-1})[\delta U]
\ge 0,
\]
and
\[
\Phi_{M_k}(\widetilde U_{k-1})-\Phi_{ M_k}(U_{k-1})
-
D\Phi_{ M_k}( U_{k-1})[-\delta U]
\ge 0.
\]
Expanding these two inequalities gives
\begin{align}
\label{eq:mtildebregg}
\begin{split}
&\big\langle \widetilde M_k,\,
g(\cdot,[D_hU_{k-1}])-g(\cdot,[D_h\widetilde U_{k-1}])\big\rangle_h
-
\big\langle \mathcal B_h(\widetilde U_{k-1}, \widetilde M_k),\,\delta U\big\rangle_h
\ge 0,
\\
&-\big\langle  M_k,\,
g(\cdot,[D_hU_{k-1}])-g(\cdot,[D_h\widetilde U_{k-1}])\big\rangle_h
+
\big\langle \mathcal B_h( U_{k-1}, M_k),\,\delta U\big\rangle_h
\ge 0.
\end{split}
\end{align}
Adding them yields
\begin{equation}
\label{eq:step4_main_ineq}
\big\langle g(\cdot,[D_h\widetilde U_{k-1}]) - g(\cdot,[D_hU_{k-1}]),\,\delta M\big\rangle_h
+
\big\langle \mathcal B_h(U_{k-1},M_k)-\mathcal B_h(\widetilde U_{k-1},\widetilde M_k),\,\delta U\big\rangle_h
\ge 0.
\end{equation}
Next, we show that when $Y\not=\widetilde{Y}$, the inequality is strict. Since $\widetilde M_k$ is strictly positive at every node, the first inequality of \eqref{eq:mtildebregg} is the weighted Bregman sum $\langle\widetilde M_k,\,D_g(\cdot\,;[D_hU_{k-1}],[D_h\widetilde U_{k-1}])\rangle_h$, which vanishes only if $D_g(x_{ij};[D_hU_{k-1}]_{ij},[D_h\widetilde U_{k-1}]_{ij})=0$ at every node. By the rigidity in \Cref{ass:disc_g}, this forces $[D_hU_{k-1}]=[D_h\widetilde U_{k-1}]$. Hence if $[D_hU_{k-1}]\neq[D_h\widetilde U_{k-1}]$,  \eqref{eq:step4_main_ineq} is strict.

\medskip
\noindent\textbf{Step 5: Conclusion.}
Adding Steps~1--4 yields
\[
\big\langle F_h^{\mathrm{td}}(Y)-F_h^{\mathrm{td}}(\widetilde Y),\,Y-\widetilde Y\big\rangle_{h,\Delta t}\ge 0.
\]
Moreover, if $Y\neq \widetilde Y$, then either $M_k\neq \widetilde M_k$ for some $k$, in which case the coupling term in Step~3 is strictly positive, or $M_k=\widetilde M_k$ for all $k$ and $U_{k-1}\neq \widetilde U_{k-1}$ for some $k$. In the latter case, the assumption
\(\langle U_{k-1}-\widetilde U_{k-1},\mathbf 1\rangle_h=0\) rules out the
possibility that \(U_{k-1}-\widetilde U_{k-1}\) is a nonzero constant grid
function. Hence, there exists some $k$ for which
\[
[D_hU_{k-1}] \neq [D_h\widetilde U_{k-1}],
\]Indeed, if $[D_hU_{k-1}]=[D_h\widetilde U_{k-1}]$, then every one-sided difference of $U_{k-1}-\widetilde U_{k-1}$ vanishes. This forces $U_{k-1}-\widetilde U_{k-1}$ to be constant. 
Thus, according to the arguments of Step~4, we get 
\[
\big\langle F_h^{\mathrm{td}}(Y)-F_h^{\mathrm{td}}(\widetilde Y),\,Y-\widetilde Y\big\rangle_{h,\Delta t}>0
\qquad\text{for all }Y\neq \widetilde Y,
\]
which proves \eqref{eq:Fh_strictly_monotone_MU}.
\end{proof}

Next, we establish the explicit formulas for the HRF method as described in \Cref{prop:explicit_disc_HRF}.
\begin{proof}[Proof of \Cref{prop:explicit_disc_HRF}]
Fix a feasible state $Y$ with $M_k\in\mathcal{M}_{++,h}$ for $k=1,\dots,N_T$ and consider the minimization problem \eqref{eq:disc_HRF_var}.
For clarification, we set
\[
\mathcal H_h(Y):=\nabla^2\mathcal E(Y)=\operatorname{diag}\bigl(1/M_1,\dots,1/M_{N_T},\,I,\dots,I\bigr).
\]
Since $M_k>0$ componentwise, the operator $\mathcal{H}_h(Y)$ is positive definite on each block, hence the quadratic map
$\Xi\mapsto \frac12\,g_{Y,\mathcal{E}}(\Xi,\Xi)$ is strictly convex on $\mathcal{X}_h$.
Therefore, the objective in \eqref{eq:disc_HRF_var} is strictly convex, and minimizing it over the linear subspace $\mathcal{T}_Y\mathcal{M}_h$
admits a unique minimizer.

Introduce Lagrange multipliers $\lambda_k\in\mathbb{R}$ for the constraints $\langle \Xi_{M_k},\mathbf{1}\rangle_h=0$ and define the Lagrangian
\[
\mathcal{L}(\Xi,\lambda)
:=
\frac12\,g_{Y,\mathcal{E}}(\Xi,\Xi)
+\langle F_h^{\mathrm{td}}(Y),\Xi\rangle_{h,\Delta t}
+\Delta t\sum_{k=1}^{N_T}\lambda_k\,\langle \Xi_{M_k},\mathbf{1}\rangle_h .
\]
We recall that $g_{Y,\mathcal{E}}(\Xi,\Psi)=\langle \mathcal{H}_h(Y)\Xi,\Psi\rangle_{h,\Delta t}$. For the unconstrained $U$-blocks, fix $k\in\{1,\dots,N_T\}$ and take an arbitrary variation $\delta\Xi_{U_{k-1}}\in\mathbb{R}^{N_x}$, keeping all other blocks fixed. Then
\[
0
=
\delta_{\Xi_{U_{k-1}}}\mathcal{L}(\Xi,\lambda)\,[\delta\Xi_{U_{k-1}}]
=
\Delta t\,\big\langle \Xi_{U_{k-1}}+s_k(Y),\,\delta\Xi_{U_{k-1}}\big\rangle_h .
\]
Since $\delta\Xi_{U_{k-1}}$ is arbitrary and $\Delta t>0$, we conclude that
\begin{equation}\label{eq:stat_U_block}
\Xi_{U_{k-1}}=-\,s_k(Y).
\end{equation}

For the constrained density blocks, fix $k\in\{1,\dots,N_T\}$ and take an arbitrary variation $\delta\Xi_{M_k}\in\mathbb{R}^{N_x}$.
Using that the $M_k$-block of $\mathcal{H}_h(Y)$ equals $\mathrm{diag}(1/M_k)$, we obtain
\[
0
=
\delta_{\Xi_{M_k}}\mathcal{L}(\Xi,\lambda)\,[\delta\Xi_{M_k}]
=
\Delta t\,\Big\langle \frac{\Xi_{M_k}}{M_k}+r_k(Y)+\lambda_k\mathbf{1},\,\delta\Xi_{M_k}\Big\rangle_h .
\]
Again $\delta\Xi_{M_k}$ is arbitrary, hence
\begin{equation}\label{eq:stat_M_block}
\frac{\Xi_{M_k}}{M_k}+r_k(Y)+\lambda_k\mathbf{1}=0,
\qquad\text{or equivalently}\qquad
\Xi_{M_k}=-\,M_k\odot\bigl(r_k(Y)+\lambda_k\mathbf{1}\bigr).
\end{equation}
Imposing the constraint $\langle \Xi_{M_k},\mathbf{1}\rangle_h=0$ and using $\langle M_k\odot a,\mathbf{1}\rangle_h=\langle M_k,a\rangle_h$ gives
\[
0=\langle \Xi_{M_k},\mathbf{1}\rangle_h
=-\langle M_k,r_k(Y)\rangle_h-\lambda_k\langle M_k,\mathbf{1}\rangle_h,
\]
so
\[
\lambda_k=-\frac{\langle M_k,r_k(Y)\rangle_h}{\langle M_k,\mathbf{1}\rangle_h}
=-\,\bar r_k(Y).
\]
Substituting this into \eqref{eq:stat_M_block} yields
\[
\Xi_{M_k}=-\,M_k\odot\bigl(r_k(Y)-\bar r_k(Y)\mathbf{1}\bigr),
\qquad k=1,\dots,N_T.
\]

Finally, identifying the minimizer $\Xi$ in \eqref{eq:disc_HRF_var} with $\dot Y(s)$ yields \eqref{eq:disc_HRF_blocks}, which completes the proof.
\end{proof}

Next, we establish the existence of the HRF dynamics and show that it preserves positivity, mass, and the slice-wise means of the value function.
\begin{proof}[Proof of \Cref{prop:global_exist_pos_mass}]
To invoke standard Euclidean ODE theory rigorously, we first view \eqref{eq:disc_HRF_blocks} as an ODE on the open set
\[
\Omega_h:=\Bigl\{Y\in\mathcal X_h:\ M_k\in\mathbb R_{++}^{N_x}\  \text{for }k=1,\dots,N_T\Bigr\}.
\]
On $\Omega_h$, the quantities $\bar r_k(Y)$ in \eqref{eq:rbar_def} are well defined, and Assumptions~\ref{ass:disc_f} and \ref{ass:disc_terms_regularity} imply that
\[
Y\mapsto (r_1(Y),\dots,r_{N_T}(Y),s_1(Y),\dots,s_{N_T}(Y))
\]
is locally Lipschitz. Hence the vector field in \eqref{eq:disc_HRF_blocks} is locally Lipschitz on $\Omega_h$, so the Picard--Lindel\"of theorem yields a unique maximal solution $Y(\cdot)$ on some interval $[0,S_{\max})$, with $S_{\max}>0$, starting from the feasible initial datum $Y(0)\in\mathcal M_h\subset\Omega_h$. The identities proved below then show that this maximal solution actually remains on $\mathcal M_h$.

Along the solution, the mass conservation follows directly from \eqref{eq:disc_HRF_blocks}. Indeed, using $\langle a\odot b,\mathbf 1\rangle_h=\langle a,b\rangle_h$ and the definition \eqref{eq:rbar_def}, for each $k$ and all $s\in[0,S_{\max})$,
\[
\frac{\mathrm d}{\mathrm ds}\langle M_k(s),\mathbf 1\rangle_h
=
-\bigl\langle M_k(s),\,r_k(Y(s))-\bar r_k(Y(s))\,\mathbf 1\bigr\rangle_h
=
-\langle M_k(s),r_k(Y(s))\rangle_h+\bar r_k(Y(s))\langle M_k(s),\mathbf 1\rangle_h
=0. 
\]
Hence
\[
\langle M_k(s),\mathbf 1\rangle_h=\langle M_k(0),\mathbf 1\rangle_h=1
\qquad\text{for all }s\in[0,S_{\max}).
\]

Next, we prove invariance of the slice-wise means of the $U$-blocks. For each $k$ and all $s\in[0,S_{\max})$,
\[
\frac{\mathrm d}{\mathrm ds}\langle U_{k-1}(s),\mathbf 1\rangle_h
=
-\langle s_k(Y(s)),\mathbf 1\rangle_h.
\]
Using \eqref{eq:sk_def_rewrite}, mass preservation, the self-adjointness of $\Delta_h$, and the divergence form of $\mathcal B_h(U_{k-1},M_k)$, we obtain
\[
\langle s_k(Y(s)),\mathbf 1\rangle_h
=
\frac{\langle M_k(s)-M_{k-1}(s),\mathbf 1\rangle_h}{\Delta t}
-\nu\langle \Delta_h M_k(s),\mathbf 1\rangle_h
+\langle \mathcal B_h(U_{k-1}(s),M_k(s)),\mathbf 1\rangle_h
=0.
\]
Hence
\[
\langle U_{k-1}(s),\mathbf 1\rangle_h=\langle U_{k-1}(0),\mathbf 1\rangle_h
\qquad\text{for all }s\in[0,S_{\max}).
\]

Positivity follows from the multiplicative structure of the density equation. Fix $k$ and a component index $\alpha$. Writing
\[
B_k(Y):=r_k(Y)-\bar r_k(Y)\mathbf 1,
\]
the $k$-th density block satisfies
\[
\dot M_k=-M_k\odot B_k(Y).
\]
Thus, 
\[
\frac{\mathrm d}{\mathrm ds}\log (M_k(s))_\alpha
=
-\,(B_k(Y(s)))_\alpha,
\]
and 
\[
(M_k(s))_\alpha
=
(M_k(0))_\alpha
\exp\!\Bigl(-\int_0^s (B_k(Y(\sigma)))_\alpha\,\mathrm d\sigma\Bigr)>0
\qquad\text{for all }s\in[0,S_{\max}).
\]
Hence, $M_k(s)$ remains strictly positive componentwise on $[0,S_{\max})$, and together with the mass constraint we have
\[
M_k(s)\in\mathcal M_{++,h}
\qquad\text{for all }s\in[0,S_{\max}).
\]

It remains to show that $S_{\max}=+\infty$. Let $\mathcal E$ be the discrete entropy \eqref{eq:disc_entropy_h}. We interpret $\nabla\mathcal E$ as the gradient with respect to the space-time pairing $\langle\cdot,\cdot\rangle_{h,\Delta t}$, i.e., $
D\mathcal E(Y)[\Xi]=\langle \nabla\mathcal E(Y),\Xi\rangle_{h,\Delta t}$,
and recall the definition of the Bregman divergence $D_{\mathcal E}$.

Fix $s\in[0,S_{\max})$. Since both $Y(s)$ and $Y^\dagger$ belong to $\mathcal M_h$, each density block of $Y^\dagger-Y(s)$ has zero mass. Hence $\delta:=Y^\dagger-Y(s)$ belongs to $\mathcal T_{Y(s)}\mathcal M_h$. Because $\dot Y(s)$ minimizes the strictly convex functional in \eqref{eq:disc_HRF_var} over the linear space $\mathcal T_{Y(s)}\mathcal M_h$, its first-order optimality condition reads
\[
g_{Y(s),\mathcal E}\bigl(\dot Y(s),\delta\bigr)
+
\bigl\langle F_h^{\mathrm{td}}(Y(s)),\delta\bigr\rangle_{h,\Delta t}=0
\qquad\text{for every }\delta\in\mathcal T_{Y(s)}\mathcal M_h.
\]
Applying this with $\delta=Y^\dagger-Y(s)$ gives
\[
g_{Y(s),\mathcal E}\bigl(\dot Y(s),\,Y^\dagger-Y(s)\bigr)
=
-\bigl\langle F_h^{\mathrm{td}}(Y(s)),\,Y^\dagger-Y(s)\bigr\rangle_{h,\Delta t}.
\]
On the other hand, by the chain rule and the definition of $D_{\mathcal E}$,
\[
\frac{\mathrm d}{\mathrm ds}D_{\mathcal E}(Y^\dagger,Y(s))
=
-\bigl\langle \nabla^2\mathcal E(Y(s))\,\dot Y(s),\,Y^\dagger-Y(s)\bigr\rangle_{h,\Delta t}
=
-\,g_{Y(s),\mathcal E}\bigl(\dot Y(s),\,Y^\dagger-Y(s)\bigr).
\]
Combining the two identities yields
\[
\frac{\mathrm d}{\mathrm ds}D_{\mathcal E}(Y^\dagger,Y(s))
=
\bigl\langle F_h^{\mathrm{td}}(Y(s)),\,Y^\dagger-Y(s)\bigr\rangle_{h,\Delta t}.
\]
Since $\widetilde F_h^{\mathrm{td}}(Y^\dagger)=0$, each density block of $F_h^{\mathrm{td}}(Y^\dagger)$ is a multiple of $\mathbf 1$ and each $U$-block vanishes. Because every density block in $Y(s)-Y^\dagger$ has zero mass, we have
\[
\bigl\langle F_h^{\mathrm{td}}(Y^\dagger),\,Y(s)-Y^\dagger\bigr\rangle_{h,\Delta t}=0.
\]
Therefore
\[
\frac{\mathrm d}{\mathrm ds}D_{\mathcal E}(Y^\dagger,Y(s))
=
-\bigl\langle F_h^{\mathrm{td}}(Y(s))-F_h^{\mathrm{td}}(Y^\dagger),\,Y(s)-Y^\dagger\bigr\rangle_{h,\Delta t}.
\]
According to the arguments of \Cref{prop:Fh_monotone_MU}, we obtain 
\[
\bigl\langle F_h^{\mathrm{td}}(Y(s))-F_h^{\mathrm{td}}(Y^\dagger),\,Y(s)-Y^\dagger\bigr\rangle_{h,\Delta t}\ge 0,
\]
and hence
\[
\frac{\mathrm d}{\mathrm ds}D_{\mathcal E}(Y^\dagger,Y(s))\le 0.
\]
Therefore $s\mapsto D_{\mathcal E}(Y^\dagger,Y(s))$ is nonincreasing on $[0,S_{\max})$.

This Lyapunov bound prevents the solution from approaching the boundary of the feasible interior in finite flow time. First, the quadratic $U$-part of the Bregman divergence is explicit:
\[
D_{\mathcal E}(Y^\dagger,Y(s))
\ge
\frac{\Delta t}{2}\sum_{k=1}^{N_T}\|U_{k-1}(s)-U_{k-1}^\dagger\|_h^2.
\]
Hence, the uniform bound
$D_{\mathcal E}(Y^\dagger,Y(s))\le D_{\mathcal E}(Y^\dagger,Y(0))$
for all $s\in[0,S_{\max})$ yields a global bound on every $U$-block along the maximal trajectory. For the density blocks, note that $D_{\mathcal E}(Y^\dagger,Y)$ contains the discrete relative-entropy terms
\[
\Delta t\sum_{k=1}^{N_T}\sum_{\alpha} h^2\Bigl[(M_k^\dagger)_\alpha\log\!\frac{(M_k^\dagger)_\alpha}{(M_k)_\alpha}-\bigl((M_k^\dagger)_\alpha-(M_k)_\alpha\bigr)\Bigr].
\]
Since $M_k^\dagger>0$ componentwise, if some component $(M_k)_\alpha$ were to tend to $0$, the corresponding term $
(M_k^\dagger)_\alpha\log\!\frac{(M_k^\dagger)_\alpha}{(M_k)_\alpha}$
would diverge to $+\infty$, contradicting the bound $
D_{\mathcal E}(Y^\dagger,Y(s))\le D_{\mathcal E}(Y^\dagger,Y(0))$.
Consequently, because the bound
$D_{\mathcal E}(Y^\dagger,Y(s))\le D_{\mathcal E}(Y^\dagger,Y(0))$
holds for every $s\in[0,S_{\max})$, there exists a constant $c_*>0$, depending only on $Y^\dagger$ and $D_{\mathcal E}(Y^\dagger,Y(0))$, such that
$(M_k(s))_\alpha\ge c_*$
for all $s\in[0,S_{\max})$, all $k$, and all components. Together with mass preservation, which yields the pointwise upper bound
$(M_k(s))_\alpha\le h^{-2}$,
and with the global bound on the $U$-blocks established above, this shows that the whole trajectory $\{Y(s):0\le s<S_{\max}\}$ is contained in a compact subset of the feasible interior.

Because the vector field in \eqref{eq:disc_HRF_blocks} is locally Lipschitz on the open set $\Omega_h$, the standard continuation criterion for ODEs in finite dimension implies that a maximal solution can be continued past any finite time as long as its image stays in a compact subset of $\Omega_h$. The previous paragraph provides such a compact subset for the entire trajectory on $[0,S_{\max})$. Hence $S_{\max}=+\infty$. The positivity and mass identities established above therefore, hold for all $s\ge 0$.
\end{proof}


The following proof establishes global convergence of the HRF on each invariant affine slice.

\begin{proof}[Proof of \Cref{thm:global_conv_discrete}]
We first explain the idea of the proof. The Bregman divergence from
$Y^\dagger$ decreases along the HRF and is bounded from below. Hence, its
decrease rate is nonnegative and integrable in artificial time. We then prove
that the trajectory stays in a compact subset of the feasible region. This gives
enough continuity to show that the decrease
rate tends to zero. Any limit point of the trajectory must therefore make the
limiting decrease rate equal to zero. By strict monotonicity, the only possible
limit point with the same averages of the value-function slices is
$Y^\dagger$. Thus, the trajectory converges to $Y^\dagger$, and the residuals
converge to zero by continuity.

\smallskip
\noindent\textbf{Step 1: The flow stays feasible and keeps the averages of
$U_k$.}
By \Cref{prop:global_exist_pos_mass}, the solution $Y(s)$ of
\eqref{eq:disc_HRF_blocks} exists for all $s\ge 0$ and remains feasible. Thus,
for every $s\ge 0$ and $k=1,\dots,N_T$, we have
$M_k(s)\in\mathcal M_{++,h}$ and $\langle M_k(s),\mathbf 1\rangle_h=1$.
Moreover, the averages of the value-function slices do not change:
\[
\langle U_{k-1}(s),\mathbf 1\rangle_h
=
\langle U_{k-1}(0),\mathbf 1\rangle_h
=
\langle U^\dagger_{k-1},\mathbf 1\rangle_h .
\]
Here, the last equality is part of the assumption on $Y^\dagger$.

\smallskip
\noindent\textbf{Step 2: The Bregman divergence is decreasing.}
As shown in the proof of \Cref{prop:global_exist_pos_mass}, along the HRF one has
\[
\frac{\mathrm d}{\mathrm ds}D_{\mathcal E}(Y^\dagger,Y(s))
=
-
\bigl\langle
\widetilde F_h^{\mathrm{td}}(Y(s))
-\widetilde F_h^{\mathrm{td}}(Y^\dagger),
Y(s)-Y^\dagger
\bigr\rangle_{h,\Delta t}.
\]
Set
\[
\Phi(s):=
\bigl\langle
\widetilde F_h^{\mathrm{td}}(Y(s))
-\widetilde F_h^{\mathrm{td}}(Y^\dagger),
Y(s)-Y^\dagger
\bigr\rangle_{h,\Delta t}.
\]
Then $\frac{\mathrm d}{\mathrm ds}D_{\mathcal E}(Y^\dagger,Y(s))=-\Phi(s)$.

We also have $\Phi(s)\ge 0$. Indeed, $Y(s)$ and $Y^\dagger$ have the same
averages of $U_k$, and every density block has mass one. Therefore the constants
removed in $\widetilde F_h^{\mathrm{td}}$ do not contribute to the pairing with
$Y(s)-Y^\dagger$. Hence
\[
\Phi(s)
=
\bigl\langle
F_h^{\mathrm{td}}(Y(s))-F_h^{\mathrm{td}}(Y^\dagger),
Y(s)-Y^\dagger
\bigr\rangle_{h,\Delta t}.
\]
The nonnegativity follows from \Cref{prop:Fh_monotone_MU}. Thus
$D_{\mathcal E}(Y^\dagger,Y(s))$ is nonincreasing.

\smallskip
\noindent\textbf{Step 3: The trajectory stays in a compact subset of the
feasible region.}
Since $D_{\mathcal E}(Y^\dagger,Y(s))\le D_{\mathcal E}(Y^\dagger,Y(0))$ for all
$s\ge 0$, and since the $U$-part of $D_{\mathcal E}$ is
$(\Delta t/2)\sum_{k=0}^{N_T-1}\|U_k(s)-U_k^\dagger\|_h^2$, all
value-function blocks $U_k(s)$ are uniformly bounded. For the density blocks, the entropy part of $D_{\mathcal E}$ is
a finite sum of terms of the form
\[
a\log\frac{a}{m}-(a-m),\qquad a>0,\quad m>0.
\]
This function is nonnegative and tends to $+\infty$ as $m\to0^+$. Therefore the
bound on $D_{\mathcal E}(Y^\dagger,Y(s))$ prevents any component of $M_k(s)$
from approaching zero. The mass constraint $\langle M_k(s),\mathbf 1\rangle_h=1$
also gives an upper bound on every component of $M_k(s)$. Hence, all components
of $Y(s)$ are uniformly bounded, and the density stays a positive distance away
from zero. Since the problem is finite-dimensional, the trajectory is contained
in a compact subset of the feasible region.

\smallskip
\noindent\textbf{Step 4: The decrease rate $\Phi(s)$ tends to zero.}
Integrating the identity
$\frac{\mathrm d}{\mathrm ds}D_{\mathcal E}(Y^\dagger,Y(s))=-\Phi(s)$ gives
\[
\int_0^\infty \Phi(s)\,\mathrm ds
=
D_{\mathcal E}(Y^\dagger,Y(0))
-\lim_{s\to\infty}D_{\mathcal E}(Y^\dagger,Y(s))
<\infty .
\]
Thus $\Phi\in L^1(0,\infty)$. On the compact set found in Step 3, the right-hand
side of \eqref{eq:disc_HRF_blocks} is continuous and bounded. Since the trajectory stays in the compact subset of the feasible region
constructed in the previous step, and since the right-hand side of
\eqref{eq:disc_HRF_blocks} is continuous there, the right-hand side is bounded
along the trajectory. Hence there exists $C>0$ such that
$\|\dot Y(s)\|\le C$ for all $s\ge 0$. Therefore, for any $s,t\ge 0$, we have
$\|Y(s)-Y(t)\|\le C|s-t|$, so $Y(\cdot)$ is uniformly continuous on
$[0,\infty)$.

Moreover, all density components stay uniformly away from zero on this compact
set. Hence $\widetilde F_h^{\mathrm{td}}$ is continuous on a neighborhood of the
trajectory. Therefore the map
$Y\mapsto \langle \widetilde F_h^{\mathrm{td}}(Y)-\widetilde F_h^{\mathrm{td}}(Y^\dagger),
Y-Y^\dagger\rangle_{h,\Delta t}$ is uniformly continuous on the compact subset
containing the trajectory. Since $Y(\cdot)$ is uniformly continuous, the
composition
$s\mapsto \langle \widetilde F_h^{\mathrm{td}}(Y(s))-\widetilde F_h^{\mathrm{td}}(Y^\dagger),
Y(s)-Y^\dagger\rangle_{h,\Delta t}$
is uniformly continuous on $[0,\infty)$. That is, $\Phi$ is uniformly
continuous. Since $\Phi\ge0$,
$\Phi\in L^1(0,\infty)$, and $\Phi$ is uniformly continuous, Barbalat's lemma
\cite{khalil2002nonlinear} implies that $\Phi(s)\to0$ as $s\to\infty$.

\smallskip
\noindent\textbf{Step 5: The trajectory converges to $Y^\dagger$.}
We have proved that $\Phi(s)\to 0$. We now use this to prove that
$Y(s)\to Y^\dagger$. Since the trajectory stays in a compact
set, every sequence $Y(s_n)$ with $s_n\to\infty$ has a convergent subsequence.
Let $\bar Y$ be the limit of such a subsequence. At this point, $\bar Y$ is only
a possible limit point of the trajectory. We will show that it must be
$Y^\dagger$.

The
quantities preserved by the flow pass to the limit. In particular,
\[
\langle \bar U_{k-1},\mathbf 1\rangle_h
=
\langle U^\dagger_{k-1},\mathbf 1\rangle_h,
\qquad k=1,\dots,N_T .
\]
Since $\Phi(s)\to0$ and $\Phi$ is continuous along convergent sequences, we get
\[
0=
\bigl\langle
\widetilde F_h^{\mathrm{td}}(\bar Y)
-\widetilde F_h^{\mathrm{td}}(Y^\dagger),
\bar Y-Y^\dagger
\bigr\rangle_{h,\Delta t}.
\]
The constants removed in $\widetilde F_h^{\mathrm{td}}$ do not affect this
pairing. Indeed, the density blocks of $\bar Y-Y^\dagger$ have zero mass, and
the value-function blocks have zero average. Hence, the same identity holds with
$F_h^{\mathrm{td}}$ in place of $\widetilde F_h^{\mathrm{td}}$:
\[
0=
\bigl\langle
F_h^{\mathrm{td}}(\bar Y)-F_h^{\mathrm{td}}(Y^\dagger),
\bar Y-Y^\dagger
\bigr\rangle_{h,\Delta t}.
\]
By the strict monotonicity in \Cref{prop:Fh_monotone_MU}, this is possible only
if $\bar Y=Y^\dagger$.

Thus, every convergent subsequence of $Y(s)$ has the same limit $Y^\dagger$.
Because the trajectory is precompact, this implies that the whole trajectory
converges:
\[
Y(s)\to Y^\dagger
\qquad\text{as }s\to\infty.
\]

\smallskip
\noindent\textbf{Step 6: The residuals converge to zero.}
Now, the convergence of the residuals follows directly from the convergence of
$Y(s)$. Since $\widetilde F_h^{\mathrm{td}}$ is continuous on the compact subset
containing the trajectory and $\widetilde F_h^{\mathrm{td}}(Y^\dagger)=0$, we
obtain
\[
\widetilde F_h^{\mathrm{td}}(Y(s))\to0.
\]
This completes the proof.
\end{proof}

Now, we are ready to compute the gradient of the inverse objective with respect to the unknown coefficients.

\begin{proof}[Proof of \Cref{prop:adjoint_grad}]
Fix a parameter $\theta\in\Theta$ and an arbitrary direction
$\delta\theta\in\mathbb R^\ell$. Since the inner state is differentiable with
respect to $\theta$ at this parameter, we denote its first-order change by
\[
\delta z := Dz^*(\theta)\,\delta\theta .
\]

\smallskip
\noindent\textbf{Step 1: Differentiate the reduced objective.}
By the definition of $\mathcal J$ in \eqref{eq:reduced-objective}, we have
\begin{equation}
\label{eq:dJ_app}
\begin{aligned}
\mathrm d\mathcal J(\theta)[\delta\theta]
=
\beta\,\bigl(Wz^*(\theta)-z^{\mathrm o}\bigr)^\top W\,\delta z
+
\gamma\,\bigl(G\theta-\theta^{\mathrm o}\bigr)^\top G\,\delta\theta
+
\alpha\,\bigl(J\theta\bigr)^\top J\,\delta\theta .
\end{aligned}
\end{equation}
The first term still contains $\delta z$, which depends on the change of the
inner solution.

\smallskip
\noindent\textbf{Step 2: Differentiate the inner equations.}
Since
\[
F(z^*(\theta),\theta)=0,
\qquad
\mathsf A z^*(\theta)=\mathsf b ,
\]
differentiating these equations in the direction $\delta\theta$ gives
\begin{equation}
\label{eq:lin-eq_app}
\frac{\partial F}{\partial z}\bigl(z^*(\theta),\theta\bigr)\,\delta z
+
\frac{\partial F}{\partial \theta}\bigl(z^*(\theta),\theta\bigr)\,\delta\theta
=0,
\end{equation}
and
\begin{equation}
\label{eq:lin-constraint_app}
\mathsf A\,\delta z=0.
\end{equation}
If there are no separate linear constraints, then \eqref{eq:lin-constraint_app}
is absent and the terms involving $\eta$ below are omitted.

\smallskip
\noindent\textbf{Step 3: Reformulation of the gradient.}
Let $(\psi,\eta)$ solve the adjoint equation \eqref{eq:adjoint}. Since
\eqref{eq:lin-eq_app} and \eqref{eq:lin-constraint_app} are equal to zero, we may
add their pairings with $\psi$ and $\eta$ to \eqref{eq:dJ_app}:
\[
0
=
\psi^\top
\left[
\frac{\partial F}{\partial z}\bigl(z^*(\theta),\theta\bigr)\,\delta z
+
\frac{\partial F}{\partial \theta}\bigl(z^*(\theta),\theta\bigr)\,\delta\theta
\right],
\qquad
0=\eta^\top \mathsf A\,\delta z .
\]
Therefore
\begin{align}
\mathrm d\mathcal J(\theta)[\delta\theta]
&=
\Bigl[
\beta\,W^\top\bigl(Wz^*(\theta)-z^{\mathrm o}\bigr)
+
\frac{\partial F}{\partial z}\bigl(z^*(\theta),\theta\bigr)^\top \psi
+
\mathsf A^\top\eta
\Bigr]^\top \delta z
\label{eq:dJ_grouped_app}
\\
&\quad+
\Bigl[
\frac{\partial F}{\partial \theta}\bigl(z^*(\theta),\theta\bigr)^\top \psi
+
\gamma\,G^\top\bigl(G\theta-\theta^{\mathrm o}\bigr)
+
\alpha\,J^\top J\,\theta
\Bigr]^\top \delta\theta .
\nonumber
\end{align}

\smallskip
\noindent\textbf{Step 4: Use the adjoint equation to remove $\delta z$.}
By the definition of $(\psi,\eta)$ in \eqref{eq:adjoint},
\[
\frac{\partial F}{\partial z}\bigl(z^*(\theta),\theta\bigr)^\top \psi
+
\mathsf A^\top\eta
=
-\beta\,W^\top\bigl(Wz^*(\theta)-z^{\mathrm o}\bigr).
\]
Hence, the coefficient of $\delta z$ in \eqref{eq:dJ_grouped_app} is zero. Thus
\[
\mathrm d\mathcal J(\theta)[\delta\theta]
=
\Bigl[
\frac{\partial F}{\partial \theta}\bigl(z^*(\theta),\theta\bigr)^\top \psi
+
\gamma\,G^\top\bigl(G\theta-\theta^{\mathrm o}\bigr)
+
\alpha\,J^\top J\,\theta
\Bigr]^\top \delta\theta .
\]
Since $\delta\theta$ is arbitrary, the gradient is
\[
\nabla_\theta \mathcal J(\theta)
=
\frac{\partial F}{\partial \theta}\bigl(z^*(\theta),\theta\bigr)^\top \psi
+
\gamma\,G^\top\bigl(G\theta-\theta^{\mathrm o}\bigr)
+
\alpha\,J^\top J\,\theta .
\]
This is exactly \eqref{eq:grad-final}. It remains to justify that the formula is independent of the chosen adjoint
solution. Let \((\psi_1,\eta_1)\) and \((\psi_2,\eta_2)\) be two adjoint
solutions, and set \(\zeta=\psi_1-\psi_2\) and \(\xi=\eta_1-\eta_2\). Then
\(F_z(z^*(\theta),\theta)^\top\zeta+\mathsf A^\top\xi=0\). For any parameter
direction \(d\), let \(\delta z_d\) be the corresponding first-order change of
the state. It satisfies
\(F_z(z^*(\theta),\theta)\delta z_d+F_\theta(z^*(\theta),\theta)d=0\) and
\(\mathsf A\delta z_d=0\). Therefore
\(\zeta^\top F_\theta(z^*(\theta),\theta)d
=-\zeta^\top F_z(z^*(\theta),\theta)\delta z_d
=\xi^\top\mathsf A\delta z_d=0\). Since this holds for every \(d\), we have
\(F_\theta(z^*(\theta),\theta)^\top\psi_1
=F_\theta(z^*(\theta),\theta)^\top\psi_2\). Thus, the adjoint contribution to
the gradient does not depend on which adjoint solution is used.
\end{proof}

Finally, we derive the update formulas for the GN method.
\begin{proof}[Proof of \Cref{prop:gn_direction}]
The GN method is applied to the least-squares objective
\[
\mathcal J(\theta)=\frac12\|r(\theta)\|_2^2 .
\]
At the current parameter $\theta$, it uses the first-order approximation
\[
r(\theta+d)\approx r(\theta)+\mathsf J_r(\theta)d .
\]
We therefore compute the first-order change of each block of $r(\theta)$ in the
direction $d$.

\smallskip
\noindent\textbf{Step 1: Differentiate the inner equations.}
The first block of $r(\theta)$ depends on $\theta$ through the inner state
$z^*(\theta)$. Let $\delta z_d$ denote the first-order change of this state in
the direction $d$. When $z^*(\theta)$ is differentiable with respect to
$\theta$, this means
\[
\delta z_d = Dz^*(\theta)d,
\qquad
z^*(\theta+\varepsilon d)
=
z^*(\theta)+\varepsilon \delta z_d+o(\varepsilon)
\quad\text{as }\varepsilon\to0.
\]
Since the inner state satisfies
\[
F(z^*(\theta),\theta)=0,
\qquad
\mathsf A z^*(\theta)=\mathsf b,
\]
differentiating these equations in the direction $d$ gives
\[
\frac{\partial F}{\partial z}\bigl(z^*(\theta),\theta\bigr)\,\delta z_d
+
\frac{\partial F}{\partial \theta}\bigl(z^*(\theta),\theta\bigr)d
=0,
\qquad
\mathsf A\,\delta z_d=0.
\]
This is exactly \eqref{eq:gn_sensitivity}.

\smallskip
\noindent\textbf{Step 2: Differentiate the least-squares residual.}
By \eqref{eq:gn-residual-def}, the residual is
\[
r(\theta)
=
\begin{bmatrix}
\sqrt{\beta}\,\bigl(Wz^*(\theta)-z^{\mathrm o}\bigr)\\[2pt]
\sqrt{\gamma}\,\bigl(G\theta-\theta^{\mathrm o}\bigr)\\[2pt]
\sqrt{\alpha}\,J\theta
\end{bmatrix}.
\]
Therefore, the first-order change of
$r(\theta)$ in the direction $d$ is
\[
\mathsf J_r(\theta)d=
\begin{bmatrix}
\sqrt{\beta}\,W\,\delta z_d\\
\sqrt{\gamma}\,Gd\\
\sqrt{\alpha}\,Jd
\end{bmatrix}.
\]

\smallskip
\noindent\textbf{Step 3: Write the GN step.}
The GN step minimizes the linearized least-squares problem
\[
\frac12\bigl\|r(\theta)+\mathsf J_r(\theta)d\bigr\|_2^2 .
\]
Differentiating this quadratic function with respect to $d$ gives
\[
\mathsf J_r(\theta)^\top\bigl(r(\theta)+\mathsf J_r(\theta)d\bigr)=0.
\]
Equivalently,
\[
\bigl(\mathsf J_r(\theta)^\top\mathsf J_r(\theta)\bigr)d
=
-\mathsf J_r(\theta)^\top r(\theta),
\]
which is \eqref{eq:gn_lm_system}.
\end{proof}

\section*{Declarations}
\noindent\textbf{Conflict of interest.}
The authors declare that they have no conflict of interest.

\medskip
\noindent\textbf{Data availability.}
The authors confirm that the data supporting the conclusions of this study are contained in the article.

\bibliographystyle{plain}
\bibliography{referencesnewnew}

\end{document}